\documentclass[a4paper,twosided,11pt]{article}
\usepackage{amsfonts}
\usepackage{latexsym}
\usepackage{amsmath}
\usepackage{amssymb}
\usepackage{amsthm}
\usepackage{amsmath}
\usepackage[shortlabels]{enumitem}
\usepackage{array}
\usepackage{hyperref}
\usepackage{verbatim}
\usepackage{bibspacing}
\usepackage{tikz}
\usepackage{tikz-cd}
\usepackage{enumitem} \usetikzlibrary{calc}

\tikzset{vertex/.style = {shape=circle,fill,inner sep=0pt,minimum size=2mm}}

\newtheorem{theorem}{Theorem}[section]
\newtheorem*{theorem*}{Theorem}

\newtheorem{corollary}[theorem]{Corollary}

\newtheorem{lem}[theorem]{Lemma}
\newtheorem{lemma}[theorem]{Lemma}

\newtheorem*{prop*}{Proposition}
\newtheorem{proposition}[theorem]{Proposition}
\newtheorem{conjecture}[theorem]{Conjecture}
\newtheorem{question}[theorem]{Question}
\newtheorem*{conjecture*}{Conjecture}
\newtheorem{definition}[theorem]{Definition}
\newtheorem*{definition*}{Definition}
\theoremstyle{definition}
\newtheorem{remark}[theorem]{\textbf{Remark}}
\newtheorem*{remark*}{Remark}
\newtheorem*{fact*}{Fact}

\newcommand{\vol}{\textrm{Vol}}
\newcommand{\norm}[1]{\left\Vert#1\right\Vert}
\newcommand{\snorm}[1]{\Vert#1\Vert}
\newcommand{\abs}[1]{\left\vert#1\right\vert}
\newcommand{\set}[1]{\left\{#1\right\}}
\newcommand{\brac}[1]{\left(#1\right)}
\newcommand{\scalar}[1]{\left \langle #1 \right \rangle}
\newcommand{\sscalar}[1]{\langle #1 \rangle}
\newcommand{\Y}{\mathbf{Y}}
\newcommand{\T}{\mathbf{T}}

\newcommand{\R}{\mathbb{R}}
\newcommand{\GG}{\mathbb{G}}

\newcommand{\G}{\mathcal{G}}

\newcommand{\N}{\mathbf{N}}
\newcommand{\Complex}{\mathbb{C}}

\newcommand{\I}{\mathcal{I}}
\newcommand{\II}{\mathrm{I\!I}}
\newcommand{\Id}{\mathrm{Id}}
\newcommand{\sym}{\text{sym}}
\newcommand{\MM}{\mathbb{M}}
\newcommand{\M}{\mathbf{M}}

\renewcommand{\H}{\mathcal{H}}
\newcommand{\HH}{\bar{\mathcal{H}}}
\newcommand{\eps}{\epsilon}

\renewcommand{\SS}{\mathbf{S}}

\newcommand{\Ric}{\text{\rm Ric}}

\newcommand{\tr}{\text{\rm tr}}
\newcommand{\D}{\mathcal{D}}
\renewcommand{\S}{\mathbb{S}}
\newcommand{\B}{\mathbb{B}}
\renewcommand{\div}{\text{\rm div}}

\newcommand{\n}{\mathfrak{n}}
\renewcommand{\c}{\mathbf{c}}
\newcommand{\C}{\mathbf{C}}
\newcommand{\F}{\mathbf{F}}

\renewcommand{\k}{\mathbf{k}}

\renewcommand{\P}{\mathbf{P}}
\newcommand{\tang}{\mathbf{t}}
\newcommand{\cyclic}{\mathcal{C}}
\newcommand{\Tr}{\mathcal{T}}
\newcommand{\NN}{\mathcal{N}}

\newcommand{\simplex}{\Delta}
\newcommand{\sspan}{\text{span}}
\renewcommand{\L}{\text{L}}
\newcommand{\per}{A}
\newcommand{\pot}{W}

\let\Im\relax
\DeclareMathOperator{\Im}{Im}
\DeclareMathOperator{\arank}{affine-rank}

\DeclareMathOperator*{\argmin}{arg\,min}
\DeclareMathOperator{\interior}{int}

\DeclareMathOperator{\Jac}{Jac}

\newlength{\defbaselineskip}
\numberwithin{equation}{section}

\begin{document}

\title{Plateau Bubbles and the Quintuple Bubble Theorem on $\S^n$} 
\date{}

\author{Emanuel Milman\textsuperscript{1} and Joe Neeman\textsuperscript{2}}

\footnotetext[1]{Department of Mathematics, Technion - Israel
Institute of Technology, Haifa, Israel. Email: emilman@tx.technion.ac.il.}
\footnotetext[2]{Department of Mathematics, University of Texas at Austin. Email: joeneeman@gmail.com.\\
The research leading to these results is part of a project that has received funding from the European Research Council (ERC) under the European Union's Horizon 2020 research and innovation programme (grant agreement No 101001677). 
This material is also based upon work supported by the National Science Foundation under Grant Nos. 2145800 and 2204449. Any opinions, findings, and conclusions or recommendations expressed in this material are those of the authors and do not necessarily reflect the views of the National Science Foundation.}

\begingroup    \renewcommand{\thefootnote}{}    
\endgroup

\maketitle

\begin{abstract}
Sullivan's multi-bubble isoperimetric conjectures in $n$-dimensional Euclidean and spherical spaces from the 1990's assert that standard bubbles uniquely minimize total perimeter among all $q-1$ bubbles enclosing prescribed volume, for any $q \leq n+2$. The double-bubble conjecture on $\R^3$ was confirmed by Hutchings--Morgan--Ritor\'e--Ros (and later extended to $\R^n$). The double-bubble conjecture on $\S^n$ ($n \geq 2$) and the triple- and quadruple- bubble conjectures on $\R^n$ and $\S^n$ (for $n \geq 3$ and $n \geq 4$, respectively) were recently confirmed in our previous work, but the approach employed there does not seem to allow extending these results further. 

In this work, we confirm the quintuple-bubble conjecture on $\S^n$ ($n \geq 5$), and as a consequence, by approximation, also the quintuple-bubble conjecture on $\R^n$ ($n \geq 5$) but without the uniqueness assertion. Moreover, we resolve the conjectures on $\S^n$ and on $\R^n$ (without uniqueness) for all $q \leq n+1$, \emph{conditioned} on the assumption that the singularities which appear at the meeting locus of several bubbles obey a higher-dimensional analogue of Plateau's laws. Another scenario we can deal with is when the bubbles are full-dimensional (``in general position"), or arrange in some good lower-dimensional configurations. 
We are able to verify that a minimizing configuration must satisfy one of the above good scenarios only when $q \leq 6$, yielding new proofs for the previously established double, triple and quadruple cases, as well as the new quintuple case. 

To this end, we develop the spectral theory of the corresponding Jacobi operator (finding analogies with the quantum-graph formalism), and a new method for deforming the bubbles into a favorable configuration. As a by-product of our analysis, we show that for all $q \leq n+1$, the Jacobi operator on a minimizing configuration always has index precisely $q-1$ and hence the corresponding isoperimetric profile is concave, answering a question of Heppes. Several compelling conjectures of PDE flavor are proposed, which would allow extending our results to all $q \leq n+1$ unconditionally. 
\end{abstract}

\section{Introduction}

A weighted Riemannian manifold $(M^n,g,\mu)$ consists of a smooth complete $n$-dimensional Riemannian manifold $(M^n,g)$ endowed with a measure $\mu$ with $C^\infty$ smooth positive density $\Psi$ with respect to the Riemannian volume measure $\vol_g$. The metric $g$ induces a geodesic distance on $M^n$, and the corresponding $k$-dimensional Hausdorff measure is denoted by $\H^k$. Let $\mu^{k} := \Psi \H^{k}$, and set $V_\mu := \mu$, the $\mu$-weighted volume. The $\mu$-weighted perimeter of a Borel subset $U \subset M^n$ of locally finite perimeter is defined as $\per_\mu(U) := \mu^{n-1}(\partial^* U)$, where $\partial^* U$ is the reduced boundary of $U$ (see Section \ref{sec:prelim} for definitions).

A \emph{$q$-cluster} $\Omega = (\Omega_1, \ldots, \Omega_q)$ is a $q$-tuple of Borel subsets with locally finite perimeter $\Omega_i \subset M^n$ called cells, such that $\set{\Omega_i}$ are pairwise disjoint, $V_\mu(M^n \setminus \bigcup_{i=1}^q \Omega_i) = 0$, $\per_\mu(\Omega_i) < \infty$ for each $i$, and $V_\mu(\Omega_i) < \infty$ for all $i=1,\ldots,q-1$. Note that when $V_\mu(M^n) = \infty$, then necessarily $V_\mu(\Omega_q) = \infty$. Also note that the cells are not required to be connected. The $\mu$-weighted volume of $\Omega$ is defined as:
\[
  V_\mu(\Omega) := (V_\mu(\Omega_1), \ldots, V_\mu(\Omega_q)) \in \simplex^{(q-1)}_{V_\mu(M^n)} ,
\]
where $\simplex^{(q-1)}_T := \{v \in \R^{q}_+  \; ; \; \sum_{i=1}^q v_i = T\}$ if $T < \infty$ and 
$\simplex^{(q-1)}_\infty := \R^{q-1}_+ \times \{ \infty\}$. 
The $\mu$-weighted total perimeter of a cluster is defined as:
\[
\per_\mu(\Omega) := \frac{1}{2} \sum_{i=1}^q \per_\mu(\Omega_i) = \sum_{1 \leq i < j \leq q} \mu^{n-1}(\Sigma_{ij}) ,
\]
where $\Sigma_{ij} := \partial^* \Omega_i \cap \partial^* \Omega_j$ denotes the $(n-1)$-dimensional interface between cells $i$ and $j$.

The isoperimetric problem for $q$-clusters consists of identifying those clusters $\Omega$ of prescribed volume $\mu(\Omega) = v$ which minimize the total perimeter $\per_\mu(\Omega)$. Note that for a $2$-cluster $\Omega = (\Omega_1,\Omega_2)$, $\per_\mu(\Omega) = \per_\mu(\Omega_1) = \per_\mu(\Omega_2)$, and so the case $q=2$ corresponds to the classical isoperimetric setup of minimizing the perimeter of a single set of prescribed volume. Consequently, the case $q=2$ is referred to as the ``single-bubble" case (with the bubble being $\Omega_1$ and having complement $\Omega_2$). Accordingly, the case $q=3$ is called the ``double-bubble" problem, the case $q=4$ is the ``triple-bubble" problem, etc... The case of general $q$ is referred to as the ``multi-bubble" problem. 

\bigskip

In this work, we will focus on (arguably) the two most natural and important settings: 
\begin{itemize}
\item Euclidean space $\R^n$, namely $(\R^n,|\cdot|)$ endowed with the Lebesgue measure $\mu$. 
\item Spherical space $\S^n$, namely the unit-sphere $(\S^n,g)$ in its canonical embedding in $\R^{n+1}$, endowed with its Haar measure $\mu$ (normalized to have total mass $1$).  
\end{itemize}
In both cases we will abbreviate $V = V_\mu$ and $\per = \per_\mu$. 
The following definition and corresponding conjectures were put forth by J.~Sullivan in the 90's \cite[Problem 2]{OpenProblemsInSoapBubbles96}. 
Recall that the (open) Voronoi cells of distinct $\set{x_1,\ldots,x_q} \subset \R^N$ are defined as:
\[
\Omega_i = \interior \set{ x \in \R^N \;;\; i \in \argmin_{j=1,\ldots,q} \abs{x-x_j} } = \set{ x \in \R^N \;;\; \argmin_{j=1,\ldots,q} \abs{x-x_j} = \{i\} } ~,~  i=1,\ldots, q ~,
\]
where $\interior$ denotes the (relative) interior operation and $\argmin$ denotes the subset of indices on which the corresponding minimum is attained.

\begin{definition*}[Standard Bubble on $\R^n$ and $\S^n$]
Given $2 \leq q \leq n+2$, the equal-volume standard $(q-1)$-bubble on $\S^n$ is the $q$-cluster $\Omega$ obtained by intersecting $\S^n$ with the Voronoi cells of $q$ equidistant points in $\S^n \subset \R^{n+1}$. \\
A standard $(q-1)$-bubble on $\R^n$ is a stereographic projection of the equal-volume standard $(q-1)$-bubble $\Omega$ on $\S^n$ with respect to some North pole in (the open) $\Omega_q$. \\
A standard $(q-1)$-bubble on $\S^n$ is a stereographic projection of a standard $(q-1)$-bubble on $\R^n$. 
\end{definition*}

\begin{figure}
    \begin{center}
            \raisebox{-0.1\height}{\includegraphics[scale=0.1]{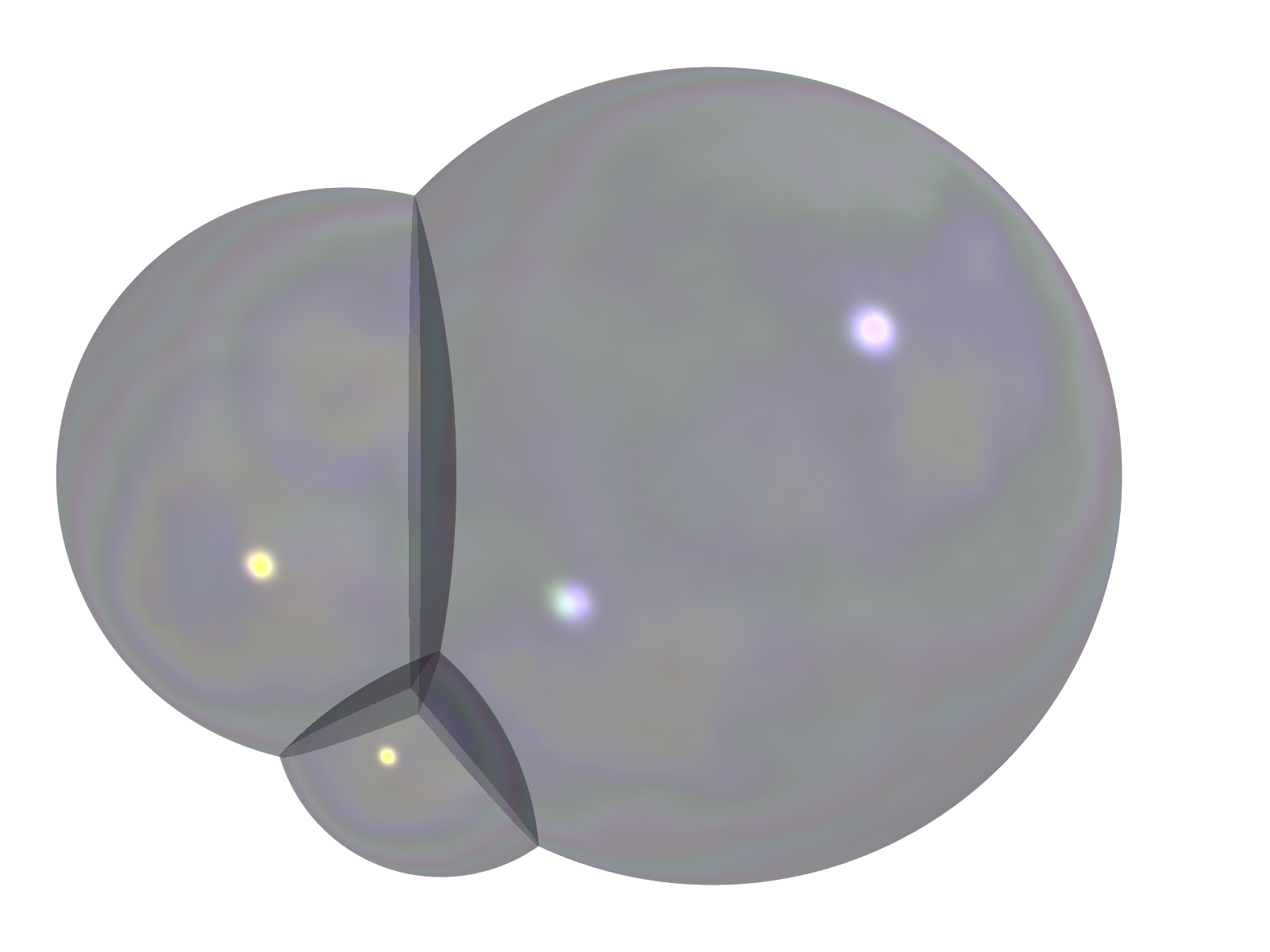}}
        \hspace{20pt}
        \begin{tikzpicture}[scale=1.2]
            \input{triple-bubble}
        \end{tikzpicture}
     \end{center}
     \caption{
         \label{fig:triple-bubble}
        Left: a standard triple-bubble in $\R^3$. Right: the $2$D cross-section through its plane of symmetry. 
     }
\end{figure}

\begin{figure}
    \begin{center}
        \includegraphics[scale=0.09]{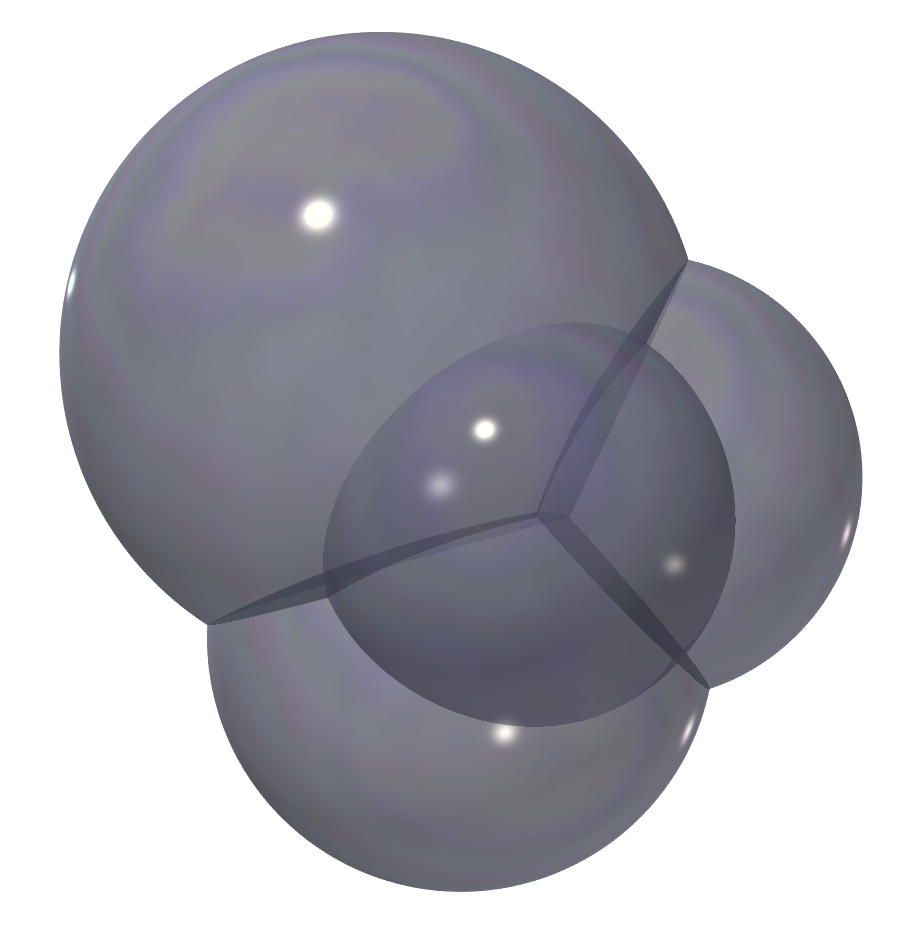}
        \includegraphics[scale=0.09]{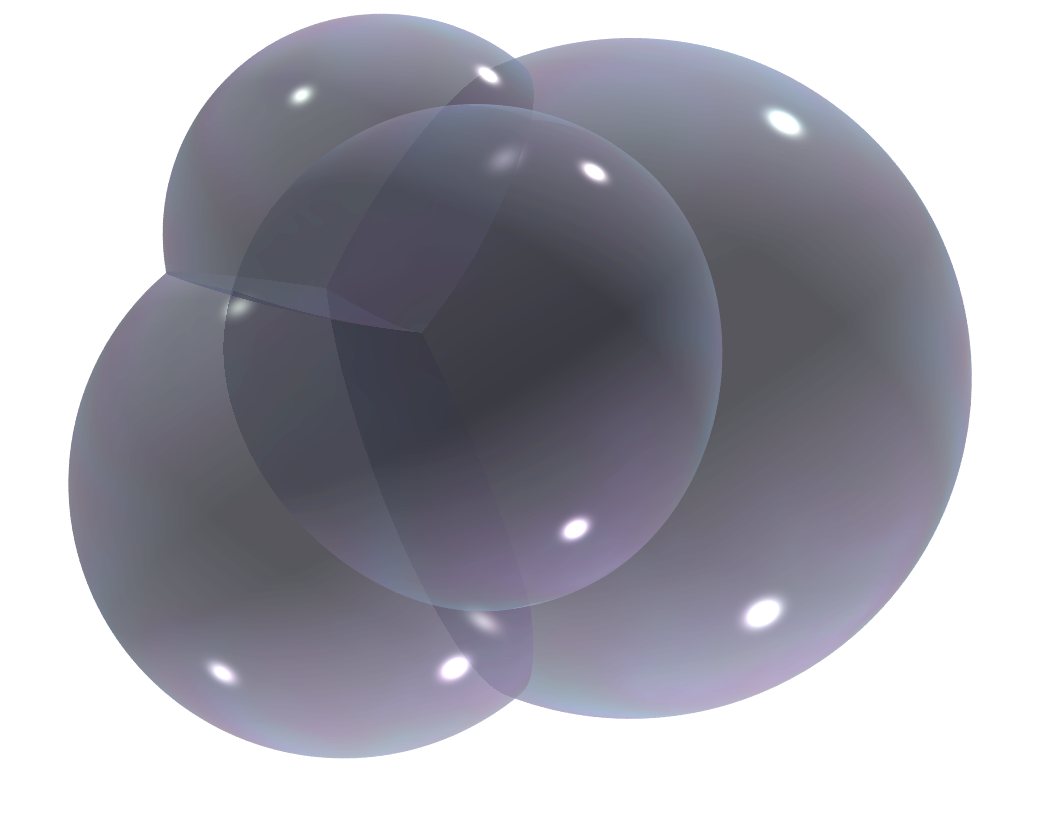}
        \includegraphics[scale=0.09]{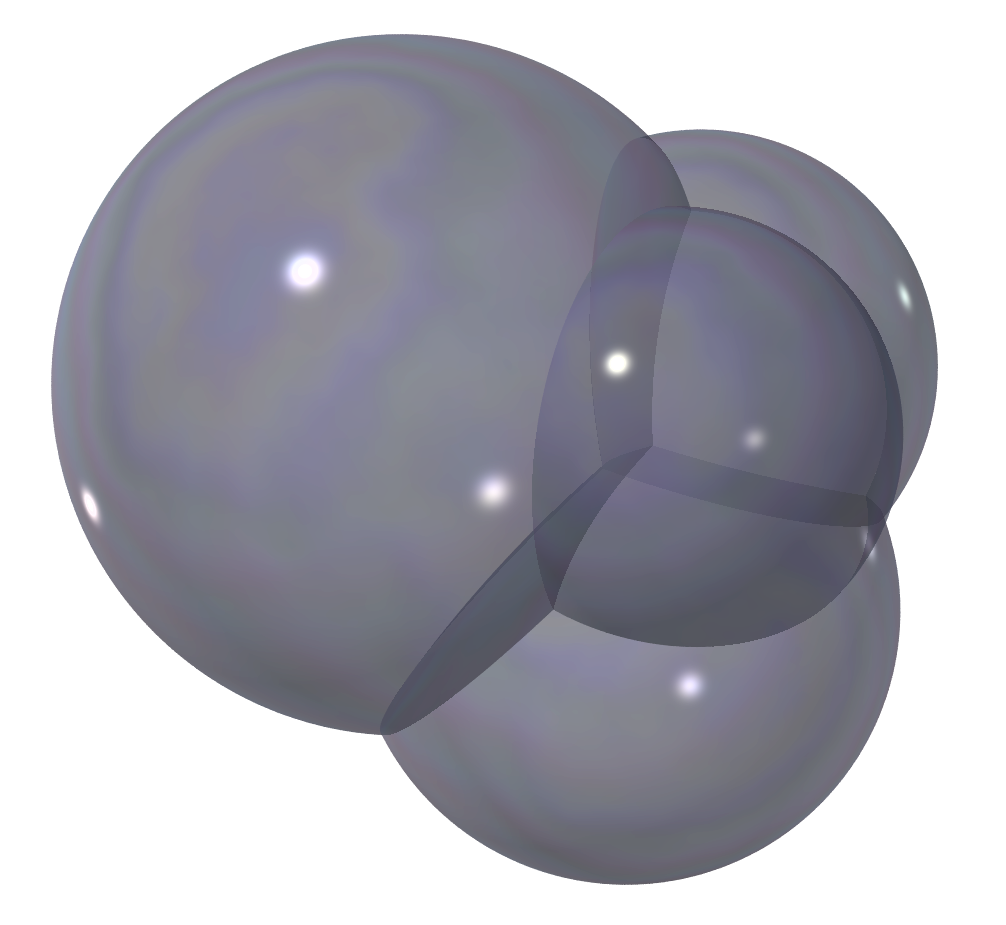}
        \includegraphics[scale=0.09]{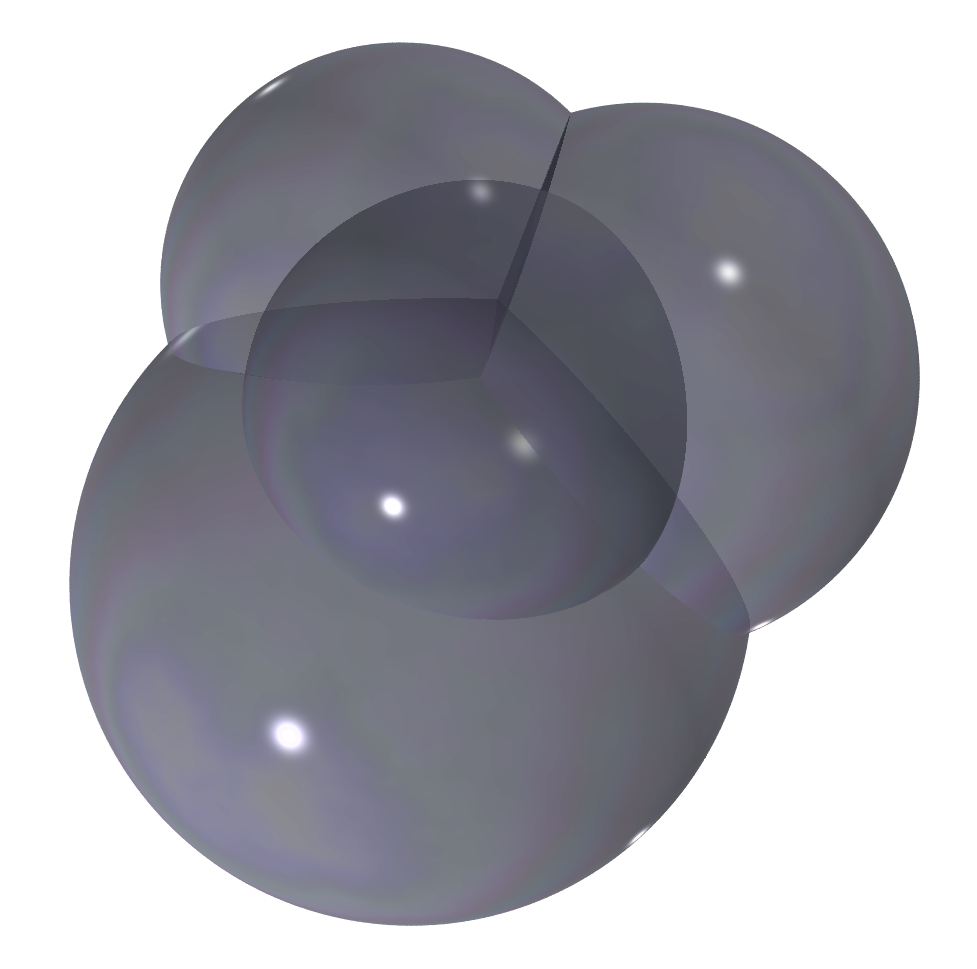}
     \end{center}
     \caption{
         \label{fig:quad-bubble-R3}
        A standard quadruple-bubble in $\R^3$ (also, the cross-section of a standard quadruple-bubble in $\R^4$ through its hyperplane of symmetry) from different angles. 
     }
\end{figure}

\begin{figure}
    \begin{center}
        \includegraphics[scale=0.17]{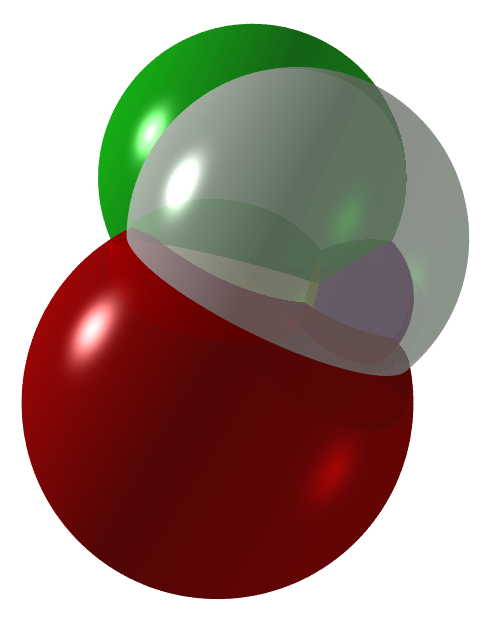}
        \includegraphics[scale=0.23]{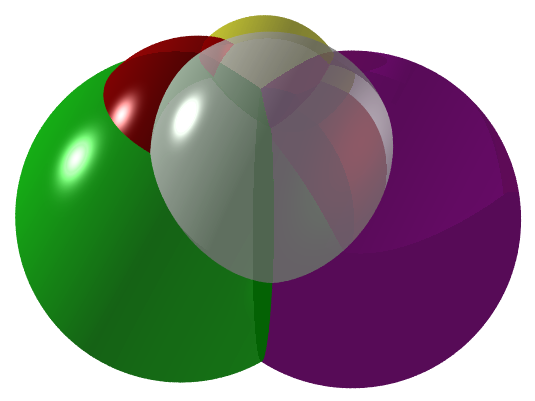}
        \includegraphics[scale=0.23]{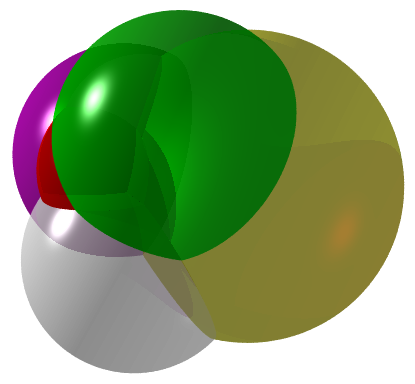}
        \includegraphics[scale=0.28]{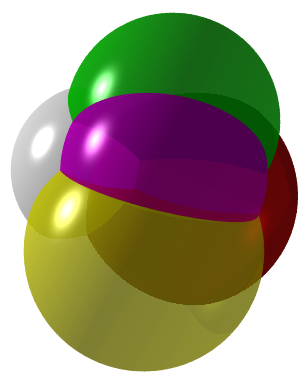}
     \end{center}
     \caption{
         \label{fig:quintuple-bubble-R3}
        Some $3$D cross-sections of various standard quintuple-bubbles in $\R^5$ (hence the $120^{\circ}$ angles of incidence at triple-points get destroyed). 
     }
\end{figure}

\begin{figure}
    \begin{center}
            \includegraphics[scale=0.1]{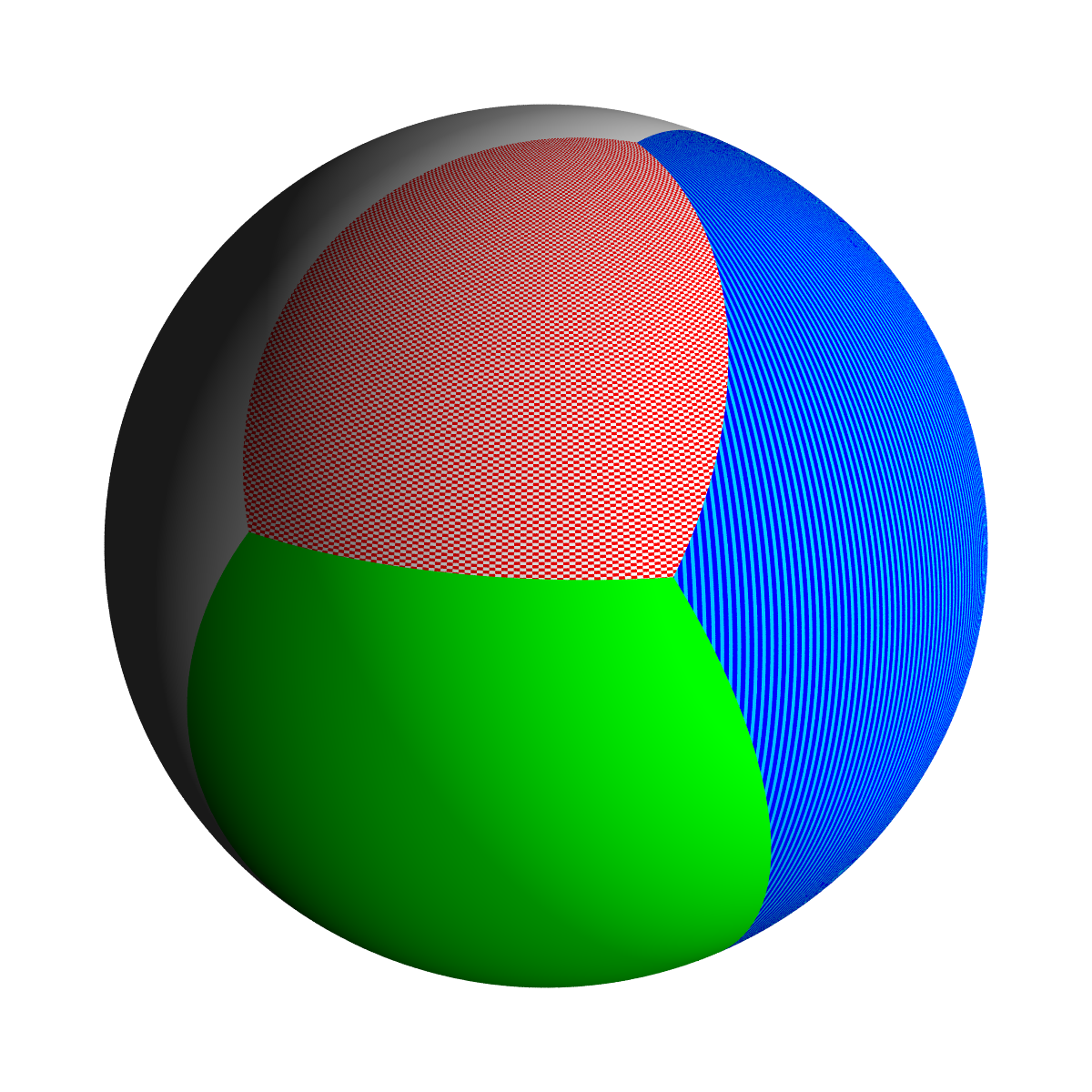}
        \raisebox{0.12\height}{\includegraphics[scale=0.122]{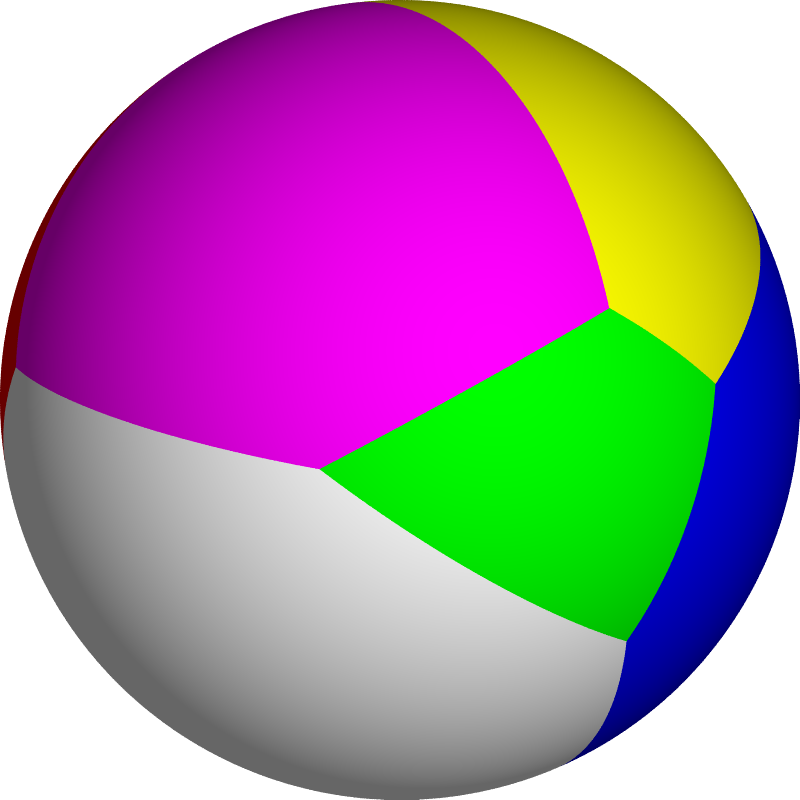}} \hspace{10pt}
         \raisebox{0.12\height}{\includegraphics[scale=0.122]{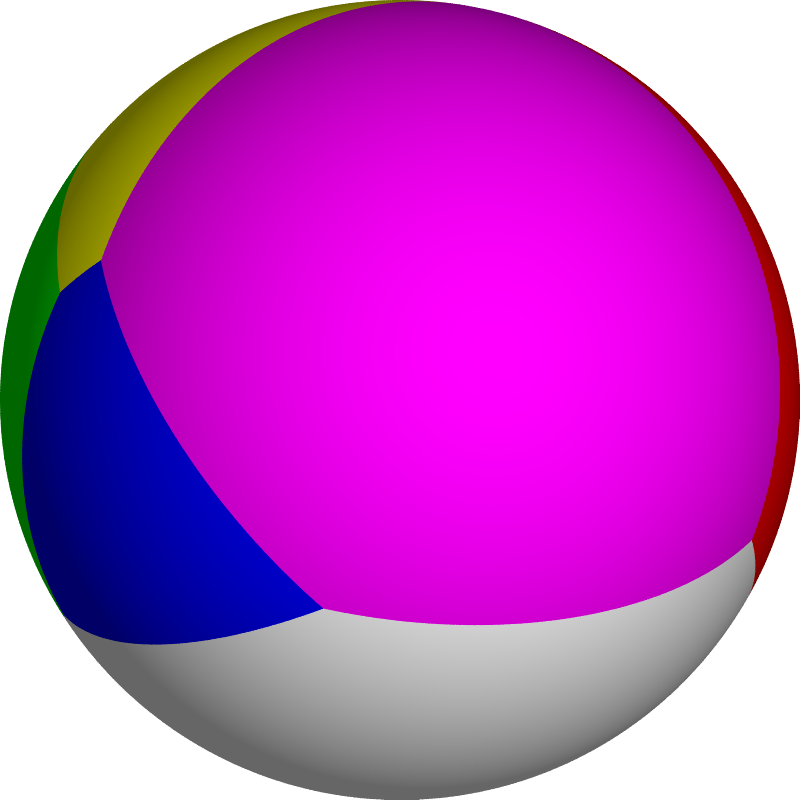}}
     \end{center}
     \caption{
         \label{fig:triple-bubble-S2}
        Left: a standard triple-bubble in $\S^2$; also, the cross-section of a standard triple-bubble in $\S^3$ through its hyperplane of symmetry. Middle and right: two $\S^2$ cross-sections of a standard quintuple-bubble in $\S^4$ (hence the $120^{\circ}$ angles of incidence at triple-points get destroyed). 
     }
\end{figure}

Note that the above construction ensures that a standard $(q-1)$-bubble $\Omega$ on $\R^n$ has a single unbounded cell $\Omega_q$, yielding a legal cluster whose boundary 
can be thought of as a soap film enclosing $q-1$ bounded (connected) air bubbles; we will henceforth omit the $q-1$ index when referring to standard bubbles. Also note that the equal-volume standard-bubble on $\S^n$ has flat interfaces lying on great spheres, and that whenever three of these interfaces meet, they do so at $120^\circ$-degree angles. As stereographic projection is conformal and preserves (generalized) spheres (of possibly vanishing curvature, i.e.~totally geodesic hypersurfaces), it follows that all standard bubbles on $\R^n$ and $\S^n$ have (generalized) spherical interfaces which meet in threes at $120^\circ$-degree angles (see Figures \ref{fig:triple-bubble}, \ref{fig:quad-bubble-R3}, \ref{fig:quintuple-bubble-R3} and \ref{fig:triple-bubble-S2}). The latter angles of incidence are a well-known necessary condition for any isoperimetric minimizing cluster candidate (see Lemma \ref{lem:boundary-normal-sum}).

\begin{conjecture*}[Multi-Bubble Isoperimetric Conjecture on $\R^n$]
For all $2 \leq q \leq n+2$, a standard bubble uniquely minimizes total perimeter among all $q$-clusters $\Omega$ on $\R^n$ of prescribed volume $V(\Omega) = v \in \interior \Delta^{(q-1)}_{\infty}$. 
\end{conjecture*}

\begin{conjecture*}[Multi-Bubble Isoperimetric Conjecture on $\S^n$]
For all $2 \leq q \leq n+2$, a standard bubble uniquely minimizes total perimeter among all $q$-clusters $\Omega$ on $\S^n$ of prescribed volume $V(\Omega) = v \in \interior \Delta^{(q-1)}_{1}$. 
\end{conjecture*}

To avoid constantly writing ``up to null sets" above and throughout this work, we will always modify an isoperimetric minimizing cluster on a null set so that its cells $\Omega_i$ are open and $\overline{\partial^* \Omega_i} = \partial \Omega_i$ (this is always possible by Theorem \ref{thm:Almgren}). Note that the single-bubble cases $q=2$ above corresponds to the classical isoperimetric problems on $\R^n$ and $\S^n$, where geodesic balls are well-known to be the unique isoperimetric minimizers \cite{BuragoZalgallerBook}; this case was included in the formulation of the conjectures for completeness. That a standard bubble in $\R^n$ exists and is unique (up to isometries) for all volumes $v \in \interior \Delta^{(q-1)}_\infty$ and $2 \leq q \leq n+2$ was proved by Montesinos-Amilibia \cite{MontesinosStandardBubbleE!}. An analogous statement equally holds on $\S^n$ for all $v \in \interior \Delta^{(q-1)}_1$ (see Lemma \ref{lem:standard-volume} proved in \cite{EMilmanNeeman-TripleAndQuadruple}). 

\subsection{Previously known and related results}

Let us describe some of the previously known results regarding the multi-bubble conjectures on $\R^n$ and $\S^n$, and refer to the excellent book by F.~Morgan \cite[Chapters 13,14,18,19]{MorganBook5Ed} for more information. Some additional related results in hyperbolic space or more general homogeneous spaces, as well as regarding locally isoperimetric partitions involving more than one cell with infinite volume, clusters with infinitely many cells, clusters with convex cells or different interface weights may be found in \cite{CottonFreeman-DoubleBubbleInSandH,NPST-ClustersViaConcentrationCompactness,ABV-LensClusters, NPT-LocallyIsoperimetricPartitions, NPST-ClustersWithInfinitelyManyCells,DeRosaTione-ConvexStationaryBubbles,BronsardNovak-DifferentTensions}. 

\begin{itemize}
\item On $\R^n$ -- 
Long believed to be true, but appearing explicitly as a conjecture in an undergraduate thesis by J.~Foisy in 1991 \cite{Foisy-UGThesis}, 
the Euclidean double-bubble problem in $\R^n$ was considered in the 1990's by various authors. In the Euclidean plane $\R^2$, the conjecture was confirmed in \cite{SMALL93} (see also \cite{MorganWichiramala-StablePlanarDoubleBubble,LawlorEtAl-PlanarDoubleBubbleViaMetacalibration,CLM-StabilityInPlanarDoubleBubble}). In $\R^3$, the equal volumes case $v_1 = v_2$ was settled in \cite{HHS95,HassSchlafly-EqualDoubleBubbles}, and the structure of general double-bubbles was further studied in \cite{Hutchings-StructureOfDoubleBubbles}. This culminated in the work of Hutchings--Morgan--Ritor\'e--Ros \cite{DoubleBubbleInR3-Announcement,DoubleBubbleInR3}, who confirmed the double-bubble conjecture in $\R^3$; their method was subsequently extended by Reichardt-et-al to resolve the conjecture in $\R^4$ and $\R^n$ \cite{SMALL03,Reichardt-DoubleBubbleInRn} (see also Lawlor \cite{Lawlor-DoubleBubbleInRn} for an alternative proof using ``unification" which applies to arbitrary interface weights). The triple-bubble case $q=4$ in the Euclidean plane $\R^2$ was confirmed by Wichiramala in \cite{Wichiramala-TripleBubbleInR2} (see also \cite{Lawlor-TripleBubbleInR2AndS2}). The triple-bubble conjecture on $\R^n$ for $n \geq 3$ and quadruple-bubble conjecture for $n \geq 4$ were recently confirmed in our previous work \cite{EMilmanNeeman-TripleAndQuadruple} (along with a new proof of the double-bubble case for $n \geq 2$). We've learned from Gary Lawlor (personal communication) that he too has managed to confirm the equal-volume case $v_1 = v_2 = v_3$ of the triple-bubble conjecture on $\R^3$.
\item On $\S^n$ --
The double-bubble and triple-bubble conjectures were resolved on $\S^2$ by Masters \cite{Masters-DoubleBubbleInS2} and Lawlor \cite{Lawlor-TripleBubbleInR2AndS2}, respectively, but on $\S^n$ for $n\geq 3$ only partial results for the double-bubble problem were known until recently \cite{CottonFreeman-DoubleBubbleInSandH, CorneliHoffmanEtAl-DoubleBubbleIn3D,CorneliCorwinEtAl-DoubleBubbleInSandG}. The double-bubble conjecture on $\S^n$ for $n \geq 2$, the triple-bubble conjecture for $n \geq 3$ and quadruple-bubble conjecture for $n \geq 4$ were recently confirmed in our previous work \cite{EMilmanNeeman-TripleAndQuadruple}. 
\end{itemize} 

\noindent
One can say a bit more in the equal-volume cases:
\begin{itemize}
\item While this falls outside the scope of Sullivan's conjecture, we mention that Paolini, Tamagnini and Tortorelli \cite{PaoliniTamagnini-PlanarQuadraupleBubbleEqualAreas,PaoliniTortorelli-PlanarQuadrupleEqualAreas} have identified a unique (up to reordering cells, scaling and isometries) minimizing quadruple-bubble of equal volumes $v_1 = v_2 = v_3 = v_4$ in the plane $\R^2$.
\item The equal-volume case $v_1 = \ldots = v_q = \frac{1}{q}$ of the multi-bubble conjecture on $\S^n$ for $2 \leq q \leq n+2$ follows immediately from the equal-volume case on $\GG^{n+1}$, described next.
\end{itemize}

The last point is a good segue for adding to our discussion Gaussian space $\GG^n$, consisting of Euclidean space $(\R^n,|\cdot|)$ endowed with the standard Gaussian probability measure $\gamma^n = (2 \pi)^{-\frac{n}{2}} \exp(-|x|^2/2) dx$. The following conjecture was confirmed in our previous work \cite{EMilmanNeeman-GaussianMultiBubble}:

\begin{theorem*}[Multi-Bubble Isoperimetric Conjecture on $\GG^n$]
For all $2 \leq q \leq n+1$, the unique Gaussian-weighted isoperimetric minimizers on $\GG^n$ of prescribed Gaussian measure $v \in \interior \simplex^{(q-1)}_1$ are simplicial clusters, obtained as the Voronoi cells of $q$ equidistant points in $\R^n$ (appropriately translated). 
\end{theorem*}

\noindent When $q=2$, the cells of a simplicial cluster are precisely halfspaces, and the single-bubble conjecture on $\GG^n$ holds by the classical Gaussian isoperimetric inequality \cite{SudakovTsirelson,Borell-GaussianIsoperimetry} and its equality cases \cite{EhrhardGaussianIsopEqualityCases, CarlenKerceEqualityInGaussianIsop}; this case was included in the formulation above for completeness. Prior partial results regarding the double-bubble conjecture on $\GG^n$ for $n \geq 2$ were obtained in \cite{CorneliCorwinEtAl-DoubleBubbleInSandG}. 
Note that the interfaces of these Gaussian simplicial clusters are all flat, as opposed to the spherical interfaces of the standard bubbles on $\R^n$ and $\S^n$; this flattening when passing from $\S^N$ to $\GG^n$ is a well-known phenomenon, due to the need to rescale $\S^N$ so as to match the ``curvature" of $\GG^n$. We can now explain the remark about the equal volume case on $\S^n$ -- both measure and perimeter on $\S^n$ and $\GG^{n+1}$ coincide for \emph{centered} cones, and the unique equal-volumes minimizer on $\GG^{n+1}$ for all $2 \leq q \leq (n+1) + 1$ is the centered simplicial cluster (whose cells are centered cones), whose intersection with $\S^n$ is exactly the equal-volume standard-bubble. 

\medskip

\subsection{Confirmation of quintuple-bubble conjecture}

Our first main result in this work is the following:

\begin{theorem} \label{thm:intro-quintuple}
The quintuple-bubble conjecture (case $q=6$) holds on $\S^n$ for all $n \geq 5$. 
\end{theorem}
By scale-invariance, approximately embedding a small cluster in $\R^n$ into $\S^n$ and applying Theorem \ref{thm:intro-quintuple}, it follows that a standard quintuple-bubble in $\R^n$ for $n \geq 5$ is indeed an isoperimetric minimizer, confirming the quintuple-bubble isoperimetric \emph{inequality} on $\R^n$ (see Subsection \ref{subsec:degenerate-elliptic}). However, uniqueness is lost in the approximation procedure, and so contrary to the $\S^n$ case, we cannot exclude the existence of additional quintuple-bubble minimizers on $\R^n$.
\begin{corollary} \label{cor:intro-quintuple-R}
The quintuple-bubble conjecture (case $q=6$) holds on $\R^n$, possibly without uniqueness, for all $n \geq 5$. 
\end{corollary}

In fact, we are able to resolve the general multi-bubble case on $\S^n$ for $q \leq n+1$, and hence by approximation the one on $\R^n$ (but without uniqueness) as well, \emph{conditioned} on either of two possible assumptions (we refer to Section \ref{sec:Voronoi} for precise definitions):
\begin{itemize}
\item
One favorable scenario for us is when a minimizing cluster is ``full-dimensional", which roughly means that the bubbles comprising the cluster are in general position, so that their centers do not lie on a linear subspace of dimension strictly smaller than $q-1$. More generally, we can handle clusters which we call ``pseudo conformally flat" (or PCF), a generalized notion of conformal flatness, which does allow particular lower-dimensional configurations. For example, a cluster $\Omega$ so that $\cap_{i=1}^q \overline{\Omega_i} \neq \emptyset$ is always PCF. 
\item 
A second favorable scenario pertains to the type of singularities a minimizing cluster can have at the meeting locus of $k$ cells. It was noted by Physicist Joseph Plateau already in the 19th century that soap bubbles (in $\R^3$) meet in threes and in fours like the cones over a regular triangle and tetrahedron, respectively. This was rigorously confirmed in $\R^3$ by Jean Taylor in \cite{Taylor-SoapBubbleRegularityInR3} for Almgren's $(\M,\eps,\delta)$ surfaces, and in particular for area minimizers. An extension of Taylor's results to certain $(\M,\eps,\delta)$ sets in $\R^n$ (and general weighted Riemannian manifolds), which in particular include the boundary of minimizing $q$-clusters, was obtained in \cite{CES-RegularityOfMinimalSurfacesNearCones}. The assumption we require to extend our results to all $q \leq n+1$ is  that a minimizing cluster satisfies a higher-order extension of Plateau's laws -- whenever $k$ cells meet in an $m$-dimensional cone (after modding out $\R^{n-m}$) for $m\leq \ell$, the cone must be a regular $m$-dimensional simplex and necessarily $k=m+1$. 
We shall call such a cluster an \emph{``$\ell$-Plateau cluster"}; when this holds for any $\ell$ we will simply say that the cluster is \emph{``Plateau"} (see Subsection \ref{subsec:Plateau} for a precise formulation). By the above results, every minimizing cluster is $3$-Plateau.
\end{itemize}

Unfortunately, we do not know how to a-priori verify that a minimizing $q$-cluster is necessarily either pseudo conformally flat or Plateau, unless $q \leq 6$, which suffices to cover the previously established double, triple and quadruple cases, and fortunately also includes the new quintuple case, yielding the (unconditional) Theorem \ref{thm:intro-quintuple}. 

\medskip

In order to explain why the method we employed in \cite{EMilmanNeeman-TripleAndQuadruple} for resolving the triple and quadruple bubble conjectures cannot be used (to the best of our understanding) to resolve the quintuple case, and in order to describe our conditional result more precisely, let us recall the main structural results for isoperimetric minimizers on $\S^n$ which we obtained in \cite{EMilmanNeeman-TripleAndQuadruple}.

\subsection{Spherical Voronoi clusters}

\begin{definition}[Spherical Voronoi Cluster] \label{def:intro-spherical-Voronoi}
A $q$-cluster $\Omega$ on $\S^n$ is called a spherical Voronoi cluster if there exist $\{ \c_i \}_{i=1,\ldots,q} \subset \R^{n+1}$ and $\{ \k_i \}_{i=1,\ldots,q} \subset \R$ so that for all $i=1,\ldots,q$:
\begin{enumerate}[(1)]
\item \label{it:intro-spherical-Voronoi}
For every non-empty interface $\Sigma_{ij} \neq \emptyset$, $\Sigma_{ij}$ is a subset of a geodesic sphere $S_{ij}$ with quasi-center $\c_{ij} = \c_i - \c_j$ and curvature $\k_{ij} = \k_i - \k_j$. \\
The quasi-center $\c$ of a geodesic sphere $S$ is the vector $\c := \n - \k p$ at any of its points $p \in S$ (where $\k$ is the curvature with respect to the unit normal $\n$). 
\item \label{it:intro-spherical-Voronoi-2}
The following Voronoi representation holds:
\[ \Omega_i  = \interior \; \set{ p \in \S^n \; ; \; \argmin_{j=1,\ldots,q} \scalar{\c_j , p} + \k_j \ni  i } =  \interior \; \bigcap_{j \neq i} \; \set{ p \in \S^n \; ;\; \scalar{\c_{ij},p} + \k_{ij} \leq  0 } .
\] \end{enumerate}
As the parameters $\{\c_i\}$ and $\{\k_i\}$ are defined up to translation, we will always employ the convention that $\sum_{i=1}^q \c_i = 0$ and $\sum_{i=1}^q \k_i = 0$. 
\end{definition}

\begin{definition}[Perpendicular Spherical Voronoi Cluster]  \label{def:intro-perpendicular}
A spherical Voronoi cluster $\Omega$ on $\S^n$ is called perpendicular if there exists $N \in \S^n$ so that $\c_i \perp N$ for all $i=1,\ldots,q$. In particular, the cells of a perpendicular spherical Voronoi cluster are invariant under reflection about $N^{\perp}$. $N$ is called a North pole of $\Omega$ and $\S^n \cap N^{\perp}$ is called an equator. A perpendicular spherical Voronoi cluster is said to have equatorial cells if all of its (open) cells intersect an equator, or equivalently, if they are all non-empty and connected.  
\end{definition}

\begin{remark} \label{rem:intro-standard-bubbles}
All standard bubbles (with $q \leq n+2$) on $\S^n$ are Plateau spherical Voronoi clusters, and when $q \leq n+1$ they are in addition perpendicular (Corollary \ref{cor:standard-Voronoi}); they are precisely characterized as those spherical Voronoi clusters for which the interfaces $\Sigma_{ij}$ are non-empty for all $1 \leq i < j \leq q$ -- see Proposition \ref{prop:standard-char}. 
\end{remark}

\begin{figure}
    \begin{center}
                 \includegraphics[scale=0.108]{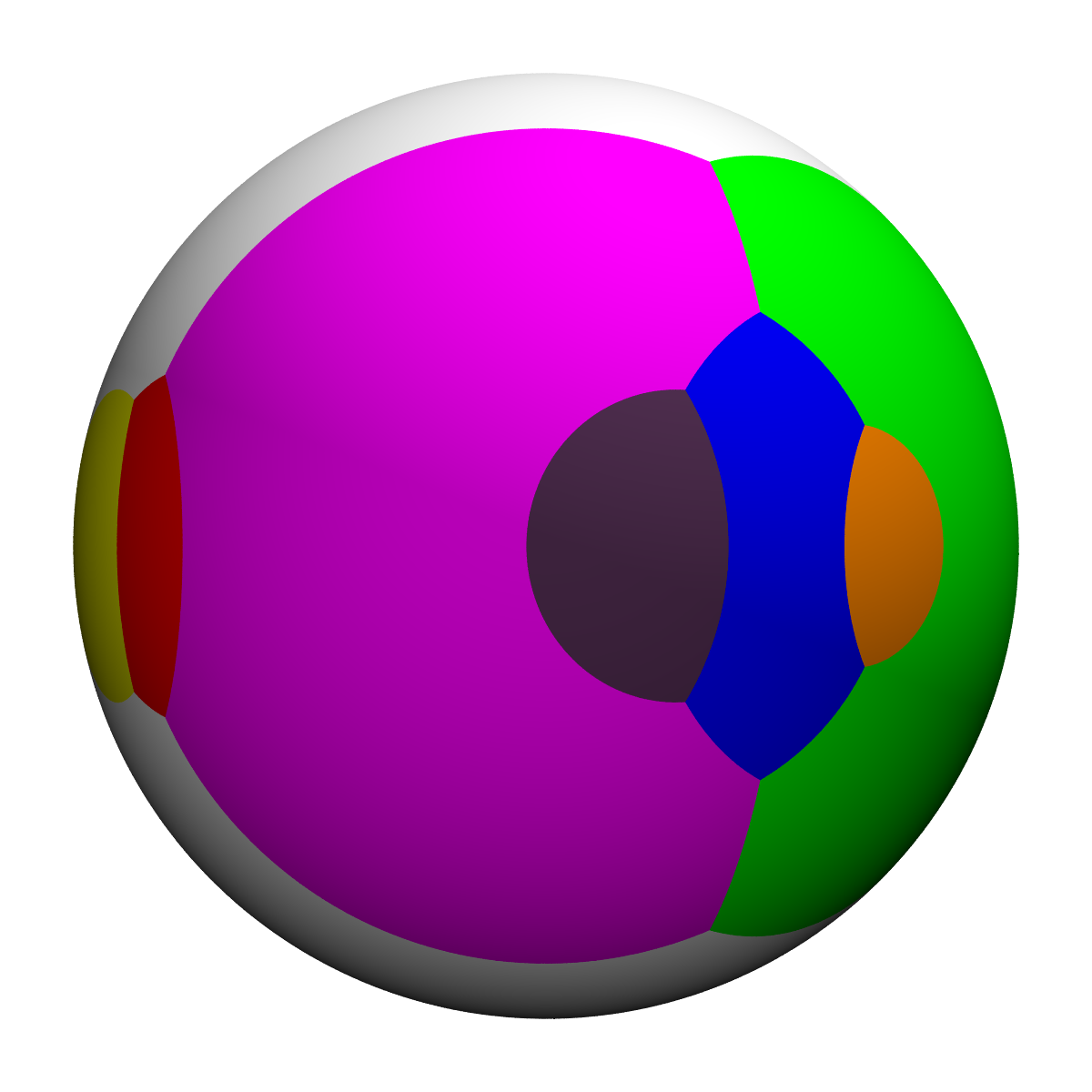}
        \includegraphics[scale=0.108]{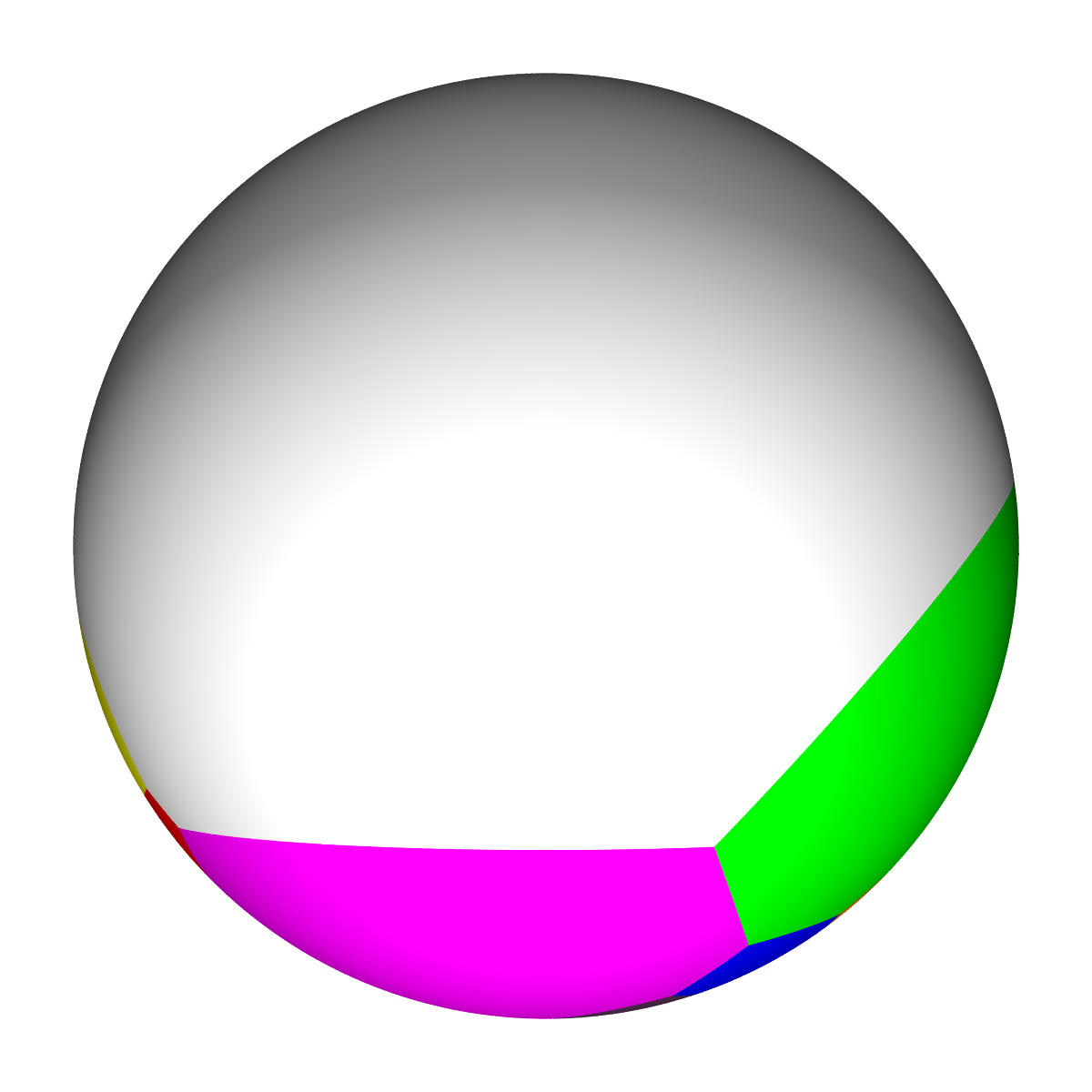}
                \raisebox{0.07\height}{
        \begin{tikzpicture}[scale=1.97]
            \input{cluster-brep-color}
        \end{tikzpicture}
        }
     \end{center}
     \caption{
         \label{fig:intro-spherical-Voronoi}
         A Plateau perpendicular spherical Voronoi cluster with equatorial cells on $\S^2$, as seen from its equatorial plane (left), from its North pole (middle), and after projection to the equatorial plane (right, colors lightened for better contrast), highlighting its convex polyhedral cells. 
     }
\end{figure}

The following structural result was established in \cite{EMilmanNeeman-TripleAndQuadruple}:

\begin{theorem}[\cite{EMilmanNeeman-TripleAndQuadruple}] \label{thm:intro-structure}
Let $\Omega$ be an isoperimetric minimizing $q$-cluster on $\S^n$, and assume that $q \leq n+1$ and $V(\Omega) \in \interior \Delta^{(q-1)}_1$. Then $\Omega$ is a perpendicular spherical Voronoi cluster with equatorial (and hence connected) cells. Moreover, $\Omega$ is regular, stable and $3$-Plateau. 
\end{theorem}

We defer the definition of regularity and stability to Section \ref{sec:prelim}. 
An analogous result on $\R^n$ was also established in \cite{EMilmanNeeman-TripleAndQuadruple}. 
In fact, it was shown that the above holds for any stable $q$-cluster with reflection-symmetry so that $\Sigma = \cup_{i=1}^q \partial \Omega_i$ is connected and satisfies some regularity assumptions. 

\subsection{The challenge}

In view of Theorem \ref{thm:intro-structure} and Remark \ref{rem:intro-standard-bubbles}, in order to confirm the multi-bubble conjecture on $\S^n$ for $q \leq n+1$, it remains to establish that the interfaces $\Sigma_{ij}$ of a minimizing $q$-cluster are non-empty for all $1 \leq i < j \leq q$. This was achieved in \cite{EMilmanNeeman-TripleAndQuadruple} in the double, triple and quadruple cases ($q=3,4,5$) by analyzing the interface adjacency graph of a minimizing cluster, and excluding all possible graphs but the complete graph on $q$ vertices. To this end, it was shown that if the adjacency graph had missing edges, it would be possible to rotate a cell (using a local isometry) without altering the volumes of the cells nor the total perimeter, until the cell touches some additional ones, thereby violating Plateau's laws. Since a minimizer is only known to be $3$-Plateau,
this argument is limited to $q \leq 5$. The problem in extending this to $q \geq 6$ is three-fold: firstly, the number of possible (isomorphism classes of) graphs on $q$ vertices grows super exponentially fast in $q$, making a case-by-case analysis essentially impossible; secondly, we will not always be able to find a ``loose" cell (or a group of cells) which we can rotate while keeping the others in place until they make contact (such a ``bubble-ring" scenario occured in \cite{EMilmanNeeman-TripleAndQuadruple}  even in the quadruple-case); and thirdly, it is known that there are non-standard area minimizing cones in $\R^m$ for $m \geq 4$, which are not cones over a regular simplex, such as the cone over a hypercube \cite{Brakke-MinimalConesOnCubes}, and so it is not clear whether the higher-order Plateau's laws should hold for a minimizing cluster. In particular, we cannot simply apply the method from \cite{EMilmanNeeman-TripleAndQuadruple} to the quintuple case, and require a different approach. 

\smallskip

To this end, we apply the PDE approach for the isoperimetric profile, which we used in \cite{EMilmanNeeman-GaussianMultiBubble} to resolve the Gaussian multi-bubble conjecture. We henceforth abbreviate $\Delta^{(q-1)} = \Delta^{(q-1)}_1$. 
Let $\I^{(q-1)} : \simplex^{(q-1)} \to \R_+$ denote the isoperimetric profile for $q$-clusters on $\S^n$, defined as:
\begin{equation} \label{eq:intro-I}
  \I^{(q-1)}(v) := \inf\{\per(\Omega) \; ; \; \text{$\Omega$ is a $q$-cluster with $V(\Omega) = v$}\}.
\end{equation}
Our goal will be to show that $\I^{(q-1)} = \I^{(q-1)}_m$ on $\simplex^{(q-1)}$, where $\I^{(q-1)}_m : \simplex^{(q-1)} \to \R_+$ for $2 \leq q \leq n+2$ denotes the multi-bubble \emph{model} profile on $\S^n$, defined for $v \in \interior \simplex^{(q-1)}$ as:
\begin{equation} \label{eq:intro-Im}
  \I^{(q-1)}_m(v) := \per(\Omega^m) \text{ where $\Omega^m$ is a standard $(q-1)$-bubble with $V(\Omega^m) = v$} 
\end{equation}
(see Lemma \ref{lem:standard-volume} to recall why this is well-defined); for $v \in \partial \Delta^{(q-1)}$ we define recursively $\I^{(q-1)}_m(v) := \I^{(q-2)}_m(v_{-i})$ if $v_i = 0$, where $v_{-i}$ denotes erasing the $i$-th coordinate from $v$ (as there is at least one empty cell). Clearly $\I^{(q-1)} \leq \I^{(q-1)}_m$, and by induction on $q$ we may assume that $\I^{(q-1)} = \I^{(q-1)}_m$ on $\partial \Delta^{(q-1)}$. We observe the following:

\begin{proposition} \label{prop:intro-profile}
Let $2 \leq q \leq n+2$. The model multi-bubble isoperimetric profile $\I_m = \I^{(q-1)}_m$ on $\S^n$ satisfies the following PDE on $\Delta^{(q-1)}$:
\begin{equation} \label{eq:intro-PDE}
\tr((-\nabla^2 \I_m)^{-1} (\Id+ \frac{2}{(n-1)^2} \nabla \I_m \otimes \nabla \I_m)) = \frac{2}{n-1} \I_m  .
\end{equation}
Moreover, $-\nabla^2 \I_m > 0$ on $T \Delta^{(q-1)}$, and hence this PDE is elliptic (albeit fully non-linear). 
\end{proposition}
\noindent See Section \ref{sec:profile} for a comparison to the ODE satisfied in the single-bubble case, and to the simpler PDE satisfied by the Gaussian multi-bubble isoperimetric profile (which does not have any first-order terms). We do not know how to obtain an analogous (strictly) elliptic PDE for the model isoperimetric profile on $\R^n$, only a degenerate one (see Subsection \ref{subsec:degenerate-elliptic}),
which explains why we are forced to deduce our results on $\R^n$ by approximation from $\S^n$ (losing uniqueness in the process). 

\smallskip

If we could show that the actual profile $\I^{(q-1)}$ on $\S^n$ satisfies (\ref{eq:intro-PDE}), in fact just as an inequality $\leq$ in an appropriate viscosity sense, then by an application of the maximum principle we would be able to deduce that $\I^{(q-1)} \geq \I^{(q-1)}_m$, thereby concluding that in fact there is equality -- see Section \ref{sec:profile} for an overview of this strategy. 
\smallskip

Upper bounds on $\nabla^2 \I^{(q-1)}$ may be obtained by perturbing the minimizing cluster $\Omega$ using a vector-field $X$ and computing the corresponding second variation, yielding the quadratic Index-Form $Q(X)$. The latter depends on the Jacobi operator $L_{Jac}$ defined on the cluster's interfaces $\Sigma^1 = \cup_{i<j} \Sigma_{ij}$, which turns out to be self-adjoint on an appropriate domain $\D_{con}$ satisfying \emph{conformal} boundary conditions whenever $3$ interfaces meet. In the Gaussian case, after showing that all interfaces of $\Omega$ must be flat, we were able to identify the entire $(q-1)$-dimensional eigenspace of eigenfunctions of $L_{Jac}$ with positive eigenvalue (equal to $1$), simply given by piecewise-constant fields. Unfortunately, things are more complicated on $\S^n$ since the interfaces will no longer be flat in general (recall that $\Sigma_{ij}$ of a spherical Voronoi cluster has curvature $\k_{ij}$). Consequently, we are forced to delve deeper into the spectral theory of $L_{Jac}$ on spherical Voronoi clusters in $\S^n$, finding analogies with the quantum-graph formalism (see Section \ref{sec:spectral}). 

\medskip

While we are not able to provide an explicit description of the subspace corresponding to positive eigenvalues of $L_{Jac}$ on $\S^n$, we are able to show that $L_{Jac}$ does have exactly $q-1$ positive eigenvalues, or equivalently, that $Q$ is of index $q-1$, which is of independent interest. 
\begin{theorem} \label{thm:intro-minimizer-q-1-positive}
Let $q \leq n+1$. Then for any minimizing $q$-cluster $\Omega$ on $\S^n$ with $V(\Omega) \in \interior \Delta^{(q-1)}$, $(L_{Jac},\D_{con})$ has exactly $q-1$ positive eigenvalues.
\end{theorem}

In fact, this holds for any regular, stable, perpendicular spherical Voronoi cluster with equatorial cells -- see Theorem \ref{thm:q-1-positive}. This is already enough to deduce the following, which was explicitly asked by Heppes in \cite[Problem 4]{OpenProblemsInSoapBubbles96} (for the case of $\R^n$, which follows modulo strictness from that of $\S^n$ by a scaling argument -- see Subsection \ref{subsec:degenerate-elliptic}):
\begin{theorem} \label{thm:intro-profile-concave}
For all $q \leq n+1$, the multi-bubble isoperimetric profile $\I^{(q-1)}$ on $\S^n$ is (strictly) concave on $\Delta^{(q-1)}$. 
 \end{theorem}
 
 However, we fall just short of confirming the PDE (\ref{eq:intro-PDE}) for $\I^{(q-1)}$ in general. What we are missing is a certain trace-identity for a symmetric operator $\F = \F(\Omega) : \R^{q} \rightarrow \R^q$ constructed from solutions to the PDE $L_{Jac} f_{ij} = (n-1) (a_i - a_j)$ on $\Sigma_{ij}$, $f \in \D_{con}$, which are minimizers of the Index-Form $Q$ under a volume constraint -- see Section \ref{sec:conformal-Jacobi} and Conjecture \ref{conj:trace-id}. We are able to verify the trace-identity for a certain relaxation $\F_0$ of $\F$, constructed in Section \ref{sec:perturbation} as the limit $\lim_{t \rightarrow 0} \F_t$ where $\F_t$ are the corresponding operators associated to a one-parameter family of conformally perturbed clusters $\Omega_t$ -- see Figure \ref{fig:conformal-perturbation}; however, we could not verify that $\F = \F_0$, i.e.~that $\F(\lim_{t \rightarrow 0} \Omega_t) = \lim_{t \rightarrow 0} \F(\Omega_t)$, which would follow if $\F(\Omega)$ were continuous in $\Omega$ (!) -- see Conjecture \ref{conj:F=F0} and Remark \ref{rem:F=F0}. A confirmation of this innocent-looking PDE question would immediately allow us to extend our results from the quintuple case to general $q \leq n+1$. 

\begin{figure}
    \begin{center}
                \includegraphics[scale=0.15]{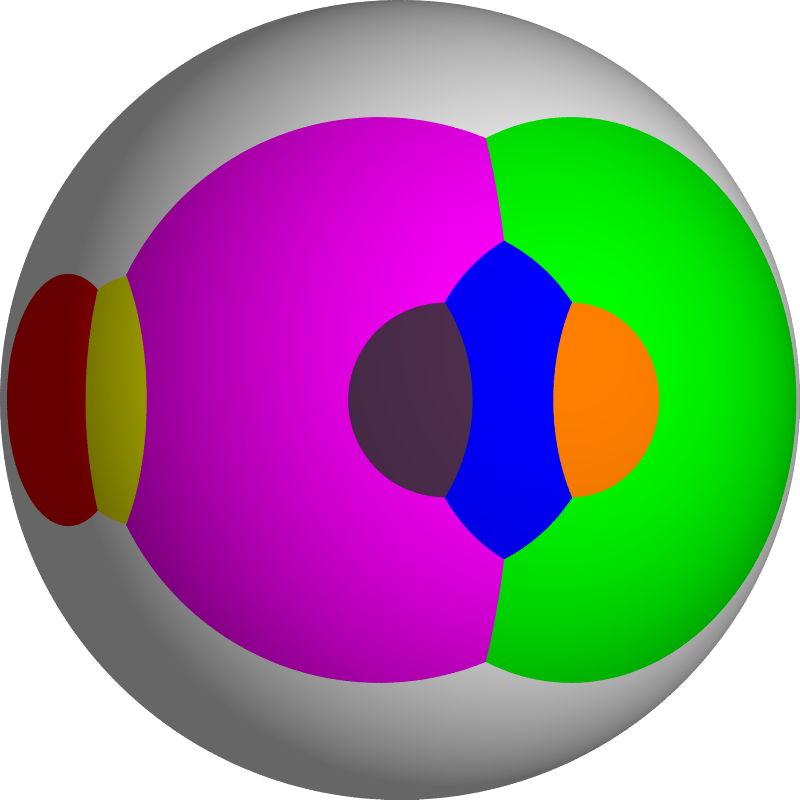}     \includegraphics[scale=0.15]{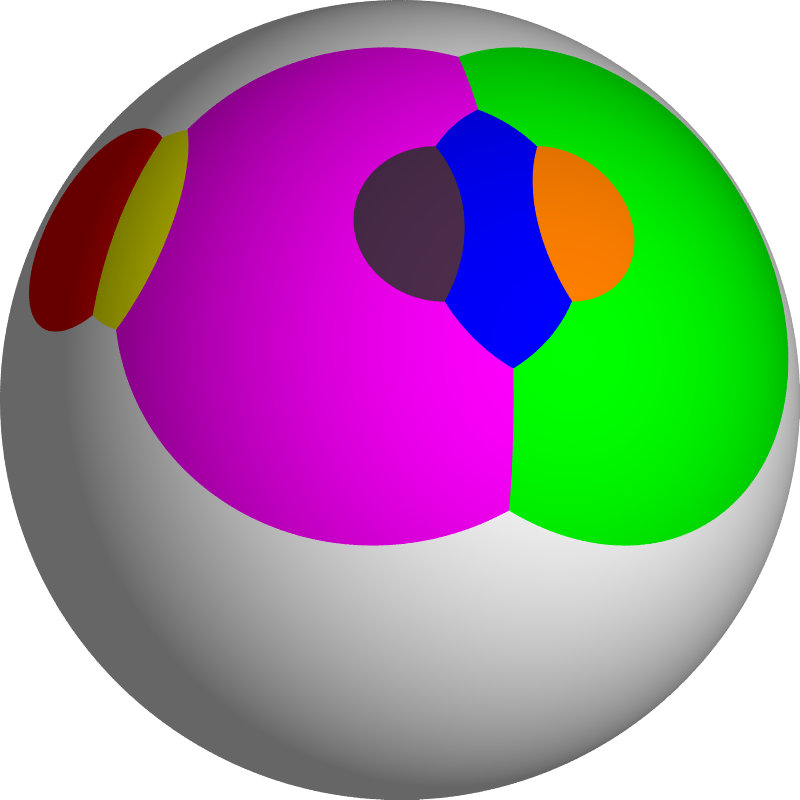}         \includegraphics[scale=0.15]{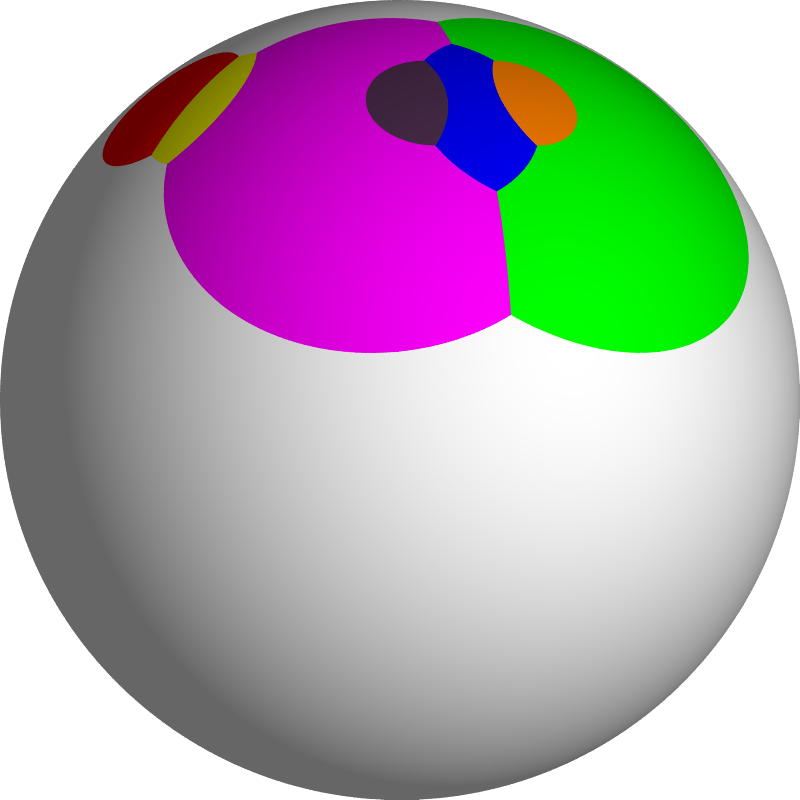}              \end{center}
     \caption{
         \label{fig:conformal-perturbation}
         A one-parameter family of conformally perturbed spherical Voronoi clusters $\{\Omega_t\}_{t \in \R}$ (original perpendicular cluster at $t=0$ on the left, followed by its perturbations at $t=0.5$ and $t=1$). The clusters are being conformally mapped towards the North pole as $t \rightarrow \infty$, and so their volumes and perimeter are changing. We do not know (!) whether $\F(\lim_{t \rightarrow 0} \Omega_t) = \lim_{t \rightarrow 0} \F(\Omega_t)$.
     }
\end{figure}

\subsection{Conditional confirmation of multi-bubble conjectures}

What we are able to show in Section \ref{sec:conditional} is the following conditional statement: 

\begin{theorem} \label{thm:intro-conditional}
Let $2 \leq q \leq n+1$, and assume that the multi-bubble conjecture for $p$-clusters holds on $\S^n$ for all $2 \leq p \leq q-1$. 
Assume that for every $v_0 \in \interior \Delta^{(q-1)}$, any minimizing $q$-cluster $\Omega$ on $\S^n$ with $V(\Omega) = v_0$ satisfies that either:
\begin{itemize}
\item $\Omega$ is pseudo conformally flat (PCF); or
\item $\Omega$ is $(q-3)$-Plateau. 
\end{itemize}
Then the multi-bubble conjecture for $q$-clusters holds on $\S^n$. 
\end{theorem}
See also Theorem \ref{thm:final} for a slightly more general statement, where only the existence of a minimizer of either of the above two forms is required in order to establish that standard bubbles are minimizers (possibly without uniqueness). The crucial feature of being $(q-3)$-Plateau is that necessarily either the cluster is fully Plateau, or it is PCF -- see Lemma \ref{lem:q-3-Plateau}. If the cluster is PCF we are able to show that $\F = \F_0$ and establish the trace-identity directly -- see Section \ref{sec:conformal-Jacobi}. Alternatively, if it is fully Plateau, we are able to slightly perturb it into a full-dimensional (and hence PCF) cluster without creating any new interfaces. Contrary to the aforementioned conformal perturbation from Figure \ref{fig:conformal-perturbation}, and contrary to the isometric rotations which were used in \cite{EMilmanNeeman-TripleAndQuadruple}, the deformation we employ here (which we call ``Gram perturbation") is not comprised of any isometries, yet nevertheless is guaranteed not to alter the cells' volume nor total perimeter (!) -- see Figure \ref{fig:Gram-perturbation} and Sections \ref{sec:LSE} and \ref{sec:Gram}. 

\begin{figure}
    \begin{center}
                \includegraphics[scale=0.12]{pokeball2.png}     \includegraphics[scale=0.12]{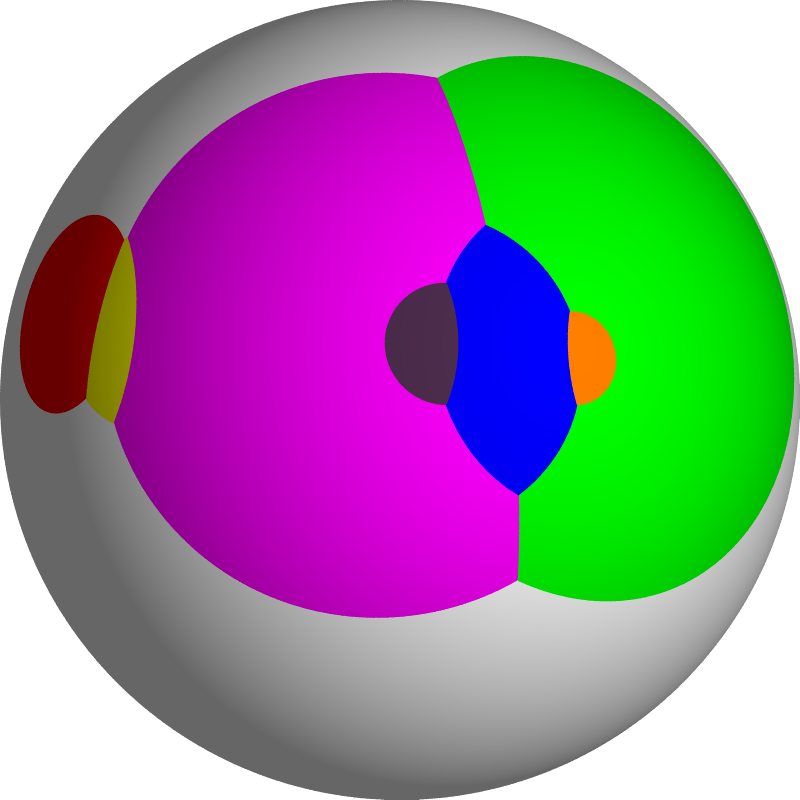}         \includegraphics[scale=0.12]{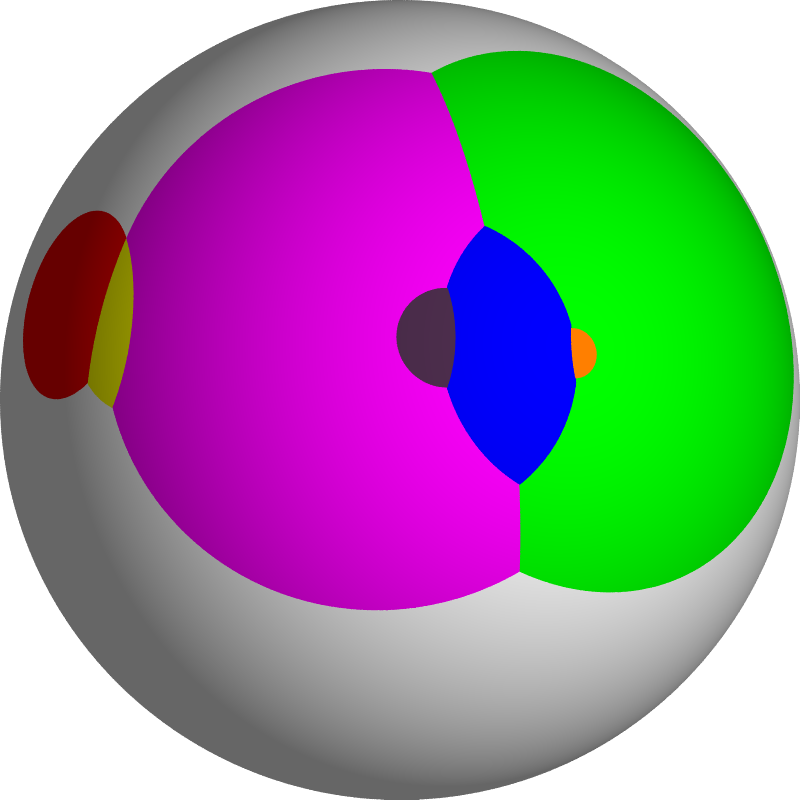}         \includegraphics[scale=0.12]{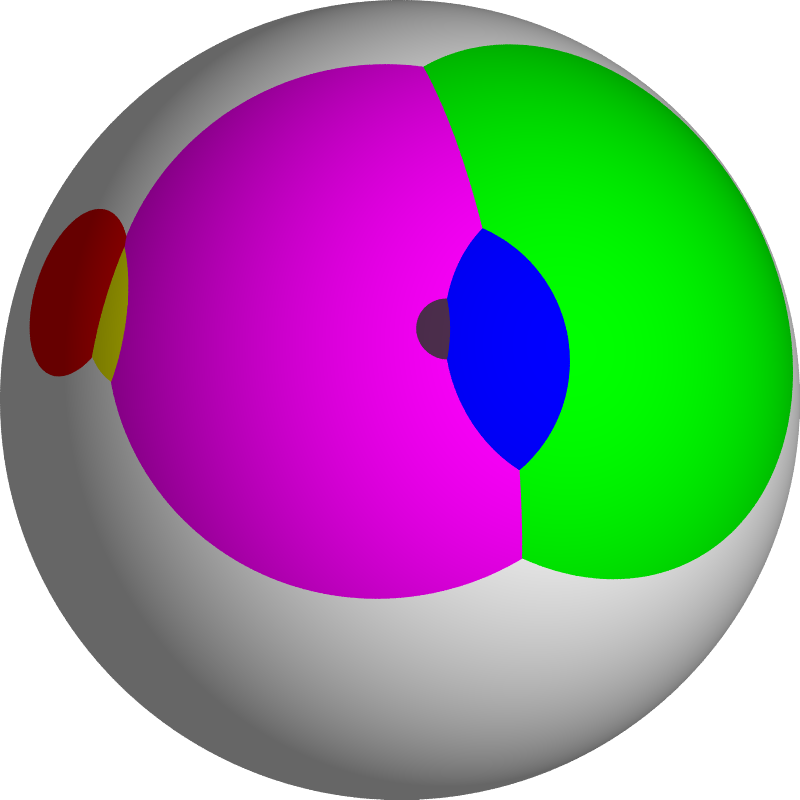}
             \end{center}
     \caption{
         \label{fig:Gram-perturbation}
         Cross-sections of a one-parameter family of Gram-perturbed spherical Voronoi clusters $\{\Omega_t\}_{t \in [0,1]}$ (original perpendicular Plateau cluster at $t=0$ on the left, followed by its perturbations at $t=0.1$, $t=0.12$ and $t=0.15$). The clusters lie on $\S^6$ but we are showing their cross-section through a fixed $\S^2$. The original cluster $\Omega_0$ has all of its cells aligned on a $2$-dimensional plane, but its Gram-perturbation $\Omega_t$ instantaneously becomes full-dimensional on $\S^6$, and there exists a time interval $t \in [0,t_0]$ so that no additional interfaces are created, the clusters remain stationary and Plateau, and their volumes and total-perimeter are guaranteed to remain fixed. However, on the $\S^2$ cross-sections, the angles of incidence at triple points are no longer $120^{\circ}$ and the volumes and perimeters appear to change. Note that by time $t=0.15$ a new interface has been created and one of the cells disappears from the cross-section.}      
\end{figure}

\smallskip

Finally, since every minimizing $q$-cluster is $3$-Plateau, we are able to apply induction on $q$ and verify the multi-bubble conjectures for all $q \leq \min(6,n+1)$, thus obtaining the new quintuple-bubble Theorem \ref{thm:intro-quintuple}, as well as a different proof of the double, triple and quadruple-bubble conjectures on $\S^n$ for $n\geq 2$, $n \geq 3$ and $n \geq 4$, respectively. Uniqueness of minimizers is handled in Section \ref{sec:uniqueness}. 

\begin{remark}
If one is able to show that for a minimizing perpendicular spherical Voronoi $7$-cluster on $\S^n$, $6$ cells cannot meet in a $4$-dimensional cone, Theorem \ref{thm:intro-conditional} would yield a confirmation of the sixtuple-bubble conjecture on $\S^n$ (and hence on $\R^n$, possibly without uniqueness) for $n \geq 6$. And similarly for higher-order multi-bubbles. 
\end{remark}

We refer to Section \ref{sec:LSE} for a more general variant of the PCF condition that we can handle in Theorem \ref{thm:intro-conditional}, which is intimately related to the joint-and-bar framework and infinitesimal rigidity of polyhedra considered by Cauchy, Alexandrov, Dehn, Whiteley and others. In Section \ref{sec:conclude} we provide some concluding remarks, including 
a compelling Locality Conjecture \ref{conj:locality} and other conjectures having a PDE flavor, which would establish that $\F = \F_0$, and thus yield the PDE (\ref{eq:intro-PDE}) for $\I^{(q-1)}$ and establish the spherical multi-bubble conjecture for all $q \leq n+1$. 

\bigskip

\noindent
\textbf{Acknowledgments.} We thank Francesco Maggi and Frank ``Chip" Morgan for their constant support.

\tableofcontents

\section{Isoperimetric minimizing clusters} \label{sec:prelim}

In this section, we recall some of the results developed in \cite{EMilmanNeeman-GaussianMultiBubble,EMilmanNeeman-TripleAndQuadruple} regarding isoperimetric minimizing clusters which will be required in this work. We formulate all results on general weighted Riemannian manifolds, as this does not really pose a greater generality over the case of $\S^n$ we are primarily interested in. However, to make the presentation more concise, we make the following simplifying assumptions on the minimizing $q$-cluster, which do hold on $\S^n$ whenever $q \leq n+1$ in view of Theorem \ref{thm:intro-structure}:
\begin{itemize}
\item The minimizing cluster's finite volume cells are all bounded (trivially in $\S^n$). Consequently, we will only work with compactly supported fields, but note it is possible to extend the theory to the case when the cells are unbounded by using appropriately admissible fields, such as in the Gaussian setting of \cite{EMilmanNeeman-GaussianMultiBubble}. 
\item The minimizing cluster has locally bounded curvature. In general, it is only known that curvature is square integrable on a minimizing cluster's codimension-one interfaces, and integrable on the codimension-two interfaces \cite[Proposition 5.7]{EMilmanNeeman-GaussianMultiBubble}, \cite[Proposition 2.23]{EMilmanNeeman-TripleAndQuadruple}. The bounded curvature assumption makes it possible to compute the effects of perturbing the cluster using general compactly supported fields, without assuming they are bounded away from the singular set $\Sigma^4$ (defined below). 
\item The measure $\mu$ is a probability measure; consequently the existence of minimizing clusters is a non-issue. 
\end{itemize}
We refer to \cite[Sections 2-3,5-6]{EMilmanNeeman-TripleAndQuadruple} and the references therein for proofs of the subsequent statements in this section.

\begin{definition}[Weighted Riemannian Manifold]
A smooth complete $n$-dimensional Riemannian manifold $(M^n,g)$ endowed with a measure $\mu$ with $C^\infty$ smooth positive density $\exp(-W)$ with respect to the Riemannian volume measure $\vol_g$ is called a weighted Riemannian manifold $(M^n,g,\mu)$. 
\end{definition}

The Levi-Civita connection on $(M,g)$ is denoted by $\nabla$. The Riemannian metric $g$ will often be denoted by $\scalar{\cdot,\cdot}$. It induces a geodesic distance on $(M,g)$, and we denote by $B(x,r)$ an open geodesic ball of radius $r >0$ in $(M,g)$ centered at $x \in M$. Recall that $\mu^k = e^{-W} \H^k$, where $\H^k$ denotes the $k$-dimensional Hausdorff measure, and that $V_\mu = \mu$ denotes the $\mu$-weighted volume.

Recall that $\simplex^{(q-1)}_T = \{ v \in \R^q_+ \; ; \; \sum_{i=1}^q v_i = T \}$ when $T < \infty$ and $\simplex^{(q-1)}_\infty := \R^{q-1}_+ \times \{ \infty\}$. We denote by $\simplex^{(q-1)} = \simplex^{(q-1)}_1$ the $(q-1)$-dimensional probability simplex. Its tangent space is denoted by $E^{(q-1)} := \{ x \in \R^q \; ; \; \sum_{i=1}^q x_i = 0 \}$. For consistency, we define $\Delta^{(0)} = \{1\}$ and $E^{(0)} = \{ 0 \}$.

Throughout this work we will often use the convention that $a_{ij}$ denotes $a_i - a_j$ whenever the individual objects $a_i$ are defined. For instance, if $\{e_i\}_{i=1,\ldots,q}$ denotes the standard basis
in $\R^q$, then $e_{ij} = e_i - e_j \in E^{(q-1)}$. Denoting by $\delta^k_i$ the delta function $\textbf{1}_{k=i}$, we also have $\delta^k_{ij} = \delta^k_i - \delta^k_j$. 

Given distinct $i,j,k \in \{1,\ldots,q\}$, we define the set of cyclically ordered pairs in $\{i,j,k\}$:
\[
\cyclic(i,j,k) := \{ (i,j) , (j,k) , (k,i) \}  .
\]

\subsection{Weighted divergence and mean-curvature}

We write $\div X$ to denote divergence of a smooth vector-field $X$, and $\div_\mu X$ to denote its weighted divergence:
\begin{equation} \label{eq:weighted-div}
\div_{\mu} X := \div (X e^{-W}) e^{+W} = \div X - \nabla_X \pot . 
\end{equation}
For a smooth hypersurface $\Sigma \subset M^n$ co-oriented by a unit-normal field $\n$, let $H_\Sigma$ denote its mean-curvature, defined as the trace of its second fundamental form $\II_{\Sigma}$. The weighted mean-curvature $H_{\Sigma,\mu}$ is defined as:
\[
H_{\Sigma,\mu} := H_{\Sigma} - \nabla_\n \pot .
\]
We write $\div_\Sigma X$ for the surface divergence of a vector-field $X$ defined on $\Sigma$, i.e. $\sum_{i=1}^{n-1} \scalar{\tang_i,\nabla_{\tang_i} X}$ where $\{\tang_i\}$ is a local orthonormal frame on $\Sigma$; this coincides with $\div X - \scalar{\n,\nabla_\n X}$ for any smooth extension of $X$ to a neighborhood of $\Sigma$. 
The weighted surface divergence $\div_{\Sigma,\mu}$ is defined as:
\[
\div_{\Sigma,\mu} X = \div_{\Sigma} X - \nabla_X \pot,
\]
so that $\div_{\Sigma,\mu} X = \div_{\Sigma} (X e^{-\pot}) e^{+\pot}$ if $X$ is tangential to $\Sigma$.  
Note that $\div_{\Sigma} \n = H_{\Sigma}$ and $\div_{\Sigma,\mu} \n = H_{\Sigma,\mu}$.  We will also abbreviate $\scalar{X,\n}$ by $X^\n$, and we will write $X^\tang$ for the tangential part of $X$, i.e. $X - X^{\n} \n$. 

Note that the above definitions ensure the following weighted version of Stokes' theorem: if $\Sigma$ is a smooth $(n-1)$-dimensional manifold with $C^1$ boundary, denoted $\partial \Sigma$, (completeness of $\Sigma \cup \partial \Sigma$ is not required), and $X$ is a smooth vector-field on $\Sigma$, continuous up to $\partial \Sigma$, with compact support in $\Sigma \cup \partial \Sigma$, then since:
\[
\div_{\Sigma,\mu} X = \div_{\Sigma,\mu} (X^\n \n) + \div_{\Sigma,\mu} X^{\tang} = H_{\Sigma,\mu} X^{\n} + \div_{\Sigma,\mu} X^{\tang} ,
\]
then:
\begin{equation} \label{eq:Stokes-classical}
\int_\Sigma \div_{\Sigma,\mu} X d\mu^{n-1} = \int_{\Sigma} H_{\Sigma,\mu} X^{\n} d\mu^{n-1} + \int_{\partial \Sigma} X^{\n_{\partial}} d\mu^{n-2} ,
\end{equation}
where $\n_{\partial}$ denotes the exterior unit-normal to $\partial \Sigma$.

Finally, we denote the surface Laplacian of a smooth function $f$ on $\Sigma$ by $\Delta_{\Sigma} f := \div_{\Sigma} \nabla^\tang f$, which coincides with $\sum_{i=1}^{n-1} \nabla^2_{\tang_i,\tang_i} f - H_{\Sigma} \nabla_{\n} f$ for any smooth extension of $f$ to a neighborhood of $\Sigma$ in $M$. The weighted surface Laplacian is defined as:
\[
 \Delta_{\Sigma,\mu} f := \div_{\Sigma,\mu} \nabla^\tang f = \Delta_{\Sigma} f - \scalar{\nabla^\tang f, \nabla^\tang \pot} . 
 \]

\subsection{Reduced boundary and perimeter}

Given a Borel set $U \subset \R^n$ with locally-finite perimeter, its reduced boundary $\partial^* U$ is defined as the subset of $\partial U$ for which there is a uniquely defined outer unit normal vector to $U$ in a measure theoretic sense (see \cite[Chapter 15]{MaggiBook} for a precise definition). The definition of reduced boundary canonically extends to the Riemannian setting by using a local chart, as it is known that $T(\partial^* U) = \partial^* T(U)$ for any smooth diffeomorphism $T$ (see \cite[Lemma A.1]{KMS-LimitOfCapillarity}). It is known that $\partial^* U$ is a Borel subset of $\partial U$, and that modifying $U$ on a null-set does not alter $\partial^* U$. If $U$ is an open set with $C^1$ smooth boundary, it holds that $\partial^* U = \partial U$  (e.g. \cite[Remark 15.1]{MaggiBook}). Recall that the $\mu$-weighted perimeter of $U$ is defined as:
\[
\per_\mu(U) = \mu^{n-1}(\partial^* U). 
\]

\subsection{Cluster interfaces}

Let $\Omega = (\Omega_1, \dots, \Omega_q)$ denote a $q$-cluster on $(M^n,g,\mu)$. 
Recall that the cells $\set{\Omega_i}_{i=1,\ldots,q}$
of a $q$-cluster $\Omega$ are assumed to be pairwise disjoint Borel subsets of $M^n$ so that $V_\mu(M^n \setminus \cup_{i=1}^q \Omega_i) = 0$ and $V_\mu(\Omega) = (V_\mu(\Omega_1),\ldots,V_\mu(\Omega_q)) \in \Delta^{(q-1)}_{V_\mu(M^n)}$. In addition, they are assumed to have locally-finite perimeter, and moreover, finite $\mu$-weighted perimeter $\per_\mu(\Omega_i) < \infty$.

We define the interface between cells $i$ and $j$ (for $i \ne j$) as:
\[
    \Sigma_{ij} = \Sigma_{ij}(\Omega) := \partial^* \Omega_i \cap \partial^* \Omega_j ,
\]
and set:
\begin{equation} \label{eq:Sigma-Sigma1}
    \Sigma := \cup_{i} \partial \Omega_i ~,~  \Sigma^1 := \cup_{i < j} \Sigma_{ij} .
\end{equation}
 It is standard to show (see \cite[Exercise 29.7, (29.8)]{MaggiBook}) that for any $S \subset \set{1,\ldots,q}$:
\begin{equation} \label{eq:nothing-lost-many}
\H^{n-1}\brac{\partial^*(\cup_{i \in S} \Omega_i) \setminus \cup_{i \in S , j \notin S} \Sigma_{ij}} = 0 .
\end{equation}
In particular: \begin{equation} \label{eq:nothing-lost}
\H^{n-1} \brac{ \partial^* \Omega_i  \setminus \cup_{j \neq i} \Sigma_{ij} } = 0 \;\;\; \forall i=1,\ldots,q ,
\end{equation}
and hence:
\[
\per_\mu(\Omega_i) = \sum_{j \neq i} \mu^{n-1}(\Sigma_{ij}) , 
\]
and:
\[
\per_\mu(\Omega) = \frac{1}{2} \sum_{i=1}^q \per_\mu(\Omega_i) = \sum_{i < j} \mu^{n-1}(\Sigma_{ij})  = \mu^{n-1}(\Sigma^1). 
\]
In addition, it follows (see \cite[(3.6)]{EMilmanNeeman-GaussianMultiBubble}) that:
\begin{equation} \label{eq:top-nothing-lost}
\forall i \;\;\; \overline{\partial^* \Omega_i} = \overline{\cup_{j \neq i} \Sigma_{ij}} .
\end{equation}

\subsection{Existence and Interface-regularity} \label{subsec:prelim-minimizing}

The following theorem is due to Almgren \cite{AlmgrenMemoirs} (see also~\cite[Chapter 13]{MorganBook5Ed} and~\cite[Chapters 29-30]{MaggiBook}). 

\begin{theorem}[Almgren] \label{thm:Almgren}
\hfill
    \begin{enumerate}[(i)]
\item \label{it:Almgren-i}
   If $\mu$ is a probability measure, then for any prescribed $v \in \simplex^{(q-1)}$, an isoperimetric $\mu$-minimizing $q$-cluster $\Omega$ satisfying $V_\mu(\Omega) = v$ exists. 
   \end{enumerate}
    For every isoperimetric minimizing cluster $\Omega$ on $(M^n,g,\mu)$:
    \begin{enumerate}[(i)]
    \setcounter{enumi}{1} 
\item \label{it:Almgren-ii}
  $\Omega$ may and will be modified on a $\mu$-null set (thereby not altering $\set{\partial^* \Omega_i}$) so that all of its cells are open, and so that for every $i$, 
$\overline{\partial^* \Omega_i} = \partial \Omega_i$  and $\mu^{n-1}(\partial \Omega_i \setminus \partial^* \Omega_i) = 0$. 
\item \label{it:Almgren-iii} 
    For all $i \neq j$ the interfaces $\Sigma_{ij} = \Sigma_{ij}(\Omega)$ are a locally-finite union of embedded
     $(n-1)$-dimensional $C^\infty$ manifolds, relatively open in $\Sigma$, and for every $x \in
    \Sigma_{ij}$ there exists $\epsilon > 0$ such that $B(x,\epsilon) \cap
    \Omega_k = \emptyset$ for all $k \neq i,j$ and $B(x,\epsilon) \cap \Sigma_{ij}$ is an embedded
     $(n-1)$-dimensional $C^\infty$ manifold.
\item \label{it:Almgren-density}
    For any compact set $K$ in $M$, there exist constants $\Lambda_K,r_K > 0$ so that:
    \begin{equation} \label{eq:Almgren-density}
    \mu^{n-1}(\Sigma \cap B(x,r)) \leq \Lambda_K r^{n-1} \;\;\; \forall x \in \Sigma \cap K \;\;\; \forall r \in (0,r_K) . 
    \end{equation}
\end{enumerate}
\end{theorem}
Whenever referring to the cells of a minimizing cluster or their topological boundary (and in particular $\Sigma$) in this work, we will always choose a representative such as in Theorem \ref{thm:Almgren} \ref{it:Almgren-ii}. 

\begin{definition} \label{def:bounded}
A cluster $\Omega$ is called bounded if its boundary $\Sigma = \cup_{i=1}^q \partial \Omega_i$ is a bounded set. 
\end{definition}

\begin{definition}[Interface--regular cluster] \label{def:interface-regular}
A cluster $\Omega$ satisfying parts \ref{it:Almgren-ii} and \ref{it:Almgren-iii} of Theorem \ref{thm:Almgren}
is called interface--regular. 
\end{definition}

The definition of interface-regular cluster should not be confused with the stronger definition of regular cluster, introduced below in Definition \ref{def:regular}.  Given an interface--regular cluster, let $\n_{ij}$ be the (smooth) unit normal field along $\Sigma_{ij}$ that points from $\Omega_i$ to $\Omega_j$. We use $\n_{ij}$ to co-orient $\Sigma_{ij}$, and since $\n_{ij} = -\n_{ji}$, note that $\Sigma_{ij}$ and $\Sigma_{ji}$ have opposite orientations. We denote by $\II^{ij}$ the second fundamental form on $\Sigma_{ij}$. 
When $i$ and $j$ are clear from the context, we will simply write $\n$ and $\II = \II_{\Sigma^1}$. We will typically abbreviate $H_{\Sigma_{ij}}$ and $H_{\Sigma_{ij},\mu}$ by $H_{ij}$ and $H_{ij,\mu}$, respectively. 

\begin{definition}[Cluster with locally bounded curvature]
We say that an interface-regular cluster $\Omega$ has locally bounded curvature if $\II_{\Sigma^1}$ is bounded on every compact set. 
\end{definition}

The following density estimate is due to Leonardi \cite[Theorem 3.1]{Leonardi-Infiltration} (cf. \cite[Lemma 2.6]{EMilmanNeeman-TripleAndQuadruple}). 
Recall that the (lower) density of a measurable set $A \subset (M^n,g)$ at a point $p \in M^n$ is defined as:
\[
\Theta(A,p) := \liminf_{r \rightarrow 0+} \frac{\H^n(A \cap B(p,r))}{\H^n(B(p,r))} . 
\]

\begin{lemma}[Leonardi's Infiltration Lemma] \label{lem:infiltration}
Let $\Omega$ be an isoperimetric $\mu$-minimizing $q$-cluster on $(M^n,g,\mu)$, and recall our convention from Theorem \ref{thm:Almgren} \ref{it:Almgren-ii}. There exists a constant $\eps > 0$, depending solely on $n$, so that for any $p \in M^n$, $i=1,\ldots,q$ and connected component $\Omega_i^\ell$ of $\Omega_i$: 
\begin{equation} \label{eq:density}
\Theta(\Omega^\ell_i,p) < \eps \;\; \Rightarrow \;\; p \notin \overline{\Omega^\ell_i} . 
\end{equation}
\end{lemma}

\subsection{First variation information -- stationary clusters}

Given a $C_c^\infty$ vector-field $X$ on $M$, let $F_t$ the associated stationary flow along $X$, defined as the family of $C^\infty$ diffeomorphisms $\set{F_t : M^n \to M^n}_{t \in \R}$ solving the following ODE:
\[ \frac{d}{dt} F_t(x) = X \circ F_t(x) ~,~ F_0 = \Id .
\] Clearly, $F_t(\Omega) = (F_t(\Omega_1), \ldots,F_t(\Omega_q))$ remains a cluster for all $t$. 
Recall that $V_\mu(M^n)$ is assumed to be finite. 
We define the $k$-th variations of weighted volume and perimeter of $\Omega$, $k \geq 1$, as:
\begin{align*}
  \delta_X^k V(\Omega)_i = \delta_X^k V(\Omega_i) & := \left. \frac{d^k}{(dt)^k}\right|_{t=0} V_\mu(F_t(\Omega_i)) \;\;  , \;\;  i = 1,\ldots,q ~,\\
  \delta_X^k A(\Omega) & := \left. \frac{d^k}{(dt)^k}\right|_{t=0} \per_\mu(F_t(\Omega)) .
\end{align*}
By \cite[Lemma 3.3]{EMilmanNeeman-GaussianMultiBubble}, $t \mapsto V_\mu(F_t(\Omega_i))$ ($i=1,\ldots,q$) and $t \mapsto \per_\mu(F_t(\Omega))$ are $C^\infty$ functions in some open neighborhood of $t=0$, and so in particular the above variations are well-defined and finite. 
 Note that $\delta_X^k V(\Omega) \in E^{(q-1)}$. When $\Omega$ is clear from the context, we will simply write $\delta_X^k V$ and $\delta_X^k A$. 

\bigskip

By testing the first-variation of a minimizing cluster, one obtains the following well-known information: 

\begin{lemma}[First-order conditions] \label{lem:first-order-conditions}
  For any isoperimetric minimizing $q$-cluster $\Omega$ on $(M^n,g,\mu)$:
  \begin{enumerate}[(i)]
    \item \label{it:first-order-constant} 
    On each $\Sigma_{ij}$, $H_{ij,\mu}$ is constant. 
    \item \label{it:first-order-cyclic}
    There exists $\lambda \in E^{(q-1)}$ such that $H_{ij,\mu} = \lambda_i - \lambda_j$ for all $i \neq j$; moreover, $\lambda \in E^{(q-1)}$ is unique whenever $V_\mu(\Omega) \in \interior \simplex^{(q-1)}_{V_\mu(M^n)}$. 
        \item \label{it:weak-angles}
    $\Sigma^1$ does not have a boundary in the distributional sense -- for every $C_c^\infty$ vector-field $X$:
      \begin{equation} \label{eq:no-boundary}
        \sum_{i<j} \int_{\Sigma_{ij}} \div_{\Sigma,\mu} X^\tang\, d\mu^{n-1} = 0.
      \end{equation}
        \end{enumerate}
\end{lemma}

\begin{remark}
The physical interpretation of $\lambda_i$ is that of air pressure inside cell $\Omega_i$; the weighted mean-curvature $H_{ij,\mu}$ is thus the pressure difference across $\Sigma_{ij}$. 
\end{remark}
 
\begin{definition}[Stationary Cluster]
An interface-regular $q$-cluster $\Omega$ satisfying the three conclusions of Lemma \ref{lem:first-order-conditions} is called stationary (with Lagrange multiplier $\lambda \in E^{(q-1)}$). \end{definition}

\noindent The following lemma provides an interpretation of $\lambda \in E^{(q-1)}$ as a Lagrange multiplier for the isoperimetric constrained minimization problem:

\begin{lemma} \label{lem:Lagrange}      
Let $\Omega$ be stationary $q$-cluster with Lagrange multiplier $\lambda \in E^{(q-1)}$.
    Then for every $C_c^\infty$ vector-field $X$: 
    \begin{align}
    \nonumber
     \delta^1_X V_i &= \sum_{j \neq i} \int_{\Sigma_{ij}} X^{\n_{ij}}\, d\mu^{n-1} \;\;\; \forall i = 1,\ldots,q \;\; , \\
     \label{eq:1st-var-area}
      \delta^1_X A &= \scalar{\lambda,\delta^1_X V} = \sum_{i<j} H_{ij,\mu} \int_{\Sigma_{ij}} X^{\n_{ij}}\, d\mu^{n-1} .
    \end{align} 
  \end{lemma}

\subsection{Higher codimension regularity} \label{subsec:regularity}

Given a minimizing cluster $\Omega$, recall (\ref{eq:Sigma-Sigma1})
 and observe that $\Sigma = \overline{\Sigma^1}$ by (\ref{eq:top-nothing-lost}) and our convention from Theorem \ref{thm:Almgren} \ref{it:Almgren-ii}. We will require additional information on the higher codimensional structure of $\Sigma$. To this end, define two special cones:
\begin{align*}
    \Y &:= \{x \in E^{(2)}\; ; \; \text{ there exist $i \ne j \in \{1,2,3\}$ with $x_i = x_j = \max_{k \in \{1,2,3\}} x_k$}\} , \\
    \T &:= \{x \in E^{(3)}\; ; \; \text{ there exist $i \ne j \in \{1,2,3,4\}$ with $x_i = x_j = \max_{k \in \{1,2,3,4\}} x_k$}\}.
\end{align*}
Note that $\Y$ consists of $3$ half-lines meeting at the origin in $120^\circ$ angles, and that $\T$ consists of $6$ two-dimensional sectors meeting in threes at $120^{\circ}$ angles along $4$ half-lines, which in turn all meet at the origin in $\cos^{-1}(-1/3) \simeq 109^{\circ}$ angles. The next theorem asserts that on the codimension-$2$ and codimension-$3$ parts of a minimizing cluster, $\Sigma$ locally looks like $\Y \times \R^{n-2}$ and $\T \times \R^{n-3}$, respectively.

\begin{theorem}[Taylor, White, Colombo--Edelen--Spolaor] \label{thm:regularity}
    Let $\Omega$ be a minimizing cluster on $(M^n,g,\mu)$, with $\mu = \exp(-W) \vol_g$ and $W \in C^\infty(M)$. Then there exist $\alpha > 0$
          and sets $\Sigma^2, \Sigma^3, \Sigma^4 \subset \Sigma$ such that:
    \begin{enumerate}[(i)]
        \item \label{it:regularity-union}
        $\Sigma$ is the disjoint union of $\Sigma^1,\Sigma^2,\Sigma^3,\Sigma^4$;             
        \item \label{it:regularity-Sigma2}         
        $\Sigma^2$ is a locally-finite union of embedded $(n-2)$-dimensional $C^{\infty}$ manifolds, and for every $p \in \Sigma^2$ there is a $C^\infty$ diffeomorphism mapping a neighborhood of $p$ in $M$ to a neighborhood of the origin in $E^{(2)} \times \R^{n-2}$, so that $p$ is mapped to the origin and $\Sigma$ is locally mapped to $\Y \times \R^{n-2}$ ; 
        \item \label{it:regularity-Sigma3}
         $\Sigma^3$ is a locally-finite union of embedded $(n-3)$-dimensional $C^{1,\alpha}$ manifolds, and for every $p \in \Sigma^3$ there is a $C^{1,\alpha}$ diffeomorphism mapping a neighborhood of $p$ in $M$ to a neighborhood of the origin in $E^{(3)} \times \R^{n-3}$, so that $p$ is mapped to the origin and $\Sigma$ is locally mapped to $\T \times \R^{n-3}$ ;
         \item \label{it:regularity-dimension}
         $\Sigma^4$ is closed and $\dim_{\H}(\Sigma^4) \leq n-4$.             
    \end{enumerate}
\end{theorem}
\begin{remark}
In the classical unweighted Euclidean setting (when $\Omega$ is a minimizing cluster with respect to the Lebesgue measure in $\R^n$), the case $n=2$ was shown by F.~Morgan in \cite{MorganSoapBubblesInR2} building upon the work of Almgren \cite{AlmgrenMemoirs}, and also follows from the results of  J.~Taylor~\cite{Taylor-SoapBubbleRegularityInR3}. The case $n=3$  was established by Taylor \cite{Taylor-SoapBubbleRegularityInR3} for general $(\M,\eps,\delta)$ sets in the sense of Almgren. 
 When $n \geq 4$, Theorem \ref{thm:regularity} was announced by B.~White \cite{White-AusyAnnouncementOfClusterRegularity,White-SoapBubbleRegularityInRn} for general $(\M,\eps,\delta)$ sets. Theorem \ref{thm:regularity} with part \ref{it:regularity-Sigma3} replaced by $\dim_{\H}(\Sigma^3) \leq n-3$ follows from the work of L.~Simon \cite{Simon-Codimension2Regularity}.
A version of Theorem \ref{thm:regularity} for multiplicity-one integral varifolds in an open set $U \subset \R^n$ having associated cycle structure, no boundary in $U$, bounded mean-curvature and whose support is $(\M,\eps,\delta)$ minimizing, was established by M.~Colombo, N.~Edelen and L.~Spolaor  \cite[Theorem~1.3, Remark~1.4, Theorem~3.10]{CES-RegularityOfMinimalSurfacesNearCones}; in particular, this applies to isoperimetric minimizing clusters in $\R^n$ \cite[Theorem~3.8]{CES-RegularityOfMinimalSurfacesNearCones}. By working in smooth charts and inserting their effect as well as that of the smooth positive density into the excess function, the latter extends to the smooth weighted Riemannian setting --  see the proof of \cite[Theorem~5.1]{EMilmanNeeman-GaussianMultiBubble} for a verification. The work of Naber and Valtorta \cite{NaberValtorta-MinimizingHarmonicMaps} implies that $\Sigma^4$ is actually $\H^{n-4}$-rectifiable and has locally-finite $\H^{n-4}$ measure, but we will not require this here. 
\end{remark}
\begin{remark} \label{rem:extend-to-Riemannian}
 In the aforementioned references, \ref{it:regularity-Sigma2} is established with only $C^{1,\alpha}$ regularity, but elliptic regularity for systems of PDEs and a classical reflection argument of Kinderlehrer--Nirenberg--Spruck \cite{KNS} allows to upgrade this to the stated $C^\infty$ regularity -- see \cite[Corollary 5.6 and Appendix F]{EMilmanNeeman-GaussianMultiBubble} for a proof on $(\R^n,|\cdot|,\mu)$. The latter argument extends to the Riemannian setting by working in a smooth local chart and inserting the effect of the Riemannian metric into the system of PDEs for the constant (weighted) mean curvatures of $\Sigma_{ij}$, $\Sigma_{jk}$ and $\Sigma_{ki}$, which are already in quasi-linear form. 
\end{remark}

 \begin{definition}[Regular Cluster] \label{def:regular}
An interface-regular cluster $\Omega$ (recall Definition \ref{def:interface-regular}) satisfying the conclusion of Theorem \ref{thm:regularity} and in addition the density estimates (\ref{eq:Almgren-density}) and (\ref{eq:density}) for some $\eps = \eps(\Omega) > 0$ (possibly depending on $\Omega$) is called regular. 
 \end{definition}
 
 \begin{remark}
In \cite{EMilmanNeeman-TripleAndQuadruple}, the definition of regularity also included the assumption that $V_\mu(\Omega) \in \interior \Delta^{(q-1)}_{\mu(M)}$, namely that all (open) cells are non-empty. This assumption was added in order to avoid pathological cases during the proofs. 
However, since all empty cells may always be removed (by reducing $q$), all of the results quoted below from \cite{EMilmanNeeman-TripleAndQuadruple} remain valid without this assumption. 
 \end{remark}

It follows that for a regular cluster, every point in $\Sigma^2$ (called the \emph{triple-point set}) belongs to the closure of exactly three cells, as well as to the closure of exactly three interfaces. Given distinct $i,j,k$, we will write $\Sigma_{ijk}$ for the subset of $\Sigma^2$ which belongs to the closure of $\Omega_i$, $\Omega_j$ and $\Omega_k$, or equivalently, to the closure of $\Sigma_{ij}$, $\Sigma_{jk}$ and $\Sigma_{ki}$. It follows that $\Sigma^2$ is the disjoint union of $\Sigma_{ijk}$ over all $i < j < k$. 
We will extend the normal fields $\n_{ij}$ to $\Sigma_{ijk}$ by continuity (thanks to $C^1$ regularity). 

\smallskip

\begin{lemma} \label{lem:Sigma2}
Let $\Omega$ be a regular cluster on $(M,g,\mu)$ with locally bounded curvature. Then for any compact subset $K \subset M$:
\begin{equation} \label{eq:Sigma2-finite}
 \mu^{n-2}(\Sigma^2 \cap K) < \infty .
 \end{equation}
\end{lemma}
\noindent The analogous estimate to (\ref{eq:Sigma2-finite}) for $\Sigma^1$ holds on any regular cluster by definition. 

\smallskip

Let us also denote:
\begin{equation} \label{eq:partial-Sigmaij}
\partial \Sigma_{ij} := \bigcup_{k \neq i,j} \Sigma_{ijk}.
\end{equation}
Note that $\Sigma_{ij} \cup \partial \Sigma_{ij}$ is a (possibly incomplete) $C^\infty$ manifold with boundary.
We define $\n_{\partial ij}$ on $\partial \Sigma_{ij}$ to be the outward-pointing unit boundary-normal to $\Sigma_{ij}$. We may extend the second fundamental form $\II^{ij}$ by continuity to $\partial \Sigma_{ij}$. 
We will abbreviate $\II^{ij}_{\partial \partial}$ for $\II^{ij}(\n_{\partial ij}, \n_{\partial ij})$ on $\partial \Sigma_{ij}$.
When $i$ and $j$ are clear from the context, we will write $\n_\partial$ for $\n_{\partial ij}$.

\subsection{Angles at triple-points} 

\begin{lemma}\label{lem:boundary-normal-sum}
    For any stationary regular cluster:
    \begin{equation} \label{eq:sum-n-zero}
    \sum_{(i,j) \in \cyclic(u,v,w)} \n_{ij} = 0 ~,~ \sum_{(i,j) \in \cyclic(u,v,w)} \n_{\partial ij} = 0  \;\;\;\;\; \forall p \in \Sigma_{uvw} . 
    \end{equation}
      In other words, $\Sigma_{uv}$, $\Sigma_{vw}$ and $\Sigma_{wu}$ meet at $\Sigma_{uvw}$ in $120^{\circ}$ angles. 
\end{lemma}

In fact, we have:
\begin{lemma} \label{lem:equivalent-stationary}
Let $\Omega$ be a regular $q$-cluster with locally bounded curvature on $(M^n,g,\mu)$. Then $\Omega$ is stationary with Lagrange multiplier $\lambda \in E^{(q-1)}$ if and only if the following two properties hold:
\begin{enumerate}[(1)]
\item \label{it:stationarity-k}
For all $i < j$, $H_{\Sigma_{ij},\mu}$ is constant and equal to $\lambda_i - \lambda_j$.
\item \label{it:stationarity-n} 
$\sum_{(i,j) \in \cyclic(u,v,w)} \n_{ij} = 0$ for all $p \in \Sigma_{uvw}$ and $u < v < w$. 
\end{enumerate}
\end{lemma}

We will frequently use that for all $p \in \Sigma_{ijk}$:
\begin{equation} \label{eq:sqrt3}
\n_{\partial ij} = \frac{\n_{ik} + \n_{jk}}{\sqrt{3}} .
\end{equation}
We also denote:
\begin{equation} \label{eq:def-II-partial}
\bar \II^{\partial ij} := \frac{\II^{ik}(\n_{\partial ik}, \n_{\partial ik}) + \II^{jk}(\n_{\partial jk}, \n_{\partial jk})}{\sqrt{3}} = \frac{ \II^{ik}_{\partial \partial} + \II^{jk}_{\partial \partial}}{\sqrt{3}} . 
\end{equation}
As usual, when integrating on $\partial \Sigma_{ij}$, we may abbreviate $\bar \II^{\partial} = \bar \II^{\partial ij}$.

\subsection{Second variation information -- stable clusters and index-form} \label{subsec:Q}

\begin{definition}[Index-Form $Q$] 
The Index Form $Q$ associated to a stationary $q$-cluster $\Omega$ on $(M,g,\mu)$ with Lagrange multiplier $\lambda \in E^{(q-1)}$ is defined as the following quadratic form on $C_c^\infty$ vector-fields $X$ on $(M,g)$:
\[
Q(X) := \delta^2_X A - \scalar{\lambda,\delta^2_X V} .
\]
\end{definition}

\begin{lemma}[Stability]\label{lem:unstable}
For any isoperimetric minimizing cluster $\Omega$ and $C_c^\infty$ vector-field $X$:
    \begin{equation}  \label{eq:Q}
      \delta^1_X V = 0 \;\;\; \Rightarrow \;\;\; Q(X) \ge 0 .
     \end{equation}
\end{lemma}

\begin{definition}[Stable Cluster]
A stationary cluster satisfying the conclusion of Lemma \ref{lem:unstable} is called stable. \end{definition}

It will be convenient to polarize $Q(X) = Q(X,X)$ as a symmetric bilinear form $Q(X,Y)$.  
The following formula for $Q(X,Y)$ may be derived under favorable conditions. 

\begin{definition}[Weighted Ricci Curvature]
The weighted Ricci curvature on $(M,g,\mu = \exp(-W) d\vol_g)$ is defined as the following symmetric $2$-tensor:
\[
\Ric_{g,\mu} := \Ric_g + \nabla^2 W . 
\]
\end{definition}

\begin{theorem} \label{thm:Q-Sigma4}
Let $\Omega$ be a stationary regular cluster on $(M^n,g,\mu)$ with locally bounded curvature. Then for any $C_c^\infty$ vector-fields $X,Y$ on $M$, $Q(X) = Q(X,X)$, where $Q(X,Y)$ denotes the following symmetric bilinear Index-Form for vector-fields:
\begin{align}
\nonumber Q(X,Y) := \sum_{i<j} \Big [ &  \int_{\Sigma_{ij}} \brac{ \scalar{\nabla^\tang X^\n,\nabla^\tang Y^\n} -(\Ric_{g,\mu}(\n,\n) + \|\II\|_2^2)  X^\n Y^\n} d\mu^{n-1}  \\
\label{eq:Q1} & - \int_{\partial \Sigma_{ij}} X^\n Y^\n \bar \II^{\partial ij} \, d\mu^{n-2}\Big] .
\end{align}
\end{theorem}

\subsection{Jacobi operator $L_{Jac}$}

\begin{definition}[Jacobi operator $L_{Jac}$]
Let $\Sigma$ be a smooth hypersurface on a weighted Riemannian manifold $(M,g,\mu)$ with unit-normal $\n$. The associated Jacobi operator $L_{Jac}$ acting on smooth functions $f \in C^\infty(\Sigma)$ is defined as:
\begin{equation} \label{eq:def-LJac}
L_{Jac} f := \Delta_{\Sigma,\mu} f + (\Ric_{g,\mu}(\n,\n) + \|\II\|_2^2) f . 
\end{equation}
\end{definition}

\begin{theorem} \label{thm:Q-LJac}
Let $\Omega$ be a stationary regular cluster on $(M,g,\mu)$ with locally bounded curvature.
Then for any $C_c^\infty$ vector-fields $X,Y$ on $M$:
\begin{equation} \label{eq:Q1-LJac}
Q(X,Y) = \sum_{i<j} \Big[ - \int_{\Sigma_{ij}} X^{\n} L_{Jac} Y^{\n} d\mu^{n-1} + \int_{\partial \Sigma_{ij}} X^{\n} \brac{\nabla_{\n_{\partial}} Y^{\n} -  \bar \II^{\partial} Y^{\n}  } d\mu^{n-2} \Big] .
\end{equation}
\end{theorem}

\begin{remark}
It is well-known and classical in the unweighted setting that the first variation of mean-curvature of a hypersurface $\Sigma$ in the normal direction is given by the unweighted Jacobi operator, and this easily extends to the weighted setting:
\[
\delta^1_{\varphi \n} H_{\Sigma,\mu} = - L_{Jac} \; \varphi . 
\]
Consequently, on a hypersufrace $\Sigma$ with constant weighted mean-curvature $H_{\Sigma,\mu}$, we have:
\begin{equation} \label{eq:LJac-deltaH}
\delta^1_X H_{\Sigma_{ij},\mu} = -L_{Jac} X^\n + \delta^1_{X^\tang} H_{\Sigma_{ij},\mu} =  -L_{Jac} X^\n . 
\end{equation}
Using this, it is easy to obtain a heuristic derivation of (\ref{eq:Q1-LJac}) by differentiating inside the integral in the formula for the first variation of area (\ref{eq:1st-var-area}). However, as explained in \cite[Section 1]{EMilmanNeeman-GaussianMultiBubble}, a rigorous derivation without a-priori assuming that $\Omega$ has locally bounded curvature is much more challenging. 
\end{remark}

\subsection{On non-oriented functions}

On some occasions, we will want to apply $L_{Jac}$ to functions $\Psi,\Phi$ defined on the entire $M$, which do not arise as some oriented scalar-fields $f_{ij} = -f_{ji}$. In that case, it is useful to note that:
\begin{lemma} \label{lem:LJac-non-oriented}
Let $\Omega$ be a stationary regular cluster on $(M,g,\mu)$ with locally bounded curvature. Then for all smooth functions $\Psi,\Phi$ on $M$:
\[
\int_{\Sigma^1}\Phi L_{Jac} \Psi d\mu^{n-1} = \int_{\Sigma^1} \Psi L_{Jac} \Phi d\mu^{n-1} , 
\]
and:
\[
\int_{\Sigma^1} \Delta_{\Sigma,\mu} \Phi d\mu^{n-1} = 0 . 
\]
\end{lemma}
\begin{proof}
Both claims are clear if we are allowed to freely integrate by parts, since then things boil down to showing:
\[
\sum_{u<v<w} \int_{\Sigma_{uvw}} \sum_{(i,j) \in \cyclic(u,v,w)} (\Phi \nabla_{\n_{\partial ij}} \Psi - \Psi \nabla_{\n_{\partial ij}} \Phi) d\mu^{n-2} = 0 ,
\]
and:
\[
\sum_{u<v<w} \int_{\Sigma_{uvw}} \sum_{(i,j) \in \cyclic(u,v,w)} \nabla_{\n_{\partial ij}} \Phi d\mu^{n-2} = 0 ,
\]
respectively, which both hold since $\sum_{(i,j) \in \cyclic(u,v,w)} \n_{\partial ij} = 0$ by (\ref{eq:sum-n-zero}). To justify integration by parts, we refer to \cite[Lemma 2.25]{EMilmanNeeman-TripleAndQuadruple} applied to
the vector-field $Z_{ij} = X^\tang$ on $\Sigma_{ij}$ for $X \in \{\nabla \Phi, \nabla \Psi\}$. Note that $\div_{\Sigma,\mu} X^\tang = \div_{\mu} X - \scalar{\n, \nabla_{\n} X} - X^\n H_{\Sigma,\mu}$, and so as $H_{\Sigma,\mu}$ is constant by stationarity, the latter expression is bounded and thus $\mu$-integrable on $\Sigma_{ij}$; also note that 
$\int_{\partial \Sigma_{ij}} |X^{\n_\partial}| d\mu^{n-2} < \infty$, which is a consequence of the locally bounded curvature and (\ref{eq:Sigma2-finite}). 
\end{proof}

\subsection{Almost piecewise-constant fields}

We denote for brevity $\Sigma^{\ge 3} = \Sigma^3 \cup \Sigma^4$. 

\begin{lemma}\label{lem:inward-fields}
    Let $\Omega$ be a stationary regular $q$-cluster on $(M,g,\mu)$.
    There is a family of $C^\infty$ vector-fields $Z_1, \dots, Z_q$ defined 
    on $M \setminus \Sigma^{\ge 3}$ such that for every $k = 1, \ldots, q$, for every $i \ne j \in \{ 1 , \ldots, q\}$,
    and for every $p \in \Sigma_{ij}$,
    \[
        |Z_k(p)| \le \sqrt {3/2} \text{ and } \scalar{Z_k(p),\n_{ij}(p)} = \delta^k_{ij} .
    \]
\end{lemma}
\begin{proof}[Proof of Lemma \ref{lem:inward-fields}]
Lemma \ref{lem:inward-fields} was proved in \cite[Section 7]{EMilmanNeeman-GaussianMultiBubble} for stationary regular clusters in Euclidean space $(\R^n,|\cdot|,\mu)$. The sign convention employed there was opposite to the one we use in this work, and so these fields were called ``inward fields" in \cite{EMilmanNeeman-GaussianMultiBubble}. 
Since the construction of the approximate inward fields in \cite{EMilmanNeeman-GaussianMultiBubble} is entirely local, by using a partition of unity and working in charts on the smooth manifold $M$, the exact same proof carried over to the Riemannian setting. \end{proof}

\subsection{Linear Algebra}

\begin{lemma}[Graph connectedness]\label{lem:LA-connected}
Let $\Omega$ be a interface-regular $q$-cluster on $(M^n,g,\mu)$ with $V_\mu(\Omega) \in \interior \simplex^{(q-1)}_{V_\mu(M^n)}$. Consider the undirected graph $G$ with vertices $\{1, \dots, q\}$ and an edge between $i$ and $j$ iff $\Sigma_{ij} \neq \emptyset$. 
Then the graph $G$ is connected.
\end{lemma}

\begin{definition}[Quadratic form $\L_A$] \label{def:LA}
Given real-valued weights $A = \{ A^{ij} \}_{i,j =1,\ldots,q}$ which are non-oriented (i.e. satisfy $A^{ij} = A^{ji}$), define the following $q$ by $q$ symmetric matrix called the discrete (weighted) Laplacian:
\begin{equation} \label{eq:L_A}
        \L_A := \sum_{1 \leq i<j \leq q} A^{ij} e_{ij} \otimes e_{ij} .
\end{equation}
We will mostly consider $\L_A$ as a quadratic form on $E^{(q-1)}$. \\
Given an interface-regular $q$-cluster $\Omega$ on $(M^n,g,\mu)$ and a family $f = \{f_{ij}\}_{i,j = 1,\ldots,q}$ of non-oriented integrable functions $f_{ij} = f_{ji}$ on $\Sigma_{ij}$, we define $\L_f := \L_A$ for $A^{ij} = \int_{\Sigma_{ij}} f_{ij} d\mu^{n-1}$. In particular, $\L_1 := \L_{\{\mu^{n-1}(\Sigma_{ij})\}}$. \end{definition}

Given an undirected graph $G$ on $\{1,\ldots,q\}$, we will say that the weights $A = \{ A^{ij} \}_{i,j=1,\ldots,q}$ are supported along the edges of $G$, and write $A = \{ A^{ij} \}_{i \sim j}$, if $A^{ij} = 0$ whenever there is no edge in $G$ between vertices $i$ and $j$. 

\begin{lemma}[$L_A$ is positive-definite] \label{lem:LA-positive}
 For all strictly positive non-oriented weights $A = \{ A^{ij} > 0\}_{i \sim j}$ supported along the edges of a connected graph $G$, $\L_A$ is positive-definite as a quadratic form on $E^{(q-1)}$ (in particular, it has full-rank $q-1$). 
\end{lemma}

\section{Spherical Voronoi clusters} \label{sec:Voronoi}

\subsection{Definitions}

\begin{definition}[Quasi-center $\c$]
Given an oriented smooth hypersurface $\Sigma$ on $\S^n \subset \R^{n+1}$, we define the quasi-center vector $\c \in \R^{n+1}$ at $p \in \Sigma$, as:
\begin{equation} \label{eq:cnk}
\c := \n - \k p ,
\end{equation}
where $\k = \frac{H_{\Sigma}}{n-1}$ is the normalized mean-curvature with respect to the unit-normal $\n$ at $p$. 
\end{definition}

\begin{remark} \label{rem:quasi-center}
Whenever $\Sigma$ is a subset of a geodesic sphere $S$ (of curvature $\k$ with respect to $\n$), we have $S = \S^n \cap E(\c,\k)$ where $E(\c,\k) := \{ p \in \R^{n+1} \; ; \; \scalar{\c,p} + \k = 0\}$. Note that this representation of the affine hyperplane $E(\c,\k)$ is not unique; however, it is uniquely determined (up to orientation) by the property that $|\c|^2 = |\n - \k p|^2 = 1 + \k^2$, and hence the signed distance of $E(\c,\k)$ from the origin is $\k / \sqrt{1 + \k^2}$. In that case, $\c$ remains constant on $S$, and we call $\c$ the quasi-center of $S$. 
\end{remark}

\begin{definition}[Affine Voronoi Cluster] \label{def:generated-Voronoi}
Given $\{ \c_i \}_{i=1,\ldots,q} \subset \R^{n+1}$ and $\{ \k_i \}_{i=1,\ldots,q} \subset \R$ 
so that $\{(\c_i,\k_i)\}_{i=1,\ldots,q}$ are distinct, the affine Voronoi $q$-cluster $\Omega$ on $\S^n$ generated by the aforementioned parameters is given by:
\begin{equation} \label{eq:Voronoi-rep}
\Omega_i  := \interior \underline \Omega_i ~,~ \underline \Omega_i := \set{ p \in \S^n \; ; \; \argmin_{j=1,\ldots,q} \scalar{\c_j , p} + \k_j \ni i } =  \bigcap_{j \neq i} \; \set{ p \in \S^n \; ;\; \scalar{\c_{ij},p} + \k_{ij} \leq 0 } .
\end{equation}
As this definition is invariant under translation of $\{ \c_i \}$ and of $\{ \k_i \}$, we will always employ the convention that $\sum_{i=1}^q \c_i = 0$ and $\sum_{i=1}^q \k_i = 0$. Consequently, $\k \in E^{(q-1)}$, and we denote:
\begin{equation} \label{eq:C}
 \C := \sum_{i=1}^q e_i \otimes \c_i : \R^{n+1} \rightarrow E^{(q-1)} . 
\end{equation}
In matrix form, $\C$ is the matrix whose rows are given by $\{ \c_i \}_{i=1,\ldots,q}$. 
\end{definition}

Recall our notation $\c_{ij} = \c_i - \c_j$, $\k_{ij} = \k_i - \k_j$, used above. Clearly an affine Voronoi cluster $\Omega$ is a legal cluster: $\Omega_i$ and $\Omega_j$ for distinct $i,j$ are pairwise disjoint, lying on either side of the (possibly degenerate, possibly empty) geodesic sphere
\begin{equation} \label{eq:Sij}
S_{ij} := \{ p \in \S^n \; ; \; \scalar{\c_{ij},p} + \k_{ij} = 0 \} ,
\end{equation}
and their union covers the entire $\S^n$ up to a null-set since $\S^n \setminus \cup_{i} \Omega_i \subset \cup_{i<j} S_{ij}$
and $(\c_{ij},\k_{ij}) \neq (0,0)$ as $\{(\c_i,\k_i)\}$ are distinct. In particular $\cup_i \overline{\Omega_i} = \S^n$. As usual, the cluster's interfaces are defined as $\Sigma_{ij} = \partial^* \Omega_i \cap \partial^* \Omega_j$ for distinct $i,j \in \{1,\ldots,q\}$, and we set $\Sigma^1 = \cup_{i<j} \Sigma_{ij}$. 

Note that a non-empty interface $\Sigma_{ij}$ of an affine Voronoi cluster lies on the geodesic sphere $S_{ij}$ given by (\ref{eq:Sij}). However, as eluded to in Remark \ref{rem:quasi-center}, there is a-priori no guarantee that $\k_{ij}$ is precisely the curvature of $S_{ij}$ (in particular, it could be that $S_{ij}$ is a degenerate sphere consisting of a single point). Adding this requirement, we obtain the definition of a spherical Voronoi cluster from the Introduction:

\begin{definition}[Spherical Voronoi Cluster (again)] \label{def:spherical-Voronoi}
An affine Voronoi cluster $\Omega$ on $\S^n$ generated by the parameters $\{ \c_i \}_{i=1,\ldots,q} \subset \R^{n+1}$ and $\{ \k_i \}_{i=1,\ldots,q} \subset \R$ is called a spherical Voronoi cluster if
\begin{enumerate}[(*)]
\item \label{it:spherical-Voronoi} 
for every non-empty interface $\Sigma_{ij} \neq \emptyset$, $\Sigma_{ij}$ is a subset of the geodesic sphere $S_{ij}$ having curvature $\k_{ij} = \k_i - \k_j$ and quasi-center $\c_{ij} = \c_i - \c_j$. 
\end{enumerate}
The vectors $\{ \c_i \}_{i=1,\ldots,q}$ and scalars $\{ \k_i \}_{i=1,\ldots,q}$ are called the cluster's quasi-center and curvature parameters, respectively, $\C : \R^{n+1} \rightarrow E^{(q-1)}$ in (\ref{eq:C}) is called its quasi-center operator, and $\k \in E^{(q-1)}$ is called its curvature vector. 
\end{definition}

\begin{remark} 
In \cite{EMilmanNeeman-TripleAndQuadruple}, a $q$-cluster on $\S^n$ was called spherical Voronoi if there exist parameters $\{ \c_i \}_{i=1,\ldots,q} \subset \R^{n+1}$ and $\{ \k_i \}_{i=1,\ldots,q} \subset \R$ so that \ref{it:spherical-Voronoi} holds after setting:
\begin{equation} \label{eq:old-def}
\Omega_i  := \set{ p \in \S^n \; ; \; \argmin_{j=1,\ldots,q} \scalar{\c_j , p} + \k_j = \{i\} } =  \bigcap_{j \neq i} \; \set{ p \in \S^n \; ;\; \scalar{\c_{ij},p} + \k_{ij} < 0 } ,
\end{equation}
employing a strict inequality on the right-hand-side of (\ref{eq:old-def}) instead of setting $\Omega_i$ to be the interior of $\underline \Omega_i$ as in (\ref{eq:Voronoi-rep}). Furthermore, the definition in  \cite{EMilmanNeeman-TripleAndQuadruple} did not a-priori require that the parameters $\{ (\c_i,\k_i) \}_{i=1,\ldots,q}$ be distinct as in Definition \ref{def:generated-Voronoi}; however, repeated parameters will just yield additional empty cells according to (\ref{eq:old-def}), and those may be removed. Clearly, the cluster given by (\ref{eq:old-def}) is contained (cell-wise) in the one given by (\ref{eq:Voronoi-rep}), and since both are $q$-clusters (by assumption and by definition, respectively), their difference must be a null-set. 
In particular, since both variants yield open cells, a cell is empty according to one variant iff it is empty according to the other. Moreover, it is easy to see that the two variants coincide whenever all cells are non-empty (according to either definition) -- see the proof of \cite[Proposition 8.14]{EMilmanNeeman-TripleAndQuadruple}. 
Consequently, the formulation of Theorem \ref{thm:intro-structure} from \cite{EMilmanNeeman-TripleAndQuadruple} which assumes that all cells are non-empty (as $V(\Omega) \in \interior \Delta^{(q-1)}$) remains valid with the definition of spherical Voronoi cluster which we employ in this work. 
\end{remark}

\begin{remark} \label{rem:quasi-centers}
It follows from Remark \ref{rem:quasi-center} that condition \ref{it:spherical-Voronoi} is equivalent to the requirement that $|\c_{ij}|^2 = 1 + \k_{ij}^2$ for all $i < j$ such that $\Sigma_{ij} \neq \emptyset$.
\end{remark}

In view of Remark \ref{rem:quasi-centers}, we record:
\begin{lemma} \label{lem:CCT}
An affine Voronoi cluster is spherical Voronoi if and only if:
\[
\scalar{\C \C^T, e_{ij} \otimes e_{ij}} = \scalar{\frac{1}{2} \Id_{E^{(q-1)}} + \k \k^T , e_{ij} \otimes e_{ij}} 
\]
for all $1 \leq i < j \leq q$ so that $\Sigma_{ij} \neq \emptyset$. 
\end{lemma}

Let $E^I := \{ x \in \R^I \; ; \; \sum_{i \in I} x_i = 0\}$, extending our notation for $E^{(q-1)} = E^{\{1,\ldots,q\}}$. 
We use $\arank(A)$ to denote the affine row rank of a matrix $A$, or of a linear operator $A : V \rightarrow E^{I}$.

\begin{definition}[Full-dimensional cluster]
A spherical Voronoi $q$-cluster $\Omega$ on $\S^n$ is called full-dimensional if:
\[
\arank(\C) = \min(q-1,n+1) . 
\]
\end{definition}

\begin{definition}[Perpendicular Spherical Voronoi Cluster, North Pole, Equator, Equatorial Cells] 
See Definition \ref{def:intro-perpendicular} from the Introduction.
\end{definition}

\begin{remark}
Note that being perpendicular, namely the existence of a North pole $N \in \S^n$ perpendicular to the affine span of $\{\c_i\}_{i=1,\ldots,q}$, is equivalent to having $\arank(\C) < n+1$. Consequently, a spherical Voronoi $q$-cluster with $q \leq n+1$ is always perpendicular. Also note that an equatorial cell is necessarily non-empty. 
\end{remark}

\begin{definition}[$\S^m$-symmetry]
A cluster $\Omega$ on $\S^n$ is said to have $\S^m$-symmetry ($m \in \{ 0, \ldots,n-1\}$), if there exists a totally-geodesic $\S^{n-1-m} \subset \S^n$ so that each cell $\Omega_i$ is invariant under all isometries of $\S^n$ which fix the points of $\S^{n-1-m}$. In particular, $\S^0$-symmetry means invariance of all cells under reflection about some hyperplane $N^{\perp}$.
\end{definition}

\begin{remark} \label{rem:symmetry}
A spherical Voronoi cluster $\Omega$ on $\S^n$ is clearly invariant under all isometries of $\S^n$ which fix the affine span of $\{\c_i\}_{i=1,\ldots,q}$. Consequently, if $r = \arank(\C) < n+1$ then $\Omega$ has $\S^m$-symmetry with $m = n-r$; in particular, when $q \leq n+1$, $\Omega$ has $\S^{n+1-q}$-symmetry. Note that being perpendicular implies having $\S^0$-symmetry, but not vice versa. 
\end{remark}

\subsection{Continuity in generating parameters} 

Recall the definition of the closed pre-cells $\underline \Omega_i$ from (\ref{eq:Voronoi-rep}), and observe that:
\begin{lemma} \label{lem:isolated}
Let $\Omega$ be an affine Voronoi cluster, and assume that $\Omega_i \neq \emptyset$. Then $\underline \Omega_i \setminus \overline \Omega_i$ consists of a finite number (possibly zero) of isolated points of $\underline \Omega_i$. In particular, $\underline \Omega_i = \overline \Omega_i$ if and only if $\underline \Omega_i$ has no isolated points. 
\end{lemma}
\begin{proof}
If $p \in \underline \Omega_i \setminus \overline \Omega_i$, we will show that $p$ is an isolated point in  $\underline \Omega_i$ (the converse direction is trivial). 
Recall that $\underline \Omega_i = \S^n \cap P_i$ where $P_i$ is a closed convex polyhedron in $\R^{n+1}$. Since $\Omega_i = \interior \underline \Omega_i$ is assumed non-empty, $P_i$ must have non-empty interior (otherwise $\S^n \cap P_i$ will have empty relative interior $\Omega_i$ in $\S^n$, as $\S^n$ is strictly positively curved). 

Let $\B^{n+1}$ denote the closed Euclidean unit-ball in $\R^{n+1}$. It cannot be that $P_i \cap \B^{n+1} = \{p\}$ since then $\underline \Omega_i = \{p\}$ and $\Omega_i  = \emptyset$. Consequently, $P_i \cap \B^{n+1}$ is a closed, convex set which must have non-empty interior (again, as $\S^n$ is strictly positively curved) -- assume that  $P_i \cap \B^{n+1}$  contains a small open ball $B(x_0,\eps)$. It is impossible for $P_i$ to contain a point $p+v$ for some $v \in T_p \S^n \setminus \{0\}$, since otherwise the (full-dimensional) convex-hull of $[p,p+v]$ with $B(x_0,\eps)$ would lie in $P_i$ (by convexity), intersect $\S^n$, and bear witness to having $p \in \overline \Omega_i$. Consequently, $P_i \cap (p + T_p \S^n) = \{p\}$, and so $p$ must be one of the finitely many vertices of $P_i$, having inward tangents forming a strictly obtuse angle with $p$, and hence $p$ must be an isolated point of $\underline \Omega_i$. 
\end{proof}

\begin{corollary} \label{cor:isolated}
Let $\Omega$ be an affine Voronoi cluster. Then the condition $\overline \Omega_i = \underline \Omega_i$ is open as a function of the generating parameters $\{\c_i\}_{i=1,\ldots,q}$ and $\{\k_i\}_{i=1,\ldots,q}$ of $\Omega$.
\end{corollary}
\begin{proof}
Recall that $\underline \Omega_i = \S^n \cap P_i$ with $P_i := \cap_{j \neq i} \{ x \in \R^{n+1} \; ; \; \scalar{x,\c_{ij}} + \k_{ij} \leq 0\}$. If $\Omega_i = \emptyset$ then $\overline \Omega_i = \underline \Omega_i$ implies that $\underline \Omega_i = \emptyset$, which is clearly an open condition. Otherwise, $\Omega_i \neq \emptyset$, which is clearly open as well. In that case, by the previous lemma, the condition $\overline \Omega_i = \underline \Omega_i$ is equivalent to $\underline \Omega_i$ having no isolated points, and since each isolated point must arise from the intersection of $\S^n$ with one of $P_i$'s finite number of vertices, it follows that this is an open condition. 
\end{proof}

For perpendicular clusters, $\underline \Omega_i$ cannot have isolated points: 

\begin{lemma} \label{lem:perp}
For a perpendicular affine Voronoi cluster $\Omega$ on $\S^n$, $\Omega_i \neq \emptyset$ implies $\overline \Omega_i = \underline \Omega_i$.
\end{lemma}

\begin{proof}
Let $N \in \S^n$ be a North pole, and consider the orthogonal projection $\Pi$ of $\S^n$ onto $N^{\perp}$; denote by $\B^n = \Pi \S^n$ the closed Euclidean unit-ball. Being perpendicular implies that $\underline \Omega_i = \S^n \cap P_i$ with $P_i$ a closed convex polyhedron perpendicular to $N^{\perp}$. If $\Omega_i \neq \emptyset$ then $K_i = \Pi \underline \Omega_i = P_i \cap \B^n$ is closed convex, and has non-empty interior (because $\Omega_i$ is open and non-empty). It follows that $K_i$ has no isolated points, and so neither does $\underline \Omega_i = \Pi^{-1} K_i$, and the assertion follows by Lemma \ref{lem:isolated}. 
\end{proof}

We finally conclude:

\begin{lemma} \label{lem:cont}
For a perpendicular affine Voronoi cluster, the closures of the non-empty cells $\overline{\Omega_i}$ and interfaces $\overline{\Sigma_{ij}}$ are locally continuous in the Hausdorff metric as a function of the generating parameters $\{\c_i\}_{i=1,\ldots,q}$ and $\{\k_i\}_{i=1,\ldots,q}$. In addition, the volumes $V(\Omega_i)$ and perimeters $\H^{n-1}(\Sigma_{ij})$ of the non-empty cells and interfaces are locally Lipschitz continuous as a function of the generating parameters. 
\end{lemma}
\begin{proof}
Since for every non-empty cell, $\overline \Omega_i = \underline \Omega_i$ in a neighborhood of $\Omega$ by Lemma \ref{lem:perp} and Corollary \ref{cor:isolated}, the assertion regarding the cells is immediate after recalling that $\underline \Omega_i$ is defined as the intersection with $\S^n$ of the closed convex polyhedron $P_i := \cap_{j \neq i} \{ x \in \R^{n+1} \; ; \; \scalar{x,\c_{ij}} + \k_{ij} \leq 0\}$. The assertion regarding the interfaces is proved identically since under our assumptions a non-empty interface satisfies $\overline{\Sigma_{ij}} = \underline \Omega_i \cap \underline \Omega_j$, and since
$\underline \Omega_i \cap \underline \Omega_j = \S^n \cap P_i \cap P_j$. 
\end{proof}

\begin{remark}
By taking a polyhedron $P_i$ having a vertex at $p \in \S^n$ with inward tangents forming a strictly obtuse angle with $p$, one may check that the claim in Lemma \ref{lem:cont} about Hausdorff continuity is false without the perpendicularity assumption. 
\end{remark}

\subsection{Non-degenerate clusters}

In principle, we will only be interested in regular spherical Voronoi clusters, as these are the only ones which can arise as isoperimetric minimizers by the results of Section \ref{sec:prelim}. However, in order to make the regularity property more tangible for spherical Voronoi clusters, let us see what can go wrong. 

In the class of affine Voronoi clusters, it is easy to construct an example which is not even interface-regular (recall Definition \ref{def:interface-regular}). Indeed, consider the (perpendicular) affine Voronoi $3$-cluster on $\S^2$ with $\c_1 = (-1,1,0)$, $\c_2 = (1,1,0)$, $\c_3 = (0,0,0)$ and $\k_1 = 1$, $\k_2 = 1$ and $\k_3 = 0$, depicted in Figure \ref{fig:cushion}. Observe that the cluster's orthogonal projection to the equatorial plane $\{ x_3 = 0\}$ is a $90^{\circ}$-degree angled wedge, and so $\Omega_1$ on the left and $\Omega_2$ on the right meet tangentially at a single point $p_0 := (0,-1,0)$, with $\Omega_3$ wedged in between. Consequently, while $\Sigma_{13}$ and $\Sigma_{23}$ are $1$-spheres, $\Sigma_{12} = \partial^* \Omega_1 \cap \partial^* \Omega_2 = \{p_0\}$ is a singleton, and so $\Omega$ is not interface-regular (which requires every non-empty interface to be a smooth $(n-1)$-dimensional manifold). Note  that $|\c_{ij}|^2 = 1 + \k_{ij}^2$ for $(i,j) \in \{(1,3), (2,3)\}$, but not for $(i,j) = (1,2)$, and so strictly speaking this is not a spherical Voronoi cluster. 

\begin{figure}
    \begin{center}
            \includegraphics[scale=0.2]{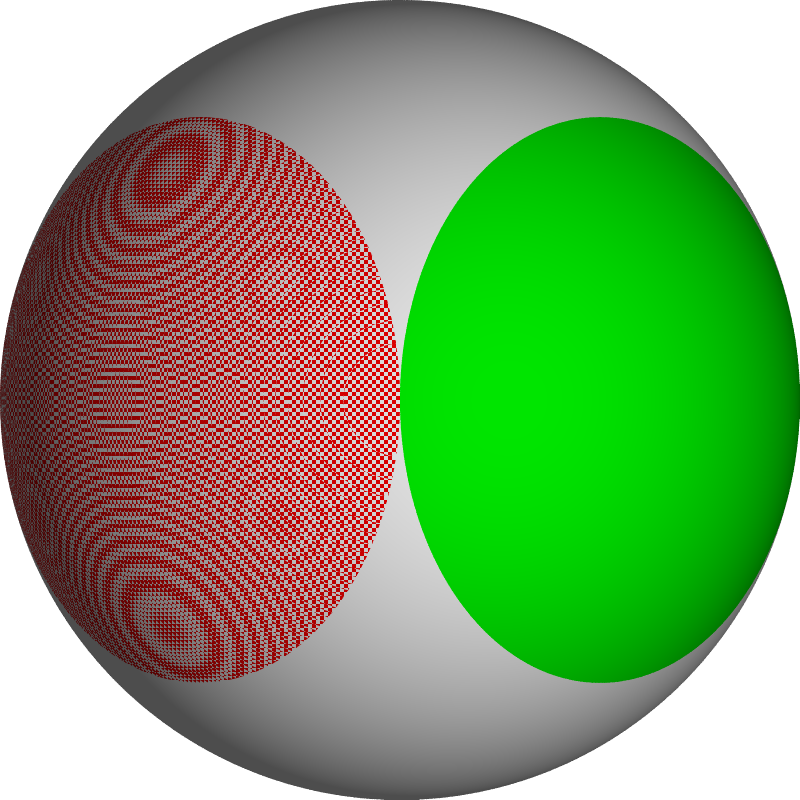}
     \end{center}
     \caption{
         \label{fig:cushion}
        A degenerate affine (yet non-spherical) Voronoi cluster on $\S^2$: the interface between the red and green cells consists of a single point, and the grey cell disappears in the blow-up around the common meeting point. 
     }
\end{figure}

Another feature of the above example pertains to the blow-up at the meeting point. For a set $A \subset M^n$ and a point $p \in M^n$, we say that $\tilde A \subset T_p M^n$ is the blow-up limit of
$A$ at $p$ if for some neighborhood $U_p$ of $p$ on which $\exp_p: T_p M^n \to M^n$ is a diffeomorphism,
\[
    \frac{\exp_p^{-1} (A \cap U_p)}{r} \to \tilde A \;\;  \text{as} \;\;  r \to 0+
\]
in $L^1_{\text{loc}}$ (note that we do not use a subsequence of $r$'s in our definition). For a cluster $\tilde \Omega^p$ in $T_p M^n$, we say that $\tilde \Omega^p$ is the blow-up limit of $\Omega$ at $p$ if each $\tilde \Omega^p_i$ is the blow-up limit of $\Omega_i$ at $p$. Clearly, if a blow-up limit exists, it is unique up to null-sets. 
\smallskip

Let $\Omega$ be an affine Voronoi cluster on $\S^n$; in particular, it has bounded curvature. It is therefore clear that for any $p \in \S^n$ the blow-up limit cluster $\tilde \Omega^p$ always exists, and that this cluster's cells are centered cones with flat interfaces -- we will call this cluster the blow-up cone. Denote:
\[
I_p :=  \{i \in \{1, \dots, q\} \; ; \; p \in \overline{\Omega_i}\} \supset \tilde I_p := \{ i \in \{1,\ldots,q\} \; ; \; \tilde \Omega^p_i \neq \emptyset \} . 
\]
Clearly $\tilde I_p \subset I_p$, but the converse containment
need not always hold. Indeed, in the previous example, we had $I_{p_0} = \{ 1,2,3\}$ and yet $\tilde \Omega^{p_0}_3 = \emptyset$ in the blow-up cone. To see that this can also happen for spherical Voronoi clusters, consider the $5$-cluster with the following parameters: $\c_1 = (0,0,0)$, $\k_1 = 0$; $\c_2 = (1,1,0)$, $\k_2 = 1$; $\c_3 = (-1,1,0)$, $\k_3=1$; $\c_4 = (1/2,y,\sqrt{3}/2)$, $\k_4 = y$; $\c_5 = (-1/2,y,\sqrt{3}/2)$, $\k_5 = y$; $y > 0$ can be arbitrary, but we'll use the value $y=2$ in Figure \ref{fig:degenerate-spherical}. Then $|\c_{ij}|^2 = 1 + \k_{ij}^2$ for all $(i,j) \in \{ (1,2),(1,3),(1,4),(1,5),(2,4),(3,5), (4,5) \}$, the indices of the 7 non-empty interfaces $\Sigma_{ij}$, so this is a legal spherical Voronoi cluster. At the meeting point $p_0 := \cup_{i=1}^5 \overline{\Omega_i}$, we have $I_{p_0} = \{1,2,3,4,5\}$ and yet $\tilde \Omega^{p_0}_1 = \emptyset$ -- see Figure \ref{fig:degenerate-spherical}. 

\begin{figure}
    \begin{center}
            \includegraphics[scale=0.2]{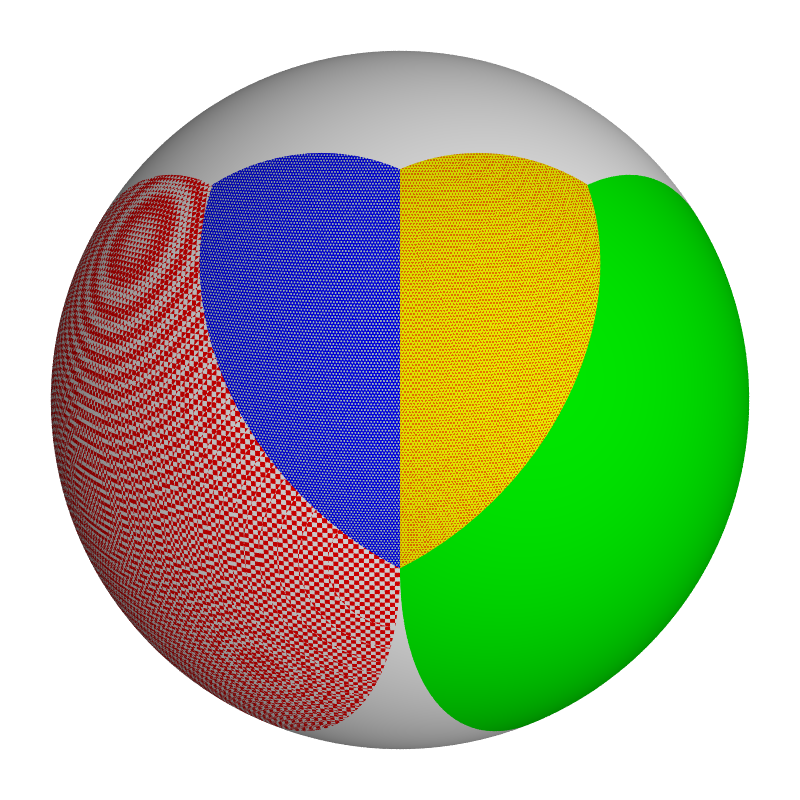}
     \end{center}
     \caption{
         \label{fig:degenerate-spherical}
        A degenerate spherical Voronoi cluster on $\S^2$. All interfaces are smooth curves, but the grey cell disappears in the blow-up around the common meeting point. 
     }
\end{figure}

\medskip

 To exclude the above pathologies, we introduce the following:
\begin{definition} \label{def:non-degenerate}
We will say that the spherical Voronoi cluster $\Omega$ on $\S^n$ is non-degenerate if:
\begin{enumerate}[(i)]
\item \label{it:nd1} For any $p \in \S^n$, $I_p = \tilde I_p$. In other words, $p \in \overline{\Omega_i}$ iff $\tilde \Omega^p_i \neq \emptyset$; and
\item \label{it:nd2} Every non-empty interface $\Sigma_{ij}$ is a \textbf{relatively open} subset of the geodesic sphere $S_{ij}$ given by (\ref{eq:Sij}). 
\end{enumerate}
\end{definition}

\noindent Recall that by definition of a spherical Voronoi cluster, if $\Sigma_{ij} \neq \emptyset$ then $S_{ij}$ is an $(n-1)$-dimensional  (in particular, non-empty and non-singleton) geodesic sphere in $\S^n$ having curvature $\k_{ij}$ and quasi-center $\c_{ij}$. Consequently, for non-degenerate spherical Voronoi clusters, the outer unit-normal $\n_{ij}$ of $\Sigma_{ij}$ pointing from $\Omega_i$ to $\Omega_j$ is well-defined, and in view of (\ref{eq:cnk}), is equal to $\n_{ij} = \c_{ij} + \k_{ij} p$. 

\begin{lemma} \label{lem:non-degenerate}
An interface-regular spherical Voronoi cluster satisfying the density estimate (\ref{eq:density}) for some $\eps = \eps(\Omega) > 0$ is non-degenerate. In particular, a regular spherical Voronoi cluster is non-degenerate. 
\end{lemma}
\begin{proof}
Property \ref{it:nd1} of non-degeneracy follows directly from the density estimate (\ref{eq:density}). Property \ref{it:nd2} of non-degeneracy follows from interface-regularity (property \ref{it:Almgren-iii} of Theorem \ref{thm:Almgren}).
\end{proof}

\noindent
Consequently, we shall only consider non-degenerate spherical Voronoi clusters from here on, as these are the only ones which can arise as isoperimetric minimizers by regularity. 

\medskip

It may be worthwhile to note that the second requirement in Definition \ref{def:non-degenerate} actually follows from the first. 

\begin{lemma} \label{lem:Sigmaij-char}
Let $\Omega$ be a spherical Voronoi cluster satisfying property \ref{it:nd1} of Definition \ref{def:non-degenerate}. Then for every $i \neq j$:
\begin{equation} \label{eq:bonus}
\Sigma_{ij} = \overline{\Omega_i} \cap \overline{\Omega_j} \setminus \cup_{k \neq i,j} \overline{\Omega_k}  .
\end{equation}
In particular, 
 $\Omega$ satisfies property \ref{it:nd2} of Definition \ref{def:non-degenerate}. 
\end{lemma}
\begin{proof}
Denote the right-hand-side of (\ref{eq:bonus}) by $\hat \Sigma_{ij}$. 
Let $p \in \hat \Sigma_{ij}$, then there exists a small neighborhood $U_p \subset \S^n$ of $p$ so that $U_p \subset \overline{\Omega_i} \cup \overline{\Omega_j} \setminus \cup_{k \neq i,j} \overline{\Omega_k}$. Since $\partial \Omega_i \cap \partial \Omega_j \subset S_{ij}$, it follows that $\Omega_i \cap U_p = \{ p \in U_p \; ; \; \scalar{\c_{ij} , p} + \k_{ij} < 0 \} $ and $\Omega_j \cap U_p = \{ p \in U_p \; ; \; \scalar{\c_{ij} , p} + \k_{ij} > 0 \}$; since neither of these sets can be empty (as this would contradict the definition of $p$), it follows that both $\Omega_i$ and $\Omega_j$ have a well-defined unique unit-normal at $p$, and so $p \in \partial^* \Omega_i \cap \partial^* \Omega_j = \Sigma_{ij}$. This also shows that $\hat \Sigma_{ij}$ is a relatively open subset of $S_{ij}$. 

For the other direction, we need to invoke property \ref{it:nd1} of Definition \ref{def:non-degenerate}. Let $p \in \partial^* \Omega_i \cap \partial^* \Omega_j$, and consider the blow-up conical cluster $\tilde \Omega^p$ at $p$. By \cite[Theorem 15.5]{MaggiBook} (or simply \cite[Corollary 15.8]{MaggiBook}), the conical cells $\tilde \Omega^p_i$ and $\tilde \Omega^p_j$ must be open half-planes (since $\Omega_i$ and $\Omega_j$ have a well-defined unique unit-normal at $p$), and since they are disjoint, this means that the other (open) conical cells $\tilde \Omega^p_k$, $k \neq i,j$, must be empty. 
By  property \ref{it:nd1} of Definition \ref{def:non-degenerate}, it follows that $I_p = \{i,j\}$, which means that $p \in \overline{\Omega_i} \cap \overline{\Omega_j} \setminus \cup_{k \neq i,j} \overline{\Omega_k} = \hat \Sigma_{ij}$. 
\end{proof}

As a corollary, we conclude that Lemma \ref{lem:non-degenerate} is in fact an equivalence.
\begin{corollary} \label{cor:non-degenerate-regular}
A non-degenerate spherical Voronoi cluster $\Omega$ is interface-regular and satisfies the density estimate (\ref{eq:density}) for some $\eps = \eps(\Omega) > 0$ depending on $\Omega$. 
\end{corollary}
\begin{proof}
The interface regularity is immediate from Definitions \ref{def:non-degenerate} and \ref{def:interface-regular} and the representation (\ref{eq:bonus}). To see the density estimate (\ref{eq:density}) for some $\eps(\Omega) > 0$, note that non-degeneracy implies that $\Theta(\Omega_i, p) > 0$ for all $p \in \overline{\Omega_i}$ (since otherwise the open blow-up cone $\tilde \Omega^p_i$ would necessarily be empty, contradicting non-degeneracy). Furthermore, we may assume that $\Omega_i \neq \emptyset$ (otherwise there is nothing to prove). Consequently, to establish the claim, it remains to show that $\overline{\Omega_i} \ni p \mapsto \Theta(\Omega_i, p)$ is lower semi-continuous. While this is in general certainly false, recall that $\underline \Omega_i = \S^n \cap P_i$ for a closed convex polyhedron $P_i$ in $\R^{n+1}$.  By Lemma \ref{lem:isolated}, we know that $\underline \Omega_i \setminus \overline \Omega_i$ consists of a finite number of isolated points and hence $\Theta(\Omega_i,p) = \Theta(\overline \Omega_i,p) = \Theta(\underline \Omega_i,p)$, and so it is enough to show that 
\[
P_i \setminus \{0\} \ni x \mapsto \liminf_{r \rightarrow 0} \frac{\H^n(P_i \cap (x + x^{\perp}) \cap B(x,r))}{r^n} 
\]
is lower semi-continuous. This follows from the convexity of $P_i$, as the intersection of a finite number of closed half-planes. 
\end{proof}

\begin{remark}
A non-degenerate spherical Voronoi cluster, all of whose (open) cells are non-empty, uniquely determines its quasi-center and curvature parameters. Indeed, the cluster is interface-regular by the previous corollary, and so $\k_{ij}$ are the (unique, well-defined) curvatures of the non-empty interfaces $\Sigma_{ij}$. Since the adjacency graph is connected by Lemma \ref{lem:LA-connected}, this determines the individual $\{ \k_i \}_{i=1,\ldots,q}$ up to an additive constant; but since we employ the convention that $\sum_{i=1}^{q} \k_i = 0$, this uniquely determines $\{ \k_i \}$. The argument for the quasi-centers is identical (coordinate-wise). 
\end{remark}

\subsection{Plateau clusters} \label{subsec:Plateau}

When $\Omega$ is a non-degenerate spherical Voronoi cluster and $p \in \overline{\Sigma_{ij}}$, then the normal vector-field $\n_{ij}$ on $\Sigma_{ij}$ is continuous up to the singularity $p$ -- we denote the limiting value by $\n_{ij}(p)$. It was observed in \cite[Lemma 11.1]{EMilmanNeeman-TripleAndQuadruple} that the blow-up cone at $p$ is itself a ``conical Voronoi cluster" whose flat interfaces are given by the normals $\{ \n_{ij}(p) ; \overline{\Sigma_{ij}} \ni p \}$.

\begin{definition}[Conical Voronoi Cluster]
A cluster $\tilde \Omega = (\tilde \Omega_i)_{i \in I}$ on Euclidean space $\R^n$ is called a conical Voronoi cluster if there exist $\{ \tilde \n_i \}_{i\in I} \subset \R^{n}$ (referred to as ``normal parameters") so that:
\begin{enumerate}[(1)]
\item \label{it:conical-Voronoi}
Every non-empty interface $\tilde \Sigma_{ij} \neq \emptyset$ lies in a hyperplane $\{x \in \R^n \; ; \; \scalar{x , \tilde \n_{ij}} =0\}$ with outer unit-normal $\tilde \n_{ij} = \tilde \n_i - \tilde \n_j$. 
\item \label{it:conical-Voronoi-2}
The following Voronoi representation holds:
\[ \tilde \Omega_i  = \interior \set{ x \in \R^n \; ; \; \argmin_{j\in I} \scalar{\tilde \n_j , x} \ni i } =  \interior \; \bigcap_{j \in I \setminus \{ i\} } \; \set{ x \in \R^n \; ;\; \scalar{\tilde \n_{ij},x}  \leq 0 } .
\] \end{enumerate}
As the normal parameters $\{\tilde \n_i\}_{i \in I}$ are defined up to translation, we will always employ the convention that $\sum_{i\in I} \tilde \n_i = 0$. 
\end{definition}

\begin{lemma}\label{lem:blow-up-cones}
    Let $\Omega$ be a spherical Voronoi $q$-cluster on $\S^n$, and let $p \in \S^n$. Then:
    \begin{enumerate}[(i)]
    \item The blow-up cone $\tilde \Omega^p = (\tilde \Omega^p_i)_{i \in I_p}$ is a conical Voronoi cluster on $T_p \S^n$ (identified with $\R^n$) with normal parameters denoted by $\{ \tilde \n_i[p]\}_{i \in I_p}$. 
    \item For all $i \in I_p$ we have $\tilde \n_i[p] = \n_i(p) - \frac{1}{|I_p|} \sum_{j \in I_p} \n_j(p)$, where $\n_i(p) := \c_i - \scalar{\c_i,p} p$ is the orthogonal projection of $\c_i$ onto $T_p \S^n$. 
            \item If $\tilde \Sigma^p_{ij} \neq \emptyset$ then $p \in \overline{\Sigma_{ij}}$. 
         \item
    Denoting:     \[ 
    \N_p := \sum_{i\in I_p} e_i \otimes \tilde \n_i[p] : T_p \S^n \rightarrow E^{I_p} ,
    \]
    we have:
        \begin{equation} \label{eq:affine-ranks}      d = \arank(\N_p) \leq \arank(\{\c_i\}_{i \in I_p}) \leq |I_p|-1 \leq q-1,
    \end{equation}
     and $\tilde \Omega^p$ is isometric to $\tilde {\underline \Omega}^p \times \R^{n-d}$ for a conical Voronoi cluster $\tilde {\underline \Omega}^p$ in $\R^d$.
    \end{enumerate}
    If in addition $\Omega$ is non-degenerate:
    \begin{enumerate}[(i)] \setcounter{enumi}{4}
    \item By definition, $I_p = \tilde I_p$, namely $\tilde \Omega^p_i \neq \emptyset$ for all $i \in I_p$. 
    \item For all $i,j \in I_p$, $\tilde \Sigma^p_{ij} \neq \emptyset$ iff $p \in \overline{\Sigma_{ij}}$.
            \item In particular, $\tilde \n_{ij}[p] = \n_{ij}(p)$ for all $i,j \in I_p$ so that $p \in \overline{\Sigma_{ij}}$.  
        \item $\text{affine-span} \{ \tilde \n_i[p] \}_{i \in I_p} = \sspan \{ \tilde \n_{ij}[p] \; ; \; \overline{\Sigma_{ij}} \ni p \}$. 
    \end{enumerate}
\end{lemma}

We now observe the following simple:

\begin{lemma} \label{lem:Plateau-at-p}
Let $\Omega$ be a spherical Voronoi cluster on $\S^n$. Then the following statements are equivalent for a given $p \in \Sigma$:
\begin{enumerate}[(a)]
\item \label{it:Plateau-1} $\arank(\N_p) = |I_p| - 1$, i.e. $\{\tilde \n_i[p]\}_{i \in I_p}$ are affinely independent. 
\item \label{it:Plateau-2} For all $i,j \in I_p$ with $i < j$, $\tilde \Sigma^p_{ij} \neq \emptyset$. 
\item \label{it:Plateau-4} $\N_p \N_p^T = \frac{1}{2} \Id_{E^{I_p}}$. 
\item \label{it:Plateau-5} $\{ \tilde \n_i[p] \}_{i \in I_p}$ are the vertices of a centered regular unit-simplex. 
\item \label{it:Plateau-6} $\tilde \Sigma^p = \partial \tilde \Omega^p$ is isometric to the product of $\R^{n-d}$ with the cone over the $(d-2)$-dimensional skeleton of a centered regular $d$-dimensional simplex, where $d = |I_p|-1$. 
\end{enumerate}
Any (and hence all) of these statements imply that $I_p = \tilde I_p$. \\
If $\Omega$ is in addition non-degenerate, any (and hence all) of these statements are equivalent to:
\begin{enumerate}[(a)] \setcounter{enumi}{5}
\item \label{it:Plateau-3} For all $i,j \in I_p$ with $i < j$, $p \in \overline{\Sigma_{ij}}$ .
\end{enumerate}
\end{lemma}
For consistency, the $-1$-dimensional skeleton of a regular $1$-dimensional simplex is defined as $\{0\}$. 
Note that necessarily $|I_p| \geq 2$ if $p \in \Sigma$. 

\begin{proof}[Proof of Lemma \ref{lem:Plateau-at-p}]
If \ref{it:Plateau-1} holds then $\tilde \Omega^p$ is the conical Voronoi cluster generated by the affinely-independent $\{ \tilde n_i[p] \}_{i \in I_p}$, and so all of its $|I_p|$ conical cells must be non-empty and we have $I_p = \tilde I_p$. Affine-independence also implies that all of its interfaces $\tilde \Sigma^p_{ij}$ are non-empty, yielding \ref{it:Plateau-2}. 
By definition of conical Voronoi cluster:
\[
\scalar{\N_p \N_p^T  , e_{ij} \otimes e_{ij}} = |\tilde \n_{ij}[p]|^2 = 1 = \scalar{\frac{1}{2} \Id_{E^{I_p}}  , e_{ij} \otimes e_{ij}} \;\;\; \;\;\; \forall i < j \text{ so that } \tilde \Sigma^p_{ij} \neq \emptyset ,
\]
and since $\{e_{ij} \otimes e_{ij}\}_{i,j \in I_p, i < j}$ span the space of all symmetric bilinear forms on $E^{I_p}$ (by comparing dimensions), we see that \ref{it:Plateau-2} implies \ref{it:Plateau-4}. 

Now, \ref{it:Plateau-4} holds iff the rows of $\N_p$ are the vertices of a regular unit-simplex (having distance $1$ between every pair of vertices), and since the rows of $\N_p$ sum to zero by our convention, the simplex must be centered and the equivalence with \ref{it:Plateau-5} is established. Recalling the definition of a conical Voronoi cluster and using that $|\tilde \n_j[p]|$ does not depend on $j \in I_p$, \ref{it:Plateau-5} implies that $\tilde \Omega^p$ is a Voronoi cluster over a centered regular  simplex:
\[
\tilde \Omega^p_i :=  \interior \set{ x \in T_p \S^n \; ; \; \argmin_{j\in I_p} |x + \tilde \n_j[p] |^2 \ni  i } \;\;\; \forall i \in I_p ,
\]
yielding \ref{it:Plateau-6}. Lastly, \ref{it:Plateau-6} implies that $\{ \tilde \n_i[p] \}_{i \in I_p}$ are affinely independent, yielding \ref{it:Plateau-1} and closing the cycle of equivalences. 
Finally, if $\Omega$ is non-degenerate, \ref{it:Plateau-2} and \ref{it:Plateau-3} are equivalent by Lemma \ref{lem:blow-up-cones}.
\end{proof}

\begin{definition}[$\ell$-Plateau cluster]
Let $\Omega$ be a non-degenerate spherical Voronoi cluster on $\S^n$. 
\begin{enumerate}[(i)]
\item We shall say that $\Omega$ is ``Plateau at $p \in \Sigma$" if any (all) of the equivalent statements \ref{it:Plateau-1} through \ref{it:Plateau-3} given by Lemma \ref{lem:Plateau-at-p} hold for $p$.  
\item We shall say that $\Omega$ is Plateau if it is Plateau at $p$ for all $p \in \Sigma$.
\item We shall say that $\Omega$ is $\ell$-Plateau if it is Plateau at $p$ for any $p \in \Sigma$ so that $\arank(\N_p) \leq \ell$. 
\end{enumerate}
Note that an $\ell$-Plateau cluster is by definition non-degenerate, for any value of $\ell$ (also negative). 
\end{definition}

\begin{lemma} \label{lem:full-Plateau}
The following are equivalent for a spherical Voronoi cluster $\Omega$ on $\S^n$: 
\begin{enumerate}[(i)]
\item \label{it:full-Plateau-1} $\Omega$ is Plateau.
\item \label{it:full-Plateau-2} $\Omega$ is $\ell$-Plateau with $\ell = \min(n,\arank(\C))$. 
\item \label{it:full-Plateau-3} For every $p \in \Sigma$ (or $p \in \S^n$), $\arank(\N_p) = |I_p|-1$. 
\end{enumerate}
\end{lemma}
\begin{proof}
\ref{it:full-Plateau-1} implies \ref{it:full-Plateau-2} trivially, which in turn implies \ref{it:full-Plateau-3} by (\ref{eq:affine-ranks}). Finally, \ref{it:full-Plateau-3} implies by Lemma \ref{lem:Plateau-at-p} that $I_p = \tilde I_p$ for every $p \in \Sigma$ (and hence for every $p \in \S^n$), and so $\Omega$ must be non-degenerate by Lemma \ref{lem:Sigmaij-char}; it then follows by definition that $\Omega$ is Plateau. 
\end{proof}

One crucial feature of being fully Plateau stems from the following fact:
\begin{lemma}
Let $\Omega$ be a spherical Voronoi cluster on $\S^n$, and assume that $\Omega$ is Plateau (in particular, non-degenerate). Then for all $i \neq j$:
\[
\Sigma_{ij} = \emptyset \;\;\; \Leftrightarrow \;\;\; \overline{\Omega_i} \cap \overline{\Omega_j} = \emptyset . 
\]
In other words, cells which do not share an interface do not touch. 
\end{lemma}
\begin{proof}
If $p \in \overline{\Omega_i} \cap \overline{\Omega_j}$, then $\{i,j\} \subset I_p$ and hence by the Plateau property and Lemma \ref{lem:Plateau-at-p} \ref{it:Plateau-3}, we would have $p \in \overline{\Sigma_{ij}}$, and in particular $\Sigma_{ij} \neq \emptyset$. The other direction is trivial. 
\end{proof}

\begin{corollary}[Same non-empty interfaces] \label{cor:Plateau-interfaces}
Let $\Omega$ be a spherical Voronoi cluster on $\S^n$ with quasi-center and curvature parameters $\{ \c_i \}_{i=1,\ldots,q}$ and $\{ \k_i \}_{i=1,\ldots,q}$, respectively. Assume that all cells are non-empty, and that $\Omega$ is perpendicular and Plateau (in particular, non-degenerate). Then there exists a neighborhood $U_\C$ of $\C : \R^{n+1} \rightarrow E^{(q-1)}$ and $U_\k$ of $\k \in E^{(q-1)}$, so that for all $\C' \in U_\C$ and $\k' \in U_\k$, denoting by $\{\c'_i\}_{i=1,\ldots,q}$ the rows of $\C'$, the affine Voronoi cluster $\Omega'$ on $\S^n$ generated by $\{\c'_i\}_{i=1,\ldots,q}$ and $\{\k'_i\}_{i=1,\ldots,q}$ has the same non-empty interfaces as $\Omega$:
\[
\Sigma_{ij} \neq \emptyset \; \; \Leftrightarrow \;\; \Sigma'_{ij} \neq \emptyset . 
\] 
\end{corollary}
\begin{proof}
Recall from Lemma \ref{lem:cont} that perpendicularity guarantees that the closures of the non-empty cells and interfaces are locally continuous in the Hausdorff metric as a function of the parameters $\C,\k$. By the previous lemma, all pairs of cells which do not share an interface do not touch thanks to the Plateau property, and hence are at a positive distance from each other.  By continuity, these pairs will remain at a positive distance from each other for sufficiently small neighborhoods $U_\C$ and $U_\k$, and hence $\Sigma_{ij} = \emptyset \Rightarrow \Sigma'_{ij} = \emptyset$. The other direction is trivial by continuity regardless of being Plateau: $\Sigma_{ij} \neq \emptyset \Rightarrow \Sigma'_{ij} \neq \emptyset$. 
\end{proof}

Another important property of being Plateau is that it is an open condition:

\begin{proposition} \label{prop:Plateau-perturbation}
With the same assumptions and notation as in Corollary \ref{cor:Plateau-interfaces}, let $[0,1] \ni t \mapsto \C(t) : \R^{n+1} \rightarrow E^{(q-1)}$ and $[0,1] \ni t \mapsto \k(t) \in E^{(q-1)}$ be continuous deformations of $\C(0) = \C$ and $\k(0) = \k$. Assume that the affine Voronoi cluster $\Omega(t)$ generated by $\C(t)$ and $\k(t)$ is in fact spherical Voronoi. 
Then there exists $t_0 \in (0,1]$ so that $\Omega(t)$ remain Plateau for all $t \in [0,t_0]$. 
\end{proposition}

By Lemma \ref{lem:full-Plateau}, the claim is that:
\[ \forall p \in \S^n \;\;\; \forall t \in [0,t_0] \;\;\; \arank(\N_p(t)) = |I_p(t)|-1 ,
\] where we use $P(t)$ to denote the parameter $P \in \{ I_p, \N_p, \n_i(p)\}$ corresponding to $\Omega(t)$. The proof is then an immediate consequence of compactness of $\S^n$, once the following local claim is established:

\begin{lemma} \label{lem:Plateau-field}
With the same assumptions and notation as in the previous proposition, for all $p_0 \in \S^n$, there exists an open neighborhood $U_{p_0}$ of $p_0$ and $t_{p_0} \in (0,1]$ so that for all $p \in U_{p_0}$ and $t \in [0,t_{p_0}]$:
\begin{enumerate}[(i)]
\item $I_p(t) \subset I_{p_0}(0)$, i.e. $U_{p_0} \cap \cup_{i \notin I_{p_0}(0)} \overline{\Omega_i(t)} = \emptyset$. 
\item $\arank(\N_p(t)) = |I_p(t)|-1$, i.e. $\{\n_i(p,t)\}_{i \in I_p(t)}$ are affinely independent. 
\end{enumerate}
\end{lemma}
\begin{proof}
Note that both requirements hold by at $p=p_0$ and $t=0$; the first trivially and the second by the Plateau property of $\Omega = \Omega(0)$. 

Let $\hat U_{p_0}$ be an open neighborhood of $p_0$ whose closure is positive distance from $\cup_{i \notin I_{p_0}(0)} \overline{\Omega_i(0)}$. Since $[0,1] \ni t \mapsto \C(t),\k(t)$ are continuous by assumption and since $\overline{\Omega_i(t)}$ is locally continuous in $\C(t),\k(t)$ with respect to the Hausdorff metric by Lemma \ref{lem:cont}, it follows that $\hat U_{p_0}$ remains separated from $\cup_{i \notin I_{p_0}(0)} \overline{\Omega_i(t)}$ for all $t \in [0,\hat t_{p_0}]$ for some $\hat t_{p_0} > 0$, establishing the first assertion. 

To show the second assertion, it is therefore enough to show that, possibly on a smaller neighborhood $U_{p_0} \subset \hat U_{p_0}$ and for $t_{p_0} \in (0,\hat t_{p_0}]$,  $\{\n_i(p,t)\}_{i \in I_{p_0}(0)}$ are affinely independent for all $p \in U_{p_0}$ and $t \in [0,t_{p_0}]$. And indeed, since $\n_i(p,t) = \c_i(t) - \scalar{\c_i(t),p} p$, we see that these vectors are continuous in $(p,t) \in \S^n \times [0,1]$, and since they are affinely independent at $(p,t) = (p_0,0)$, there exists an open neighborhood where they remain so. 
\end{proof}

\subsection{Regular spherical Voronoi clusters}

The results collected in Section \ref{sec:prelim} imply that:

\begin{lemma}\label{lem:3-Plateau}
An isoperimetric minimizing spherical Voronoi cluster on $\S^n$ is $3$-Plateau (and in particular, non-degenerate). 
\end{lemma}
\begin{proof}
Any regular (and in particular minimizing) spherical Voronoi cluster on $\S^n$ is non-degenerate by Lemma \ref{lem:non-degenerate}. 
Now let $p \in \Sigma$ with $d = \arank(N_p) \leq 3$. Note that $d \geq 1$ because $|I_p| \geq 2$ and hence $p \in \Sigma = \overline{\Sigma^1}$. By Lemma~\ref{lem:blow-up-cones}, the blow-up cluster $\tilde \Omega^p$ at $p$ is isometric to ${\tilde {\underline \Omega}^p} \times \R^{n-d}$ for a conical Voronoi cluster ${\tilde {\underline \Omega}^p}$ in $\R^d$. By~\cite[Theorem 21.14]{MaggiBook}, the blow-up limit's boundary $\tilde \Sigma^p = \partial \tilde \Omega^p$ is an area-minimizing cone. Therefore (see e.g. \cite[Theorem 5.4.8]{FedererBook}) ${\tilde {\underline \Sigma}^p} =  \partial {\tilde {\underline \Omega}^p}$ is also an area-minimizing cone in $\R^d$,
 and so by Taylor's classification~\cite{Taylor-SoapBubbleRegularityInR3} it must be a $\T$ cone if $d=3$, a $\Y$ cone if $d=2$ and a hyperplane if $d=1$. Since $|I_p|$ is exactly the number of non-empty cells in $\tilde \Omega^p$, it follows that $|I_p| = d+1$, establishing that $\Omega$ is Plateau at $p$. 
\end{proof}

\begin{remark}
It is almost also true that a general regular spherical Voronoi cluster on $\S^n$ is $3$-Plateau, just by invoking Lemma \ref{lem:non-degenerate} as above and the property of regular clusters asserting that $\Sigma^2$ is locally diffeomorphic to a $\Y$-cone and that $\Sigma^3$ is locally diffeomorphic to a $\T$-cone. However, our definition of regularity allows for adding an arbitrary subset $B \subset \Sigma$ with  $\dim_{\H}(B) \leq n-4$ to the ``singular set" $\Sigma_4$, and so strictly speaking, we are not able to immediately guarantee the $2$-Plateau property on $B \cap \Sigma^2$, nor the $3$-Plateau property on $B \cap \Sigma^3$. While for the results of Sections \ref{sec:spectral} through \ref{sec:LSE} it is enough to consider regular spherical Voronoi clusters, we will prefer to state our results in those sections for $3$-Plateau spherical Voronoi clusters, since the Plateau definition is much more concrete for a spherical Voronoi cluster than that of regularity. 
\end{remark}

\begin{figure}
    \begin{center}
            \includegraphics[scale=0.2]{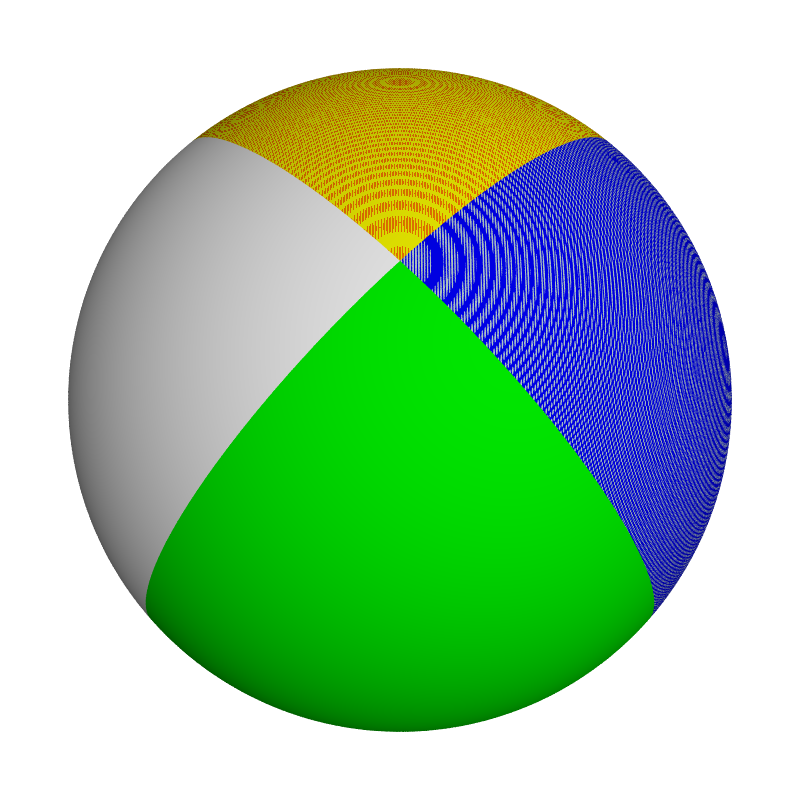}
     \end{center}
     \caption{
         \label{fig:quad}
        A non-degenerate spherical Voronoi $4$-cluster which is not $2$-Plateau. 
     }
\end{figure}

See Figure \ref{fig:intro-spherical-Voronoi} for an example of a $2$-Plateau spherical Voronoi cluster, and Figure \ref{fig:quad} for an example of a non-degenerate spherical Voronoi cluster which is not $2$-Plateau, and hence cannot be an isoperimetric minimizer. For a $2$-Plateau spherical Voronoi cluster, we define for distinct $i,j,k$
\[
\Sigma_{ijk} := \overline{\Omega_i} \cap \overline{\Omega_j} \cap \overline{\Omega_k} \setminus \cup_{\ell \neq i,j,k} \overline{\Omega_\ell},
\]
 and set $\Sigma^2 := \cup_{i<j<k} \Sigma_{ijk}$. For a $3$-Plateau spherical Voronoi cluster, we define for distinct $i,j,k,\ell$
 \[
 \Sigma_{ijk\ell} :=  \overline{\Omega_i} \cap \overline{\Omega_j} \cap \overline{\Omega_k}  \cap \overline{\Omega_\ell} \setminus \cup_{m \neq i,j,k,\ell} \overline{\Omega_m},
 \]
 set $\Sigma^3 := \cup_{i<j<k<\ell} \Sigma_{ijk\ell}$, and finally
 \[
 \Sigma^4 := \Sigma \setminus (\Sigma^1 \cup \Sigma^2 \cup \Sigma^3) = \cup_{i<j<k<\ell<m} \overline{\Omega_i} \cap \overline{\Omega_j} \cap \overline{\Omega_k} \cap \overline{\Omega_\ell}  \cap \overline{\Omega_m} .  
 \]
 In either case, we set $\Sigma^{\geq 3} = \Sigma \setminus (\Sigma^1 \cup \Sigma^2)$. 
 
\medskip
 
With these definitions, it is straightforward to verify that:
\begin{lemma} \label{lem:Voronoi-regular}
A $3$-Plateau spherical Voronoi cluster is regular. 
\end{lemma}
\begin{proof}
By Corollary \ref{cor:non-degenerate-regular}, non-degeneracy implies that the cluster is interface-regular and satisfies the density estimate (\ref{eq:density}) for some $\eps >0$. The density estimate (\ref{eq:Almgren-density}) is trivial for a spherical Voronoi cluster. The definitions of $\Sigma^2$ and $\Sigma^3$ above and the $3$-Plateau property (Lemma \ref{lem:Plateau-at-p} \ref{it:Plateau-6}) imply that the boundaries of the blow-up clusters at $p \in \Sigma^2$ and $p \in \Sigma^3$ are $\Y$- and $\T$-cones, respectively. Arguing as in the proof of Lemma \ref{lem:Sigmaij-char}, $\Sigma_{ijk}$ is a relatively open subset of $\{ p \in \S^n \; ; \; \scalar{p,\c_i} + \k_i = \scalar{p,\c_j} + \k_j = \scalar{p,\c_k} + \k_k \}$, and it follows that $\Sigma$ is locally diffeomorphic in a neighborhood of $p \in \Sigma_{ijk}$ to the $\Y$-cone at $p$. Similarly, $\Sigma$ is locally diffeomorphic in a neighborhood of $p \in \Sigma_{ijk\ell}$ to the $\T$-cone at $p$. 

Lastly, we need to show that $\dim_\H(\Sigma^4) \leq n-4$. The $3$-Plateau property implies that $\arank(\N_p) \geq 4$ for all $p \in \Sigma^4$ since $|I_p| \geq 5$ at those points. Arguing as before, it follows that for every $p \in \Sigma^4$ there is a neighborhood of $p$ where $\Sigma^4$ is a smooth embedded manifold of codimension $\arank(N_p) \geq 4$, concluding the proof. 
\end{proof}

Consequently, all of the assertions for regular clusters from Section \ref{sec:prelim} hold for such clusters. In particular, the assertions of Lemma \ref{lem:boundary-normal-sum} and (\ref{eq:sqrt3}) on the $120^{\circ}$-degrees of incidence at triple-points $\Sigma^2$ are in force.  In fact, it is not hard to check the proofs in \cite{EMilmanNeeman-GaussianMultiBubble,EMilmanNeeman-TripleAndQuadruple} and verify that most of the properties from Section \ref{sec:prelim} remain valid for $2$-Plateau clusters, but we will not insist on this extraneous generality here, and formulate the results of Sections \ref{sec:conformal} through \ref{sec:LSE} for $3$-Plateau spherical Voronoi clusters. 

\begin{lemma} \label{lem:Voronoi-stationary}
A $3$-Plateau spherical Voronoi cluster is stationary (with Lagrange multiplier $\lambda = (n-1) \k$).
\end{lemma}
\begin{proof}
Immediate from Definition \ref{def:spherical-Voronoi} and Lemma \ref{lem:equivalent-stationary} (as the cluster is of bounded curvature and regular by the previous comments), after recalling that $\n_{ij} = \c_{ij} + \k_{ij} p$ on each non-empty interface $\Sigma_{ij}$. 
\end{proof}

\subsection{Standard bubbles}

By Liouville's classical theorem \cite{Udo-MobiusDifferentialGeometry,Blair-InversionTheory}, all conformal automorphisms of $\bar \R^n$, the one-point-at-infinity compactification of $\R^n$, when $n \geq 3$, are given by M\"obius transformations, obtained by composing isometries (orthogonal linear transformations and translations) with scaling ($p \mapsto e^\lambda p$) and spherical inversion ($p \mapsto p / |p|^2$). Using stereographic projection, this also classifies all conformal automorphisms of $\S^n$.
Note the degrees of freedom when selecting the stereographic projection: once the North pole is chosen on $\S^n \subset \R^{n+1}$, one is allowed to isometrically embed $\R^n$ in $\R^{n+1}$ perpendicularly to the North-South axis arbitrarily (as long as it doesn't pass through the North pole).
Set $\bar \S^n = \S^n$ and let $\{ \MM_1 , \MM_2 \} = \{\R , \S\}$. 
The composition of two stereographic projections, the first from $\bar \MM_1^n$ to $\bar \MM^n_2$ and the second from $\bar \MM^n_2$ back to $\bar \MM_1^n$, is a M\"obius automorphism of $\bar \MM_1^n$, and in fact any M\"obius automorphism of $\bar \MM_1^n$ may be obtained in this manner.
The standard bubbles on $\R^n$ and $\S^n$ are thus generated from the equal-volume bubble by M\"obius automorphisms (in the case of $\R^n$, automorphisms which send a point in the interior of $\Omega_q$ to infinity, to ensure that $\Omega_q$ remains the unique unbounded cell). We refer to \cite[Section 10]{EMilmanNeeman-TripleAndQuadruple} and the references therein for proofs of the statements in this subsection and the ensuing one.

\begin{lemma}[M\"obius Automorphisms Preserve $\ell$-Plateau Spherical Voronoi Clusters] \label{lem:Mobius-preserves-Voronoi}
Let $\Omega$ be a spherical Voronoi cluster on $\S^n$. Then for any M\"obius automorphism $T$ of $\S^n$, $T(\Omega) = (T (\Omega_1),\ldots,T(\Omega_q))$ is a spherical Voronoi cluster. For any $\ell$, if $\Omega$ is $\ell$-Plateau, then so is $T(\Omega)$. 
\end{lemma}
\begin{proof}
The first assertion was established in \cite[Lemma 10.2]{EMilmanNeeman-TripleAndQuadruple}. To see the second part, which is local in nature, note that since the differential $dT$ is non-singular, $T$ preserves both non-degeneracy and the $\ell$-Plateau property at any given point.
\end{proof}

Since the equal-volume bubble on $\S^n$ is spherical Voronoi and Plateau (see \cite[Section 10]{EMilmanNeeman-TripleAndQuadruple}), we obtain:

\begin{corollary}[Standard Bubbles are Spherical Voronoi] \label{cor:standard-Voronoi}
For all $2 \leq q \leq n+2$, standard $(q-1)$-bubbles on $\S^n$ are spherical Voronoi Plateau clusters. If $2 \leq q \leq n+1$, they are in fact perpendicular spherical Voronoi Plateau clusters with $\S^{n+1-q}$-symmetry. \end{corollary}

\begin{proposition} \label{prop:standard-char}
The following are equivalent for a spherical Voronoi $q$-cluster $\Omega$ on $\S^n$ with $2 \leq q \leq n+2$:
\begin{enumerate}[(i)]
\item \label{it:standard-char-i} $\Omega$ is a standard-bubble. 
\item \label{it:standard-char-ii} All interfaces are non-empty: $\Sigma_{ij} \neq \emptyset$ for all $1 \leq i < j \leq q$.
\item \label{it:standard-char-iii} $\C \C^T  = \frac{1}{2} \Id_{E^{(q-1)}} + \k \k^T$. 
\end{enumerate}
\end{proposition}

\begin{lem}[Standard Bubbles of Prescribed Curvature] \label{lem:standard-curvature}
For all $2 \leq q \leq n+2$ and $\k \in E^{(q-1)}$, there exists a standard bubble on $\S^n$ with $q$-cells and spherical Voronoi curvature parameters $\{\k_i\}_{i=1,\ldots,q}$, i.e. the curvature of the spherical interface $\Sigma_{ij}$ is $\k_{ij} = \k_i - \k_j$. Moreover, such a standard bubble (with the same ordering of its cells) is unique up to orthogonal transformations of $\R^{n+1}$. 
\end{lem}

\begin{lem}[Standard Bubbles of Prescribed Volume]  \label{lem:standard-volume}
For all $2 \leq q \leq n+2$ and $v \in \interior \Delta^{(q-1)}$, there exists a standard bubble $\Omega$ on $\S^n$ with $q$-cells and $V(\Omega) = v$. Moreover, such a standard bubble (with the same ordering of its cells) is unique up to orthogonal transformations of $\R^{n+1}$. 
\end{lem}

\begin{remark}
Analogues of Lemmas \ref{lem:standard-curvature} and \ref{lem:standard-volume} regarding standard bubbles in Euclidean space $\R^n$ and in Gauss space $\GG^n$ were proved by Montesinos Amilibia \cite{MontesinosStandardBubbleE!} and by the authors \cite{EMilmanNeeman-GaussianMultiBubble}, respectively. For standard double-bubbles in $\S^2$ and $\S^n$ and triple-bubbles in $\S^2$, these results are due to Masters \cite{Masters-DoubleBubbleInS2}, Cotton-Freeman \cite{CottonFreeman-DoubleBubbleInSandH} and Lawlor \cite{Lawlor-TripleBubbleInR2AndS2}, respectively. 
\end{remark}

\subsection{Pseudo conformally flat (PCF) clusters}

\begin{definition}[Conformally Flat]
A spherical Voronoi cluster $\Omega$ on $\S^n$ is called conformally flat if there exists a M\"obius automorphism $T$ of $\S^n$ so that the spherical Voronoi cluster $T(\Omega)$ is flat, i.e.~all of its non-empty interfaces lie on geodesic spheres with zero curvature. 
\end{definition}

\begin{lemma} \label{lem:conformally-flat}
A spherical Voronoi $q$-cluster on $\S^n$ is conformally flat if and only if:
\[
\exists \xi \in \R^{n+1} \;\;\; |\xi| < 1 \;\; \text{ so that } \;\; \scalar{\c_i , \xi} + \k_i = 0 \;\;\; \forall i=1,\ldots,q ,
\]
(where $\{\c_i\}_{i=1,\ldots,q}$ and $\{ \k_i \}_{i=1,\ldots,q}$ are the cluster's quasi-center and curvature parameters, respectively). 
\end{lemma}

In view of the previous lemma, we generalize the property of being conformally flat by removing the restriction that $|\xi| < 1$. 
\begin{definition}[Pseudo Conformally Flat Cluster] \label{def:PCF}
A spherical Voronoi cluster $\Omega$ on $\S^n$ is called pseudo conformally flat (PCF) if:
\[
\exists \xi \in \R^{n+1} \;\; \text{ so that } \;\; \scalar{\c_i , \xi} + \k_i = 0 \;\;\; \forall i=1,\ldots,q .
\]
In this case we shall say that $\Omega$ is PCF with compatibility parameter $\xi \in \R^{n+1}$. 
\end{definition}

\begin{remark} \label{rem:all-intersect}
In view of our convention that $\sum \c_i = 0$ and $\sum \k_i = 0$, this is the same as requiring that $\scalar{\c_{ij} , \xi} + \k_{ij} = 0$ for all $i,j =1 ,\ldots, q$. In view of the polyhedral cell representation $\underline \Omega_i = P_i \cap \S^n$ for $P_i := \bigcap_{j \neq i} \set{ x \in \R^{n+1} \; ; \; \scalar{\c_{ij} ,x} + \k_{ij} \leq 0 }$, it follows that $\Omega$ is PCF iff $P := \bigcap_{i=1}^q P_i$ is non-empty, i.e.~that all closed convex polyhedra $\{P_i\}$ have a common meeting point $\xi \in \R^{n+1}$.  Similarly, a spherical Voronoi cluster on $\S^n$ is conformally flat iff $P \cap \{ x \in \R^{n+1} \; ; \; |x| < 1\} \neq \emptyset$. For a perpendicular spherical Voronoi cluster (with North pole at $N$), this is the same as witnessing the common intersection point after orthogonally projecting onto $N^{\perp}$, yielding a clear geometric characterization. 
\end{remark}

\begin{remark}
Clearly, when $q-1 \leq n+1$, a full-dimensional cluster is pseudo conformally flat. 
\end{remark}

\subsection{A property of $(q-3)$-Plateau clusters}

The significance of the number $q-3$ in Theorem \ref{thm:intro-conditional} can already be explained in the following:
\begin{lemma} \label{lem:q-3-Plateau}
Let $\Omega$ be a spherical Voronoi $q$-cluster on $\S^n$ with $q \leq n+2$. Assume that $\Omega$ is $(q-3)$-Plateau. Then either $\Omega$ is fully Plateau, or else it is pseudo conformally flat. 
\end{lemma}
\begin{proof}
Let $\C$ denote the cluster's quasi-center operator, and set $d := \arank(\C)$. If $d = q-1$ then the cluster is full-dimensional and hence PCF. If $d \leq q-3$ then by assumption and Lemma \ref{lem:full-Plateau} the cluster is Plateau. 
\smallskip

It remains to handle the case when $d=q-2$. Denote $Q := \cap_{i=1}^q \overline{\Omega_i}$ and consider two subcases:
\begin{enumerate}[(i)]
\item
$Q \neq \emptyset$. In view of Remark \ref{rem:all-intersect}, this means that any $\xi \in Q$ will satisfy $\scalar{\c_i,\xi} + \k_i = 0$ for all $i=1,\ldots,q$, and hence the cluster is PCF with compatibility parameter $\xi$. 
\item 
$Q = \emptyset$. This means that $|I_p| < q$ for all $p \in \S^n$. Given $p \in \Sigma$, either $\arank(\N_p) \leq q-3$ and so $\Omega$ is Plateau at $p$ by assumption, or else $\arank(N_p) \geq q-2$. In that case, since $\arank(N_p) \leq |I_p|-1$ by  (\ref{eq:affine-ranks}) and $|I_p| \leq q-1$, it follows that $\arank(N_p) = |I_p|-1 = q-2$, and hence $\Omega$ is again Plateau at $p$ by Lemma \ref{lem:Plateau-at-p} \ref{it:Plateau-1}. Consequently $\Omega$ is Plateau at $p$ for all $p \in \Sigma$, so it is Plateau. 
\end{enumerate}
This concludes the proof. 
\end{proof}

\section{Strategy of the proof - a PDE for the isoperimetric profile} \label{sec:profile}

Let $2 \leq q \leq n+2$, and recall the definitions (\ref{eq:intro-I}) and (\ref{eq:intro-Im}) of the spherical isoperimetric multi-bubble profile $\I = \I^{(q-1)} : \Delta^{(q-1)} \rightarrow \R_+$ and its model counterpart $\I_m = \I_m^{(q-1)} : \interior \Delta^{(q-1)} \rightarrow \R_+$ from the Introduction, as well as the natural extension of $\I_m$ to the boundary $\partial \Delta^{(q-1)}$. For consistency, the case $q=1$ is treated by setting $\Delta^{(0)} = \{1\}$ and $\I^{(0)}(1) = \I^{0}_m(1) = 0$. 
 Recall Proposition \ref{prop:intro-profile}, which asserts that the model profile satisfies the following PDE on $\Delta^{(q-1)}$:
\begin{equation} \label{eq:PDE}
\tr((-\nabla^2 \I_m)^{-1} (\Id+ \frac{2}{(n-1)^2} \nabla \I_m \otimes \nabla \I_m)) = \frac{2}{n-1} \I_m  ,
\end{equation}
and that moreover, $-\nabla^2 \I_m > 0$ on $T \Delta^{(q-1)} \simeq E^{(q-1)}$ (and hence this PDE is elliptic). 
While it is possible to derive this by direct computation, we will establish Proposition \ref{prop:intro-profile} as a by-product of our analysis in Section \ref{sec:conformal-Jacobi}; see Proposition \ref{prop:Im-PDE-proof} for a proof. 

\medskip

Since we use the convenient normalization of having $V(\S^n) = 1$, we define $\HH^{n-1} := \frac{1}{|\S^n|} \H^{n-1}$ to be the normalized $(n-1)$-dimensional Hausdorff measure, where $|\S^n|$ denotes the actual $n$-dimensional volume of $\S^n$. With these conventions, we have $\mu = \mu^n = V$, $\mu^{n-1} = \HH^{n-1}$, and $\per(U) = \HH^{n-1}(\partial^* U)$ for all Borel subsets $U \subset \S^n$ of finite perimeter.

\subsection{Comparison with one-dimensional and Gaussian cases}

The case $q=2$ of (\ref{eq:PDE}) agrees with the following easily derived ODE satisfied by the single-bubble isoperimetric profile $\I = \I_m$ on $\S^n$. Identifying $\Delta^{(1)}$ with $[0,1]$, we have $\I_m(v) = \varphi \circ \Phi^{-1}$ where $\varphi(t) = \frac{|\S^{n-1}|}{|\S^n|} \sin^{n-1}(t)$ and $\Phi(s) = \int_0^t \varphi(t) dt$, and as $\I_m' = (\log \varphi)' \circ \Phi^{-1} = (n-1) \cot \circ \Phi^{-1}$ we deduce:
\[
 -\I_m \I_m'' = -(\log \varphi)'' \circ \Phi^{-1} = \frac{n-1}{\sin^2 \circ \Phi^{-1}} = (n-1) (1 + \cot^2) \circ \Phi^{-1} = (n-1) \brac{1 + \frac{1}{(n-1)^2} (\I_m')^2 } . 
 \]
This is precisely (\ref{eq:PDE}) in the single-bubble case, after noting that the identification between $\Delta^{(1)}$ and $[0,1]$ stretches the domain metric by a factor of $\sqrt{2}$, yielding this factor with each derivation of $\I_m$.

\medskip

In addition, (\ref{eq:PDE}) should be compared with the simpler elliptic PDE satisfied by Gaussian Multi-Bubble isoperimetric profile $\I_{\GG^n} = \I_{\GG^n,m}$, as established in \cite{EMilmanNeeman-GaussianMultiBubble}:
\begin{equation} \label{eq:Gaussian-PDE}
\tr( (-\nabla^2 \I_{\GG^n,m})^{-1}) = 2 \I_{\GG^n,m} . 
\end{equation}
As a sanity check, note that by a simple rescaling, the PDE satisfied by the model profile $\I_m$ on $(\S^n , (n-1) g_{can},\lambda_{\S^n})$, i.e. on the rescaled sphere $\sqrt{n-1} \cdot \S^n$, is:
\[
\tr((-\nabla^2 \I_m)^{-1} (\Id+ \frac{2}{n-1} \nabla \I_m \otimes \nabla \I_m)) = 2 \I_m ,
\]
and so when $n \rightarrow \infty$, we formally recover (\ref{eq:Gaussian-PDE}) for the Gaussian measure, in accordance with the well-known procedure for approximating the Gaussian measure as the uniform measure on an infinite dimensional rescaled sphere (see e.g. \cite{Borell-GaussianIsoperimetry,CorneliCorwinEtAl-DoubleBubbleInSandG}). By following the argument of Borell from \cite{Borell-GaussianIsoperimetry}, it is not hard to rigorously deduce that $\I^{(q-1)}_{\GG^n} = \I^{(q-1)}_{\GG^n,m}$ from knowing that $\I^{(q-1)}_{\S^N} = \I^{(q-1)}_{\S^N,m}$ for all $N$ sufficiently large. However, the analysis of equality cases will break down in this approximation procedure, as in the derivation of Corollary \ref{cor:intro-quintuple-R} from Theorem \ref{thm:intro-quintuple}.

\subsection{The conjectured PDI for the actual isoperimetric profile} \label{subsec:viscosity} \label{subsec:max-principle}

Our goal in this work is to develop the tools for showing that the same PDI as (\ref{eq:PDE}) holds for the actual profile $\I$:
\begin{equation} \label{eq:PDI}
\nabla^2 \I  < 0 ~,~\tr \brac{(-\nabla^2 \I)^{-1} (\Id+ \frac{2}{(n-1)^2} \nabla \I \otimes \nabla \I)} \leq \frac{2}{n-1} \I . 
\end{equation}
More precisely, we will want to show that for all $v_0 \in \interior \Delta^{(q-1)}$, there exists a quadratic form $\F = \F(v_0)$ on $E^{(q-1)}$, so that at $v_0$:
\begin{equation} \label{eq:PDI-F}
\F > 0 ~,~ \F^T \nabla^2 \I \F \leq -(n-1) \F ~,~ \tr \brac{\F \; (\Id+ \frac{2}{(n-1)^2} \nabla \I \otimes \nabla \I)} \leq 2 \I 
\end{equation}
(where the first two inequalities are understood in the positive-definite sense). 
Of course, $\I$ may not be a-priori smooth, so (\ref{eq:PDI-F}) is meant in the viscosity sense. Let us define what we mean by this term in the context of this work. 

\begin{definition} \label{def:viscosity}
We say that (\ref{eq:PDI-F}) holds in the viscosity sense at $v_0 \in \interior \Delta^{(q-1)}$, if there exists a minimizer $\Omega$ with $V(\Omega) = v_0$ having Lagrange multiplier $\lambda \in E^{(q-1)}$, so that for all $a \in E^{(q-1)}$ and $\eps > 0$, there exists a $C_c^\infty$ vector-field $X = X_{\eps,a}$ on $\S^n$ so that:
\begin{equation} \label{eq:def-viscosity}
\delta^1_X V(\Omega) = \F a ~,~ Q_{\Omega}(X) \leq -(n-1) a^T \F a + \eps ,
\end{equation}
and:
\begin{equation} \label{eq:F-trace-inq}
 \F > 0 ~,~ \tr \brac{\F (\Id+ \frac{2}{(n-1)^2} \lambda \otimes \lambda)} \leq 2\I . 
 \end{equation}
\end{definition}

\medskip

The following was proved in \cite{EMilmanNeeman-GaussianMultiBubble} for the analogous elliptic PDE (\ref{eq:Gaussian-PDE}) satisfied by the Gaussian multi-bubble isoperimetric profile $\I_{\GG^n}$; for completeness, we repeat the proof in our setting.

\begin{proposition}[Maximum Principle] \label{prop:max-principle}
Let $2 \leq q \leq n+2$, and assume that (\ref{eq:PDI-F}) holds in the viscosity sense of Definition \ref{def:viscosity} and that $\I^{(q-1)} = \I^{(q-1)}_m$ on $\partial \Delta^{(q-1)}$. Then $\I^{(q-1)}  = \I^{(q-1)}_m$ on the entire $\Delta^{(q-1)}$. 
\end{proposition}
\begin{proof}
We abbreviate $\I = \I^{(q-1)}$, $\I_m = \I_m^{(q-1)}$ and $\Delta = \Delta^{(q-1)}$. 
 By definition $\I \leq \I_m$, so we need to show the converse inequality. As $\I_m$ satisfies the elliptic PDE (\ref{eq:PDE}), it is smooth by elliptic regularity on $\interior \Delta$. It is easy to check using compactness of $\S^n$ that $\I_m$ is also continuous up to the boundary $\partial \Delta$. As for $\I$, it is lower semi-continuous on $\Delta$ by lower semi-continuity of the perimeter (see e.g. \cite[Lemma 9.2]{EMilmanNeeman-GaussianMultiBubble}), and so $\I - \I_m$ must attain a global minimum at some $v_0 \in \Delta$. Assume in the contrapositive that $\I(v_0) - \I_m(v_0) < 0$; since we assume $\I^{(q-1)} = \I^{(q-1)}_m$ on $\partial \Delta$, necessarily $v_0 \in \interior \Delta$.

Let $\Omega$ be a minimizing $q$-cluster with $V(\Omega) = v_0$ as in Definition \ref{def:viscosity}.  
    By the results of Section \ref{sec:prelim}, $\Omega$ is regular and stationary with Lagrange multiplier $\lambda \in E^{(q-1)}$.
    Since $v_0$ is a minimum of $\I - \I_m$, for any smooth flow $F_t$ along a $C_c^\infty$ vector-field $X$ on $\S^n$,
    we have
    \begin{equation}\label{eq:I_m-and-I}
        \I_m(V(F_t(\Omega))) \leq \I(V(F_t(\Omega)) - (\I - \I_m)(v_0) \leq \per(F_t(\Omega)) - (\I - \I_m)(v_0).
    \end{equation}
        The two sides are equal when $t = 0$, and twice differentiable in $t$. By comparing first and second
    derivatives at $t=0$, it follows that:
    \begin{equation}
        \label{eq:compare-first}
        \scalar{\nabla \I_m(v_0),\delta_X V} = \delta_X A ,
    \end{equation}
    \begin{equation}
        \label{eq:compare-second}
        (\delta_X V)^T \nabla^2 \I_m(v_0) \delta_X V  + \scalar{\nabla \I_m(v_0),\delta_X^2 V} \leq \delta_X^2 A.
    \end{equation}
    
    Applying the assumption that (\ref{eq:PDI-F}) holds in the viscosity sense of Definition \ref{def:viscosity}, we may find a vector field $X$ with $\delta^1_X V = \F a$ where $\F > 0$ is full-rank on $E^{(q-1)}$ and $a \in E^{(q-1)}$ is arbitrary. Recalling 
   that $\delta_X A = \scalar{\lambda,\delta_X V}$ by stationarity (Lemma \ref{lem:Lagrange}) and comparing with~\eqref{eq:compare-first}, it follows that:
    \begin{equation} \label{eq:compare-fourth}
     \nabla \I_m(v_0) = \lambda .
     \end{equation}
    Plugging this into~\eqref{eq:compare-second}, we obtain:
    \begin{equation} \label{eq:compare-third}
        (\delta_X V)^T \nabla^2 \I_m(v_0) \delta_X V  \leq \delta_X^2 A - \scalar{\lambda,\delta_X^2 V} = Q(X).
    \end{equation}
    Applying (\ref{eq:def-viscosity}) again, it follows that for all $a \in E^{(q-1)}$ and $\eps > 0$:
    \[
    a^T \F \nabla^2 \I_m(v_0) \F a \leq -(n-1) a^T \F a + \eps .
    \]
    Taking the limit $\eps \rightarrow 0$, we deduce that as quadratic forms on $E^{(q-1)}$:
    \[
    \F \nabla^2 \I_m(v_0) \F \leq -(n-1) \F . 
    \]
    Since $\F>0$ was assumed positive-definite on $E^{(q-1)}$, it follows that $\nabla^2 \I_m(v_0) < 0$ and:
    \begin{equation} \label{eq:compare-fifth}
    0 < (-\nabla^2 \I_m(v_0))^{-1} \leq \frac{1}{n-1} \F . 
    \end{equation}
    
    Finally, recall that $A,B \geq 0$ implies that $\tr(A B) \geq 0$. Consequently, multiplying (\ref{eq:compare-fifth}) by the positive-definite $\Id/2 + \frac{1}{(n-1)^2} \nabla \I_m(v_0) \otimes \nabla \I_m(v_0) : E^{(q-1)} \rightarrow E^{(q-1)}$ and tracing, recalling (\ref{eq:compare-fourth}) and the PDE (\ref{eq:PDE}) satisfied by $\I_m$, it follows that:
    \begin{align*}
        \frac{1}{n-1} \I_m(v_0) & = \tr\brac{(-\nabla^2 \I_m(v_0))^{-1} \brac{\Id/2 + \frac{1}{(n-1)^2} \nabla \I_m(v_0) \otimes \nabla \I_m(v_0)} } \\
        & \leq  \frac{1}{n-1} \tr \brac{ \F \brac{\Id/2 + \frac{1}{(n-1)^2} \lambda \otimes \lambda} } \leq \frac{1}{n-1} \I(v_0) ,
        \end{align*}
        where the last inequality is due to our assumption (\ref{eq:F-trace-inq}). We deduce $\I_m(v_0) \leq \I(v_0)$, contradicting our original assumption that $\I(v_0) < \I_m(v_0)$ and concluding the proof. 
\end{proof}

Note that the assumption $\I^{(q-1)} = \I_m^{(q-1)}$ on $\partial \Delta^{(q-1)}$ is the same as the assumption that $\I^{(q-2)} = \I_m^{(q-2)}$ on the entire $\Delta^{(q-2)}$ (since at least one of the $v_i$'s vanishes for $v \in \partial \Delta^{(q-1)}$, thereby effectively reducing the number of cells). Consequently, we may use induction on $q$ in conjunction with Proposition \ref{prop:max-principle} to establish that $\I^{(q-1)}  = \I^{(q-1)}_m$ on $\Delta^{(q-1)}$. As the basis of induction, we can either start at $q=2$ and use the known single-bubble spherical isoperimetric inequality $\I^{(1)} = \I_m^{(1)}$, or start at $q=1$ (which holds trivially since $\I^{(0)}(1) = \I_m^{(0)}(1) = 0$) and derive an independent proof of the single-bubble case $q=2$ using Proposition \ref{prop:max-principle}. We thus obtain:

\begin{corollary}
Let $2 \leq p \leq n+2$,  and assume that for all $q=2,\ldots,p-1$, (\ref{eq:PDI-F}) holds for $\I = \I^{(q-1)}$ in the viscosity sense on $\Delta^{(q-1)}$. Then $\I^{(p-1)} = \I_m^{(p-1)}$ on the entire $\Delta^{(p-1)}$.
\end{corollary}

\subsection{What we can show}

We can show that for every $v_0 \in \interior \Delta^{(q-1)}$, 
\begin{equation} \label{eq:PDI-F1}
\F^T \nabla^2 \I \F \leq -(n-1) \F
\end{equation}
holds in the viscosity sense for a certain positive semi-definite $\F = \F(v_0) \geq 0$ (to be defined later on), and that:
\begin{equation} \label{eq:PDI-Feps}
\F^T_{\eps} \nabla^2 \I \F_{\eps} \leq -(n-1) \F_{\eps}
\end{equation}
 for the positive-definite $\F_{\eps} = \F + \eps \P_{\ker \F} > 0$ and all $\eps > 0$, where $\P_{\ker \F}$ denotes orthogonal projection onto the kernel $\ker \F$. 

\medskip

Unfortunately, we are not able to show in full generality the trace inequality (\ref{eq:F-trace-inq}), and thus conclude (\ref{eq:PDI-F}).  
Instead, we are able to show that a certain relaxation of $\F$, denoted $\F_0$, is positive-definite and satisfies 
the trace inequality (\ref{eq:F-trace-inq}) with equality; we conjecture that $\F = \F_0$, but we do not know how to establish this in general. Note that (\ref{eq:PDI-Feps}) is already enough to conclude that $\I$ is strictly concave in the viscosity sense, and hence in the usual sense, yielding Theorem \ref{thm:intro-profile-concave}.

\medskip

In this work, we are able to establish the trace inequality (\ref{eq:F-trace-inq}) (in fact, with equality) at $v_0 \in \interior \Delta^{(q-1)}$ if there exists a minimizer $\Omega$ with $V(\Omega) = v_0$ which satisfies either of the following two conditions:
\begin{enumerate}
\item $\Omega$ is pseudo conformally flat (PCF) -- in this case we show that $\F = \F_0$ (and in fact, we may also handle a generalization of the PCF condition described in Section \ref{sec:LSE}); or
\item $\Omega$ is $(q-3)$-Plateau -- in this case, by Lemma \ref{lem:q-3-Plateau}, the cluster is either PCF or fully Plateau. In the latter scenario, we will show that it is possible to deform the cluster into a full-dimensional one without changing its volume or total perimeter. In both scenarios, we will have found a minimizer of volume $v_0$ which is PCF,  thereby reducing to the first case. 
\end{enumerate}
This outlines our strategy for establishing Theorem \ref{thm:intro-conditional} from the Introduction. Recalling that every minimizing cluster is $3$-Plateau by Lemma \ref{lem:3-Plateau}, this establishes the quintuple-bubble case $q=6$ whenever $n \geq 5$ (Theorem \ref{thm:intro-quintuple}), as well as providing an alternative proof of the double, triple and quadruple cases from \cite{EMilmanNeeman-TripleAndQuadruple}.

\section{Spectral theory of $Q^0$ and $L_{Jac}$} \label{sec:spectral}

In this section we extend the definitions of the symmetric quadratic form $Q$ and symmetric operator $L_{Jac}$ defined for smooth vector- and scalar-fields, to a closed hermitian form and self-adjoint operator with compact resolvent (respectively) defined on appropriate dense domains of $L^2(\Sigma^1)$. In particular, we will see that the conformal boundary conditions arise naturally via this procedure, and we will make a noteworthy (yet tangential) connection to the quantum graphs formalism \cite{BerkolaikoKuchment-QuantumGraphs}. 

\subsection{Sobolev Spaces} 

We assume throughout this section that $\Omega$ is a bounded regular stationary cluster on $(M,g,\mu = \Psi \vol_g)$ with bounded curvature. 
Recall that $\mu^{n-1}(\Sigma^1) < \infty$ and that $\mu^{n-2}(\Sigma^2) < \infty$ by Lemma \ref{lem:Sigma2}; as the cluster is bounded and $\Psi$ is bounded away from $0$ and $\infty$ on compact sets, this is equivalent to $\H^{n-1}(\Sigma^1) < \infty$ and $\H^{n-2}(\Sigma^2)  < \infty$, respectively. We will implicitly assume that all integration in the Lebesgue and Sobolev spaces introduced below are with respect to $\mu^{n-1}$ on $\Sigma^1$ and $\mu^{n-2}$ on $\Sigma^2$ (instead of $\H^{n-1}$ and $\H^{n-2}$, respectively) -- as the measures $\mu^k$ and $\H^k$ are equivalent on compact sets, this makes absolutely no difference in the resulting spaces, qualitative estimates, etc...  

Let $H^1(\Sigma_{ij})$ denote the Sobolev space of $L^2$-integrable functions $u_{ij}$ on $\Sigma_{ij}$ having distributional tangential derivative $\nabla^\tang u_{ij}$ in $L^2(\Sigma_{ij})$, and set:
\[
\norm{u_{ij}}_{H^1(\Sigma_{ij})}^2 := \norm{u_{ij}}_{L^2(\Sigma_{ij})}^2 + \norm{\nabla^\tang u_{ij}}_{L^2(\Sigma_{ij})}^2 = \int_{\Sigma_{ij}} (|\nabla^\tang u_{ij}|^2  + |u_{ij}|^2) d\mu^{n-1} < \infty . 
\]
To obtain the desired spectral theory, we will initially consider complex-valued Sobolev functions below, and endow $L^2$ with its standard sesquilinear scalar product $\scalar{\cdot,\cdot}_{L^2}$. However, after the spectral theory is obtained, we will tacitly revert back to using real-valued functions in all of our function classes. 

Recall our notation $\partial \Sigma_{ij} = \cup_{k \neq i,j} \Sigma_{ijk}$  and that $\Sigma_{ij} \cup \partial \Sigma_{ij}$ is a (possibly incomplete) $C^\infty$ Riemannian manifold with $C^\infty$ boundary. For consistency, we denote $\partial \Sigma^1 := \Sigma^2$. 
 For simplicity, we will in addition assume that:
 
 \begin{enumerate}[(Lip)]
 \item \label{eq:Lip}
 For all $i,j,k$, $\Sigma_{ij}$ and $\Sigma_{ijk}$  have Lipschitz topological boundaries $\overline{\partial \Sigma_{ij}}$ and $\overline{\partial \Sigma_{ijk}}$.
 \end{enumerate}

\noindent This certainly holds for any $3$-Plateau spherical Voronoi cluster on $\S^n$, since each cell $\Omega_i$ is obtained as (the interior of) the intersection of a closed convex polyhedron $P_i$ with $\S^n$, and Lemma \ref{lem:isolated} and non-degeneracy ensure that the boundary of each non-empty cell remains Lipschitz after the intersection with $\S^n$, as the half-spaces whose intersection comprises $P_i$ locally around $p \in \partial \Omega_i$ will have normals with distinct projections to $T_p \S^n$.
See also \cite{RozanovaPierrat-Survey} for more general conditions (such as Ahlfors regularity) which ensure the properties we require. Under \ref{eq:Lip}, the trace space $H^{1/2}(\overline{\partial \Sigma_{ij}})$ of $H^1(\Sigma_{ij})$ is well-defined by using a partition of unity and pushing-forward $H^{1/2}(\R^{n-2})$ using the Lipschitz local boundary map \cite{Marschall-TraceOfSobolev, Grisvard-Book}.  Since $\overline{\partial \Sigma_{ij}} \setminus \partial \Sigma_{ij} \subset \Sigma^3 \cup \Sigma^4$, and in particular $\H^{n-2}(\overline{\partial \Sigma_{ij}} \setminus \partial \Sigma_{ij}) = 0$, we make no distinction between $L^2(\overline{\partial \Sigma_{ij}})$ and $L^2(\partial \Sigma_{ij})$ below. Moreover, as $C^\infty(\Sigma_{ijk})$ is dense in $H^{1/2}(\Sigma_{ijk})$ and $\Sigma_{ijk}$ is assumed to have Lipschitz topological boundary, a standard cut-off argument (cf.~\cite[Theorem 1.4.2.4]{Grisvard-Book}) ensures that actually $C_c^\infty(\Sigma_{ijk})$ is dense in $H^{1/2}(\Sigma_{ijk})$ (much in contrast to the situation with $H^s(\Sigma_{ijk})$ when $s > 1/2$). 
  It follows that $C_c^\infty(\partial \Sigma_{ij})$ is dense in $H^{1/2}(\partial \Sigma_{ij})$, and that $H^{1/2}(\overline{\partial \Sigma_{ij}}) \simeq \oplus_k H^{1/2}(\Sigma_{ijk}) \simeq H^{1/2}(\partial \Sigma_{ij})$, and so we will make no distinction between $H^{1/2}(\overline{\partial \Sigma_{ij}})$ and $H^{1/2}(\partial \Sigma_{ij})$ below as well.
 In addition, the following properties hold \cite{Grisvard-Book,LionsMagenes-Vol1}:
\begin{itemize}
\item $H^1(\Sigma_{ij})$ embeds compactly in $L^2(\Sigma_{ij})$ (Rellich--Kondrachov Theorem). 
\item There exists a well-defined, linear, continuous and surjective trace operator 
\[
H^1(\Sigma_{ij}) \ni u_{ij} \mapsto \Tr u_{ij} \in H^{1/2}(\partial \Sigma_{ij}) ,
\]
extending the usual restriction operator of $C^1(\overline{\Sigma_{ij}})$ functions onto $\partial \Sigma_{ij}$. Moreover, the trace has a linear and continuous right inverse $\mathcal{E} : H^{1/2}(\partial \Sigma_{ij}) \rightarrow H^1(\Sigma_{ij})$, so that $\Tr \circ \mathcal{E} = \Id$. 
\item There exists  $\alpha_0 > 0$ so that for all $i,j$:
\[
\forall 0 < \alpha < \alpha_0 \;\; \exists M_\alpha < \infty \;\;\;
 \int_{\partial \Sigma_{ij}} |\Tr u_{ij}|^2 d\mu^{n-2} \leq \alpha \int_{\Sigma_{ij}}|\nabla^\tang u_{ij}|^2 d\mu^{n-1} + M_\alpha \int_{\Sigma_{ij}} |u_{ij}|^2 d\mu^{n-1}  .
\]
\item $H_0^1(\Sigma_{ij})$, the closure of $C_c^\infty(\Sigma_{ij})$ in $H^1(\Sigma_{ij})$, coincides with $\{ u_{ij} \in H^1(\Sigma_{ij}) \; ; \; \Tr u_{ij} \equiv 0\}$ and is dense in $L^2(\Sigma_{ij})$. 
\item Given $u_{ij}\in H^1(\Sigma_{ij})$, $\Delta_{\Sigma,\mu} u_{ij} \in H^{-1}(\Sigma_{ij})$ denotes its distributional (weighted) surface Laplacian via duality with $H^1_0(\Sigma_{ij})$. In particular, $ \Delta_{\Sigma,\mu} u_{ij} \in L^2(\Sigma_{ij})$ iff there exists $f_{ij} \in L^2(\Sigma_{ij})$ so that:
\begin{equation} \label{eq:distributional-Delta}
\int_{\Sigma_{ij}} \scalar{\nabla^\tang u_{ij}, \nabla^\tang v_{ij}} d\mu^{n-1} = -\int_{\Sigma_{ij}} f_{ij} v_{ij} d\mu^{n-1} \;\;\; \forall v_{ij} \in H^1_0(\Sigma_{ij}) ,
\end{equation}
and in that case $\Delta_{\Sigma,\mu} u_{ij} = f_{ij}$ (uniquely defined by the previous bullet point). 
\item The divergence theorem holds in the following form for all $u,v \in H^1(\Sigma_{ij})$ so that $\Delta_{\Sigma,\mu} u \in L^2(\Sigma_{ij})$:
\begin{equation} \label{eq:divergence-Sobolev}
\int_{\Sigma_{ij}} \scalar{\nabla^\tang u_{ij}, \nabla^\tang v_{ij}} d\mu^{n-1} = -\int_{\Sigma_{ij}} (\Delta_{\Sigma,\mu} u_{ij}) v_{ij} d\mu^{n-1} + \scalar{\NN u_{ij} , \Tr v_{ij}}_{H^{-\frac{1}{2},\frac{1}{2}}(\partial \Sigma_{ij})}  , 
\end{equation}
where $\NN u_{ij} \in H^{-1/2}(\partial \Sigma_{ij})$ is a continuous linear functional acting on elements in $H^{1/2}(\partial \Sigma_{ij})$ (via the coupling $\scalar{\cdot,\cdot}_{H^{-\frac{1}{2},\frac{1}{2}}(\partial \Sigma_{ij})}$), which extends the normal partial derivative $\nabla_{\n_{\partial ij}} f_{ij}$ on $\partial \Sigma_{ij}$ of $C^1(\overline{\Sigma_{ij}})$ functions $f_{ij}$, acting by integration on $\partial \Sigma_{ij}$ with respect to $\mu^{n-2}$.
\end{itemize}

\medskip

 Denote $H^1(\Sigma^1) := \oplus_{i<j} H^1(\Sigma_{ij})$ and $H^1_0(\Sigma^1) := \oplus_{i<j} H^1_0(\Sigma_{ij})$ the direct sum of the respective Sobolev spaces on the individual interfaces $\Sigma_{ij}$, and similarly for $L^2(\Sigma^1)$. We force these spaces to be oriented by defining $u_{ij} := -u_{ji}$ when $i > j$. An element $u \in H^1(\Sigma^1)$ is thus comprised of oriented Sobolev functions $u = (u_{ij})_{i \neq j}$ with $u_{ij} = -u_{ji}\in H^1(\Sigma_{ij})$, so that:
\[
\norm{u}^2_{H^1(\Sigma^1)} := \sum_{i<j} \norm{u_{ij}}_{H^1(\Sigma_{ij})}^2 = \norm{u}_{L^2(\Sigma^1)}^2 + \norm{\nabla^\tang u}_{L^2(\Sigma^1)}^2 < \infty . 
\]
As usual, $\Delta_{\Sigma,\mu} u \in H^{-1}(\Sigma^1)$ denotes the distributional (weighted) surface Laplacian of $u \in H^1(\Sigma^1)$ via duality with $H^1_0(\Sigma^1)$; in particular $\Delta_{\Sigma,\mu} u \in L^2(\Sigma^1)$ iff $\exists f \in L^2(\Sigma^1)$ so that:
\[
\int_{\Sigma^1} \scalar{\nabla^\tang u, \nabla^\tang v} d\mu^{n-1} = -\int_{\Sigma^1} f v d\mu^{n-1} \;\;\; \forall v \in H^1_0(\Sigma^1) ,
\]
and in that case $\Delta_{\Sigma,\mu} u = f$. 

We will also use the notation:
\begin{align*}
& H^1(\Sigma^1) \ni u \mapsto \Tr u \in H^{1/2}(\partial \Sigma^1, \Complex^3) ~,~\\
& \Tr u := (\Tr u_{ij} , \Tr u_{jk} , \Tr u_{ki}) \text{ on } \Sigma_{ijk} ~,~  i < j < k ,
\end{align*}
and:
\begin{align*}
& H^1(\Sigma^1) \ni u \mapsto \NN u \in H^{-1/2}(\partial \Sigma^1, \Complex^3) ~,~ \\
& \NN u := (\NN u_{ij} , \NN u_{jk} , \NN u_{ki} ) \text{ on } \Sigma_{ijk} ~,~ i < j < k  . 
\end{align*}
We endow $\Complex^3$ with its standard sesquilinear scalar-product $\scalar{\cdot,\cdot}_{\Complex^3}$, and use it to sesquilinearly extend the action of $N \in H^{-1/2}(\partial \Sigma^1, \Complex^3)$ on $T \in H^{1/2}(\partial \Sigma^1, \Complex^3)$, namely:
\[
\scalar{N, T}_{H^{-\frac{1}{2},\frac{1}{2}}(\partial \Sigma^1)} = \scalar{N_1,\bar T_1}_{H^{-\frac{1}{2},\frac{1}{2}}(\partial \Sigma^1)} + \scalar{N_2,\bar T_2}_{H^{-\frac{1}{2},\frac{1}{2}}(\partial \Sigma^1)} + \scalar{N_3,\bar T_3}_{H^{-\frac{1}{2},\frac{1}{2}}(\partial \Sigma^1)} . 
\]

\subsection{Closed semi-bounded sesquilinear hermitian form}

Observe that the definition from (\ref{eq:Q1}) of the symmetric bilinear Index-Form $Q(X,Y)$ for $C_c^\infty$ vector fields $X,Y$ only depends on their normal components $X^\n$ and $Y^\n$. Consequently, we may define an analogous form $Q^0(f,g)$ for scalar-fields $f,g$ which do not necessarily arise as the normal component of a smooth (or for that matter, any) vector-field. This was done in \cite{EMilmanNeeman-TripleAndQuadruple} for Lipschitz scalar-fields, but for deriving a proper spectral-theory, which we shall require in this work, we will need a definition which applies to arbitrary Sobolev scalar-fields. 
Consequently, we define the following sesquilinear hermitian form for arbitrary $u,v \in H^1(\Sigma^1)$:

\begin{definition}[Index-Form $Q^0$ for $u,v \in H^1(\Sigma^1)$]
\[ Q^0(u,v) := \sum_{i<j} \Big [ \int_{\Sigma_{ij}} \brac{\scalar{\nabla^\tang u_{ij} , \nabla^\tang \bar v_{ij}} - (\Ric_{g,\mu}(\n,\n) + \|\II\|_2^2) u_{ij} \bar v_{ij}} d\mu^{n-1}  - \int_{\partial \Sigma_{ij}} \bar \II^{\partial ij} \Tr u_{ij} \Tr \bar v_{ij}  \, d\mu^{n-2} \Big ] . 
\] As usual, we set $Q^0(u) := Q^0(u,u)$.
\end{definition}
After developing the corresponding spectral theory for $Q^0$ over $\Complex$, we will tacitly revert back to treating $Q^0$ as a symmetric bilinear form on real-valued Sobolev functions. 

\medskip

Denoting $K := \sup \{\norm{\II(p)}_{op} \; ; \; p \in \Sigma^1 \}$, we have $|\bar \II^{\partial ij}| \leq \frac{2 K}{\sqrt{3}}$ and $\|\II\|_2^2 \leq (n-1) K^2$. Further denoting 
$R_+ := \sup \{ \Ric_{g,\mu}(\n,\n) ; p \in \Sigma^1 \}$ and $R_- := \inf \{ \Ric_{g,\mu}(\n,\n) ;  p \in \Sigma^1 \}$, it easily follows that:
\begin{equation} \label{eq:upper-bounded}
Q^0(u) \leq \brac{ \brac{ 1 + \frac{2 K}{\sqrt{3}} \frac{\alpha_0}{2}} \norm{\nabla^\tang u}_{L^2(\Sigma^1)}^2 + (M_{\alpha_0/2} \frac{2 K}{\sqrt{3}} - R_-) \norm{u}_{L^2(\Sigma^1)}^2 } \;\;\; \forall u \in H^1(\Sigma^1) ,
\end{equation}
and that moreover:
\begin{equation} \label{eq:semi-bounded}
Q^0(u) \geq \frac{1}{2} \norm{\nabla^\tang u}_{L^2(\Sigma^1)}^2 -C \norm{u}_{L^2(\Sigma^1)}^2 \geq -C \norm{u}_{L^2(\Sigma^1)}^2  \;\;\; \forall u \in H^1(\Sigma^1)  
\end{equation}
where $C = M_{\alpha_1} \frac{2 K}{\sqrt{3}} + (R_+ + (n-1) K^2)$ and $\alpha_1 := \frac{1}{2} \min(\alpha_0,  \frac{\sqrt{3}}{2 K})$. In particular, $Q^0$ is $L^2(\Sigma^1)$-bounded-from-below. Introducing the following notation:
\[
\norm{u}_{Q,C}^2 := (C+1) \norm{u}^2_{L^2(\Sigma^1)} + Q^0(u) ~,~ u \in H^1(\Sigma^1) ,
\]
it follows that $\norm{u}_{Q,C} \geq 0$ with $\norm{u}_{Q,C} = 0$ iff $u=0$, and we conclude that $\norm{\cdot}_{Q,C}$ is a norm on $H^1(\Sigma^1)$. 

 Note that (\ref{eq:upper-bounded}) and (\ref{eq:semi-bounded}) immediately imply that:
 \begin{equation} \label{eq:QC-equivalent-H1}
 \text{  $\norm{\cdot}_{Q,C}$ is equivalent to $\norm{\cdot}_{H^1(\Sigma^1)}$ on $H^1(\Sigma^1)$. }
 \end{equation}
 As simple consequences of this equivalence, we record:

\begin{lemma} \label{lem:Q0-bounds} \hfill
\begin{itemize}
\item There exists $M > 0$ so that for all $u,v \in H^1(\Sigma^1)$:
\[
\abs{Q^0(u,v)} \leq M \norm{u}_{H^1(\Sigma^1)} \norm{v}_{H^1(\Sigma^1)} . 
\]
\item
In particular:
\[
|Q^0(u) - Q^0(v)| \leq M (2 \norm{u-v}_{H^1(\Sigma^1)} \norm{v}_{H^1(\Sigma^1)} + \norm{u-v}^2_{H^1(\Sigma^1)}) ,
\]
and $H^1(\Sigma^1) \ni u \rightarrow Q^0(u)$ is continuous.
\end{itemize}
\end{lemma} 

Recall (e.g \cite[Chapter 4]{Davies-SpectralTheoryBook}) that an $L^2(\Sigma^1)$-semi-bounded sesquilinear hermitian form $Q^0$ defined on a linear subspace $\D$ of $L^2(\Sigma^1)$ is called \emph{closed} if $\D$ is complete with respect to $\norm{\cdot}_{Q,C}$. As $\norm{\cdot}_{Q,C}$ and $\norm{\cdot}_{H^1(\Sigma^1)}$ are equivalent on $H^1(\Sigma^1)$, and since the latter linear space is complete with respect to $\norm{\cdot}_{H^1(\Sigma^1)}$, it follows that $Q^0$ is closed on any \emph{closed} linear subspace of $H^1(\Sigma^1)$.

\subsection{Boundary projectors, form domain and scalar-fields}

Introduce the following two orthogonal projections $P_D, P_R$ on $\Complex^3$ given by $P_D := \frac{1}{3} \mathbf{1} \otimes \mathbf{1}$ and $P_R := \Id - P_D$.  
The 'D' stands for Dirichlet and the 'R' for Robin, as will be apparent soon; as we have no Neumann boundary conditions in our setting, we refrain from introducing the unnecessary $P_N = 0$. Note that $P_D a = 0$ iff $a_1 + a_2 + a_3 = 0$.

We now finally introduce the $Q^0$ form-domain:
\[
\D_Q := \set{ u \in H^1(\Sigma^1) \; ; \; P_D \Tr u = 0 } ,
\]
which corresponds to elements satisfying the Kirchoff-Dirichlet boundary conditions. Since $H_0^1(\Sigma_{ij})$ is dense in $L^2(\Sigma_{ij})$, clearly $\D_Q$ is dense in $L^2(\Sigma^1)$. Furthermore, since the trace operator is continuous, $\D_Q$ is a closed linear subspace of $H^1(\Sigma^1)$, and so by the preceding subsection it follows that $Q^0$ is closed on $\D_Q$. 

\medskip

Elements of $\D_Q$ are called Sobolev scalar-fields. It will be useful to consider the following concrete sub-class of $\D_Q$:

\begin{definition}[Smooth scalar-fields]
A tuple $f = (f_{ij})_{i \neq j \in \{1,\ldots,q\}}$ of $C_c^\infty$-smooth functions defined on the closure $\overline{\Sigma_{ij}}$ is called a (compactly supported) smooth scalar-field on $\Sigma$, denoted $f \in C^\infty_c(\Sigma)$, if:
\begin{enumerate}
\item $f$ is oriented, i.e. $f_{ji} = -f_{ij}$; and
\item $f$ satisfies Dirichlet-Kirchoff boundary conditions: at every triple-point in $\Sigma_{ijk}$, we have $f_{ij} + f_{jk} + f_{ki} = 0$. 
\end{enumerate} 
\end{definition}

\subsection{Conformal boundary conditions}

It is interesting and useful at this point to invoke the quantum graphs formalism -- we refer to \cite[Section 1.4]{BerkolaikoKuchment-QuantumGraphs} for a general overview on the subject. Note that for $u \in \D_Q$, since $P_D \Tr u = 0$ then $P_R \Tr u = \Tr u$. Consequently, for $u,v \in \D_Q$ we may write:
\[
Q^0(u,v) := \int_{\Sigma^1} \brac{\scalar{\nabla^\tang u , \nabla^\tang \bar v} - (\Ric_{g,\mu}(\n,\n) + \|\II\|_2^2) u \bar v} d\mu^{n-1}  - \int_{\partial \Sigma^1} \scalar{\Lambda P_R \Tr u  , P_R \Tr v}_{\Complex^3} \, d\mu^{n-2} ,
\]
where $\Lambda \in C^\infty(\partial \Sigma^1 , \MM(\Im P_R))$ is a field of hermitian operators on $\Im P_R$ given by $\Lambda = P_R \Lambda_0 P_R$, where $\Lambda_0 := \text{diag}(\bar \II^{\partial ij} , \bar \II^{\partial jk} , \bar \II^{\partial ki})$ on $\Sigma_{ijk}$ for $i < j < k$. Note that by cluster regularity and boundedness of curvature, $\Lambda$ is indeed $C^\infty$ smooth and bounded on each $\Sigma_{ijk}$. 
 
\begin{definition}[Conformal boundary conditions] \label{def:conformal-BCs}
A scalar-field $u \in \D_Q$ is said to satisfy conformal boundary conditions (BCs) if $u \in \D_{con}$, where:
\begin{align*}
\D_{con} & :=  \{ u \in \D_Q \; ; \; \; \Delta_{\Sigma,\mu} u \in L^2(\Sigma^1) \;,\; P_R \NN u = \Lambda P_R \Tr u \}  \\
& = \{ u \in H^1(\Sigma^1) \; ; \;  \Delta_{\Sigma,\mu} u \in L^2(\Sigma^1) \;,\; P_R \NN u = \Lambda P_R \Tr u \; , \; P_D \Tr u = 0 \} .
\end{align*}
\end{definition}

\begin{remark}
In the quantum graphs setting $\D_{con}$ corresponds to Kirchoff-Dirichlet BCs ($P_D \Tr u = 0$) coupled with Robin BCs ($P_R \NN u = \Lambda P_R \Tr u$) on the triple-point set $\Sigma^2 = \partial \Sigma^1$. Note that we must assume that $\Delta_{\Sigma,\mu} u \in L^2(\Sigma^1)$ in order for $\NN u$ to even be defined. 
\end{remark}

\begin{remark} \label{rem:conformal-BCs}
The reason for the nomenclature ``conformal BCs" will be made clear in Lemma \ref{lem:conformal-boundary} in the next section. For now, let us just point out that 
for smooth scalar-fields, $f \in C_c^\infty(\Sigma)$ satisfies conformal BCs iff $\nabla_{\n_{\partial ij}} f_{ij} - \bar \II^{\partial ij} f_{ij}$ is independent of $(i,j) \in \cyclic(u,v,w)$ on $\Sigma_{uvw}$, for all $u < v < w$. Indeed, since $P_D \Tr f = 0$ then $P_R \Tr f = \Tr f$,  and so $P_R \NN f = \Lambda P_R \Tr f$ is the same as $P_R (\NN f - \Lambda_0 \Tr f) = 0$. Recalling that for $f \in C_c^\infty(\Sigma)$, $\NN f_{ij} = \nabla_{\n_{\partial ij}} f_{ij}$ and $\Tr f_{ij} = f_{ij}$ on $\partial \Sigma_{ij}$, the assertion follows. 
\end{remark}

\begin{remark}
By elliptic regularity, it certainly holds when $\Delta_{\Sigma,\mu} u \in L^2_{loc}(\Sigma^1)$ that $u \in H^2_{loc}(\Sigma^1)$. However, since we only assume that $\overline{\partial \Sigma_{ij}}$ is Lipschitz, it is known that even in the presence of vanishing Dirichlet BCs, this does not guarantee in general that $u_{ij} \in H^2(\Sigma_{ij})$ \cite[Theorem 1.4.5.3]{Grisvard-Book} (cf. \cite{Savare-RegularityInLipDomains}). It might be possible to use the fact that $\Omega_i$ is obtained as the intersection of a convex polyherdon with $\S^n$ to show global $H^2$ regularity afterall, but we do not pursue this here. 
\end{remark}

By invoking the divergence theorem for Sobolev functions (\ref{eq:divergence-Sobolev}), recalling the definition of $L_{Jac}$ from (\ref{eq:def-LJac}), and using that $P_R$ is an orthogonal projection on $\Complex^3$, we immediately obtain:
\begin{lemma} \label{lem:Q^0-by-parts}
 For all $u,v \in \D_Q$ so that $\Delta_{\Sigma,\mu} u \in L^2(\Sigma^1)$:
\begin{equation} \label{eq:Q^0-by-parts}
Q^0(u,v) = - \scalar{L_{Jac} u , v}_{L^2(\Sigma^1)}  + \scalar{P_R \NN u - \Lambda P_R \Tr u , P_R \Tr v}_{H^{-\frac{1}{2},\frac{1}{2}}(\partial \Sigma^1)} .
\end{equation}
In particular:
\begin{equation} \label{eq:Q^0-Dcon}
Q^0(u,v) = - \scalar{L_{Jac} u , v}_{L^2(\Sigma^1)} \;\;\;  \forall u \in \D_{con} ~,~ \forall v \in \D_Q . 
\end{equation}
\end{lemma}

The hermitianity of $Q^0$ immediately verifies that $L_{Jac}$ is symmetric on $\D_{con}$. It turns out that in fact $(L_{Jac},\D_{con})$ is self-adjoint, as we shall verify in the next subsection.

\subsection{The induced self-adjoint operator}

By Kato's representation theorem, the closed sesquilinear hermitian semi-bounded form $Q^0$, defined on the dense subset $\D_Q$ of Hilbert space $L^2(\Sigma^1)$, induces a self-adjoint operator $A$ on $L^2(\Sigma^1)$ with $\H^1(\Sigma^1)$-dense domain $\D_A := \text{Dom}(A)$ in $\D_Q$,
in the following manner \cite[Section 4.4]{Davies-SpectralTheoryBook}, \cite[Section 3.3]{Helffer-Book}:
\[
\D_{A} := \{ u \in \D_Q \; ; \; \exists f_u \in L^2(\Sigma^1) \;\; Q^0(u,v) = \scalar{f_u , v}_{L^2(\Sigma^1)} \;\; \forall v \in \D_Q \} ~,~ A u := f_u .
\]
The density of $\D_Q$ in $L^2(\Sigma^1)$ clearly implies that $A$ is well-defined. 
By semi-boundedness (\ref{eq:semi-bounded}), clearly $A \geq -C \; \Id$, and furthermore, it is known that $\D_Q = \text{Dom}( (A + C \Id)^{1/2})$. It follows by the Rellich-Kondrachov theorem that $A + C \Id$ has compact resolvent, and thus the spectrum of $A$ is real and discrete, consisting of a sequence of eigenvalues of finite multiplicity increasing to infinity. In particular, $A$ is Fredholm. 

Recalling Lemma \ref{lem:Q^0-by-parts}, it should now be clear that $(A,\D_A)$ coincides with $(-L_{Jac}, \D_{con})$, thus obtaining an intrinsic interpretation of our conformal BCs as the natural ones arising in the latter standard procedure for constructing a self-adjoint operator from a closed sesquilinear hermitian semi-bounded form. Let us quickly verify this:

\begin{proposition} \label{prop:enter-LJac}
$(A , \D_A) = (-L_{Jac} , \D_{con})$. 
\end{proposition}
\begin{proof}
Clearly $\D_{con} \subset \D_{A}$ and $A|_{\D_{con}} = L_{Jac}$, since if $u \in \D_{con}$ then $\Delta_{\Sigma,\mu} u$ and hence $L_{Jac} u$ are in $L^2(\Sigma^1)$ and so Lemma \ref{lem:Q^0-by-parts} applies and the boundary term in (\ref{eq:Q^0-by-parts}) vanishes. 

Conversely, if $u \in \D_{A}$, then $Q^0(u,v) = \scalar{A u , v}_{L^2(\Sigma^1)}$ for some $A u \in L^2(\Sigma^1)$ and all $v \in \D_Q$. Testing this for $v \in H^1_0(\Sigma^1) \subset \D_Q$,  it follows from (\ref{eq:distributional-Delta}) that $-\Delta_{\Sigma,\mu} u  = A u + (\Ric_{g,\mu}(\n,\n) + \|\II\|_2^2) u \in L^2(\Sigma^1)$. Consequently, we may invoke Lemma \ref{lem:Q^0-by-parts} to deduce that:
\begin{equation} \label{eq:punch}
\scalar{A u + L_{Jac} u , v}_{L^2(\Sigma^1)} = \scalar{P_R \NN u - \Lambda P_R \Tr u , P_R \Tr v}_{H^{-\frac{1}{2},\frac{1}{2}}(\partial \Sigma^1)} \;\;\; \forall v \in \D_Q . 
\end{equation}
Note that we've already verified that $Au + L_{Jac} u = 0$, and so the left-hand-side vanishes. 
Given $T \in H^{1/2}(\partial \Sigma^1,\Im P_R)$, since $\Tr : H^1(\Sigma_{ij}) \mapsto H^{1/2}(\partial \Sigma_{ij})$ is surjective, we may find $v_T \in H^1(\Sigma^1)$ so that $\Tr v_T = T$; since $P_D T = 0$ in fact $v_T \in \D_Q$. Applying (\ref{eq:punch}) with $v = v_T$, we deduce that:
\[
0 = \scalar{P_R \NN u - \Lambda P_R \Tr u , T}_{H^{-\frac{1}{2},\frac{1}{2}}(\partial \Sigma^1)}  . 
\]
As $T \in H^{1/2}(\partial \Sigma^1,\text{Im} P_R)$ was arbitrary and as $P_R \NN u - \Lambda P_R \Tr u \in H^{-1/2}(\partial \Sigma^1,\text{Im} P_R)$, it follows that necessarily $P_R \NN u - \Lambda P_R \Tr u = 0$, i.e. that $u \in \D_{con}$.
\end{proof}

\subsection{Density of smooth scalar-fields in $\D_Q$}

\begin{definition}[Smooth scalar-fields compactly supported away from $\Sigma^{\geq 3}$]
A smooth scalar-field $f \in C_c^\infty(\Sigma)$ is said to be compactly supported away from $\Sigma^{\geq 3}$, denoted $f \in C^\infty_{c}(\Sigma \setminus \Sigma^{\geq 3})$, if the (compact) support of each $f_{ij}$ in $\overline{\Sigma_{ij}}$ is disjoint from $\Sigma^{\geq 3}$.
\end{definition}

\begin{proposition} \label{prop:dense}
$C_c^\infty(\Sigma \setminus \Sigma^{\geq 3})$ is $H^1(\Sigma^1)$-dense in $\D_Q$.
\end{proposition}

\begin{proof}Let $u = (u_{ij}) \in \D_Q$, and fix $\eps > 0$. Denote $u_{ij,k} := \Tr u_{ij}|_{\Sigma_{ijk}} \in H^{1/2}(\Sigma_{ijk})$ for all $i,j,k$. As $C_c^\infty(\Sigma_{ijk})$ is dense in $H^{1/2}(\Sigma_{ijk})$, we may find $g_{ij,k} \in C_c^\infty(\Sigma_{ijk})$ with $\snorm{u_{ij,k} - g_{ij,k}}_{H^{1/2}(\Sigma_{ijk})} \leq \eps$. Moreover, for each $i<j<k$, as $u_{ij,k} + u_{jk,i} + u_{ki,j} = 0$, we may ensure that $g_{ij,k} + g_{jk,i} + g_{ki,j} = 0$ as well by redefining $g_{ki,j} :=  - (g_{ij,k} + g_{jk,i})$; clearly by the triangle inequality $\snorm{u_{ki,j} - g_{ki,j}}_{H^{1/2}(\Sigma_{ijk})} \leq 2\eps$. Setting $g_{ij} := g_{ij,k}$ on $\Sigma_{ijk}$ and zero on $\overline{\partial \Sigma_{ij}} \setminus \partial \Sigma_{ij}$, the regularity of the cluster asserted in Theorem \ref{thm:regularity} \ref{it:regularity-Sigma2} ensures that $g_{ij}$ is a smooth function on $\overline{\partial \Sigma_{ij}}$ (in the sense of \cite[Chapter 2]{Lee-SmoothManifolds}); consequently, $g_{ij}$ extends to a global smooth function $h_{ij}$ defined on $\overline{\Sigma_{ij}}$ by e.g. \cite[Lemma 2.26]{Lee-SmoothManifolds}. 
Define $v_{ij} := u_{ij} - h_{ij} \in H^1(\Sigma_{ij})$, and observe that $\snorm{\Tr v_{ij}|_{\Sigma_{ijk}}}_{H^{1/2}(\Sigma_{ijk})} \leq 2 \eps$ for all $k$. Recalling that $H^{1/2}(\partial \Sigma_{ij}) \simeq \oplus_k H^{1/2}(\Sigma_{ijk})$, it follows that  $\snorm{\Tr v_{ij}}_{H^{1/2}(\partial \Sigma_{ij})} \leq C_{ij} \eps$. As the trace operator has a continuous right inverse on $H^{1/2}(\partial \Sigma_{ij})$, it follows that there exists $v^0_{ij} \in H^1(\Sigma_{ij})$ so that $\Tr v^0_{ij} = \Tr v_{ij}$ and $\snorm{v^0_{ij}}_{H^1(\Sigma_{ij})} \leq C'_{ij} \eps$. Now $v_{ij} - v^0_{ij} \in H_0^1(\Sigma_{ij})$, and by density of $C_c^\infty(\Sigma_{ij})$ in $H_0^1(\Sigma_{ij})$, there exists $h^0_{ij} \in C_c^\infty(\Sigma_{ij})$ so that $\snorm{v_{ij} - v^0_{ij} - h^0_{ij}}_{H^1(\Sigma_{ij})} \leq \eps$. Setting $f_{ij} := h_{ij} + h^0_{ij} \in C_c^\infty(\overline{\Sigma_{ij}} \setminus \Sigma^{\geq 3})$, as $u_{ij} - f_{ij} = (v_{ij} - v^0_{ij}- h^0_{ij} ) + v^0_{ij}$, we conclude that $\snorm{u_{ij} - f_{ij}}_{H^1(\Sigma_{ij})} \leq (C'_{ij}+1) \eps \leq C'' \eps$ for $C'' := \max_{i,j} (C'_{ij}+1)$. Moreover, $f_{ij}|_{\Sigma_{ijk}} = g_{ij,k}$ and therefore $f = (f_{ij})$ satisfies the Dirichlet-Kirchoff BCs. We conclude that $f = (f_{ij})$ is a smooth scalar-field compactly supported away from $\Sigma^{\geq 3}$ which approximates $u = (u_{ij})$ arbitrarily well in $H^1(\Sigma^1)$. 
 \end{proof}

\subsection{Density of physical fields}

\begin{definition}[Physical Scalar-Field]
A scalar-field $f = \{f_{ij} \} \in C^\infty_c(\Sigma)$ is called physical if there exists a $C^\infty_c$ vector-field $X$ on $(M,g)$ so that $f_{ij} = X^{\n_{ij}}$ on $\Sigma_{ij}$ (we will say that $f$ is derived from the physical vector-field $X$). Otherwise, it is called non-physical. 
\end{definition}

\begin{definition}[Scalar-field first variation of volume $\delta^1_u V$]
Given a Sobolev scalar-field $u \in \D_Q$, the first variation of volume $\delta^1_u V = \delta^1_u V(\Omega) \in E^{(q-1)}$ is defined as:
\[
(\delta^1_u V)_i = \delta^1_u V(\Omega_i) := \int_{\partial \Omega_i} u_{ij} d\mu^{n-1} = \sum_{j \neq i} \int_{\Sigma_{ij}} u_{ij} d\mu^{n-1} ~,~ i=1,\ldots,q . 
\]
Equivalently:
\[
\delta^1_u V = \delta^1_u V(\Omega) := \sum_{i<j} \int_{\Sigma_{ij}} u_{ij} d\mu^{n-1} e_{ij}  \in E^{(q-1)} . 
\]
\end{definition}

Note that by Lemma \ref{lem:Lagrange},  Theorem \ref{thm:Q-Sigma4}, 
whenever $f$ is a physical scalar-field derived from a $C_c^\infty$ vector-field $X$, respectively, we have:
\[
\delta^1_f V = \delta^1_X V \text{ and } Q^0(f) = Q(X) . 
\]

\medskip

Working with scalar-fields is in practice much more convenient than with vector-fields. However, given a general scalar-field $f \in C_c^\infty(\Sigma)$ on a stationary regular cluster,  it is seldom the case that $f$ will be physical --
 see \cite[Section 4]{EMilmanNeeman-TripleAndQuadruple} for a discussion. This is a genuine issue even when the cluster is known to be completely regular, for instance when it is a spherical Voronoi cluster. The reason is that Taylor's classification of minimizing cones \cite{Taylor-SoapBubbleRegularityInR3} is only available in dimensions two and three, and does not extend to dimension four and higher, where additional non-simplicial minimizing cones are known to exist \cite{Brakke-MinimalConesOnCubes}. Consequently, $m \geq 6$ cells of a minimizing spherical cluster  could potentially meet in a strange cone of affine dimension strictly smaller than $m-1$ (but necessarily strictly larger than $3$). This would incur various linear dependencies between the normals $\n_{ij}$ at the meeting point, and prevent writing $f_{ij} = X^{\n_{ij}}$ for a well-defined vector-field $X$. So in order to extract information arising from a (necessarily physical) perturbation of the cluster, an approximation of $f$ by physical fields is necessary even in this simple case. Fortunately, we have:

\begin{proposition} \label{prop:physical}
Scalar-fields in $C_c^\infty(\Sigma \setminus \Sigma^{\geq 3})$ are physical. 
\end{proposition}

While it is a bit easier to show that physical fields are $H^1(\Sigma^1)$-dense in $C_c^\infty(\Sigma \setminus \Sigma^{\geq 3})$, we make the extra effort of showing the actual containment. To this end, we require the following:
\begin{lemma} \label{lem:tripod-extension}
Let $\Sigma := \Y \times \R^{n-2}$ be the standard $Y$-cone in $E^{(2)} \times \R^{n-2}$. As usual, denote by $\Sigma_{12}, \Sigma_{23}, \Sigma_{31}$ its three interfaces, and by $\Sigma_{123} = \overline{\Sigma_{12}} \cap \overline{\Sigma_{23}} \cap \overline{\Sigma_{31}} = \{ 0 \} \times \R^{n-2}$. Let $f_{ij} \in C^\infty(\overline{\Sigma_{ij}})$ for $(i,j) \in \cyclic(1,2,3)$ so that $f_{12} + f_{23} + f_{31} = 0$ on $\Sigma_{123}$. Then there exist smooth extensions $F_{ij}$ of $f_{ij}$ to $E^{(2)} \times \R^{n-2}$ so that $F_{12} + F_{23} + F_{31} = 0$ on the entire $E^{(2)} \times \R^{n-2}$.
\end{lemma}
\begin{proof}
For all $(i,j) \in \cyclic(1,2,3)$, as $\overline{\Sigma_{ij}}$ is closed, one can extend $f_{ij}$ to a $C^\infty$-smooth function on the entire $E^{(2)} \times \R^{n-2}$ \cite[Lemma 2.26]{Lee-SmoothManifolds}, 
which we continue to denote by $f_{ij}$. Let $P_{ijk}$ denote a (non-orthogonal) linear projection on $E^{(2)} \times \R^{n-2}$ so that $\Im P_{ijk} = \text{span}(\Sigma_{ij})$, $P_{ijk}|_{\Sigma_{123}} = \Id_{\Sigma_{123}}$ and $\ker P_{ijk} = \text{span}(\Sigma_{jk}) \cap \Sigma_{123}^{\perp}$, and let $P_0$ denote the orthogonal projection onto $\Sigma_{123}$. Define:
\[
F_{12} = f_{12} ~,~ F_{23}(x) = f_{23}(P_{231} x) - (f_{12} + f_{31})(P_{132} x) - f_{23}(P_0 x) ~,~ F_{31} = -(F_{12} + F_{23}) . 
\]
Using that $f_{12} + f_{23} + f_{31} = 0$ on $\Sigma_{123}$, one easily checks that $F_{23}$ and $F_{31}$ coincide with $f_{23}$ and $f_{31}$ on $\overline{\Sigma_{23}}$ and $\overline{\Sigma_{31}}$, respectively, thereby concluding the proof. 
\end{proof}

\begin{proof}[Proof of Proposition \ref{prop:physical}]
Let $f = (f_{ij}) \in C_c^\infty(\Sigma \setminus \Sigma^{\geq 3})$. Let $K$ denote the union of the compact support of the $f_{ij}$'s, which by assumption is disjoint from $\Sigma^{\geq 3}$.
We shall construct a $C_c^\infty$ vector-field $X$ on $(M,g)$ so that $X^{\n_{ij}} = f_{ij}$ on $\Sigma_{ij}$ for all $i,j$. Using a partition of unity, it is enough to do this in some neighborhood $N_p$ of every given $p \in M$. We may always select a partition which is subordinate to an open cover consisting of $M \setminus K$ and $\{M_p\}_{p \in K}$, where $M_p$ is a neighborhood of $p$ in $(M,g)$ disjoint from $\Sigma^{\geq 3}$ so that, by assumption, for all $i,j$, $f_{ij}$ on $\overline{\Sigma_{ij}} \cap M_p = (\Sigma_{ij} \cup \partial \Sigma_{ij}) \cap M_p$ is the restriction of a smooth function $\Psi^p_{ij}$ defined on $M_p$. Consequently, for $p \in M \setminus (\Sigma_1 \cup \Sigma_2)$, we simply set $X \equiv 0$ in $N_p$. For $p \in \Sigma_1 \cup \Sigma_2$ we proceed as follows.

Recall from Lemma \ref{lem:inward-fields} that $Z_1,\ldots,Z_q$ are $C^\infty$ vector-fields constructed on $M \setminus \Sigma^{\geq 3}$ so that $Z_k^{\n_{ij}} = \delta^k_{ij}$ on $\Sigma_{ij}$ (and hence, by regularity of the cluster, on $\Sigma_{ij\ell}$ as well for all $\ell$). Consequently, if $p \in \Sigma_{ij}$ we simply set $X = \Psi^p_{ij} Z_i$ on $N_p$, noting that $X^{\n_{ij}} = f_{ij}$ on $\Sigma_{ij} \cap N_p$ as required. When $p \in \Sigma_{ijk}$, regularity of the cluster (as in Theorem \ref{thm:regularity} \ref{it:regularity-Sigma2}) implies that there exists a $C^\infty$-diffeomorphism $\varphi$ locally mapping $\Sigma$ in some neighborhood $N_p$ to a standard $\Y$-cone. Applying Lemma \ref{lem:tripod-extension} to the functions $\varphi_*(f_{ab})$, $(a,b) \in \cyclic(i,j,k)$, and denoting by $F_{ab}$ the pull back via $\varphi$ of the resulting smooth extensions, the $F_{ab}$'s extend the $f_{ab}$'s to the entire $N_p$ and satisfy $F_{ij} + F_{jk} + F_{ki} = 0$ there. Consequently, we may apply a linear transformation to obtain smooth $F_i, F_j , F_k$ on $N_p$ so that $F_a - F_b = F_{ab}$ for all $(a,b) \in \cyclic(i,j,k)$. It remains to set $X = F_i Z_i + F_j Z_j + F_k Z_k$ on $N_p$, noting that 
 $X^{\n_{ab}} = F_{ab} = f_{ab}$ on $\Sigma_{ab} \cap N_p$ for all $(a,b) \in \cyclic(i,j,k)$, as required. 
\end{proof}

Combining Propositions \ref{prop:physical} and \ref{prop:dense}, we obtain:
\begin{corollary} \label{cor:physical-dense}
Physical scalar-fields are $H^1(\Sigma^1)$-dense in $\D_Q$. 
\end{corollary}

\subsection{Approximation by physical fields}

From here on in this section, we add the assumption that $V_\mu(\Omega) \in \interior \Delta^{(q-1)}_{\mu(M)}$. In other words, our standing assumption is that $\Omega$ is a bounded stationary regular $q$-cluster with bounded curvature satisfying \ref{eq:Lip} on $(M,g,\mu)$ and $V_\mu(\Omega) \in \interior \Delta^{(q-1)}_{\mu(M)}$.

\medskip

Before stating our general approximation result, we shall require: 
\begin{lemma}[Volume Correction] \label{lem:volume-offset}
For any $C_c^\infty$ vector-field $Y$ and $v \in E^{(q-1)}$, there exists a $C_c^\infty$ vector-field $Z$ so that:
\begin{itemize}
\item $\delta^1_Z V = \delta^1_Y V + v$, and
\item $|Q(Z) - Q(Y)| \leq A \norm{Y^\n}_{H^1(\Sigma^1)}|v| + B |v|^2$,
\end{itemize}
where $A,B > 0$ depend on $\Omega$ but are independent of $Y,v$. 
\end{lemma} 
\begin{proof} 
Consider the Sobolev scalar-field $\delta^k = (\delta^k_{ij}) \in \D_Q$. By Corollary \ref{cor:physical-dense}, there exists a $C_c^\infty$ vector-field $X_k$ so that $\norm{X_k^\n - \delta^k}_{H^1(\Sigma^1)} \leq \eps$ for any fixed $\eps \in (0,1]$. 
Given $a \in \R^q$, set $X_a = \sum_{k=1}^q a_k X_k$, and let $a= (a_{ij})$ denote the piecewise constant scalar-field $a = \sum_{k=1}^q a_k \delta^k$. Consequently,
\begin{equation} \label{eq:A00}
 \norm{X_a^\n - a}_{H^1(\Sigma^1)} \leq \sqrt{q} |a| \eps,
\end{equation}
and since $\norm{a}_{H^1(\Sigma^1)} = \norm{a}_{L^2(\Sigma^1)} \leq \sqrt{R} |a|$, where $R := \mu^{n-1}(\Sigma^1) <\infty$, it follows that:
\begin{equation} \label{eq:A0}
\norm{X_a^\n}_{H^1(\Sigma^1)} \leq A_0 |a| ~,~ A_0 : = \sqrt{R} + \sqrt{q} . 
\end{equation}

 Denote the linear map $E^{(q-1)} \ni a \mapsto \delta^1_{X_a} V \in E^{(q-1)}$ by $\tilde L_1 a$, and note that
 \[
\delta^1_a V = \sum_{i<j} \int_{\Sigma_{ij}} d\mu^{n-1} a_{ij} e_{ij} = L_1 a, 
\]
where $L_1$ was introduced in Definition \ref{def:LA}. Since for $f \in \D_Q$ we have by Cauchy-Schwarz:
\begin{equation} \label{eq:sillyCS}
|\delta^1_f V|  \leq \sum_{k=1}^q |\int_{\partial \Omega_i} f_{ij} d\mu^{n-1}| \leq 2 \int_{\Sigma^1} |f_{ij}| d\mu^{n-1} \leq  2 \sqrt{R} \norm{f}_{L^2(\Sigma^1)} ,
\end{equation}
it follows by (\ref{eq:A00}) that:
\[
|\tilde L_1 a - L_1 a| = |\delta^1_{X^\n_a - a} V| \leq 2 \sqrt{R q} |a| \eps . 
\]
Since $V_\mu(\Omega) \in \interior \Delta^{(q-1)}_{\mu(M)}$, it follows by Lemmas \ref{lem:LA-connected} and \ref{lem:LA-positive} that $L_1$ is positive-definite and in particular invertible on $E^{(q-1)}$. Consequently, $\tilde L_1$ is also invertible for small enough $\eps > 0$. We proceed by fixing such an $\eps \in (0,1]$. 

Given $v \in E^{(q-1)}$, we define $a \in E^{(q-1)}$ to be such that $\delta^1_{X_a} V = \tilde L_1 a = v$. Given a $C_c^\infty$ vector-field $Y$, set $Z = Y + X_a$. Clearly $\delta^1_Z V = \delta^1_Y V + v$. By Lemma \ref{lem:Q0-bounds} we have:
\[
|Q(Z) - Q(Y)| = |Q^0(Z^\n) - Q^0(Y^\n)| \leq M \brac{2 \norm{Y^\n}_{H^1(\Sigma^1)} \norm{X_a^\n}_{H^1(\Sigma^1)} + \norm{X_a^\n}^2_{H^1(\Sigma^1)}} ,
\]
for some $M <\infty$. 
It follows from (\ref{eq:A0}) that:
\[
|Q(Z) - Q(Y)| \leq 2 M A_0 \norm{Y^\n}_{H^1(\Sigma^1)} |a| + M A_0^2 |a|^2 . 
\]
Setting $A = 2 M A_0 \snorm{(\tilde L_1)^{-1}}$ and $B = M A_0^2 \snorm{(\tilde L_1)^{-1}}^2$, the proof is complete. 
\end{proof}

We can now finally prove:

\begin{theorem}[Index-Form Approximation Theorem] \label{thm:approximation}
For any Sobolev scalar-field $u = \{u_{ij}\} \in \D_Q$ and $\eps > 0$, there exists a $C_c^\infty$ vector-field $Z$ on $(M,g)$ so that:
\begin{enumerate}
\item  $\delta^1_{Z} V(\Omega) = \delta^1_{u} V(\Omega)$; and
\item $|Q(Z) - Q^0(u)| \leq \eps$. 
\end{enumerate}
In particular, if $\Omega$ is in addition assumed to the stable, then for any $u \in \D_Q$ we have:
\begin{equation} \label{eq:Q0-non-negative}
\delta^1_u V(\Omega) = 0 \;\; \Rightarrow \;\;  Q^0(u)  \geq 0 .  
\end{equation}
\end{theorem}
\begin{proof}
Fix $u \in \D_Q$ and $\eps > 0$. By Corollary \ref{cor:physical-dense}, there exists a $C_c^\infty$ field $Y$ so that $\norm{u - Y^\n}_{H^1(\Sigma^1)} \leq \eps$. Lemma \ref{lem:Q0-bounds} yields:
\[
|Q(Y) - Q^0(u)| \leq M (2 \eps \norm{u}_{H^1(\Sigma^1)} + \eps^2) ,
\]
for some $M <\infty$ depending on $\Omega$. 
Recalling that $R:=\mu^{n-1}(\Sigma^1) < \infty$, we also have by (\ref{eq:sillyCS}) that $|\delta^1_u V - \delta^1_Y V| \leq 2 \sqrt{R} \norm{u - Y^\n}_{L^2(\Sigma^1)} \leq  2 \sqrt{R} \eps$. By Lemma \ref{lem:volume-offset}, there exists a $C^\infty_c$ vector field $Z$ with $\delta^1_Z V = \delta^1_u V$ and:
\[
|Q(Z) - Q(Y)| \leq A (\norm{u}_{H^1(\Sigma^1)} + \eps) 2 \sqrt{R} \eps + 2 B  R \eps^2 , 
\]
for some $A,B < \infty$ depending on $\Omega$.  Applying the triangle inequality and adjusting constants, the proof is complete. 
\end{proof}

\subsection{At most $q-1$ positive eigenvalues}

The following corollary is immediate from stability:

\begin{corollary} \label{cor:at-most-positive-eigenvalues}
Assume in addition that $\Omega$ is stable. 
Then $(L_{Jac},\D_{con})$ has at most $q-1$ positive eigenvalues (counting multiplicity).  
\end{corollary}
\begin{proof}
Otherwise, there would be $q$ or more orthonormal (in $L^2(\Sigma^1)$) eigenfunctions  $u^{(k)} \in \D_{con}$ with positive eigenvalues $\lambda_k > 0$. We would thus be able to find a non-trivial linear combination $u = \sum_k c_k u^{(k)} \in \D_{con}$ with $E^{(q-1)} \ni \delta^1_u V = \vec 0$. But by Lemma \ref{lem:Q^0-by-parts}:
\[
Q(u) = -\scalar{L_{Jac} u , u}_{L^2(\Sigma^1)} = - \sum_k |c_k|^2 \lambda_k < 0 ,
\]
a contradiction to stability and (\ref{eq:Q0-non-negative}). 
\end{proof}

We will show in Theorem \ref{thm:q-1-positive} that for $3$-Plateau, stable, perpendicular spherical Voronoi clusters on $\S^n$ with equatorial cells, there are \emph{exactly} $q-1$ positive eigenvalues. In particular, by Theorem \ref{thm:intro-structure}, this applies to minimizing $q$-clusters $\Omega$ on $\S^n$ with $q \leq n+1$ and $V(\Omega) \in \interior \Delta^{(q-1)}$.

\subsection{Elliptic regularity}

Finally, we will require the following useful lemma proved in \cite[Lemma 3.8]{DoubleBubbleInR3} for physical scalar-fields when $\Sigma = \Sigma^1 \cup \Sigma^2$, $q=3$ and there are no singularities. \begin{lemma}[Elliptic regularity] \label{lem:elliptic-regularity}
Assume in addition that $\Omega$ is stable, and let $f \in \D_Q$ be a Sobolev scalar-field so that:
\begin{equation} \label{eq:zero-minimizer}
\delta^1_f V(\Omega) = 0 \text{ and } Q^0(f) = 0 . 
\end{equation}
Then there exists $a \in E^{(q-1)}$ so that for every $i<j$, $L_{Jac} f_{ij} \equiv a_{ij}$ on $\Sigma_{ij}$. In particular, $f_{ij}$ is $C^\infty$-smooth on (the relatively open) $\Sigma_{ij}$. 
\end{lemma}
\begin{proof}
This was shown in \cite[Lemma 5.15]{EMilmanNeeman-TripleAndQuadruple} for Lipschitz scalar-fields $f$ of a particular form. The proof immediately extends to general Sobolev scalar-fields using the approximation Theorem \ref{thm:approximation}. For a self-contained proof see Lemma \ref{lem:minimizer-is-Jacobi}, which also establishes that $f \in \D_{con}$. 
\end{proof}

\section{Special fields} \label{sec:conformal}

In this section we recall the construction and properties of several families of vector and scalar fields from \cite[Sections 6-7]{EMilmanNeeman-TripleAndQuadruple}. 

\subsection{Conformal Killing fields}

 Recall that a vector-field $X$ on $(M,g)$ is called a Killing field if $X$ generates a one-parameter family of isometries, or equivalently, if $(\nabla X)^{\sym} = \mathcal{L}_X g = 0$, where $\mathcal{L}_X$ denotes the Lie derivative and $(\nabla X)^{\sym}$ the symmetrized Levi-Civita covariant derivative $(\nabla X)^{\sym}(a,b) := \frac{1}{2} \brac{\nabla X(a,b) + \nabla X(b,a)}$.
 More generally, $X$ is called a conformal Killing field if $X$ generates a one-parameter family of conformal mappings $F_t$, or equivalently, if $(\nabla X)^{\sym}= f_X \Id$ for some function $f_X$ called the conformal factor of $X$. 
  If these properties hold on a subset $\Omega \subset M$, $X$ is called a Killing (conformal Killing) field on $\Omega$, and $X$ is said to act isometrically (conformally) on $\Omega$. 
 
\medskip
 
The following was shown in \cite[Lemma 6.1]{EMilmanNeeman-TripleAndQuadruple}:

\begin{lemma} \label{lem:conformal-boundary}
Let $\Omega$ be a stationary regular cluster on $(M,g,\mu)$. Then for any $C^1$ vector-field $Y$ on $(M,g)$ so that $Y$ acts conformally at a point $p \in \Sigma_{uvw}$, $\nabla_{\n_{\partial ij}} Y^{\n_{ij}} - \bar \II^{\partial ij} Y^{\n_{ij}}$ is independent of $(i,j) \in \cyclic(u,v,w)$. \\
In particular, if $Y$ is a smooth conformal Killing field in a neighborhood of $\Sigma$ then its normal component satisfies the conformal BCs of Definition \ref{def:conformal-BCs} and Remark \ref{rem:conformal-BCs}: $Y^{\n} \in \D_{con}$. 
\end{lemma}

The first statement below is completely classical; the second is a simple computation, cf. \cite[Lemmas 6.6-6.7]{EMilmanNeeman-TripleAndQuadruple}:
\begin{lemma}\label{lem:LJac-Killing}
Let $(M,g,\vol_g)$ be an \emph{unweighted} Riemannian manifold, and let $\Sigma$ denote a smooth hypersurface with unit-normal $\n$ and constant mean-curvature $H_{\Sigma}$. Then:
\begin{enumerate}
\item For any Killing field $W$ in a neighborhood of $p \in \Sigma$ in $M^n$, we have at $p$:
\[
L_{Jac} W^\n = 0 . 
\]
\item For any conformal Killing field $W$ in a neighborhood of $p \in \Sigma$ in $M^n$ with conformal factor $f_W$, we have at $p$:
\[
L_{Jac} W^\n = f_W H_{\Sigma} - (n-1) \scalar{\nabla f_W , \n} .
\]
\end{enumerate}
\end{lemma}

\subsection{M\"obius fields on $\S^n$} \label{subsec:Mobius}

By Liouville's classical theorem \cite{Udo-MobiusDifferentialGeometry,Blair-InversionTheory}, all conformal automorphisms of $\bar \R^n$, the one-point-at-infinity compactification of $\R^n$, when $n \geq 3$, are given by M\"obius transformations, obtained by composing isometries (orthogonal linear transformations and translations) with scaling ($p \mapsto e^\lambda p$) and spherical inversion ($p \mapsto p / |p|^2$). 
Using stereographic projection, this also classifies all conformal automorphisms of $\S^n$, as well as all global conformal Killing fields on both model spaces. The conformal Killing fields which are non-Killing (i.e.~do not generate isometries) constitute an $(n+1)$-dimensional linear space -- on $\S^n$, these are precisely given by the $(n+1)$-dimensional family of ``M\"obius fields", introduced below following \cite[Subsection 6.3]{EMilmanNeeman-TripleAndQuadruple}:

\begin{definition}[M\"obius Field $W_\theta$]  \label{def:dilation-field}
\hfill \\
The M\"obius field $W_\theta$ on $\S^n \subset \R^{n+1}$ in the direction of $\theta \in \R^{n+1}$ is the section of $T \S^n$ given by:
\[
W_\theta(p) := \theta - \scalar{\theta,p} p . 
\]
\end{definition}

\begin{remark} \label{rem:dilation-conformal} 
Note that $\nabla W_\theta(p) = -\scalar{p,\theta} \Id$, 
confirming that this is indeed a conformal Killing field with conformal factor $f_{W_\theta} = -\scalar{p,\theta}$. It is easy to check the geometric significance of the field $W_\theta$:  it is obtained by pulling back the scaling field $p \mapsto |\theta| p$ on $\R^n$ via the stereographic projection from $\S^n$ using $\theta / |\theta| \in \S^n$ as the North pole. 
\end{remark}

\subsection{$L_{Jac}$ of some special fields}

\begin{lemma}  \label{lem:LJac-dilation}
Let $\Sigma$ denote a smooth hypersurface of constant mean-curvature $H = (n-1) \k$ on $\S^n$. We denote by $\II_0 := \II - \k \Id$ the traceless part of its second fundamental form $\II$ with respect to the unit-normal $\n$. Then for any $\theta \in \R^{n+1}$:
\begin{itemize}
\item $L_{Jac} 1 = (n-1) (1 + \k^2) + \norm{\II_0}^2$. 
\item $L_{Jac} W_\theta^\n = L_{Jac} \scalar{\theta,\n} = (n-1) \scalar{\theta,\c}$. 
\item 
$L_{Jac} \scalar{\theta,p} = -(n-1) \k \scalar{\theta,\c} + \norm{\II_0}^2 \scalar{\theta,p}$. 
\item 
$L_{Jac} \scalar{\theta,\c} =  (n-1)(1+\k^2) \scalar{\theta,\c} - \k \norm{\II_0}^2 \scalar{\theta,p}$. 
\end{itemize}
If $\Sigma$ is a sphere in $\S^n$ of curvature $\k$ and quasi-center $\c$, and $N \in \S^n$ with $N \perp \c$, then:
\begin{itemize}
\item $1 - \frac{1}{2} \Delta_{\Sigma} \scalar{N,p}^2 = n (1 + \k^2) \scalar{N,p}^2$. 
\end{itemize}
\end{lemma}
\begin{proof}
Recall that on $\S^n$ we have $L_{Jac} = \Delta_{\Sigma}  + (n-1) + \norm{\II}^2 = \Delta_{\Sigma} + (n-1) (1 + \k^2) + \norm{\II_0}^2$, and so the formula for $L_{Jac} 1$ immediately follows. To see the formula for $L_{Jac} W_\theta^\n$, 
recall from Remark \ref{rem:dilation-conformal} that the conformal factor of $W_\theta$ is $f_{W_\theta} = -\scalar{p,\theta}$. Consequently, by Lemma \ref{lem:LJac-Killing}, we have:
\[
L_{Jac} W_\theta^\n = - H_{\Sigma} \scalar{p,\theta} + (n-1) \scalar{\theta, \n} = (n-1) \scalar{\theta , \n - \k p} . 
\]
Next, we have:
\[
\nabla^\tang \scalar{\theta,p} = \theta - \scalar{\theta,p} p - \scalar{\theta,n} n ,
\]
and consequently:
\[
\Delta_{\Sigma} \scalar{\theta,p} = \div_{\Sigma} \; \nabla^\tang \scalar{\theta,p} = - (n-1) \scalar{\theta,p} - \scalar{\theta,n} H . 
\]
Recalling that $\c = \n - \k p$ with $\k$ constant, the formulas for $L_{Jac} \scalar{\theta,p}$ and $L_{Jac} \scalar{\theta,\c}$ are obtained. 

Finally, if $\Sigma$ is a sphere, since $0 = \scalar{N,\c} = \scalar{N,\n} - \k \scalar{N,p}$ on $\Sigma$, we see that:
\begin{align*}
& \frac{1}{2} \Delta_{\Sigma} \scalar{N,p}^2 = \scalar{N,p} \Delta_{\Sigma} \scalar{N,p} + |\nabla^\tang \scalar{N,p}|^2 \\
& = -(n-1) \scalar{N,p}^2 - (n-1) \k \scalar{N,\n} \scalar{N,p} + 1 - \scalar{N,p}^2 - \scalar{N,\n}^2  = 1- n (1 + \k^2) \scalar{N,p}^2 . 
\end{align*}
\end{proof}

\subsection{Piecewise-constant fields}

Given $a \in \R^q$, we will also use $a = (a_{ij})$ to denote the scalar-field in $\D_Q$ which is equal to the constant $a_{ij} = a_i - a_j$ on $\Sigma_{ij}$. Note that this definition is invariant on the fibers of the quotient map from $\R^q$ to $E^{(q-1)}$ parallel to $\textbf{1} = (1,\ldots,1)$, and so we will typically assume that $a \in E^{(q-1)}$. We will frequently use that for all $f \in \D_Q$:
\[
\scalar{a , f}_{L^2(\Sigma^1)} = \sum_{i<j} \int_{\Sigma_{ij}} a_{ij} f_{ij} d\HH^{n-1} = \scalar{a , \delta^1_f V}_{E^{(q-1)}} . 
\]
Note that $\delta^1_a V = \sum_{i<j} \int_{\Sigma_{ij}} a_{ij}  d\HH^{n-1} e_{ij}  = L_1 a$, where 
$e_{ij} := e_i - e_j$, $\{e_i\}_{i=1,\ldots,q}$ is the standard basis of $\R^q$, and
\[
L_1 := \sum_{i<j} \int_{\Sigma_{ij}} e_{ij} \otimes  e_{ij} d\HH^{n-1} 
\]
was introduced in Definition \ref{def:LA}. 

\subsection{Regular and truncated skew fields}

   Let $\Omega$ be a perpendicular $3$-Plateau spherical Voronoi cluster on $\S^n$ with North pole $N \in \S^n$. 
The following definition and properties were established in \cite[Sections 7,9]{EMilmanNeeman-TripleAndQuadruple}:

\begin{definition}[Regular and Truncated Skew Fields] 
Given $a \in E^{(q-1)}$: 
\begin{itemize}
\item 
The scalar-field $g^{(a)} \in C_c^\infty(\Sigma)$ given by $g^{(a)}_{ij} := a_{ij} \scalar{p,N}$ is called a skew-field. 
\item 
The scalar-field $h^{(a)} \in \D_Q$ given by $h^{(a)}_{ij} := a_{ij} \abs{\scalar{p,N}}$ is called a truncated skew-field. 
\end{itemize}
\end{definition}

\begin{proposition} \label{prop:skew-fields}
For all $a \in E^{(q-1)}$:
\begin{enumerate}
\item
$g^{(a)} \in \D_{con}$, $L_{Jac}\;  g^{(a)} = 0$ and $\delta^1_{g^{(a)}} V = 0$. 
\item 
$Q^0(h^{(a)}) = 0$ and $\delta^1_{h^{(a)}} V = L_{|\scalar{p,N}|} a$. 
\end{enumerate}
\end{proposition}

\begin{remark} \label{rem:crazy-pos-def}
Every non-empty $\Sigma_{ij}$ is a (relatively open) subset of the geodesic sphere $S_{ij} = \{ p \in \S^n \; ; \; \scalar{\c_{ij} , p} + \k_{ij} = 0 \}$ with $\c_{ij} \perp N$. As $S_{ij}$ intersects the equator $\S^n \cap N^{\perp}$ perpendicularly, it follows that $\int_{\Sigma_{ij}} |\scalar{p,N}| dp > 0$ for all non-empty $\Sigma_{ij}$. Consequently, whenever $V(\Omega) \in \interior \Delta^{(q-1)}$, $L_{|\scalar{p,N}|}$ is a positive-definite quadratic form on $E^{(q-1)}$ by Lemmas \ref{lem:LA-connected} and \ref{lem:LA-positive}. 
\end{remark}

\section{Conformal Jacobi fields} \label{sec:conformal-Jacobi}

Throughout the rest of this section, we assume that $\Omega$ is a $3$-Plateau spherical Voronoi $q$-cluster with $q \leq n+2$. As usual, we denote its quasi-center and curvature parameters by $\{\c_i\}_{i=1,\ldots,q} \subset \R^{n+1}$ and $\{\k_i\}_{i=1,\ldots,q} \subset \R$, and use $\C : \R^{n+1} \rightarrow E^{(q-1)}$ and $\k \in E^{(q-1)}$ to denote its corresponding quasi-center operator and curvature vector, respectively. Note that such a cluster is automatically stationary by Lemma \ref{lem:Voronoi-stationary}. At times, we will also assume that:
\begin{itemize}
\item $V(\Omega) \in \interior \Delta^{(q-1)}$ (equivalently, all cells are non-empty); or
\item $q \leq n+1$; or
\item $\Omega$ is stable; or
\item $\Omega$ is perpendicular, i.e. that $\c_i \perp N$ for some North pole $N \in \S^n$ (so in particular $\Omega$ is symmetric with respect to reflection about $N^{\perp}$); or
\item the perpendicular $\Omega$ has equatorial cells, i.e. that the equator $\S^{n-1} := \S^n \cap N^{\perp}$ intersects all (open) cells $\Omega_i$ (so in particular they are all non-empty and connected);
\end{itemize}
in those cases, we will explicitly state these additional assumptions below. Recall by Theorem \ref{thm:intro-structure} that when $q \leq n+1$, a minimizing cluster $\Omega$ with $V(\Omega) \in \interior \Delta^{(q-1)}$ is necessarily a $3$-Plateau, stable, perpendicular spherical Voronoi cluster with equatorial cells.

\begin{definition}[Conformal Jacobi field]
A Sobolev scalar-field $f = (f_{ij}) \in \D_Q$ is called a conformal Jacobi field (with conformal parameter $a \in E^{(q-1)}$) if:
\begin{enumerate}
\item It satisfies conformal boundary conditions, i.e. $f \in \D_{con}$; and
\item It satisfies $L_{Jac} f_{ij} \equiv (n-1) a_{ij}$ on $\Sigma_{ij}$ for all $i,j$.  
\end{enumerate}
A conformal Jacobi field with conformal parameter $0$ is simply called a Jacobi field, corresponding to an element in the kernel of $(L_{Jac}, \D_{con})$. 
\end{definition}
\begin{remark} \label{rem:conformal-Jacobi-smooth}
Note that there is no requirement that $f$ be physical, nor that $f \in C_c^\infty(\Sigma)$. However, by local elliptic regularity, it follows that $f_{ij}$ is $C^\infty$-smooth on each $\Sigma_{ij}$ (but we make no such claim on $\overline{\Sigma_{ij}}$ nor even on $\overline{\Sigma_{ij}} \setminus \Sigma^{\geq 3}$). 
\end{remark}
\begin{remark}
By (\ref{eq:Q^0-Dcon}), we see that for any Sobolev scalar-field $g \in \D_Q$:
\begin{equation} \label{eq:Q0-Jacobi}
Q^0(f,g) = -(n-1) \scalar{a , g}_{L^2(\Sigma^1)} = -(n-1) \scalar{a , \delta^1_g V}_{E^{(q-1)}} . 
\end{equation}
\end{remark}

Recalling Lemmas \ref{lem:conformal-boundary} and \ref{lem:LJac-Killing}, we deduce the classical fact that the normal component of a Killing field is a Jacobi field. The next lemma establishes an analogous statement for M\"obius fields:

\begin{lemma} \label{lem:Mobius-Jacobi}
The normal component of a M\"obius field $W_\theta$ is a (physical) conformal Jacobi field with conformal parameter $a = \C \theta$. 
\end{lemma}
\begin{proof}
Lemma \ref{lem:conformal-boundary} shows that the physical field $W_\theta^{\n} \in C_c^\infty(\Sigma)$ satisfies conformal boundary conditions, so $W_\theta^\n \in \D_{con}$. It remains to invoke Lemma \ref{lem:LJac-dilation}, which verifies that:
\[
 L_{Jac} W_\theta^{\n_{ij}} = (n-1) \scalar{\c_{ij} ,  \theta} .
 \]
 As the cluster is spherical Voronoi, $\c_{ij} \equiv \c_i - \c_j$ is constant on $\Sigma_{ij}$; setting $a_i = \scalar{\c_i, \theta}$, the assertion follows. 
\end{proof}

\subsection{Conformally complete system of conformal Jacobi fields}

\begin{definition}[Conformally complete system of conformal Jacobi fields]
We shall say that a conformally complete system of conformal Jacobi fields exists for $\Omega$ if for all $a \in E^{(q-1)}$, there exists a conformal Jacobi field $f^{a}$ with conformal parameter $a$. 
\end{definition}

As an immediate corollary of Lemma \ref{lem:Mobius-Jacobi}, we have:
\begin{corollary} \label{cor:full-dim-F} 
If $q \leq n+2$ and $\Omega$ is full-dimensional, i.e. $\C$ is of full affine-rank, then a conformally complete system of conformal Jacobi fields exists.
\end{corollary}
However, for lower-dimensional clusters, the existence of a complete system as above is a-priori not clear. We can show:

\begin{theorem} \label{thm:conformally-complete} For any $3$-Plateau, stable, perpendicular spherical Voronoi cluster with equatorial cells, a conformally complete system of conformal Jacobi fields exists.
\end{theorem}

The proof is based on volumetric non-degeneracy, defined next.

\subsection{Volumetric non-degeneracy}

Recall that in the context of geometric variational problems, a (say, closed) hypersurface $\Sigma$ is called \emph{non-degenerate} if $\ker L_{Jac} = \{0\}$, i.e. there are no (non-zero) Jacobi fields on $\Sigma$. Furthermore, it is called  \emph{equivariantly non-degenerate} \cite{BPS-DeformationsOfCMC} 
if \emph{every} Jacobi field on $\Sigma$ is trivial, arising as the normal component of a Killing field. It is a rather difficult problem to establish that a given hypersurface is (equivariantly) non-degenerate, even if it is known to be of constant mean-curvature and stable, 
and often in the literature (equivariant) non-degeneracy is taken as a postulate. Of course, symmetry is an expected source of degeneracy, and in our setting when $\Sigma^1$ possesses $\S^0$-symmetry, it follows by Proposition \ref{prop:skew-fields} that the scalar fields $f_{ij} = a_{ij} \scalar{p,N}$ are Jacobi fields for all $a \in E^{(q-1)}$; while for standard bubbles these will all be the normal component of a rotation, this will not necessarily be the case for general lower-dimensional spherical Voronoi clusters (as these fields may not even be physical). Determining all Jacobi fields given a hypersurface $\Sigma$ is a major challenge in differential geometry, and in some sense, the reason why we are not able to establish the trace inequality (\ref{eq:F-trace-inq}) in full-generality is because we lack such a characterization in our setting. While we cannot a-priori rule-out equivariant degeneracy, we can establish the following volumetric version of non-degeneracy. 

\begin{definition}[Volumetric Non-Degeneracy]
A cluster $\Omega$ is called volumetrically non-degenerate if $L_{Jac} u = 0$ implies $\delta^1_u V = 0$ for all $u \in \D_{con}$. 
\end{definition}
In other words, while a Jacobi field is not required to arise as the normal component of a Killing field, it is required to preserve the volume of each cell $\Omega_i$ to first order.

\begin{proposition} \label{prop:Fredholm}
A cluster $\Omega$ is volumetrically non-degenerate iff it has a conformally complete system of conformal Jacobi fields.
\end{proposition}
\begin{proof}
By the Fredholm alternative for $(L_{Jac},\D_{con})$ on $L^2(\Sigma^1)$, a solution $u \in \D_{con}$ to $L_{Jac} u  = (n-1) a$ exists for a given $a \in E^{(q-1)}$ iff for any $g \in \text{Dom}(L_{Jac}^*)$ with $L_{Jac}^* g = 0$ one has $0 = \scalar{a,g}_{L^2(\Sigma^1)} = \sum_{i=1}^q a_i \delta^1_g V_i$, where $L_{Jac}^*$ denotes the adjoint operator.  In other words, a solution exists for all $a \in E^{(q-1)}$  iff $L_{Jac}^* g = 0$ implies $\delta^1_g V = 0$. But as $(L_{Jac},\D_{con})$ is self-adjoint, this just reduces to the volumetric non-degeneracy. \end{proof}

Consequently, the proof of Theorem \ref{thm:conformally-complete} reduces to showing: 
\begin{theorem} \label{thm:volumetric non-degeneracy}
A $3$-Plateau, stable, perpendicular spherical Voronoi cluster with equatorial cells is volumetrically non-degenerate. 
\end{theorem}
\begin{proof}
Otherwise, there exists $f \in \D_{con}$ with $L_{Jac} f = 0$ and $E^{(q-1)} \ni v := \delta^1_f V \neq 0$.  

As $\Omega$ is perpendicular and all of its cells are non-empty, Remark \ref{rem:crazy-pos-def} implies that $L_{|\scalar{p,N}|} > 0$ is positive-definite and in particular invertible on $E^{(q-1)}$. Consequently, we may set $E^{(q-1)} \ni b := L^{-1}_{|\scalar{p,N}|} v \neq 0$ and consider the truncated skew-field $h = h^{(b)} \in \D_Q$ given by $h_{ij} := b_{ij} \abs{\scalar{p,N}}$. By Proposition \ref{prop:skew-fields}, $\delta^1_h V = v$ and $Q^0(h) = 0$. 

Now consider the scalar field $z := f - h \in \D_Q$. By Lemma \ref{lem:Q^0-by-parts} we have:
\[
Q^0(f,u) = - \scalar{L_{Jac} f , u}_{L^2(\Sigma^1)} = 0 \;\;\; \forall u \in \D_Q ,
\]
and hence $Q^0(f) = Q^0(f,h) = 0$. It follows from bilinearity and symmetry that:
\[
Q^0(z) = Q^0(f) - 2 Q^0(f,h) + Q^0(h)  = 0 . 
\]
As $\delta^1_z V = \delta^1_f V - \delta^1_h V = v -v = 0$ and $\Omega$ is stable, Lemma \ref{lem:elliptic-regularity} implies that $z_{ij}$ must be $C^\infty$-smooth on every (relatively open) non-empty $\Sigma_{ij}$. 

Now consider the cluster $\tilde \Omega$ on the equator $\S^{n-1}$ given by $\tilde \Omega_i = \Omega_i \cap \S^{n-1}$. Since all cells of $\Omega$ are equatorial, we have that $V(\tilde \Omega) \in \interior \Delta^{(q-1)}$. Denote the interfaces of $\tilde \Omega$ by $\tilde \Sigma_{ij}$; clearly, $\tilde \Omega$ is interface-regular. 
Consider the undirected graph $\tilde G$ on the set of vertices $\{1,\ldots,q\}$, having an edge between $i,j \in \{1,\ldots,q\}$ iff  $\tilde \Sigma_{ij} \neq \emptyset$ (i.e.~$\Omega_i$ and $\Omega_j$ are adjacent through an interface passing \emph{through the equator}). As $\S^{n-1}$ satisfies a single-bubble isoperimetric inequality $V(U) \in (0,\vol(\S^{n-1})) \Rightarrow \per(U) > 0$, it follows by Lemma \ref{lem:LA-connected} that $\tilde G$ must be connected. Consequently, there exist two adjacent vertices $i \neq j$ in $\tilde G$ so that $b_{ij} = b_i - b_j \neq 0$ (otherwise $b \in E^{(q-1)}$ would necessarily be $0$). As $\Sigma_{ij}$ is non-empty, relatively open, and intersects the equator $\S^{n-1}$ perpendicularly,  if follows that $h_{ij} = b_{ij} \abs{\scalar{p,N}}$ is \emph{not} $C^\infty$-smooth on $\Sigma_{ij}$. Since $f_{ij}$ is itself smooth by Remark \ref{rem:conformal-Jacobi-smooth}, it follows that $z_{ij} = f_{ij} - h_{ij}$ is non-smooth on $\Sigma_{ij}$, contradicting Lemma \ref{lem:elliptic-regularity} and concluding the proof. 
\end{proof}

\subsection{Conformal Jacobi fields as minimizers of $Q^0$}

Conformal Jacobi fields have the following important property. Given $V \in E^{(q-1)}$, we denote by $\D_Q[V]$ the following affine subspace of $\D_Q$ (of codimension $q-1$ if all cells are non-empty):
\[
\D_{Q}[V] := \{ f \in \D_Q \; ; \; \delta^1_f V = V \} . 
\]

\begin{proposition} \label{prop:conformal-Jacobi-prop}
Let $f$ be a conformal Jacobi field, and set $V = \delta^1_f V$. Then for any Sobolev scalar-field $g \in \D_Q[V]$:
\begin{enumerate}[(i)]
\item $Q^0(g-f) = Q^0(g) - Q^0(f)$. 
\item \label{it:minimizer}
If $\Omega$ is stable then $Q^0(g) \geq Q^0(f)$. In other words, on a stable $\Omega$, conformal Jacobi fields minimize $Q^0(g)$ among all $g \in \D_Q[V]$. 
\item \label{it:non-positive}
If $\Omega$ is stable and perpendicular then $Q^0(f) \leq 0$. 
\end{enumerate}
\end{proposition}
\begin{proof}\hfill
\begin{enumerate}[(i)]
\item We have $L_{Jac} f = (n-1) a$ for some $a \in E^{(q-1)}$. By (\ref{eq:Q0-Jacobi}):
\[
Q^0(f,g) = -(n-1) \scalar{a, \delta^1_g V} = -(n-1) \scalar{a, \delta^1_f V} = Q^0(f,f) ,
\]
and the claim follows by bilinearity and symmetry of $Q^0$.
\item As $\delta^1_{g-f} V = 0$, we have $Q^0(g-f) \geq 0$ by stability and (\ref{eq:Q0-non-negative}), and so the second assertion follows immediately from the first. 
\item The third assertion follows from the second applied to $g = h^{(a)}$, a truncated skew-field. By Proposition \ref{prop:skew-fields}, $Q^0(h^{(a)}) = 0$, and so setting $a = L_{|\scalar{p,N}|}^{-1} \delta^1_f V$ ensures that $\delta^1_{h^{(a)}} V = \delta^1_f V$; note that by removing empty-cells from $\Omega$, which doesn't alter $\delta^1_f V$ apart from truncating some zero entries, we may assume that all cells are non-empty, and hence $L_{|\scalar{p,N}|}$ is invertible by Remark \ref{rem:crazy-pos-def}. 
\end{enumerate}
\end{proof}

The converse to Proposition \ref{prop:conformal-Jacobi-prop} \ref{it:minimizer} is equally true without any stability:

\begin{lemma} \label{lem:minimizer-is-Jacobi}
Assume that $V(\Omega) \in \interior \Delta^{(q-1)}$. 
If $f$ is a minimizer of $Q^0$ on $\D_Q[V]$ for some $V \in E^{(q-1)}$, then $f$ is a conformal Jacobi field (with some conformal parameter $a \in E^{(q-1)}$), i.e. $f \in \D_{con}$ and $L_{Jac} f_{ij} = (n-1) a_{ij}$ on $\Sigma_{ij}$. 
\end{lemma}
\begin{proof}
Recall our convention given $b \in \R^{q}$ for using $b$ to also denote the piecewise-constant scalar field $(b_{ij})$; we may always assume that $b \in E^{(q-1)}$. Also recall that $\delta^1_b V = L_1 b$ and that $L_1 > 0$ is positive-definite by Lemmas \ref{lem:LA-connected} and \ref{lem:LA-positive}, since we assume $V(\Omega) \in \interior \Delta^{(q-1)}$. 

As a piecewise-constant field is in $L^2$, by Proposition \ref{prop:enter-LJac} and definitions of $\D_{Con}$ and $L_{Jac}$ we need to establish that:
\begin{equation} \label{eq:verify}
\exists a \in E^{(q-1)} \;\;\;  \forall g \in \D_Q \;\;\; Q^0(f , g) = -(n-1) \scalar{a , g}_{L^2} = -(n-1) \scalar{a , \delta^1_g V}_{E^{(q-1)}} . 
\end{equation}

Now, for any $g_0 \in \D_Q[0]$, $f + \eps g_0 \in \D_Q[V]$ and hence:
\[
Q^0(f) \leq Q^0(f  + \eps g_0) = Q^0(f) + 2 \eps Q^0(f , g_0) + \eps^2 Q^0(g_0) \;\;\; \forall \eps \in \R . 
\]
Consequently $Q^0(f,g_0) = 0$ for all $g_0 \in \D_Q[0]$.

Let $b \in \R^q$ be given by $b_k := Q^0(f,e_k)$, and note that $\sum_{k=1}^q b_k = Q^0(f , \mathbf{1}) = Q^0(f,0) = 0$, so that $b \in E^{(q-1)}$. Given $g \in \D_Q$, since $g_0 := g - L_1^{-1} \delta^1_g V \in \D_Q[0]$, it follows that:
\[
Q^0(f,g) = Q^0(f,L_1^{-1} \delta^1_g V) = \sum_{k=1}^q (L_1^{-1} \delta^1_g V)_k Q^0(f,e_k) = \scalar{L_1^{-1} \delta^1_g V , b}_{E^{(q-1)}} = \scalar{\delta^1_g V, L_1^{-1} b}_{E^{(q-1)}} . 
\]
Consequently, setting $a := - \frac{1}{n-1} L_1^{-1} b$, (\ref{eq:verify}) is verified.
\end{proof}

The question of whether a minimizer of $Q^0$ over $\D_Q[V]$ exists, and in particular, whether $Q^0$ is even bounded from below on $\D_Q[V]$, is much more delicate, and will be touched upon below. 

\subsection{The conformal-to-volume operator $\F$}

Assume in this subsection that $\Omega$ is a spherical Voronoi cluster having a conformally complete system of conformal Jacobi fields. 
For each $k=1,\ldots,q$, let $f^k$ be a conformal Jacobi field so that:
\[
L_{Jac} f^k_{ij} = (n-1) \delta^k_{ij} = (n-1) (\delta^k_i - \delta^k_j) \;\; \text{ on $\Sigma_{ij}$ }. 
\]
Given $a \in E^{(q-1)}$, we set $f^a := \sum_{k=1}^q a_k f^k$, which is a conformal Jacobi field with conformal parameter $a$. We shall  extend this definition for a brief moment to all $a \in \R^q$, after noting that the resulting conformal parameter of $f^a$ is in any case defined modulo the identification between $\R^q$ and its quotient space $E^{(q-1)}$. 

\begin{definition}
The conformal-to-volume linear operator $\F : \R^q \rightarrow E^{(q-1)}$ is defined as:
\[
 \F a = \delta^1_{f^a} V \;\; \text{ i.e. } \;\;  \F =  \sum_{i< j} \int_{\Sigma_{ij}} e_{ij} \otimes (f^1_{ij}, \ldots,f^q_{ij}) d\H^{n-1} . 
\]
\end{definition}

\begin{proposition} \label{prop:F-properties}
Let $\Omega$ be a $3$-Plateau spherical Voronoi cluster on $\S^n$, having a conformally complete system of conformal Jacobi fields.
\begin{enumerate}[(1)]
\item $\F$ is canonical -- it does not depend on the choice of conformal Jacobi fields $\{f^k\}_{k=1,\ldots,q}$. Consequently, $\F = \F_{\Omega}$ is called the conformal-to-volume operator associated to $\Omega$. 
\item $\F \textbf{1} = 0$. Consequently, we shall treat $\F$ as a linear operator $\F : E^{(q-1)} \rightarrow E^{(q-1)}$. 
\item For all $a,b \in E^{(q-1)}$, $Q^0(f^a,f^b) = -(n-1) a^T \F b$. 
\item $\F$ is symmetric: $\F = \F^T$. 
\item \label{it:pos-def}
If in addition $\Omega$ is stable and perpendicular then $\F$ is positive semi-definite: $\F \geq 0$.
\end{enumerate}
\end{proposition}
\begin{proof}
\hfill
\begin{enumerate}
\item Any other conformal Jacobi fields $\{\tilde f^k\}_{k=1,\ldots,q}$ satisfying $L_{Jac} \tilde f^k_{ij} = (n-1) \delta^k_{ij}$ would differ from $\{f^k\}_{k=1,\ldots,q}$ by elements in $\ker L_{Jac}$. By Proposition \ref{prop:Fredholm}, $\Omega$ is volumetrically non-degenerate, and hence $\delta^1_{\tilde f^k} V = \delta^1_{f^k} V$, resulting in the same operator $\F$ (as $\F a = \delta^1_{f^a} V = \delta^1_{\tilde f^a} V$). 
\item Since $L_{Jac} f^{\textbf{1}} = 0$, it follows by volume non-degeneracy that $\F \textbf{1} = \delta^1_{f^{\textbf{1}}} V = 0$. 
\item
By (\ref{eq:Q^0-Dcon}), 
\begin{align*}
Q^0(f^a,f^b) & = -\scalar{L_{Jac} f^a, f^b}_{L^2(\Sigma^1)} = -(n-1) \sscalar{a , f^b}_{L^2(\Sigma^1)} \\
& = -(n-1) \sscalar{a , \delta^1_{f^b} V}_{E^{(q-1)}} = -(n-1) a^T \F b . 
\end{align*}
\item The symmetry of $\F$ follows immediately from the symmetry of $Q^0$. 
\item When $\Omega$ is stable and perpendicular then $\F \geq 0$ since $-(n-1) a^T \F a = Q^0(f^a) \leq 0$ for all $a \in E^{(q-1)}$ by Proposition \ref{prop:conformal-Jacobi-prop} \ref{it:non-positive}. \end{enumerate} 
\end{proof}

\begin{definition}
The operator $\N : \R^{n+1} \rightarrow E^{(q-1)}$ is defined by:
\[
\N \theta = \delta^1_{W_\theta} V \;\; \text{ i.e. } \;\; \N := \sum_{i< j} \int_{\Sigma_{ij}} e_{ij} \otimes \n_{ij} d\HH^{n-1} . 
\]
\end{definition}
\noindent
Indeed, note that:
\[
\delta^1_{W_\theta} V = \sum_{i<j} \int_{\Sigma_{ij}} e_{ij} \scalar{W_\theta,\n_{ij}} d\HH^{n-1} = \sum_{i<j} \int_{\Sigma_{ij}} e_{ij} \scalar{\theta,\n_{ij}} d\HH^{n-1} . 
\]

\begin{proposition} \label{prop:FC=N}
$\F \C = \N$ and $\tr( \F \C \C^T) = \HH^{n-1}(\Sigma)$. 
\end{proposition}
\begin{proof}
Lemma \ref{lem:Mobius-Jacobi} directly implies that $\F \C = \N$. Consequently, since $\c_{ij} = \n_{ij} - \k_{ij} p$, we have:
\begin{align*}
\tr(\F \C \C^T) & = \tr(\C^T \N ) = \sum_{i<j} \int_{\Sigma_{ij}} \scalar{\C^T e_{ij} , \n_{ij}} d\HH^{n-1} \\
& =  \sum_{i<j} \int_{\Sigma_{ij}} \scalar{\c_{ij}, \n_{ij}} d\HH^{n-1} = \HH^{n-1}(\Sigma^1) = \HH^{n-1}(\Sigma). 
\end{align*}
\end{proof}

\begin{corollary}
For all $q \leq n+2$, the following trace identity holds on a standard bubble $q$-cluster $\Omega$:
\begin{equation} \label{eq:trace-id}
 \tr( \F \brac{\Id/2 + \k \otimes \k}) = \HH^{n-1}(\Sigma) . 
\end{equation}
\end{corollary}
\begin{proof}
By Proposition \ref{prop:standard-char} a standard bubble cluster is characterized by $\C \C^T = \Id/2 + \k \otimes \k$, and so the assertion follows from the previous proposition. 
\end{proof}

\begin{conjecture} \label{conj:trace-id}
For any $3$-Plateau, stable, perpendicular spherical Voronoi $q$-cluster on $\S^n$ with equatorial cells and $q \leq n+1$,
the trace identity (\ref{eq:trace-id}) holds as an inequality:
\begin{equation} \label{eq:trace-inequality}
 \tr( \F \brac{\Id/2 + \k \otimes \k}) \leq \HH^{n-1}(\Sigma) .
\end{equation}
\end{conjecture}

Establishing the validity of this conjecture
is precisely the key obstacle which we were not able to overcome or bypass in this work, and confirming it would immediately extend all of our results in this work from the quintuple case to arbitrary $q \leq n+1$ -- see Remark \ref{rem:conjectures}.  
 In the sequel we shall establish that the trace identity (\ref{eq:trace-id}) does hold on general $3$-Plateau PCF clusters, and obtain several stronger variants of Conjecture \ref{conj:trace-id}.

\subsection{Volume complete system of conformal Jacobi fields} \label{subsec:volume-complete}

\begin{definition}[Volume complete system of conformal Jacobi fields]
We shall say that a volume complete system of conformal Jacobi fields exists for $\Omega$ if for all $V \in E^{(q-1)}$, there exists a conformal Jacobi field $f$ so that $\delta^1_f V = V$. 
\end{definition}

Throughout this subsection, we shall assume in addition that $V(\Omega) \in \interior \Delta^{(q-1)}$, since otherwise a volume complete system of conformal Jacobi fields can never exist; by removing empty cells from $\Omega$, one may always reduce to this case. 

\medskip

Assume that a \emph{conformally} complete system of conformal Jacobi fields exists for $\Omega$. It follows that the system is \emph{volume} complete if and only if the conformal-to-volume operator $\F : E^{(q-1)} \rightarrow E^{(q-1)}$ is full-rank. 
On a stable cluster, the question of whether a given $V \in E^{(q-1)}$ lies in the image $\Im \F$ turns out to depend on whether $Q^0$ is bounded from below on $\D_Q[V]$, or equivalently, on:
\[
\D_{con}[V] := \D_{con} \cap \D_Q[V] . 
\]
Note that $\D_{con}[V]$ is non-empty, since we assume $V(\Omega) \in \interior \Delta^{(q-1)}$ and we can simply take an appropriate linear combination of smooth scalar fields compactly supported in (the relatively open) $\Sigma_{ij}$ (i.e. vanishing on $\Sigma^2$ and hence satisfying the conformal boundary conditions), such as the ones constructed in \cite[Lemma C.2]{EMilmanNeeman-GaussianMultiBubble}.
Also note that as $\F$ is symmetric, it acts invariantly on $\Im \F = (\ker \F)^{\perp}$. 

\begin{lemma} \label{lem:bounded-below}
Assume that a conformally complete system of conformal Jacobi fields exists for $\Omega$, and that in addition $\Omega$ is stable and $V(\Omega) \in \interior \Delta^{(q-1)}$. 
Then for every $V \in E^{(q-1)}$, the following are equivalent:
\begin{enumerate}
\item $V \in \Im \F$. 
\item $Q^0$ is bounded below on $\D_Q[V]$. 
\item $Q^0$ is bounded below on $\D_{con}[V]$. 
\end{enumerate}
\end{lemma}
\begin{proof}
If $V = \F a$ then the conformal Jacobi field $f^a$ with conformal parameter $a \in E^{(q-1)}$ is a minimizer of $Q^0$ on $\D_Q[V]$ by Proposition \ref{prop:conformal-Jacobi-prop} and stability, and hence $Q^0$ is bounded below on $\D_Q[V]$, and in particular on $\D_{con}[V]$.  

On the other hand, if $V \notin \Im \F = (\ker \F)^{\perp}$, then there exists $a \in E^{(q-1)}$ so that $\F a = 0$ and $\scalar{a,V} \neq 0$. Let $f^a \in \D_{con}$ be a conformal Jacobi field with conformal parameter $a$, and note that $\delta^1_{f^a} V = \F a = 0$ and $Q^0(f^a) = -(n-1) a^T \F a = 0$.  If $g_V \in \D_{con}[V]$, then by (\ref{eq:Q0-Jacobi}):
\begin{align}
\nonumber
Q^0(g_V + \lambda f^a) & = Q^0(g_V) + 2 \lambda Q^0(f^a,g_V) = Q^0(g_V) - 2 \lambda (n-1) \scalar{a,\delta^1_{g_V} V}_{E^{(q-1)}} \\
\label{eq:unbounded-below} & = Q^0(g_V) - 2 \lambda (n-1) \scalar{a,V} . 
\end{align}
Taking $|\lambda| \rightarrow \infty$, since $g_V + \lambda f^a \in \D_{con}[V]$ and $\scalar{a,V} \neq 0$, we conclude that $Q^0$ is not bounded below on $\D_{con}[V]$. 
\end{proof}

\begin{proposition} \label{prop:volume-complete}
Let $\Omega$ be a $3$-Plateau, stable, perpendicular spherical Voronoi cluster with equatorial cells. In particular, a conformally complete system of conformal Jacobi fields exists for $\Omega$ by Theorem \ref{thm:conformally-complete}. Let $\F : E^{(q-1)} \rightarrow E^{(q-1)}$ denote its corresponding conformal-to-volume operator. Then the following are equivalent:
\begin{enumerate}[(1)]
\item \label{it:vc1} A volume complete system of conformal Jacobi fields exists for $\Omega$. 
\item \label{it:vc2} $\F$ is full-rank. 
\item \label{it:vc3} $\F > 0$ is positive-definite. 
\item \label{it:vc4} For all $V \in E^{(q-1)}$, $Q^0$ is bounded below on $\D_Q[V]$. 
\item \label{it:vc5} For all $V \in E^{(q-1)}$, $Q^0$ is bounded below on $\D_{con}[V]$. 
\end{enumerate}
\end{proposition}

\begin{proof}
By definition of $\F$ we have \ref{it:vc1} iff \ref{it:vc2}. Since $\F \geq 0$ by Proposition \ref{prop:F-properties} \ref{it:pos-def} (using that $\Omega$ is perpendicular and stable), \ref{it:vc2} iff \ref{it:vc3}. Finally, Lemma \ref{lem:bounded-below} confirms the equivalence between \ref{it:vc2}, \ref{it:vc4} and \ref{it:vc5}. 
\end{proof}

Given a cluster $\Omega$ as in Proposition \ref{prop:volume-complete}, the question of whether a volume complete system of conformal Jacobi fields exists, or equivalently, whether $\F > 0$, is an interesting one, and we conjecture that this is always the case when $q \leq n+1$ (see the stronger Conjecture \ref{conj:F=F0}). The most tractable way to establish this seems to be by showing that $Q^0$ is bounded below on $\D_Q[V]$, but we do not pursue this here. The reason is that whenever $Q^0$ is \textbf{not} bounded below on $\D_Q[V]$, this actually works in our favor when attempting to establish the desired PDI (\ref{eq:PDI-F}) or for establishing that the isoperimetric profile $\I$ is concave (in the viscosity sense of Definition \ref{def:viscosity}). Indeed, if $\F$ is not full-rank, we can simply consider $\F^{\eps} := \F + \eps \P_{\ker \F}$ for an arbitrarily small $\eps > 0$, where $\P_{\ker \F}$ is the orthogonal projection onto the kernel $\ker \F$. Note that $\F^{\eps} > 0$ since $\F|_{(\ker \F)^{\perp}} > 0$. 
Then for all $a \notin \Im \F$, since $\F^{\eps} a \notin \Im \F$ then $Q^0$ is not bounded below on $\D_Q[\F^{\eps} a]$, and we can find a scalar field $g^a$ with $\delta^1_{g^a} V = \F^{\eps} a$ and with $Q^0(g^a) \leq -(n-1) a^T \F^{\eps} a$. 

\medskip

For our purposes, it will be enough to show that $\F > 0$ for general $3$-Plateau perpendicular PCF clusters (with $V(\Omega) \in \interior \Delta^{(q-1)}$) -- this will be established in Proposition \ref{prop:PCF-F=F0} in the sequel. 

\medskip
In the meantime, we can already derive the following consequence of Lemma \ref{lem:bounded-below} and Proposition \ref{prop:volume-complete}, which is of independent interest. 
\begin{theorem} \label{thm:q-1-positive}
Let $\Omega$ be a $3$-Plateau, stable, perpendicular spherical Voronoi cluster with equatorial cells. Then $(L_{Jac},\D_{con})$ has  exactly $q-1$ positive eigenvalues. 
\end{theorem}
\begin{proof}
We've already seen in Corollary \ref{cor:at-most-positive-eigenvalues} that stability implies that $L_{Jac}$ has at most $q-1$ positive eigenvalues. On the other hand, by the minimax principle, it is enough to construct a $(q-1)$-dimensional linear subspace $G$ of $\D_{con}$ on which $L_{Jac}$ is positive-definite. 

To this end, recall that a conformally complete system of conformal Jacobi fields exists for $\Omega$ by Theorem \ref{thm:conformally-complete}, let $\F : E^{(q-1)} \rightarrow E^{(q-1)}$ denote its corresponding conformal-to-volume operator, and consider an orthonormal basis $\{ v^i \}_{i=1,\ldots,q-1}$ of $E^{(q-1)}$ so that the first $k$ vectors span $\ker \F$ and the other vectors span $\Im \F$. Given $d = \sum_{i=1}^{q-1} c_i v^i \in E^{(q-1)}$, set $f^d := \sum_{i=1}^{q-1} c_i f^{v^i}$ where $f^{v^i}$ is a conformal Jacobi field with conformal parameter $v^i$, so that $f^d$ is a conformal Jacobi field with conformal parameter $d$. For $i=1,\ldots,k$, select $h^{v^i} \in \D_{con}[v^i]$, and denote $h^b = \sum_{i=1}^k c_i h^{v^i}$ for $b = \sum_{i=1}^k c_i v^i \in \ker \F$, so that $\delta^1_{h^b} V = b$. Set 
\[
M := \max\brac{1 , \frac{1}{n-1} \max \{ Q^0(h^b) \; ; \; b \in \ker \F , |b|=1 \}} ,
\]
which exists and is finite by compactness. Now define $g^b = h^b + M f^b$, and note that by (\ref{eq:unbounded-below}):
\begin{align*}
Q^0(g^b) & = Q^0(h^b + M f^b) = Q^0(h^b) - 2 M (n-1) \scalar{b,\delta^1_{h^b} V}_{E^{(q-1)}} \\
& = Q^0(h^b) - 2 M (n-1) |b|^2 \leq - M (n-1) |b|^2 . 
\end{align*}
Furthermore, $\delta^1_{g^b} V = b + M \F b = b$. Define the following linear subspace of $\D_{con}$: 
\[
G := \{f^a + g^b \; ; \;  a \in \Im \F , b \in \ker \F\}.
\]
Note that $G$ is of dimension $q-1$ since $\delta^1_{f^a + g^b} V = \F a + b$, which spans the entire $E^{(q-1)}$ as $\F$ acts invariantly on $\Im \F$. On $G$ we have:
\begin{align*}
Q^0(f^a + g^b) & = Q^0(f^a) + 2 Q^0(f^a,g^b) + Q^0(g^b) \leq  -(n-1) a^T \F a + 2 \sscalar{a, \delta^1_{g^b} V}_{E^{(q-1)}} - (n-1) |b|^2\\
& = -(n-1) a^T \F a - (n-1) |b|^2 = -(n-1) (a+b)^T \F^1 (a+b) ,
\end{align*}
where $\F^1 = \F + \P_{\ker \F} > 0$. 
Since for all $f \in \D_{con}$, $Q^0(f) = - \scalar{L_{Jac} f , f}$, we confirm that $L_{Jac}$ is positive definite on $G$, thereby concluding the proof. 
\end{proof}

In particular, this applies to minimizing $q$-clusters $\Omega$ when $q \leq n+1$ and $V(\Omega) \in \interior \Delta^{(q-1)}$ by Theorem \ref{thm:intro-structure}, yielding Theorem \ref{thm:intro-minimizer-q-1-positive}. 

\subsection{On pseudo conformally flat clusters}

It will be important for us to know that a conformally complete system exists for PCF clusters, regardless of whether they are stable or perpendicular. 

\begin{proposition} \label{prop:PCF-F}
Let $\Omega$ be a $3$-Plateau PCF cluster with compatibility parameter $\xi \in \R^{n+1}$. On $\S^n$ define the function $\Psi_\xi$ given by:
\[
\Psi_\xi(p) := 1 - \scalar{p,\xi} .
\]
Then for all $a \in E^{(q-1)}$, $f^a = (f^a_{ij}) = (a_{ij} \Psi_\xi)$ is a conformal Jacobi field with conformal parameter $a$. In particular, a conformally complete system of conformal Jacobi fields exists for $\Omega$, and we have:
\[
\F = L_{\Psi_\xi} . 
\]
\end{proposition}
\begin{proof}
Applying Lemma \ref{lem:LJac-dilation} and using that $\II_0 = 0$ on a spherical cluster, we have on $\Sigma_{ij}$:
\begin{equation} \label{eq:LJac-Psi}
L_{Jac} \Psi_\xi = (n-1) ( 1+ \k_{ij}^2 + \k_{ij} \scalar{\c_{ij},\xi}) = n-1 , 
\end{equation}
where we have used that $\scalar{\c_{ij},\xi} = - \k_{ij}$ by definition of the PCF compatibility parameter $\xi \in \R^{n+1}$. 
Consequently $L_{Jac} f^a_{ij} = (n-1) a_{ij}$ on $\Sigma_{ij}$. 

It remains to check that $f^a$ satisfies conformal BCs (recall Remark \ref{rem:conformal-BCs}). Note that $\scalar{\n_{ij},\xi} = - \k_{ij} (1 - \scalar{p,\xi})) = -\k_{ij} \Psi_\xi$ on $\Sigma_{ij}$, and hence for $p \in \Sigma_{ijk}$:
\[
\nabla_{\n_{\partial ij}} \Psi_{\xi} = - \scalar{\n_{\partial ij},\xi} = - \scalar{ \frac{\n_{ik} + \n_{jk}}{\sqrt{3}} , \xi } = \frac{\k_{ik} + \k_{jk}}{\sqrt{3}} \Psi_\xi = \bar \II^{\partial ij} \Psi_{\xi} . 
\]
Consequently, we have $\nabla_{\n_{\partial ij}} f^a_{ij} = \bar \II^{\partial ij} f^a_{ij}$ on $\partial \Sigma_{ij}$, and 
we confirm that $\nabla_{\n_{\partial ij}} f^a_{ij} - \bar \II^{\partial ij} f^a_{ij} = 0$ is independent of $(i,j) \in \cyclic(u,v,w)$ on $\Sigma_{uvw}$.

Since $f^a_{ij} = \scalar{\Psi_\xi e_{ij} , a}$ on $\Sigma_{ij}$, it follows that:
\[
\F = \sum_{i<j} \int_{\Sigma_{ij}} e_{ij} \otimes \Psi_\xi e_{ij} d\HH^{n-1} = \sum_{i<j} \int_{\Sigma_{ij}} \Psi_\xi d\HH^{n-1} \; e_{ij} \otimes e_{ij}  = L_{\Psi_{\xi}} . 
\]
\end{proof}

\begin{proposition} \label{prop:PCF-trace}
On a $3$-Plateau PCF cluster $\Omega$, the trace-identity (\ref{eq:trace-id}) always holds. 
\end{proposition}
\begin{proof}
Recall that $L_{Jac}(1) = (n-1)(1+\k_{ij}^2)$ on $\Sigma_{ij}$ by Lemma \ref{lem:LJac-dilation} and that $L_{Jac}(\Psi_{\xi}) = n-1$ by (\ref{eq:LJac-Psi}). Since $\F = L_{\Psi_\xi}$ and $\tr ((\Id/2 + \k \otimes \k)  (e_{ij} \otimes e_{ij})) = 1 + \k_{ij}^2$, we conclude from Lemma \ref{lem:LJac-non-oriented} that:
\begin{align*}
\tr(\F (\Id/2 + \k\otimes \k)) & = \sum_{i<j} \int_{\Sigma_{ij}} (1 + \k_{ij}^2) \Psi_\xi d\HH^{n-1} = \frac{1}{n-1} \int_{\Sigma^1} L_{Jac}(1) \Psi_\xi d\HH^{n-1} \\
&  = \frac{1}{n-1} \int_{\Sigma^1} 1 L_{Jac}(\Psi_\xi) d\HH^{n-1} = \int_{\Sigma^1} d\HH^{n-1}  . 
\end{align*}
\end{proof}

\subsection{PDE for model isoperimetric profile}

Note that when the cluster $\Omega$ is conformally flat, in particular PCF with $|\xi| < 1$ by Lemma \ref{lem:conformally-flat}, then $\Psi_{\xi} > 0$, and so when in addition $V(\Omega) \in \interior \Delta^{(q-1)}$, $\F = L_{\Psi_\xi} > 0$ is positive definite on $E^{(q-1)}$ by Lemma \ref{lem:LA-connected} and \ref{lem:LA-positive}. In particular, this holds for standard bubbles, and so we can easily deduce:
\begin{proposition} \label{prop:Im-PDE-proof}
For all $q \leq n+2$, the model isoperimetric profile $\I_m : \Delta^{(q-1)} \rightarrow \R_+$ satisfies the following (fully non-linear yet elliptic) PDE on $\interior(\Delta^{(q-1)})$:
\[
\nabla^2 \I_m < 0 ~,~ \tr \brac{(-\nabla^2 \I_m)^{-1} \; (\Id+ \frac{2}{(n-1)^2} \nabla \I_m \otimes \nabla \I_m)} = \frac{2}{n-1} \I_m . 
\]
\end{proposition}
\begin{proof}
Let $\Omega$ be a standard-bubble $q$-cluster (with $q \leq n+2$) of volume $v_0 \in \interior \Delta^{(q-1)}$; it always exists by Lemma \ref{lem:standard-volume}. Let $\C : \R^{n+1} \rightarrow E^{(q-1)}$ be the associated quasi-center operator, which is necessarily full-rank. Recall from Corollary \ref{cor:full-dim-F} that a conformally complete system of conformal Jacobi fields $\{f^a\}_{a \in E^{(q-1)}}$ exists for $\Omega$, which may be chosen to be the normal component of any M\"obius field $W_\theta$ so that $a = C \theta$. 
Let $\F_m : E^{(q-1)} \rightarrow E^{(q-1)}$ denote the associated conformal-to-volume operator, so that:
\[
 \delta^1_{W_\theta} V = \delta^1_{f^a} V = \F_m a ~,~ Q(W_\theta) = Q^0(f^a) = -(n-1) a^T \F_m a
\]
 by Proposition \ref{prop:F-properties}. As explained above, we know that $\F_m > 0$ since $\Omega$ is conformally flat, and also that the trace identity (\ref{eq:trace-id}) holds for $\F = \F_m$ by Proposition \ref{prop:PCF-trace} (a standard-bubble is Plateau). 
 
Being a standard-bubble, $\Omega$ is clearly stationary with Lagrange multipler $\lambda = (n-1) \k \in E^{(q-1)}$ by Lemma  \ref{lem:equivalent-stationary}. Now it is crucial to observe that the flow $\Phi_t$ generated by a given M\"obius field $W_\theta$ is a conformal automorphism of $\S^n$ and hence $\Omega_t := \Phi_t(\Omega)$ remains a standard-bubble by Lemma \ref{lem:Mobius-preserves-Voronoi}, and we have equality $\per(\Omega_t) = \I_m(V(\Omega_t))$ for all $t \in \R$. Repeating the argument in Subsection \ref{subsec:viscosity}, we deduce after taking first and second derivatives at $t=0$:
\begin{itemize}
\item  $\nabla \I_m(v_0) = \lambda = (n-1) \k$, and hence the trace identity (\ref{eq:trace-id}) translates to:
\begin{equation} \label{eq:model-trace}
\tr\brac{\F_m \brac{\Id/2 + \frac{1}{(n-1)^2} \lambda \otimes \lambda}} = \HH^{n-1}(\Sigma) = \I_m(v_0) . 
\end{equation}
\item We have:
\[
(\delta^1_{W_\theta} V)^T  \nabla^2 \I_m  \delta^1_{W_\theta} V = Q(W_\theta) \;\;\; \forall \theta \in \R^{n+1} ,
\]
which translates to the following identity as symmetric quadratic forms:
\[
\F_m \nabla^2 \I_m \F_m = -(n-1) \F_m .
\]
As $\F_m > 0$, we deduce that:
\begin{equation} \label{eq:model-hessian}
 \nabla^2 \I_m = -(n-1) \F^{-1}_m  < 0 . 
 \end{equation}
\end{itemize}
Combining (\ref{eq:model-trace}) and (\ref{eq:model-hessian}), the asserted PDE for $\I_m$  is obtained. 
\end{proof}

On the other hand, for a general $3$-Plateau PCF cluster, even if one assumes that the cluster is stable, perpendicular and with equatorial cells, the positivity of $\F = L_{\Psi_{\xi}}$ is not clear at all, as we may not have $|\xi| < 1$. In the next section, we shall establish that $\F > 0$ for general $3$-Plateau \emph{perpendicular} PCF clusters (all of whose cells are non-empty).

\section{Conformal perturbation} \label{sec:perturbation}

It will be very useful for us to apply a conformal perturbation of an initial spherical Voronoi cluster $\Omega$. Throughout this section, we denote 
\begin{equation} \label{eq:Phit}
\Phi_t := \exp(t W_N) : \S^n \rightarrow \S^n ,
\end{equation}
the conformal (M\"obius) automorphism obtained by flowing along the M\"obius field $W_N = N - \scalar{N,p} p$ for time $t \in \R$. We shall always assume that $N \in \S^n$, and as our notation suggests, $N$ will be chosen to be the North pole of $\Omega$ whenever $\Omega$ is assumed perpendicular. 

\begin{lemma} \label{lem:conformal-factor}
The conformal factor of $\Phi_t$ is:
\[
\frac{1}{\cosh t +  \scalar{N,p} \sinh t } . 
\]
\end{lemma}
\begin{proof}
Direct computation verifies that: \[
 \Phi_t(p) = \frac{p - \scalar{p,N} N + (\scalar{p,N} \cosh t   + \sinh t) N}{\cosh t + \scalar{p,N} \sinh t } ,
   \]
satisfies the (uniquely defining) ODE $\frac{d}{dt} \Phi_t = N - \scalar{N,\Phi_t} \Phi_t$. 
Its conformal factor is the coefficient of $p$, yielding the assertion. 
\end{proof}

\begin{lemma} \label{lem:kc-evolution}
For all $t \in \R$, $\Phi_t(\Omega)$ is a spherical Voronoi cluster whose curvature and quasi-center parameters are given by:
\begin{align*}
\k_i(t) & = \k_i(0) \cosh t - \scalar{\c_i(0), N} \sinh t ,\\
\c_i(t)  = \c_i(0) - \scalar{\c_i(0),N} N & + (\scalar{\c_i(0),N} \cosh t - \k_i(0) \sinh t ) N . 
\end{align*}
\end{lemma}
\begin{proof}
$\Phi_t(\Omega)$ remains a spherical Voronoi cluster by Lemma \ref{lem:Mobius-preserves-Voronoi}. To calculate its curvature and quasi-center parameters, it will be enough to observe that they satisfy the following system of first-order ODEs:
\[
\frac{d}{dt} {\k_{i} \choose \c_{i}} = {-\scalar{\c_{i},N} \choose -\k_{i} N} ,
\]
so that the unique solution to this system is given by the asserted equations. 

By linearity, connectedness of the interface adjacency graph, and our convention to have $\sum_{i=1}^q \k_i(t) = 0$ and $\sum_{i=1}^q \c_i(t) = 0$, it is enough to verify the above ODEs for $\k_{ij}$ and $\c_{ij}$ on $\Sigma_{ij}$. Recall that $\k_{ij} = \frac{H_{ij}}{n-1}$ and hence by (\ref{eq:LJac-deltaH}) and Lemma \ref{lem:LJac-dilation}:
\[
\delta^1_{W_N} \k_{ij} = - \frac{1}{n-1} L_{Jac} W_N^{\n_{ij}} =  - \frac{1}{n-1} L_{Jac} \scalar{\n_{ij},N} = - \scalar{\c_{ij},N} ,
\]
yielding the first ODE. For the second, we use the Euclidean flat connection $\nabla_{\R^{n+1}}$ for all computations. If $F_t : \R^{n+1} \rightarrow \R^{n+1}$ is the flow generated by $W_N$ in $\R^{n+1}$, then:
\[
\frac{d}{dt} d_{\R^{n+1}} F_t = \nabla_{\R^{n+1}} W_{N} = -p \otimes N - \scalar{N,p} \Id_{\R^{n+1}} . 
\]
It follows that $\Sigma_{ij}$ remains parallel to first order, and hence $P_{T \S^n} \delta^1_{W_N} \n_{ij} = 0$. However, $\delta^1_{W_N} \n_{ij}$ does have a component in the radial direction $p$, which is seen to be $-\scalar{\n_{ij},N} p$. 
Recalling that $\c_{ij} = \n_{ij} - \k_{ij} p$, we conclude that:
\[
\delta^1_{W_N} \c_{ij} = \delta^1_{W_N} \n_{ij} - (\delta^1_{W_N} \k_{ij}) p - \k_{ij} \delta^1_{W_N} p = 
-\scalar{\n_{ij},N} p + \scalar{\c_{ij},N} p - \k_{ij} W_N = - \k_{ij} N . 
\]
\end{proof}

\begin{corollary} \label{cor:Omega_t-PCF}
Let $\Omega$ be a perpendicular spherical Voronoi cluster with North pole $N \in \S^n$. Then $\Omega_t := \Phi_t(\Omega)$ is a PCF cluster with compatibility parameter $\xi_t := \coth (t) N$ for all $t \neq 0$. 
\end{corollary}
\begin{proof}
Since $N \perp \c_{i}$ for all $i=1,\ldots,q$, we obtain from the previous lemma that $\Omega_t$ is a spherical Voronoi cluster with parameters:
\begin{equation} \label{eq:Mobius-kct}
\k_{i}(t) = \k_{i}(0) \cosh(t) ~,~ \c_i(t) = \c_{i}(0)  -  \sinh (t) \k_{i}(0)  N . 
\end{equation}
Therefore:
\[
\scalar{\c_{i}(t) , N} = - \sinh(t) \k_{i} = -\tanh(t) \k_{i}(t) , 
\]
and hence:
\[
\scalar{\c_{i}(t) , \xi_t} = -\k_{i}(t) . 
\]
\end{proof}

See Figure \ref{fig:conformal-perturbation} in the Introduction for a graphical depiction of the conformally perturbed $\{\Omega_t\}$. 

\subsection{The relaxed conformal-to-volume operator $\F_0$} \label{subsec:relaxed}

We henceforth only consider perpendicular spherical Voronoi clusters $\Omega$. By the previous corollary, we know that $\Omega_t$ is necessarily a PCF cluster with compatibility parameter $\xi_t$ for all $t \neq 0$. It follows by Proposition \ref{prop:PCF-F} that a conformally complete system of conformal Jacobi fields exists for $\Omega_t$ for all $t \neq 0$, and that its associated conformal-to-volume operator is given by:
\[
\F_t := L_{\Psi_{\xi_t}} = \sum_{i<j} \int_{\Sigma_{ij}} \Psi_{\xi_t} d\HH^{n-1} \;  e_{ij} \otimes e_{ij}  ,
\]
where:
\[
  \Psi_{\xi_t} := 1 - \scalar{p,\xi_t} = 1 - \coth(t) \scalar{p,N} . 
\]

Even though $\Psi_{\xi_t}$ is blowing up as $t \rightarrow 0$, let us verify that the limit $\lim_{t \rightarrow 0} \F_t$ nevertheless exists and compute it. As  these are finite-dimensional linear operators, convergence is understood in any norm (as they are all equivalent). 

\begin{proposition} \label{prop:F0}
Let $\Omega$ be a perpendicular spherical Voronoi cluster with North pole $N \in \S^n$. Then the limit $\lim_{t \rightarrow 0} F_t$ exists and is equal to:
\[
\lim_{t \rightarrow 0} \F_t = \F_0 := n L_{\scalar{p,N}^2} = n \sum_{i<j} \int_{\Sigma_{ij}} \scalar{p,N}^2 d\HH^{n-1}(p) \; 
 e_{ij} \otimes e_{ij}.
 \]
The operator $\F_0 : E^{(q-1)} \rightarrow E^{(q-1)}$ shall be called the \emph{relaxed} conformal-to-volume operator associated to $\Omega$. 
\end{proposition}
\begin{proof}
As all operators are symmetric and act on $E^{(q-1)}$, it is enough to verify pointwise convergence of off-diagonal entries in a matrix representation with respect to the standard basis $e_1,\ldots,e_q$.

The conformal factor of the diffeomorphism $\Phi_t$ is $(\cosh(t) + \sinh(t) \scalar{N,p})^{-1} = 1- t \scalar{N,p} + o(t)$, and so its Jacobian is $\Jac \Phi_t(p) = 1 - (n-1) \scalar{N,p} t + o(t)$. In addition, $\xi_t = N \coth(t) = \frac{1}{t} N +  o(1)$ and $\Phi_t(p) = p + W_N(p) t + o(t)$. Consequently, denoting $\Sigma_{ij}(t) = \Phi_t(\Sigma_{ij})$ and applying the change of variables $p_t = \Phi_t(p)$, we obtain for any $i < j$:
\begin{align*}
-(\F_t)_{ij} & =  \int_{\Sigma_{ij}(t)} (1 - \scalar{p_t,\xi_t}) d\HH^{n-1}(p_t) = \int_{\Sigma_{ij}} (1 - \scalar{\Phi_t(p),\xi_t}) \Jac \Phi_t(p) dp \\
& =  \int_{\Sigma_{ij}} (1 - \sscalar{p + t W_N(p) , N/t} + o(1))(1 - (n-1) \scalar{p,N} t + o(t)) d\HH^{n-1}(p)  . 
\end{align*}
As $\scalar{W_N(p),N} = 1 - \scalar{p,N}^2$, we obtain:
\[
= \int_{\Sigma_{ij}} (\scalar{p,N}^2 - \sscalar{p,N/t}) (1 - (n-1) \scalar{p,N} t) dp + o(1) .
\]
But since $\Sigma_{ij}$ is symmetric with respect to reflection about $N^{\perp}$ and $\scalar{p,N}$ is odd, the term $\scalar{p,N/t}$ which appears to blow-up integrates to zero, and its only contribution is after being multiplied by the term $- (n-1) \scalar{p,N} t$, yielding a grand total of:
\[
= n \int_{\Sigma_{ij}} \scalar{p,N}^2 dp + o(1) = -(\F_0)_{ij} + o(1) . 
\]
\end{proof}

\begin{proposition} \label{prop:trace-id-F0}
Let $\Omega$ be a $3$-Plateau perpendicular spherical Voronoi cluster. Then:
\begin{enumerate}[(i)]
\item If  $V(\Omega) \in \interior \Delta^{(q-1)}$ then $\F_0 > 0$ is strictly positive-definite. 
\item The trace identity (\ref{eq:trace-id}) holds for $\F_0$ in place of $\F$:
\begin{equation} \label{eq:trace-id-F0}
\tr( \F_0 \brac{\Id/2 + \k \otimes \k}) = \HH^{n-1}(\Sigma) . 
\end{equation}
\end{enumerate}
\end{proposition}
\begin{proof} 
\hfill
\begin{enumerate}[(i)]
\item
Since $\Omega$ is perpendicular, we have $\int_{\Sigma_{ij}} \scalar{p,N}^2 d\HH^{n-1}(p) > 0$ for all non-empty interfaces $\Sigma_{ij}$, and since $V(\Omega) \in \interior \Delta^{(q-1)}$, we deduce that $F_0 = n L_{\scalar{p,N}^2}  > 0$ is strictly positive-definite by Lemmas \ref{lem:LA-connected} and \ref{lem:LA-positive}. 
\item 
Using that $\Omega$ is perpendicular again and recalling Lemmas \ref{lem:LJac-dilation}, as well as the regularity (Lemma \ref{lem:Voronoi-regular}), stationarity (Lemma \ref{lem:Voronoi-stationary})  and Lemma \ref{lem:LJac-non-oriented}, we confirm:
\begin{align*}
& \tr( \F_0 \brac{\Id/2 + \k \otimes \k}) = n \sum_{i<j} \int_{\Sigma_{ij}} (1+\k_{ij})^2 \scalar{p,N}^2 d\HH^{n-1}(p) \\
& = \int_{\Sigma^1} \brac{1 - \frac{1}{2} \Delta_{\Sigma} \scalar{p,N}^2} d\HH^{n-1}(p) = \int_{\Sigma^1} d\HH^{n-1} = \HH^{n-1}(\Sigma) . 
\end{align*}

\end{enumerate}
\end{proof}

\subsection{Relation to the actual conformal-to-volume operator $\F$}

We were unable to verify the following:
\begin{conjecture} \label{conj:F=F0}
Let $\Omega$ be a $3$-Plateau, stable, perpendicular spherical Voronoi $q$-cluster with equatorial cells and $q \leq n+1$. Then $\F = \F_0$, and in particular, $\F > 0$.
\end{conjecture}
\begin{remark} \label{rem:F=F0}
In other words, the conjecture is that $\F( \lim_{t \rightarrow 0} \Omega_t) = \lim_{t \rightarrow 0} \F(\Omega_t)$, namely that $\F$ is continuous on the one-parameter family of spherical Voronoi clusters $\Omega_t$ given by (\ref{eq:Mobius-kct}), obtained as a conformal deformation $\Omega_t = \Phi_t(\Omega)$ with $\Phi_t$ given by (\ref{eq:Phit}).  
\end{remark}

\begin{remark}
Note that in view of (\ref{eq:trace-id-F0}), Conjecture \ref{conj:F=F0} is stronger than Conjecture \ref{conj:trace-id}. Consequently, a confirmation of Conjecture \ref{conj:F=F0} would immediately extend all of our results in this work from the quintuple case to arbitrary $q \leq n+1$ -- see Remark \ref{rem:conjectures}. 
\end{remark}

What we can show is the following partial confirmation of Conjecture \ref{conj:F=F0}:

\begin{proposition} \label{prop:FC=F0C}
Let $\Omega$ be a perpendicular spherical Voronoi cluster having a conformally complete system of conformal Jacobi fields. Then $\F \C = \F_0 \C$. 
\end{proposition}
\begin{proof}
Denote $\Omega_t = \Phi_t(\Omega)$, the spherical Voronoi cluster obtained after perturbing $\Omega$ using the conformal automorphism $\Phi_t$. Denote by $\F^t$, $\C^t$ and $\N^t$ the usual operators corresponding to $\Omega_t$; note that $\F^0 = \F$ is well-defined by assumption, and that $\F^t$ for $t \neq 0$ is well-defined by Proposition \ref{prop:PCF-F} since $\Omega_t$ is a PCF cluster by Corollary \ref{cor:Omega_t-PCF}. Recall by Proposition \ref{prop:FC=N} that $\F^t \C^t = \N^t$.  

Now, since $\C^t \rightarrow \C^0=\C$ and $\N^t \rightarrow \N^0 =\N$ trivially, it follows that:
\[
\F_0 \C = (\lim_{t \rightarrow 0} \F^t) \C =  \lim_{t \rightarrow 0} (\F^t \C^t) =  \lim_{t \rightarrow 0}  \N^t = \N = \F \C . 
\]
\end{proof}

We obtain the following corollary, which may be of independent interest:
\begin{corollary}
Let $\Omega$ be a perpendicular spherical Voronoi cluster having a conformally complete system of conformal Jacobi fields. Then $\R^{n+1} \ni \theta \mapsto Q(W_\theta)$ is negative semi-definite, and moreover:
\[
Q(W_\theta) = - (n-1) n \sum_{i<j} \int_{\Sigma_{ij}} \scalar{\theta,\c_{ij}}^2 \scalar{p,N}^2 d\HH^{n-1}(p) \;\;\; \forall \theta \in \R^{n+1}.
\]
\end{corollary}
\begin{proof}
By Lemma \ref{lem:Mobius-Jacobi}, $W_\theta^\n$ is a conformal Jacobi field with conformal parameter $a = \C \theta$, and hence by Propositions \ref{prop:F-properties} and \ref{prop:FC=F0C}:
\[
Q(W_\theta) = Q^0(W_\theta^\n) = -(n-1) a^T \F a = -(n-1) \theta^T \C^T \F \C \theta = -(n-1) \theta^T \C^T \F_0 \C \theta .
\]
Recalling that $\F_0 = n L_{\scalar{p,N}^2}$, the assertion follows. 
\end{proof}
\begin{remark}
We don't know whether $Q(W_\theta)$ is negative semi-definite on a general spherical Voronoi cluster, since the general formula we have (applying Lemma \ref{lem:LJac-dilation}) is:
\[
Q(W_\theta) = Q^0(W_\theta^\n) = -(n-1) \scalar{L_{Jac} \scalar{\theta,\n} , \scalar{\theta,\n}}_{L^2(\Sigma^1)} = -(n-1) \sum_{i<j} \int_{\Sigma_{ij}} \scalar{\theta,\c_{ij}} \scalar{\n_{ij}, \theta} d\HH^{n-1} ,
\]
which does not have a clear sign. This is in contrast to the Gaussian setting, where $Q(T_\theta) \leq 0$ for translation fields $T_\theta$ \cite[Theorem 4.10]{EMilmanNeeman-GaussianMultiBubble}. 
\end{remark}

Another immediate corollary is:
\begin{corollary} \label{cor:F=F0-full-dim}
Let $\Omega$ be a full-dimensional perpendicular spherical Voronoi $q$-cluster with $q \leq n+2$ (in particular, this holds if $\Omega$ is a perpendicular standard bubble when $q \leq n+1$). Then $\F = \F_0$, and in particular, $\F > 0$ if $V(\Omega) \in \interior \Delta^{(q-1)}$. 
\end{corollary}
\begin{proof}
$\Omega$ has a conformally complete system of conformal Jacobi fields by Corollary \ref{cor:full-dim-F}. As $\C$ is of full affine-rank and $q-1 \leq n+1$, $\F \C = \F_0 \C$ immediately implies $\F = \F_0$. 
\end{proof}

Let us now extend this to general $3$-Plateau perpendicular PCF clusters. 

\subsection{On perpendicular pseudo conformally flat clusters}

Recall that we still don't know that $\F = L_{\Psi_\xi} > 0$ for a perpendicular PCF cluster which is not full-dimensional, since $\Psi_\xi$ will change sign on $\S^n$ if $\abs{\xi} > 1$. We are now finally ready to establish this in the following:

\begin{proposition} \label{prop:PCF-F=F0}
Let $\Omega$ be a $3$-Plateau perpendicular PCF $q$-cluster. Then $\F = \F_0$, and in particular, $\F > 0$ if $V(\Omega) \in \interior \Delta^{(q-1)}$. 
\end{proposition}
\begin{proof}
Let $\xi \in \R^{n+1}$ denote compatibility parameter of $\Omega$, so that  $\scalar{\c_{ij} , \xi } = -\k_{ij}$ for all $i<j$. Since $\Omega$ is perpendicular we have $\scalar{\c_{ij} , N} = 0$, and so we may assume that $\xi \perp N$. 
Denote $\Omega_t = \Phi_t(\Omega)$, the $3$-Plateau spherical Voronoi cluster (by Lemma \ref{lem:Mobius-preserves-Voronoi}) obtained after perturbing $\Omega$ using the conformal automorphism $\Phi_t$, and let $\Sigma^t$ denote its corresponding boundary. Recall from Lemma \ref{lem:kc-evolution} and Corollary \ref{cor:Omega_t-PCF} that the corresponding quasi-centers and curvatures are given by $\c_{ij}(t) = \c_{ij} - \k_{ij} \sinh(t) N$ and $\k_{ij}(t) = \k_{ij} \cosh(t)$, and that $\scalar{\c_{ij}(t) , \xi_t} = -\k_{ij}(t)$ for $\xi_t = \coth(t) N$, so that $\Omega_t$ is a PCF cluster. 
It follows that:
\[
\scalar{\c_{ij}(t) , N - \sinh(t) \xi} = 0 . 
\]

Now, recall from Proposition \ref{prop:PCF-F} that $\Omega_t$ has a conformally complete system of conformal Jacobi fields; moreover, $a_{ij}(1-\scalar{p_t,z_t})$ is a conformal Jacobi field on $p_t \in \Sigma^t$ with conformal parameter $a \in E^{(q-1)}$ for any $z_t \in \R^{n+1}$ so that $\scalar{\c_{ij}(t),z_t} = -\k_{ij}(t)$ for all $i<j$, and so adding a multiple of $N - \sinh(t) \xi$ to $z_t$ does not alter this property. Since $z_t = \xi_t$ satisfies this property, we deduce that all of the following fields are all conformal Jacobi fields on $\Sigma^t$ corresponding to the same conformal parameter $a$:
\[
 a_{ij} (1 - \scalar{p_t, \xi_t - \beta ( N - \sinh(t) \xi )} ) \;\;\; \forall \beta \in \R . 
\]
Setting $\beta = \coth(t)$, we obtain:
\[
a_{ij} (1 - \cosh(t) \scalar{ p_t,  \xi }) . 
\]
The point is that contrary to the original conformal Jacobi fields $a_{ij} (1 - \scalar{p_t,\xi_t})$, these conformal Jacobi fields do not blow up as $t \rightarrow 0$. 

Let us now inspect the conformal-to-volume operator $\F_t$ of $\Omega_t$, which is canonical by Proposition \ref{prop:F-properties} and so does not depend on the particular conformal Jacobi fields one uses for its computation. Recall by Proposition \ref{prop:F0} that $F_t = n L_{\scalar{p,N}^2} + o(1)$, and so inspecting the off-diagonal entry $-e_i^T \F_t e_j$ for $i<j$, we obtain:
\[
\int_{\Sigma^t_{ij}} (1 -  \cosh(t) \scalar{p_t, \xi }) d\HH^{n-1}(p_t) = n \int_{\Sigma_{ij}} \scalar{p,N}^2 d\HH^{n-1}(p) + o(1) . 
\]
Taking the limit as $t \rightarrow 0$, the boundedness of the integrand on the left-hand-side implies:
\[
\int_{\Sigma_{ij}} (1 - \scalar{p, \xi}) d\HH^{n-1}(p) = n \int_{\Sigma_{ij}} \scalar{p,N}^2 d\HH^{n-1}(p) . 
\]
As the left-hand-side is precisely $-e_i^T \F e_j$ by Proposition \ref{prop:PCF-F}, we confirm that $\F = \F_0$. 
\end{proof}

\subsection{Establishing the PDI on PCF clusters}

We can finally conclude:
\begin{theorem} \label{thm:PDE-for-PCF}
Let $\Omega$ be a minimizing $q$-cluster with $q \leq n+1$ and $V(\Omega) = v_0 \in \interior \Delta^{(q-1)}$. Assume that $\Omega$ is pseudo-conformally flat (in particular, this holds if $\Omega$ is full-dimensional). Then there exists a quadratic form $\F$ on $E^{(q-1)}$ so that the PDI:
\[
\F > 0 ~,~ \F^T \nabla^2 \I \F \leq -(n-1) \F ~,~ \tr \brac{\F \; (\Id+ \frac{2}{(n-1)^2} \nabla \I \otimes \nabla \I)} \leq 2 \I 
\]
holds at $v_0$ in the viscosity sense of Definition \ref{def:viscosity}. 
\end{theorem}
\begin{proof}
By the results of Section \ref{sec:prelim} and Theorem \ref{thm:intro-structure}, a minimizing cluster is necessarily regular, stationary with Lagrange multiplier $\lambda \in E^{(q-1)}$ and stable. Moreover, it is perpendicular spherical Voronoi with equatorial cells and $3$-Plateau. By Theorem \ref{thm:conformally-complete}, there exists a conformally complete system of conformal Jacobi fields for $\Omega$, denoted $\{f^a\}_{a \in E^{(q-1)}}$. Let $\F : E^{(q-1)} \rightarrow E^{(q-1)}$ denote the associated conformal-to-volume operator, so that $\delta^1_{f^a} V(\Omega) = \F a$ with $Q^0(f^a) = -(n-1) a^T \F a$ by Proposition \ref{prop:F-properties}. By Theorem \ref{thm:approximation}, this means that for every $\eps > 0$ there exists a smooth vector-field $X_{a,\eps}$ on $\S^n$ so that:
\[
\delta^1_{X_{a,\eps}} V = \F a ~,~   Q(X_{a,\eps}) \leq -(n-1) a^T \F a + \eps .
\]

Now, using that $\Omega$ is in addition PCF, it follows by Proposition \ref{prop:PCF-F=F0} that $\F = \F_0 > 0$; in addition, the trace-identity (\ref{eq:trace-id}) holds by either Proposition \ref{prop:PCF-trace} or \ref{prop:trace-id-F0}. Recalling that $\lambda = (n-1) \k$, it follows that:
\[
\tr\brac{\F \brac{\Id/2 + \frac{1}{(n-1)^2} \lambda \otimes \lambda}} = \HH^{n-1}(\Sigma) = \I(v_0) . 
\]
This concludes the verification of the PDI in the viscosity sense of Definition \ref{def:viscosity}.
\end{proof}

\begin{remark} \label{rem:conjectures}
Note that in order to extend the above proof to arbitrary minimizing clusters, and thereby verify the Multi-Bubble Isoperimetric Conjecture on $\S^n$ for all $q \leq n+1$, what we are missing is precisely the confirmation of Conjecture \ref{conj:trace-id}, or better yet, the stronger Conjecture \ref{conj:F=F0}. A compelling intermediate conjecture will be put forth in Section \ref{sec:conclude}.
\end{remark}

\section{Linearized structure equations} \label{sec:LSE}

In order to be able to handle spherical Voronoi clusters which are not PCF, we now consider a more general condition. Assume that $\{\c_i(t)\}_{i=1,\ldots,q}$ and $\k(t) = \{\k_i(t)\}_{i =1,\ldots,q}$ are smoothly varying quasi-center and curvature parameters (respectively) of a one-parameter family of non-degenerate spherical Voronoi clusters $\Omega(t)$ for $t \in [0,\eps)$ and some $\eps > 0$. When $t=0$ we simply write $\Omega = \Omega(0)$, $\{\c_i\} = \{\c_i(0)\}$ and $\k = \k(0)$. 
The spherical Voronoi condition means that:
\[
|\c_{ij}(t)|^2 = 1 + \k_{ij}(t)^2 \;\;\; \forall i \neq j \text{ so that } \Sigma_{ij}(t) \neq \emptyset. 
\]
Since $\Sigma_{ij} \neq \emptyset$ implies $\Sigma_{ij}(t) \neq \emptyset$ for all $t \in [0,\eps)$ (by continuity, as $\Sigma_{ij}$ is relatively open in $S_{ij}$ thanks to non-degeneracy), it follows that:
\[
|\c_{ij}(t)|^2 = 1 + \k_{ij}(t)^2 \;\;\; \forall i \neq j \text{ so that } \Sigma_{ij} \neq \emptyset. 
\]
Differentiating at time $t=0$, and denoting $\delta \c_i = \frac{d}{dt}|_{t=0} \c_i$ and $\delta \k  = \frac{d}{dt}|_{t=0} \k$, we obtain:
\begin{equation} \label{eq:structure}
\scalar{\c_{ij} , \delta \c_{ij}} = \k_{ij} \delta \k_{ij} \;\;\; \forall i \neq j \text{ so that } \Sigma_{ij} \neq \emptyset ,
\end{equation}
which may be seen as a linearized ``structure equation" system for spherical Voronoi clusters. 
This serves as motivation for the following:

\begin{definition}[Solvability of linearized structure equations (LSE)]
Given a spherical Voronoi cluster $\Omega$ with quasi-centers $\{ \c_i \}_{i=1,\ldots,q} \subset \R^{n+1}$ and curvature vector $\k \in E^{(q-1)}$, respectively, we shall say that its linearized structure equation (LSE) system is solvable with $\delta \k = a \in E^{(q-1)}$, if there exists $\{ \delta \c_i \}_{i=1,\ldots,q}$ so that (\ref{eq:structure}) holds with $\delta \k = a$. \\
We shall say that a complete system of solutions to the LSE exists if the equation system (\ref{eq:structure}) is solvable with $\delta \k = a$ for any $a \in E^{(q-1)}$. 
\end{definition}

There are two scenarios when we can assert the solvability of the linearized structure equations:
\begin{lemma} \label{lem:PCF-LSE}
Let $\Omega$ be a PCF cluster. Then a complete system of solutions to the LSE exists.
\end{lemma}
\begin{proof}
By the PCF property, there exists $\xi \in \R^{n+1}$ so that $\scalar{\c_i , \xi} + \k_i = 0$ for all $i=1,\ldots,q$. Therefore, given $a \in E^{(q-1)}$, simply set $\delta \c_i = -a_i \xi$ which solves (\ref{eq:structure}) with $\delta \k = a$ (in fact, for all $i \neq j$). 
\end{proof}

\begin{lemma}
For a general spherical Voronoi cluster $\Omega$, the LSE is solvable with $\delta \k = a$ for all $a \in \Im \C$. 
\end{lemma}
\begin{proof}
If $a = \C \theta$ then $a_{ij} = \scalar{\c_{ij},\theta}$. Consequently $\delta \c_i := \k_i \theta$ is a solution to (\ref{eq:structure}) with $\delta \k = a$ (in fact, for all $i \neq j$).
\end{proof}

Note that both scenarios above yield a rank-one solution to the LSE, so the general case may be thought of as a higher-rank extension of both scenarios.  Also note that both scenarios imply in particular that a complete system of solutions exists whenever the cluster is full-dimensional. 

\begin{remark}
One may attempt to solve (\ref{eq:structure}) by constructing a smooth one-parameter perturbation $\Omega(t)$ of $\Omega$ with $\frac{d}{dt}|_{t=0} \k(t) = \delta \k$ for a given $\delta \k \in E^{(q-1)}$. However, the smoothness at $t=0$ turns out to be a genuine issue due to the need to take a square-root of the Gram matrix $\C(t) \C(t)^T$ (see Section \ref{sec:Gram}), and it seems that we can only ensure the smoothness in the above two scenarios. 
\end{remark}

\subsection{Constructing conformal Jacobi fields}

As noted above, the existence of a complete system of solutions to the LSE should be thought of as a ``higher-rank" PCF condition, and it is therefore not surprising that we are able to extend all of our results from the previous sections regarding PCF clusters to the higher-rank case. In particular, we are able to show that if a complete system of solutions to the LSE exists for a given $3$-Plateau spherical Voronoi cluster, then a conformally complete system of conformal Jacobi fields exists as well, and that moreover $\F = \F_0$; in particular, $\F > 0$ if $V(\Omega) \in \interior \Delta^{(q-1)}$ and the trace identity (\ref{eq:trace-id}) holds by Proposition \ref{prop:trace-id-F0}. We therefore pose the following:
\begin{question}
Is it true that for any $3$-Plateau, stable, perpendicular spherical Voronoi $q$-cluster with equatorial cells and $q \leq n+1$, a complete system of solutions to the LSE (\ref{eq:structure}) always exists? 
\end{question}
In view of the previous discussion, a positive answer would confirm Conjecture \ref{conj:F=F0} and thereby extend the verification of the multi-bubble conjectures to all $q \leq n+1$. However, we do not have a clear sense of how reasonable it would be to conjecture that the answer is positive. There are actually fascinating connections between this question and the joint-and-bar framework and infinitesimal rigidity of polyhedra considered by Cauchy, Alexandrov, Dehn and Whiteley (see e.g. \cite{Whiteley-InfinitesimalRigidFrameworks,Connelly-Rigidity} and the references therein), and while we encourage the reader to explore this direction, we shall not expand on this here. 

\medskip

Consequently, we will not verify here that $\F = \F_0$, as this is obtained by repeating the argument of Proposition \ref{prop:PCF-F=F0} and will not be required to establish the main results of this work. But to establish the main result of the next section, we do require the following:

\begin{proposition} \label{prop:LSE}
Let $\Omega$ be a $3$-Plateau spherical Voronoi cluster, and assume that $\{\delta \c_i\}_{i=1,\ldots,q}$ is a solution to the LSE with $\delta \k = a \in E^{(q-1)}$. Then the scalar-field:
\[
f_{ij} := a_{ij} + \scalar{\delta \c_{ij},p} \text{ on $\Sigma_{ij}$ }
\]
is a conformal Jacobi field with conformal parameter $a$, namely $f = (f_{ij}) \in \D_{con}$ and $L_{Jac} f_{ij} = (n-1) a_{ij}$. 
\end{proposition}

For the proof, we first observe:
\begin{lemma} \label{lem:cyclic}
For $a,b \in \R^3$ and $(i,j) \in \cyclic(u,v,w)$, the difference
\[
a_{ij} (b_{ik} + b_{jk}) - (a_{ik} b_{ik} - a_{jk} b_{jk}) 
\]
only depends on $a,b$ and does not depend on $(i,j)$ as long as they are in cyclic order. 
\end{lemma}
\begin{proof}
Denoting by $f_{ijk}$ the left-hand term and by $g_{ijk}$ the right-hand one, it is enough to verify that $f_{ijk} - g_{ijk} = f_{jki} - g_{jki}$. This follows by explicit computation, utilizing that $b_{ij} + b_{jk} + b_{ki} = 0$ and $c_{ij} = -c_{ji}$, $c \in \{a,b\}$. 
\end{proof}
\begin{corollary} \label{cor:cyclic}
Assume that $\{\delta \c_i\}_{i=1,\ldots,q}$ is a solution to the LSE with $\delta \k = a \in E^{(q-1)}$. Then for any $\Sigma_{uvw} \neq \emptyset$, the difference:
\[
\scalar{\c_{ik} + \c_{jk}, \delta \c_{ij}} - (\k_{ik} + \k_{jk}) a_{ij} 
\]
is independent of $(i,j) \in \cyclic(u,v,w)$. 
\end{corollary}
\begin{proof}
If $\Sigma_{uvw} \neq \emptyset$ then necessarily $\Sigma_{uv}, \Sigma_{vw}, \Sigma_{wu} \neq \emptyset$. 
Applying Lemma \ref{lem:cyclic} twice (coordinate-wise) and the LSE (\ref{eq:structure}), we have:
\[
\scalar{\c_{ik} + \c_{jk}, \delta \c_{ij}} = M + \scalar{\c_{ik} , \delta \c_{ik}} -  \scalar{\c_{jk} , \delta \c_{jk}} = M + \k_{ij} a_{ik} - \k_{jk} a_{jk} = M' + (\k_{ik} + \k_{jk}) a_{ik} ,
\]
for some $M,M' \in \R$ independent of $(i,j) \in \cyclic(u,v,w)$. 
\end{proof}

\begin{proof}[Proof of Proposition \ref{prop:LSE}]
By Lemma \ref{lem:LJac-dilation} we have:
\[
L_{Jac} f_{ij} = (n-1) (1 + \k_{ij}^2) a_{ij} - (n-1) \k_{ij} \scalar{\delta \c_{ij} , \c_{ij}} = 
(n-1) (1 + \k_{ij}^2) a_{ij} - (n-1) \k_{ij}^2 a_{ij} = (n-1) a_{ij} . 
\]
To verify the conformal BCs (defined in the setting of Section \ref{sec:spectral}), we must show by Remark \ref{rem:conformal-BCs} that for all $\Sigma_{uvw} \neq \emptyset$, $\nabla_{\n_{\partial ij}} f_{ij} - \bar \II^{\partial ij} f_{ij}$ is independent of $(i,j) \in \cyclic(u,v,w)$. 
And indeed, recalling that $\c = \n - \k p$, (\ref{eq:sqrt3}) and (\ref{eq:def-II-partial}):
\[
\sqrt{3} (\nabla_{\n_{\partial ij}} f_{ij} - \II^{\partial ij} f_{ij}) = \scalar{ \n_{ik} + \n_{jk} - (\k_{ik} + \k_{jk}) p  , \delta \c_{ij}} - (\k_{ik} + \k_{jk}) a_{ij} = \scalar{\c_{ik} + \c_{jk} , \delta \c_{ij}} - (\k_{ik} + \k_{jk}) a_{ij} ,
\]
which is independent of $(i,j) \in \cyclic(u,v,w)$ by Corollary \ref{cor:cyclic}.
\end{proof}

\section{Gram perturbation of Plateau clusters} \label{sec:Gram}

So far, we've developed tools to handle spherical Voronoi clusters which are full-dimensional, or more generally, are pseudo conformally flat, or more generally, have a complete system of solutions to their linearized structure equations. Unfortunately, this is still not enough to establish Theorem \ref{thm:intro-quintuple} regarding quintuple bubbles. We require one last tool in the form of:

\begin{proposition} \label{prop:Gram-perturbation}
Let $\Omega$ be a spherical Voronoi $q$-cluster with $q \leq n+2$. Assume that all cells are non-empty and that $\Omega$ is perpendicular and \emph{Plateau}. Let $\k \in E^{(q-1)}$ denote its curvature vector and let $\C : \R^{n+1} \rightarrow E^{(q-1)}$ denote it quasi-center operator. Then for any neighborhood $U_\C$ of $\C$, there exists a \textbf{full-dimensional} Plateau spherical Voronoi cluster $\Omega'$ with quasi-center operator $\C' \in U_\C$ and curvature vector $\k' = \k$, so that $V(\Omega') = V(\Omega)$ and $\per(\Omega') = \per(\Omega)$. 
\end{proposition}

For the proof, we are going to construct $\Omega'$ from an appropriate perturbation of $\Omega$'s positive semi-definite Gram operator $\G := \C \C^T$ as follows:
\begin{lemma} \label{lem:Gram}
Let $\Omega$ be a spherical Voronoi $q$-cluster with $q \leq n+2$, quasi-center operator $\C$ and curvature vector $\k$. Denote $\G := \C \C^T : E^{(q-1)} \rightarrow E^{(q-1)}$ and set:
\[
\G_t := (1-t) \G + t (\Id/2 + \k \k^T) : E^{(q-1)} \rightarrow E^{(q-1)} . 
\]
Then:
\begin{enumerate}[(i)]
\item \label{it:Gram1} 
$\G_t$ is strictly positive-definite for all $t \in (0,1]$. 
\item \label{it:Gram2} 
For all $t \in [0,1]$, and $i \neq j$ so that $\Sigma_{ij} \neq \emptyset$, $\scalar{\G_t , e_{ij} \otimes e_{ij}} = \scalar{ \Id /2 + \k \k^T , e_{ij} \otimes e_{ij}}$. 
\item \label{it:Gram3}
There exists a map  $[0,1] \ni t \mapsto \C_t : \R^{n+1} \rightarrow E^{(q-1)}$, which is continuous on $[0,1]$ and smooth on $(0,1]$, so that $\C_0 = \C$ and $\C_t \C_t^T = \G_t$. 
\end{enumerate}
\end{lemma}
\begin{proof}
The first assertion is trivial and the second follows from Lemma \ref{lem:CCT}. To show the third assertion, recall that a positive semi-definite $A : \R^m \rightarrow \R^m$ has a unique positive semi-definite square root, denoted by $\sqrt{A} : \R^m \rightarrow \R^m$. It is known (e.g. \cite[Theorem 6.2.37, Theorem 6.6.30]{HornJohnson-TopicsInMatrixAnalysis})
 that the square root operator is continuous on the closed convex cone of positive semi-definite matrices of fixed dimension and smooth on the open convex cone of positive definite ones. 
   In particular, $t \mapsto \sqrt{\G_t}$ is continuous on $[0, 1]$ and smooth on $(0,1]$. 
Now, $\G = \C \C^T : E^{(q-1)} \rightarrow E^{(q-1)}$, and since $q-1 \leq n+1$, there exists $U : \R^{n+1} \rightarrow E^{(q-1)}$ with orthonormal rows such that
$\C = \sqrt{\G} U$ (see e.g. \cite[Theorem 7.3.11]{HornJohnson-MatrixAnalysis}).
Therefore, setting $\C_t := \sqrt{\G_t} U$, the third assertion is established. 
\end{proof}

\begin{lemma} \label{lem:good-time}
Let $\Omega$ be a spherical Voronoi $q$-cluster with $q \leq n+2$.
Assume that all cells are non-empty and that $\Omega$ is perpendicular and Plateau.
 Let $[0,1] \ni t \mapsto \C_t$ and $\k$ be as in the previous lemma, and let $\Omega(t)$ be the affine Voronoi cluster generated by the rows of $\C_t$ and $\k$. Then there exists $t_e \in (0,1]$ so that for all $t \in (0,t_e]$, $\Omega(t)$ is a full-dimensional Plateau spherical Voronoi cluster with the same non-empty interfaces as $\Omega$ (i.e. $\Sigma_{ij}(t) \neq \emptyset$ iff $\Sigma_{ij} \neq \emptyset$). 
\end{lemma}
\begin{proof}
Since all cells are non-empty and $\Omega$ is assumed to be perpendicular and Plateau, the continuity of the map $t \mapsto \C_t$ and Corollary \ref{cor:Plateau-interfaces} imply the existence of $t_e \in (0,1]$ so that $\Sigma_{ij}(t) \neq \emptyset$ iff $\Sigma_{ij} \neq \emptyset$ for all $t \in (0,t_e]$. Consequently, Lemma \ref{lem:Gram} \ref{it:Gram2} and Lemma \ref{lem:CCT} imply that $\Omega(t)$ is actually a spherical Voronoi cluster for all $t \in (0,t_e]$. Proposition \ref{prop:Plateau-perturbation} implies that by choosing $t_e \in (0,1]$ small-enough, $\Omega(t)$ will remain Plateau for all $t \in [0,t_e]$. Finally, since for all $t \in (0,1]$, $\G_t = \C_t \C_t^T$ is positive-definite and hence full-rank by Lemma \ref{lem:Gram} \ref{it:Gram1}, $\C_t$ is also full-rank and hence $\Omega(t)$ is full-dimensional for all $t \in (0,t_e]$. 
\end{proof}

The continuity of the map $t \mapsto \C_t$ implies that $\C_t \in U_{\C}$ for any neighborhood $U_{\C}$ of $\C$ and all sufficiently small $t >0$ (depending on $U_{\C}$). Consequently, to conclude the proof of Proposition \ref{prop:Gram-perturbation}, it remains to show that $V(\Omega(t)) = V(\Omega)$ and $\per(\Omega(t)) = \per(\Omega)$ for all $t \in [0,t_e]$. Since $V(\Omega(t))$ and $\per(\Omega(t))$ are Lipschitz continuous as a function of $\C_t$ by Lemma \ref{lem:cont}, we conclude in conjunction with Lemma \ref{lem:Gram} \ref{it:Gram3} that these functions are continuous on $[0,t_e]$ and locally Lipschitz continuous on $(0,t_e]$. Therefore, it is enough to show that:

\begin{lemma} \label{lem:vol-constant}
With the same assumptions and notation as in the previous lemma,
\[
\frac{d}{dt} V(\Omega(t)) = 0 ~,~ \frac{d}{dt} \per(\Omega(t)) = 0 \;\;\; \forall t \in (0,t_e) . 
\]
\end{lemma}

\noindent
We will show this in the next two subsections - first conditionally and then in full. 

\subsection{Volumes and total perimeter remain fixed -- conditional verification}

It may be initially quite surprising that the volume and perimeter remain constant throughout the Gram perturbation of the cluster, 
since all cells will change their global configuration, much in contrast to the procedure we employed in \cite{EMilmanNeeman-TripleAndQuadruple} where local isometries were used to rotate a single chosen bubble. Indeed, in the latter method the affine-rank of $\C_t$ will increase by at most $1$, whereas Gram perturbation instantaneously makes $\C_t$ of full affine-rank (hence it is easier to spot the difference for higher values of $q$) -- see Figure \ref{fig:Gram-perturbation} from the Introduction.  It is actually not too hard to see why volume and perimeter are preserved, at least conditioned on a certain seemingly technical assumption. 
\medskip

Assume that for $t \in (0,t_e]$, $\Omega(t) = F_t(\Omega(t_e))$ for some flow $F_t : \S^n \rightarrow \S^n$ driven by a time-dependent smooth vector-field $X(t)$ on $(0,t_e] \times \S^n$: 
\begin{equation} \label{eq:t-flow}
\frac{d}{dt} F_t = X(t) \circ F_t ~,~ F_{t_e} = \Id ~,~ t \in (0,t_e] . 
\end{equation}
Note that we start the flow at time $t=t_e$ and not $t=0$, since the deformation $t \mapsto \Omega(t)$ is not guaranteed to be smooth at $t=0$ in Lemma \ref{lem:Gram}, but this will not play any role below. 
 Given $p \in \Sigma_{ij}(t_0)$, denote $p(t) = F_t(p) \in \Sigma_{ij}(t)$ and observe that:
\[
\scalar{p(t) , \c_{ij}(t)} + \k_{ij}(t) = 0 . 
\]
For the sake of generality, we allow for the moment the curvatures $\k(t)$ to change with time.
Now differentiate this at a fixed time $t$, denoting as usual $\delta \c_i(t) = \frac{d}{dt} \c_i(t)$ and $\delta \k(t)  = \frac{d}{dt} \k(t)$. Since $\scalar{X(t)^{\tang} , \c_{ij}(t)} = 0$ we obtain (removing for brevity the time dependence):
\[
X^{\n_{ij}} + \scalar{p , \delta \c_{ij}} = \scalar{ X , \c_{ij}} + \scalar{p , \delta \c_{ij}} = - \delta k_{ij} ,
\]
so that:
\[
X^{\n_{ij}} = - \delta k_{ij} - \scalar{p , \delta \c_{ij}} . 
\]
This should not be surprising in view of Proposition \ref{prop:LSE}.
Back to our setting, $\k(t)$ remains constant and hence $\delta \k(t) = 0$. We would thus conclude that, assuming such a vector field $X(t)$ exists, it must satisfy for all $t \in (0,t_e]$:
\begin{equation} \label{eq:explicit-magic}
X(t)^{\n_{ij}(t)} =  - \scalar{p(t) , \delta \c_{ij}(t)} \text{ on $\Sigma_{ij}(t)$.}
\end{equation}

We now observe:
\begin{lemma} \label{lem:zero-field}
With the same assumptions and notation as in Lemma \ref{lem:good-time}, denote $f(t) = (f_{ij}(t))$ the scalar field on $\Sigma(t)$, $t \in (0,t_e]$, given by:
\begin{equation} \label{eq:explicit-Jacobi}
f_{ij}(t) :=  - \scalar{p , \frac{d}{dt} \c_{ij}(t)} \text{ on $\Sigma_{ij}(t)$.}
\end{equation}
Then for all $t \in (0,t_e]$, $f(t)$ is a Jacobi field (i.e. $f(t) = (f_{ij}(t)) \in \D_{con}(\Sigma(t))$ and $L_{Jac,\Sigma(t)} f(t) \equiv 0$), and $\delta^1_{f(t)} V(\Omega(t)) = 0$. 
\end{lemma}
\begin{proof}
Since $\Omega(t)$ remains a Plateau spherical Voronoi cluster with the same non-empty interfaces as $\Omega$ for all $t \in [0,t_e]$, we know that $\{\delta \c_{i}(t)\}$ and $\delta \k(t) = 0$ satisfy the linearized structure equations (\ref{eq:structure}). It follows by Proposition \ref{prop:LSE} that $f(t)$ is a Jacobi field. Since for $t \in (0,t_e]$ the Plateau spherical Voronoi clusters $\Omega(t)$ are full-dimensional and hence \emph{volumetrically non-degenerate} by Proposition \ref{prop:Fredholm} and Corollary \ref{cor:full-dim-F}, it follows that $\delta^1_{f(t)} V(\Omega(t)) = 0$ for all $t \in (0,t_e]$.
\end{proof}

Comparing the definition of $f_{ij}(t)$ from (\ref{eq:explicit-Jacobi}) with the requirement from $X(t)^{\n_{ij}(t)}$ in (\ref{eq:explicit-magic}), we can now conditionally establish Lemma \ref{lem:vol-constant} as follows:
\begin{lemma}
With the same assumptions and notation as in Lemma \ref{lem:good-time}, assume there exists a smooth vector field $X(t)$ on $(0,t_e] \times \S^n$ 
so that $\Omega(t) = F_t(\Omega(t_e))$ with $F_t$ given by (\ref{eq:t-flow}). 
Then the conclusion of Lemma \ref{lem:vol-constant} holds. 
\end{lemma}
\begin{proof}
 If such a vector field $X(t)$ exists, we have seen in (\ref{eq:explicit-magic}) that necessarily $X(t)^{\n_{ij}(t)} = f_{ij}(t)$ on  $\Sigma_{ij}(t)$. 

Given $t_1 \in (0,t_e)$, set $X := X(t_1)$ and note that the formulas for the first variations of volume and perimeter from Lemma \ref{lem:Lagrange} (derived for a stationary flow along $X$) clearly apply in the non-stationary setting as well. 
Invoking Lemma \ref{lem:zero-field}, we therefore deduce:
\[
\frac{d}{dt}|_{t=t_1} V(\Omega(t)) = \delta^1_{X(t_1)} V(\Omega(t_1)) = \delta^1_{f(t_1)} V(\Omega(t_1)) = 0.
\]
Finally, since $\Omega(t)$ remains stationary for all $t \in (0,t_e]$ (with the same Lagrange multiplier $\lambda = (n-1) \k$), we have  by Lemma \ref{lem:Lagrange}:
\[
 \frac{d}{dt}|_{t=t_1} \per(\Omega(t)) = \delta^1_{X(t_1)} A(\Omega(t_1)) = \sscalar{\lambda , \delta^1_{X(t_1)} V(\Omega(t_1))} = 0 .
\]
\end{proof}

\subsection{Gram perturbation is a flow}

It remains to establish:
\begin{proposition} \label{prop:Plateau-field}
With the same assumptions and notation as in Lemma \ref{lem:good-time}, perhaps after further reducing $t_e \in (0,1]$, 
there exists a smooth vector field $X(t)$ on $(0,t_e] \times \S^n$ 
so that $\Omega(t) = F_t(\Omega(t_e))$ with $F_t$ given by (\ref{eq:t-flow}). 
\end{proposition}

While this may seem like a mere technical issue, we will actually need to employ our assumption that $\Omega$ is Plateau once again.

\begin{proof}[Proof of Proposition \ref{prop:Plateau-field}]
 By Lemma \ref{lem:perp} and Corollary \ref{cor:isolated}, as $\Omega$ is perpendicular and without empty cells, then by decreasing $t_e \in (0,1]$ if necessary, we may assume that $\overline{\Omega_i}(t) = \underline \Omega_i(t)$ for all $t \in [0,t_e]$ and $i=1,\ldots,q$. Consequently, for all $p \in \S^n$ and $t \in [0,t_e]$, we have:
 \begin{align*}
  I_p(t) & = \{ i \in \{1,\ldots,q\} \; ; \; p \in \overline{\Omega_i}(t) \} = \{  i \in \{1,\ldots,q\} \; ; \; p \in \underline \Omega_i(t) \} \\
  & =  \{ i \in \{1,\ldots,q\} \; ; \; \scalar{\c_i(t),p} + \k_i = \min_{j = 1,\ldots,q} \scalar{\c_j(t),p} + \k_j \} .
 \end{align*}
 To construct the desired diffeomorphism $F_t : \S^n \rightarrow \S^n$ so that $F_t(\Omega(t_e)) = \Omega(t)$ for all $t \in (0,t_e]$, our task is to ensure that $I_{F_t(p_1)}(t)$ remains constant for all $p_1 \in \S^n$ and $t \in (0,t_e]$. 
 
Recall that $\Omega$ is in addition assumed Plateau. Consequently, given $p_0 \in \S^n$, we may invoke Lemma \ref{lem:Plateau-field}, which ensures the existence of a neighborhood $U_{p_0}$ of $p_0$ so that for all $p \in U_{p_0}$ and $t \in [0,t_{p_0}]$, $I_p(t) \subset I_{p_0}(0)$ and $\arank(\N_p(t)) = |I_p(t)|-1$. Recalling that $t_e \in (0,1]$ was obtained by invoking Proposition \ref{prop:Plateau-perturbation}, whose proof is based on compactness and Lemma \ref{lem:Plateau-field}, we may assume that $t_{p_0} \geq t_e$ without loss of generality. Now consider the map:
\[
U_{p_0} \times [0,t_e] \ni (p,t) \mapsto R_{p_0}(p,t) := ( \scalar{\bar \c_i(t),p} + \bar \k_i )_{i \in I_{p_0}(0)} \in E^{I_{p_0}(0)} ,
\]
where:
\[
\bar \c_i(t) := \c_i(t) - \frac{1}{|I_{p_0}(0)|} \sum_{j \in I_{p_0}(0)} \c_j(t) ~,~ \bar \k_i := \k_i - \frac{1}{|I_{p_0}(0)|} \sum_{j \in I_{p_0}(0)} \k_j . 
\]
Since $I_{p}(t) \subset I_{p_0}(0)$ on $U_{p_0} \times [0,t_e]$, we have for all $t \in [0,t_e]$:
\[
p \in U_{p_0} \;\; \Rightarrow \;\; I_{p}(t) = \{ i \in I_{p_0}(0)\; ;\; \scalar{\bar \c_i(t),p} + \bar \k_i = \min_{j \in I_{p_0}(0)} \scalar{\bar \c_j(t),p} + \bar \k_j \} . 
\]
Consequently, we first construct a smooth vector field $X(p,t)$ on $U_{p_0} \times (0,t_e]$ so that $R_{p_0}(F_t(p_1),t)$ remains locally constant by equating the time-derivative to zero:
\begin{align*}
\vec 0 = \partial_p R_{p_0}(p,t) X(p,t) + \partial_t R_{p_0}(p,t) & = (\scalar{\bar \c_i(t),X(p,t)} + \sscalar{\frac{d}{dt} \bar \c_i(t) , p})_{i \in I_{p_0}(0)} \\
& =  (\scalar{\bar \n_i(p,t),X(p,t)} + \sscalar{\frac{d}{dt} \bar \c_i(t) , p})_{i \in I_{p_0}(0)} ,
\end{align*}
where $\bar \n_i(p,t)  = \bar \c_i(t) - \scalar{\bar \c_i(t),p} p$. Indeed, since $\{ \bar \n_i(p,t) \}_{i \in I_{p_0}(0)}$ vary smoothly on $U_{p_0} \times (0,t_e]$ and remain affinely independent while summing to zero, and $(\sscalar{\frac{d}{dt} \bar \c_i(t) , p})_{i \in I_{p_0}(0)} \in E^{I_{p_0}(0)}$ varies smoothly on $U_{p_0} \times (0,t_e]$ as well, $X(p,t)$ may be chosen as some smooth section of the corresponding smooth vector subbundle of $T (U_{p_0} \times (0,t_e])$. 

It remains to construct $X(\cdot,t)$ globally on $\S^n$ using a partition of unity subordinate to the open cover by $\{U_{p_0}\}_{p_0 \in \S^n}$. This ensures that if $\frac{d}{dt} F_t = X(t) \circ F_t$ then $R_{p_0}(F_t(p_1),t)$ remains locally constant in $t$ for all $p_0$, and hence $I_{F_t(p_1)}(t)$ remains constant for all $t \in [0,t_e]$ and $p_1 \in \S^n$, concluding the proof. 
\end{proof}

\section{Proof of the conditional multi-bubble isoperimetric inequality} \label{sec:conditional}

We are now finally ready to provide a proof of the conditional Theorem \ref{thm:intro-conditional}, postponing the analysis of the cases of equality to the next section. In other words, we establish:

\begin{theorem} \label{thm:final}
Let $2 \leq q \leq n+1$, and assume that $\I^{(q-2)} = \I_m^{(q-2)}$ on the entire $\Delta^{(q-2)}$ (equivalently, 
that $\I^{(q-1)} = \I_m^{(q-1)}$ on $\partial \Delta^{(q-1)}$). 
Assume that for every $v_0 \in \interior \Delta^{(q-1)}$, there exists a minimizing $q$-cluster $\Omega$ on $\S^n$ with $V(\Omega) = v_0$ so that either:
\begin{itemize}
\item $\Omega$ is pseudo conformally flat; or
\item $\Omega$ is $(q-3)$-Plateau. 
\end{itemize}
Then $\I^{(q-1)} = \I_m^{(q-1)}$ on the entire $\Delta^{(q-1)}$.  
\end{theorem}

To see the relation with the assumption used in Theorem \ref{thm:intro-conditional} from the Introduction, note that if the multi-bubble conjecture for $p$-clusters on $\S^n$ holds then $\I^{(p-1)} = \I_m^{(p-1)}$ on $\interior \Delta^{(p-1)}$; consequently, if the multi-bubble conjecture for $p$-clusters holds for all $2 \leq p \leq q-1$ then $\I^{(q-2)} = \I_m^{(q-2)}$ on the entire $\Delta^{(q-2)}$. As explained in Subsection \ref{subsec:max-principle}, this is the same as $\I^{(q-1)} = \I_m^{(q-1)}$ on $\partial \Delta^{(q-1)}$. 

\begin{proof}[Proof of Theorem \ref{thm:final}]
By Proposition \ref{prop:max-principle}, it is enough to show for every $v_0 \in \interior \Delta^{(q-1)}$  that (\ref{eq:PDI-F}) holds in the viscosity sense of Definition \ref{def:viscosity}. Given $v_0 \in \interior \Delta^{(q-1)}$, let $\Omega$ be a minimizing $q$-cluster on $\S^n$ with $V(\Omega) = v_0$. By Theorem \ref{thm:intro-structure}, $\Omega$ is a regular stable perpendicular spherical Voronoi cluster with equatorial cells, which is $3$-Plateau. 
Our assumption guarantees by Lemma \ref{lem:q-3-Plateau} that either $\Omega$ is pseudo conformally flat (PCF) or is fully Plateau. If it is Plateau, then by Proposition \ref{prop:Gram-perturbation} we can deform $\Omega$ into a full-dimensional (and so in particular, PCF) spherical Voronoi cluster $\Omega'$ without changing its cell volumes or total perimeter, and so $\Omega'$ is still an isoperimetric minimizer with $V(\Omega') = v_0$. In either case, we have found a minimizing PCF spherical Voronoi cluster of volume $v_0$. Theorem \ref{thm:PDE-for-PCF} then guarantees that (\ref{eq:PDI-F}) holds in the viscosity sense, thereby concluding the proof.
\end{proof}

\section{Uniqueness of minimizers} \label{sec:uniqueness}

Having established that $\I^{(q-1)} = \I_m^{(q-1)}$ on $\Delta^{(q-1)}$, we are now ready to establish the uniqueness of minimizers, as follows:

\begin{theorem} \label{thm:final-equality}
Let $2 \leq q \leq n+1$, and assume that $\I^{(q-1)} = \I_m^{(q-1)}$ on $\Delta^{(q-1)}$. 
Given $v_0 \in \interior \Delta^{(q-1)}$, assume that \emph{any} minimizing $q$-cluster $\Omega$ on $\S^n$ with $V(\Omega) = v_0$ satisfies that either:
\begin{itemize}
\item $\Omega$ is pseudo conformally flat; or
\item $\Omega$ is $(q-3)$-Plateau. 
\end{itemize}
Then $\Omega$ is necessarily a standard $(q-1)$-bubble (up to null-sets). 
\end{theorem}

In conjunction with Theorem \ref{thm:final}, this will conclude the proof of Theorem \ref{thm:intro-conditional}. 
Since every minimizing $q$-cluster is $3$-Plateau by Lemma \ref{lem:3-Plateau}, we can apply induction on $q$ as in Subsection \ref{subsec:max-principle} and verify the multi-bubble conjectures for all $q \leq \min(6,n+1)$, thus obtaining the new quintuple-bubble Theorem \ref{thm:intro-quintuple}, as well as a different proof of the double, triple and quadruple-bubble conjectures on $\S^n$ for $n\geq 2$, $n \geq 3$ and $n \geq 4$, respectively.

\medskip

The key step in the proof of Theorem \ref{thm:final-equality} is the following:

\begin{proposition} \label{prop:equality}
Let $2 \leq q \leq n+1$, and let $\Omega$ be a minimizing $q$-cluster on $\S^n$ of volume $v_0 \in \interior \Delta^{(q-1)}$. Let $\F,\F_0 : E^{(q-1)} \rightarrow E^{(q-1)}$ denote the associated conformal-to-volume operator and its relaxed counterpart, respectively. Assume that $\F = \F_0$, and that the spherical isoperimetric profile and its model counterpart coincide $\I = \I_m$ in a neighborhood of $v_0$. Then $\Omega$ must be a standard bubble (up to null-sets). 
\end{proposition}
\begin{proof}
Recall our convention of modifying $\Omega$ by null-sets as in Theorem \ref{thm:Almgren} \ref{it:Almgren-ii}. 
By the results of Section \ref{sec:prelim}, $\Omega$ is necessarily regular, stationary with Lagrange multiplier $\lambda \in E^{(q-1)}$ and stable. Moreover, by Theorem \ref{thm:intro-structure} it is a perpendicular spherical Voronoi $3$-Plateau cluster with equatorial cells, and we have $\lambda = (n-1) \k$. By Theorem \ref{thm:conformally-complete}, there exists a conformally complete system of conformal Jacobi fields $\{f^a\}_{a \in E^{(q-1)}}$ for $\Omega$, so $\F$ is well-defined. 

We first claim that for any smooth vector-field $X$ on $\S^n$:
\begin{equation} \label{eq:minimizer-variation-conc}
  (\delta_X V)^T \nabla^2 I(v_0) \delta_X V \leq Q(X) . 
\end{equation}
Indeed, if $F_t$ is the flow on $\S^n$ generated by $X$, we know that $\I(V(F_t(\Omega))) \leq \per(F_t(\Omega))$, with equality at $t=0$. Differentiating twice at $t=0$ (since we assume that $\I=\I_m$ in a neighborhood of $v_0$ and hence smooth there), we obtain:    
\[
        \scalar{\nabla \I(v_0),\delta_X V} = \delta_X A ,
\]
\[
        (\delta_X V)^T \nabla^2 \I(v_0) \delta_X V  + \scalar{\nabla \I(v_0),\delta_X^2 V} \leq \delta_X^2 A.
\]
On the other hand, recall that $\delta_X A = \scalar{\lambda,\delta_X V}$ by Lemma \ref{lem:Lagrange}. We may select a family of smooth vector fields $X$ so that $\delta_X V$ spans the entire $E^{(q-1)}$, e.g. by using all linear combinations of approximate piecewise-constant fields as in the proof of Lemma \ref{lem:volume-offset}.  
Consequently, it follows that $\nabla \I(v_0) = \lambda$, and recalling the definition of $Q(X)$, (\ref{eq:minimizer-variation-conc}) is established.

We now apply (\ref{eq:minimizer-variation-conc}) to the smooth vector fields constructed in Theorem \ref{thm:approximation} whose normal component approximates the conformal Jacobi field $f^a$; since $\delta^1_{f^a} V = \F a$ and $Q(f^a) = -(n-1) a^T \F a$ by Proposition \ref{prop:F-properties}, we conclude that:
\[
a^T \F \nabla^2 \I(v_0) \F a \leq  -(n-1) a^T \F a \;\;\; \forall a \in E^{(q-1)} . 
\]
By assumption $\F = \F_0$, and $\F_0 > 0$ by Proposition \ref{prop:trace-id-F0}. Consequently, we deduce that in the positive-definite sense:
\begin{equation} \label{eq:nabla2IleqF}
\nabla^2 \I(v_0) < 0 ~,~ (-\nabla^2 \I(v_0))^{-1} \leq \frac{1}{n-1} \F . 
\end{equation}

On the other hand, since $\nabla \I(v_0) = (n-1) \k$, $\I = \I_m$ in a neighborhood of $v_0$, $\I_m$ satisfies the PDE (\ref{eq:PDE}),  the trace-identity (\ref{eq:trace-id-F0}) holds for $\F_0$ by Proposition \ref{prop:trace-id-F0}, and finally $\F_0 = \F$ by assumption, we have:
\begin{align*}
& \tr( (-\nabla^2 \I(v_0))^{-1} (\Id/2 + \k \otimes \k) ) =  
\tr( (-\nabla^2 \I(v_0))^{-1} (\Id/2 + \frac{1}{(n-1)^2} \nabla \I(v_0) \otimes \nabla \I(v_0)) ) \\
& = \tr( (-\nabla^2 \I_m(v_0))^{-1} (\Id/2 + \frac{1}{(n-1)^2} \nabla \I_m(v_0) \otimes \nabla \I_m(v_0)) ) =  \frac{1}{n-1} \I_m(v_0) = \frac{1}{n-1} \I(v_0) \\
& = \frac{1}{n-1} \HH^{n-1}(\Sigma) = \frac{1}{n-1} \tr( \F_0 (\Id/2 + \k \otimes \k)) = \frac{1}{n-1} \tr( \F (\Id/2 + \k \otimes \k)) .
\end{align*}
As $\Id /2 + \k \otimes \k > 0$, together with (\ref{eq:nabla2IleqF}) we deduce that:
\[
(-\nabla^2 \I(v_0))^{-1} = \frac{1}{n-1} \F . 
\]

Now, let $\Omega^m$ denote a standard bubble with $V(\Omega^m) = v_0$, which exists by Lemma \ref{lem:standard-volume}, and let $\Sigma^m$ denote its boundary.  Let $\F_m$ denote its conformal-to-volume operator, which is well-defined by Corollary \ref{cor:full-dim-F}. Since by (\ref{eq:model-hessian}) we also have:
\[
(-\nabla^2 \I_m(v_0))^{-1} = \frac{1}{n-1} \F_m ,
\]
and since $\I = \I_m$ in a neighborhood of $v_0$, we conclude that:
\[
\F = \F_m . 
\]

Since $q \leq n+1$, the standard bubble $\Omega^m$ is itself a perpendicular spherical Voronoi cluster by Corollary \ref{cor:standard-Voronoi} with North pole $N^m$, and hence $\F_m = (\F_m)_0$ by Corollary \ref{cor:F=F0-full-dim}. Recalling Proposition \ref{prop:F0}, it follows that:
\[
n L_{\Sigma,\scalar{p,N}^2} = \F_0 = \F = \F_m = (\F_m)_0 = n L_{\Sigma^m, \scalar{p^m,N^m}^2} . 
\]
In other words (inspecting the off-diagonal elements):
\[
\int_{\Sigma_{ij}} \scalar{p,N}^2 d\HH^{n-1}(p) = \int_{\Sigma^m_{ij}} \scalar{p^m,N^m}^2 d\HH^{n-1}(p^m) \;\;\; \forall i < j . 
\]
But since a standard bubble's interfaces are all non-empty, i.e. $\Sigma^m_{ij} \neq \emptyset$ for all $i < j$, and since they are all perpendicular to $N^m$, it follows that the right-hand side is strictly positive for all $i < j$. In particular, we deduce that the interfaces of the minimizer $\Omega$ are also all non-empty: $\Sigma_{ij} \neq \emptyset$ for all $i < j$. Applying Proposition \ref{prop:standard-char}, we conclude that $\Omega$ must be a standard bubble. 
\end{proof}

\begin{proof}[Proof of Theorem \ref{thm:final-equality}]
Let $\Omega$ be a minimizing $q$-cluster with $V(\Omega) = v_0 \in \interior \Delta^{(q-1)}$. 
As usual, we modify $\Omega$ by null-sets as in Theorem \ref{thm:Almgren} \ref{it:Almgren-ii}. 
By Theorem \ref{thm:intro-structure}, $\Omega$ is a perpendicular spherical Voronoi cluster with equatorial cells. 
Our assumption and Lemma \ref{lem:q-3-Plateau} imply that $\Omega$ is either PCF or fully Plateau. 

We now claim that $\Omega$ must be PCF, in which case $\F = \F_0$ by Proposition \ref{prop:PCF-F=F0}, and hence $\Omega$ must be a standard-bubble by Proposition \ref{prop:equality}. 
Assume in the contrapositive that $\Omega$ were not PCF, and so in particular, it cannot be full-dimensional. Let $\k$ denote its curvature vector, let $\C$ denote its quasi-center operator, and let $\G = \C \C^T : E^{(q-1)} \rightarrow E^{(q-1)}$ denote the corresponding Gram operator, which therefore must be singular. Since $\Omega$ is not PCF, it must be Plateau, and we can use Proposition \ref{prop:Gram-perturbation} to deform it into another minimizer $\Omega'$ with $V(\Omega') = v_0$ so that $\Omega'$ is a full-dimensional (and in particular PCF) spherical Voronoi cluster with $\C' \in U_\C$ for any given neighborhood $U_\C$ of $\C$ and the same curvature vector $\k$. But according to the first part of the proof, it follows that necessarily $\Omega'$ must be a standard-bubble, regardless of how small we choose the neighborhood $U_\C$. This is of course impossible, since on one hand, denoting $\G' = \C' (\C')^T$, $\det(\G')$ can be made arbitrarily close to $\det(\G) = 0$ by continuity of the determinant; but on the other hand, $\G'$ coincides with the Gram matrix $\G^m$ of a standard-bubble of curvature $\k$, and since $\G^m = \Id / 2 + \k \k^T > 0$ by Proposition \ref{prop:standard-char}, $\det(G') > 0$ remains positively uniformly bounded below. 
This concludes the proof. 
\end{proof}

\section{Concluding Remarks} \label{sec:conclude}

\subsection{Locality conjecture}

Let $\Omega$ be spherical Voronoi $q$-cluster having a conformally complete system of conformal Jacobi fields.  
Let $\F : E^{(q-1)} \rightarrow E^{(q-1)}$ denote its associated conformal-to-volume operator, which we know is symmetric. We may therefore uniquely express $\F$ as $\sum_{i<j} \F_{ij}  \, e_{ij} \otimes e_{ij}$ for some coefficients $\F_{ij}$. We conjecture the following:

\begin{conjecture}[Locality Conjecture] \label{conj:locality}
Let $\Omega$ be a $3$-Plateau, stable, perpendicular spherical Voronoi $q$-cluster with equatorial cells and $q \leq n+1$ in $\S^n$. Then $\{\F_{ij}\}$ is supported on the set of non-empty interfaces:
\begin{equation}  \label{eq:locality}
\forall i<j \;\;\; \Sigma_{ij} = \emptyset \;\; \Rightarrow \;\; \F_{ij} = 0 . 
\end{equation}
\end{conjecture}
Note that the existence of $\F$ for $\Omega$ as above is ensured by Theorem \ref{thm:conformally-complete}. In other words, the conjecture states that a conformal Jacobi field $f^i$ with conformal parameter $e_i$ satisfies $\delta^1_{f^i} V(\Omega_j) = 0$ for all cells $j=1,\ldots,q$ which do not share an interface with cell $i$. So in some sense, the first order effect on the volume of cells induced by conformal Jacobi fields is ``local", depending only on the adjacency structure of the interface graph. 

Note that whenever it holds that for some (integrable) function $\Psi$:
\[
 \F = L_{\Psi} = \sum_{i<j} \int_{\Sigma_{ij}} \Psi d\HH^{n-1} \;  e_{ij} \otimes e_{ij},
 \]
 then (\ref{eq:locality}) indeed holds. In particular, the Locality Conjecture holds for any regular PCF cluster by Proposition \ref{prop:PCF-F}, and hence on any full-dimensional cluster and on standard-bubbles. 
 
 \bigskip
 
Interestingly, the Locality Conjecture lies in between Conjectures \ref{conj:trace-id} and \ref{conj:F=F0}, and so it would be enough to verify it in order to extend our main results from the quintuple-bubble case to the general case for all $q \leq n+1$. Indeed, if $\F = \F_0 = n L_{\scalar{p,N}^2}$ as per Conjecture \ref{conj:F=F0} then the Locality Conjecture holds by the previous comments. To see that the Locality Conjecture implies the validity of Conjecture \ref{conj:trace-id}, we  first observe:

\begin{lemma}\label{lem:magic-C}
For any spherical Voronoi cluster, 
\[
\C \C^T =  \Id/2 + \k \otimes \k  +  \SS , 
\] 
for some symmetric operator $\SS : E^{(q-1)} \rightarrow E^{(q-1)}$  which is perpendicular to $\text{span} \{ e_{ij} \otimes e_{ij}\}_{\Sigma_{ij} \neq \emptyset}$.
\end{lemma}
Here of course we emloy the scalar product $\scalar{A,B} := \tr(A^T B)$ on the space of linear operators. 
\begin{proof}
By definition of spherical Voronoi cluster, as $\c_{ij} = \n_{ij} - \k_{ij} p$ for any non-empty interface $\Sigma_{ij}$, we have:
\[
\scalar{ e_{ij} \otimes e_{ij} ,\C \C^T} = |\c_{ij}|^2 = 1+\k_{ij}^2 = \scalar{ e_{ij} \otimes e_{ij} , \Id/2 + \k \otimes \k} \;\;\; \forall \Sigma_{ij} \neq \emptyset . 
\]
Consequently, the difference $\SS := \C \C^T - (\Id/2 + \k \otimes \k)$ satisfies the assertion.  
\end{proof}

We can now verify that the Locality Conjecture implies Conjecture \ref{conj:trace-id}:
\begin{lemma} \label{lem:support-F0}
If $\F$ satisfies (\ref{eq:locality}) then the trace-identity (\ref{eq:trace-id}) holds for $\F$.
\end{lemma}
\begin{proof}
By the previous lemma, if $\F$ satisfies (\ref{eq:locality}) then $\tr(\F \SS) = 0$. Consequently, by Proposition \ref{prop:FC=N}:
\[
\tr( \F (\Id/2 + \k \otimes \k)) = \tr( \F (\C \C^T - \SS)) = \tr (\F \C \C^T) = \tr(\C^T \N) = \HH^{n-1}(\Sigma) ,
\]
thereby verifying the trace-identity (\ref{eq:trace-id}).
\end{proof}

\subsection{Bounded conformal Jacobi fields conjecture}

A conjecture of more apparent PDE flavor is the following:
\begin{conjecture} \label{conj:bounded}
Let $\Omega$ be a $3$-Plateau, stable, perpendicular spherical Voronoi $q$-cluster with equatorial cells and $q \leq n+1$ in $\S^n$. Let $\Omega_t = \Phi_t(\Omega)$ be the PCF cluster resulting from the conformal perturbation $\Phi_t$ from Section \ref{sec:perturbation}. Then there exist $C , t_0 > 0$ so that for all $t \in [0,t_0]$ and $a \in E^{(q-1)}$, there is a conformal Jacobi field $f^a_t$ on $\Sigma_t$ with conformal parameter $a$, so that $\norm{f^a_t} \leq C |a|$ uniformly in $t \in [0,t_0]$. 
\end{conjecture}
\noindent Some useful norms $\norm{\cdot}$ to use in the above conjecture would be $C^0(\Sigma^1_t)$ or even just $L^2(\Sigma^1_t)$. 

\medskip

Recall that the explicit conformal Jacobi fields we employed in Section \ref{sec:perturbation} for calculating $\F_t$ and hence $\F_0 = \lim_{t \rightarrow 0} \F_t$ were blowing up as $t \rightarrow 0$. If true, conjecture \ref{conj:bounded} would allow to use other conformal Jacobi field representatives which do not blow up, and hence would most likely allow us to conclude that $\F = \F_0$, confirming Conjecture \ref{conj:F=F0} and thus the multi-bubble conjecture on $\S^n$ for all $q \leq n+1$. 

\medskip

Another perspective regarding Conjecture \ref{conj:bounded} is as follows. It seems possible to show the $\Gamma$-convergence of the Index-Forms $Q^0_{\Sigma_t}$ to $Q^0_{\Sigma}$ as $t \rightarrow 0$. However, to show that the minimum of $Q^0_{\Sigma_t}$ on $\D_{Q_{\Sigma_t}}[V]$ converges to the minimum of $Q^0_{\Sigma}$ on $\D_{Q_{\Sigma}}[V]$ requires some type of equicoercivity of $Q^0_{\Sigma_t}$ modulo $\ker L_{Jac,\Sigma_t}$, and this where Conjecture \ref{conj:bounded} would come in handy. The problem is that the dimensionality of $\ker L_{Jac,\Sigma_t}$ can jump up at $t=0$ due to the extra symmetries that the cluster $\Omega = \Omega_0$ gains in the limit, and this is in some sense the source of the obstacle for establishing that $\F = \F_0$.

\subsection{Gram perturbation can only increase perimeter conjecture}

In Section \ref{sec:Gram} we've used Gram perturbation of a (perpendicular) spherical Voronoi cluster to deform it into a full-dimensional cluster. Under the assumption that \emph{the initial cluster is Plateau}, we were able to show that no new interfaces would be created in the process, and so as a consequence, that the resulting cluster would remain spherical Voronoi, and that the volumes of the cells and total perimeter would not be altered. 

If the initial cluster is not necessarily Plateau, it is certainly possible that new interfaces would instantaneously be created, the cluster would cease to be stationary (as we would get non $120^{\circ}$-angles at triple points), and in particular the spherical Voronoi property would instantaneously be ruined. The volumes of the cells would not remain fixed, and neither would the total perimeter. However, we conjecture that in that case, it would be possible to decrease the total perimeter in a rate which is strictly faster than (i.e. super linear in) the rate of change in the cells' volumes, and so after compensating for the latter, we would be able to construct a cluster having the same volume and yet strictly smaller total perimeter. But this would be impossible if the initial cluster were an isoperimetric minimizer. In other words:

\begin{conjecture}
Assume that $\Omega$ is a spherical Voronoi $q$-cluster on $\S^n$ with $q \leq n+1$ which is an isoperimetric minimizer. Then  the Gram perturbation $\{\Omega_t\}$ does not create new interfaces for all $t \in [0,t_e]$ and some $t_e > 0$. 
\end{conjecture}

If true, this would allow us to always run the Gram perturbation argument of Section \ref{sec:Gram}, and we would be able to extend our results to all $q \leq n+1$ unconditionally.

\subsection{Upper bounds on $\F$}

In connection to our attempts to establish the trace-identity (\ref{eq:trace-id}), we were able to show the following theorem, which we state without proof: \begin{theorem}
Let $\Omega$ be a minimizing $q$-cluster on $\S^n$ ($q \leq n+1$) with $\S^m$ symmetry ($m \geq 1$). Then as positive semi-definite operators:
\[
\F \leq \frac{n-1}{n} \frac{m+1}{m} \F_0 . 
\]
In particular, by (\ref{eq:trace-id-F0}):
\[
\tr( \F (\Id/2 + \k \otimes \k)) \leq \frac{n-1}{n} \frac{m+1}{m} \HH^{n-1}(\Sigma),
\]
and the trace-inequality (\ref{eq:trace-inequality}) holds up to a factor of $\frac{n-1}{n} \frac{m+1}{m}$. 
\end{theorem}
This means that we can obtain the PDE (\ref{eq:PDI}) up to a multiplicative factor of $\frac{n-1}{n} \frac{m+1}{m}$. 
In general, Theorem \ref{thm:intro-structure} and Remark \ref{rem:symmetry} only imply the $\S^m$-symmetry of a minimizing cluster for $m = n+1-q$, and so unless $q=2$ (the single-bubble case), this factor will be strictly larger than $1$. Nevertheless, the resulting non-sharp PDE will yield some non-trivial information on the isoperimetric profile $\I^{(q-1)}$. 

\medskip

For completeness, let us also mention some additional information we have on the spherical multi-bubble isoperimetric profile $\I^{(q-1)}_{\S^n} : \Delta^{(q-1)} \rightarrow \R_+$:
\begin{itemize}
\item $\I^{(q-1)}_{\S^n}$ is strictly concave by Theorem \ref{thm:intro-profile-concave}. 
\item As mentioned in the Introduction, the multi-bubble conjecture on $\S^n$ in the \emph{equal-volume case} holds true for all $q \leq n+2$ by the Gaussian multi-bubble theorem on $\GG^{n+1}$ from \cite{EMilmanNeeman-GaussianMultiBubble}:
\[
\I^{(q-1)}_{\S^n}(1/q,\ldots,1/q) = \I^{(q-1)}_{\S^n,m}(1/q,\ldots,1/q) \brac{ = I^{(q-1)}_{\GG^{n+1},m}(1/q,\ldots,1/q) }. 
\]
\item The same argument actually shows that: 
\[
\I^{(q-1)}_{\S^n}(v) \geq \I_{\GG^{n+1},m}^{(q-1)}(v) \;\; \; \forall v \in \Delta^{(q-1)} . 
\]
Indeed, the cone over a given cluster $\Omega$ on $\S^n$ with apex at the origin is a cluster on $\R^{n+1}$ having the same Gaussian volume and Gaussian total-perimeter as the volume and total-perimeter of $\Omega$ with respect to the Haar probability measure on $\S^n$, yielding the asserted inequality. In the equal volumes case we get an equality, since the Gaussian minimizer is a simplicial cluster centered at the origin, whose intersection with $\S^n$ is exactly an equal-volume standard-bubble. 
\end{itemize}

\subsection{PDE for the multi-bubble isoperimetric profile on $\R^n$} \label{subsec:degenerate-elliptic}

By scaling the metric on $\S^n$, it is immediate to check that the model $(q-1)$-bubble isoperimetric profile $\I_{K,m}$ on the rescaled sphere $\sqrt{\frac{n-1}{K}} \S^n$ of Ricci curvature $K > 0$ satisfies the following (strictly) elliptic PDE:
\[ \tr\brac{(-\nabla^2 \I_{K,m})^{-1} \brac{\frac{K}{2} \Id+ \frac{1}{n-1} \nabla \I_{K,m} \otimes \nabla \I_{K,m}}} = \I_{K,m}  .
\] Taking the limit as the Ricci curvature $K$ tends to $0$, and using that the model Euclidean $(q-1)$-bubble isoperimetric profile $\I_{0,m} := \I^{(q-1)}_{\R^n,m}$ on $\R^n$ is homogeneous under scaling, namely $\I_{0,m}(\lambda v) = \lambda^{\frac{n-1}{n}} \I_{0,m}(v)$ for all $\lambda > 0$, one may check that $\I_{0,m}$ must satisfy the following PDE on $\R^{q-1}_+$ (which indeed respects the latter homogeneity):
\begin{equation} \label{eq:degenerate-PDE}
\tr((-\nabla^2 \I_{0,m})^{-1} \nabla \I_{0,m} \otimes \nabla \I_{0,m})) = (n-1) \I_{0,m} . 
\end{equation}

Unfortunately, (\ref{eq:degenerate-PDE}) is only degenerate elliptic, and so there is no uniqueness for the corresponding Dirichlet problem. 
For this reason, we do not know how to obtain uniqueness of minimizers in the Euclidean case, and so we approach this case by taking a limit from the strictly elliptic spherical case (forfeiting uniqueness in the process).  

An identical scaling argument applies to the actual isoperimetric profiles $\I^{(q-1)}_{\R^n}$ and $\I^{(q-1)}_{\S^n}$, yielding for all $v \in \R^{q-1}_+$ that:
\[
\I^{(q-1)}_{\R^n}(v) = \lim_{\eps \rightarrow 0+} \eps^{\frac{1-n}{n}} \I_{\S^n}^{(q-1)}(\eps v , 1-\eps  (v_1+\ldots+v_{q-1})) 
\]
(where we pad the last coordinate to obtain a vector in $\Delta^{(q-1)}$). This explains why resolving the multi-bubble isoperimetric problem on $\S^n$, namely establishing $\I_{\S^n}^{(q-1)} \equiv \I_{\S^n,m}^{(q-1)}$, implies the corresponding resolution in $\R^n$, namely $\I^{(q-1)}_{\R^n} \equiv \I^{(q-1)}_{\R^n,m}$ (but without the equality cases), as well as why the concavity of the isoperimetric profile on $\S^n$ is inherited by the one on $\R^n$.

\bibliographystyle{plain}
\bibliography{../../../ConvexBib}

\end{document}